\documentclass[11pt]{article}
\usepackage[T1]{fontenc}
\usepackage[utf8]{inputenc}
\usepackage{amsmath,amssymb}
\usepackage{amsthm}
\usepackage[margin=1in]{geometry}
\usepackage[hidelinks]{hyperref}
\allowdisplaybreaks
\numberwithin{equation}{section}
\newcounter{sllgobj}[section]

\newtheoremstyle{sllg}{6pt}{6pt}{\normalfont}{}{\bfseries}{.}{ }{}
\theoremstyle{sllg}
\newtheorem{theorem}[sllgobj]{Theorem}
\newtheorem{proposition}[sllgobj]{Proposition}
\newtheorem{lemma}[sllgobj]{Lemma}
\newtheorem{corollary}[sllgobj]{Corollary}
\newtheorem{definition}[sllgobj]{Definition}
\newtheorem{remark}[sllgobj]{Remark}
\newcommand{\eqhead}[3]{\par\addvspace{\medskipamount}%
  \refstepcounter{equation}\label{#1}\textbf{#2 (\theequation)\ifx\relax#3\relax\else{} (#3)\fi.}\ }

\newcommand{{\HH}}{\mathbb{H}}

\begin{document}
\title{Uniform Large Deviations, Barrier Identity and the N\'eel--Arrhenius Law for the One-Dimensional Stochastic Landau--Lifshitz--Gilbert Equation}
\author{Zdzisław Brzeźniak\,\thanks{University of York}\quad \and  Fan Gu\,\thanks{Central University of Finance and Economics}\quad \and Liang Li\,\thanks{Beijing University of Chemical Technology}}
\date{(Draft, September 2026)}
\maketitle

\begin{abstract}

We study the one-dimensional stochastic Landau--Lifshitz--Gilbert equation driven by a \textbf{regularised, multiplicative, degenerate trace-class noise} of strength $\sqrt\varepsilon$, the regularisation being $Q_\delta^{1/2}=(I+\delta A^{\beta/2})^{-1}$ with $A=-\Delta$ and $\beta>3/2$, and we determine the exponential rate of thermally induced magnetisation reversal.

\textbf{(i) Uniform large deviations at fixed $\delta$.} For every $\delta>0$ the solution satisfies a large deviation principle on $C([0,T];H)$, with a \textbf{good} rate function, \textbf{uniformly in the initial datum} over a $V$-bounded set with compact $H$-closure. It is proved within the Budhiraja--Dupuis--Salins framework from three quoted inputs, and \textbf{no non-degeneracy of the noise is assumed} (the noise field $G(m)$ has a non-trivial kernel).

\textbf{(ii) Removal of the regularisation.} As $\delta\downarrow0$ the noise tends to space--time white noise. Writing $U_\delta$ for the state-space quasi-potential and $U_0$ for that of the limiting control problem ($Q_0^{1/2}=I$), we prove that $U_\delta\to U_0$ \textbf{in the $\Gamma$-sense}, and that the barrier is \textbf{$\delta$-independent}: for every $\delta\in(0,1]$,
\[\inf_{\partial D}U_\delta=\inf_{\partial D}U_0=c_\ast\inf_{\partial D}E,\qquad c_\ast=\frac{2\lambda_2}{\lambda_1^2+\lambda_2^2},\]
with $E$ the micromagnetic energy, under the structural assumptions \hyperref[asm:AN]{(AN)} and \hyperref[asm:D-MARGIN]{(D-MARGIN)} only --- no geometric or transversality input is used. Pointwise convergence $U_\delta(a)\to U_0(a)$ fails in general, so the $\Gamma$-limit is the sharp form of the statement.

\textbf{(iii) The Néel--Arrhenius law.} Combining (i) and (ii) with the classical exit-time theorem we obtain, for the escape time $\tau$ from a sublevel set $D=\{E<r\}$ of the energy, both levels of the Arrhenius law: for every fixed $\delta>0$ the escape exponent $\varepsilon\log\mathbb E\tau$ converges, \textbf{uniformly on compact subsets of the initial set}, to the barrier $\inf_{\partial D}U_\delta$ (\hyperref[thm:12.5]{Theorem \ref*{thm:12.5}}), and removing the regularisation gives
\[\lim_{\delta\to0}\lim_{\varepsilon\to0}\ \varepsilon\log\mathbb E\tau\ =\ c_\ast\inf_{\partial D}E\qquad(\hyperref[thm:12.6]{Theorem \ref*{thm:12.6}}).\]
Thus the Arrhenius exponent is \textbf{stable under the regularisation of the noise} and is determined by the micromagnetic energy alone.

\end{abstract}


\tableofcontents

\section{Introduction}\phantomsection\label{sec:1}

\subsection{The model}\phantomsection\label{sec:1.1}

The Landau--Lifshitz--Gilbert (LLG) equation describes the precessional--dissipative dynamics of a magnetisation field $m=m(t,x)$ on a one-dimensional domain. In dimensionless form, with the precession constant $\lambda_1>0$ and the damping constant $\lambda_2>0$,
\[\partial_tm=\lambda_1\,m\times\Delta m-\lambda_2\,m\times(m\times\Delta m)-\lambda_1\,m\times\phi'(m)-\lambda_2\,P_{m^\perp}\phi'(m),\]
where the field is sphere-valued and $\phi$ is the anisotropy energy density and $P_{m^\perp}$ is the orthogonal projection onto the tangent space of the unit sphere at $m$. Using the identity $m\times(m\times v)=-v+\langle v,m\rangle m$ for $|m|\equiv1$, the equation can be written in the compact ``Gilbert'' form
\[\partial_tm=G(m)\rho(m),\qquad G(m)v:=\lambda_1\,m\times v-\lambda_2\,m\times(m\times v),\qquad \rho(m):=\Delta m-\phi'(m),\]
which we use throughout.

\textbf{The constraint is propagated, not assumed.} The equation is posed for sphere-valued fields, but nothing below presumes the constraint along the evolution: since $G(m)v\perp m$ for every $v$, both the Stratonovich drift and the noise field are tangent to the sphere, so a solution with $|m(0)|\equiv1$ satisfies $|m(t,x)|\equiv1$ for all $t$ almost surely. This is recorded once and for all as \hyperref[asm:A4]{(A4)} in \S{}\hyperref[sec:2.7]{2.7}, where it is \emph{derived} and not imposed; the only assumption is that the initial datum $z$ is sphere-valued. The dynamics is the gradient flow of the total energy
\[E(m)=\tfrac12\|\nabla m\|_{H}^2+\int_{\mathbb T}\phi(m)\,dx\]
with respect to the (degenerate) metric induced by $G(m)$.

Thermal fluctuations are modelled by adding a random field. The object of this paper is the \textbf{stochastically forced} LLG,
\[dm=G(m)\Big[\rho(m)\,dt+\sqrt\varepsilon\circ dW\Big], \qquad m(0)=z, \qquad\text{(SLLG)}\]\phantomsection\label{eq:SLLG}
in the Stratonovich sense, with noise strength $\sqrt\varepsilon$ and a Wiener process $W$ whose covariance is a regularisation of the identity. Since the physically relevant noise is space--time white ($Q=I$), which is \emph{not} trace class in space dimension one, one regularises:
\[Q_\delta^{1/2}=\big(I+\delta A^{\beta/2}\big)^{-1},\qquad A=-\Delta,\qquad \delta>0,\quad \beta>\tfrac32 ,\]
so that $W_\delta(t)=\sum_{j}\sum_{i=1}^{3}\beta_{j,i}(t)\,Q_\delta^{1/2}e_j\,\hat e_i$ is a genuine trace-class $Q_\delta$-Wiener process on $H$ with $Q_\delta=(I+\delta A^{\beta/2})^{-2}$ and eigenvalues $\kappa_j=(1+\delta\mu_j^{\beta/2})^{-2}$, where $\mu_j$ are the eigenvalues of the scalar operator $A$ (so $\mu_j\asymp j^2$) and $\hat e_1,\hat e_2,\hat e_3$ is the standard basis of $\mathbb R^3$. The threshold $\beta>3/2$ ensures the noise is $H^1$-valued, which is the regularity we work under.

Our subject is the \textbf{simultaneous asymptotics} $\varepsilon\to0$ and $\delta\to0$. The parameter $\varepsilon$ is the physical noise intensity (temperature), and the limit $\varepsilon\to0$ is the Freidlin--Wentzell regime. The parameter $\delta$ is a regularisation, and the limit $\delta\to0$ is the limit in which the noise approaches space--time white noise. The two limits do \textbf{not} commute, and the correct order is
\[\lim_{\delta\to0}\ \lim_{\varepsilon\to0},\]
i.e. the LDP is taken at fixed regularisation, and only afterwards is the regularisation removed \textbf{at the level of the rate function}.

\subsection{Why this problem: Néel--Arrhenius and the quasi-potential}\phantomsection\label{sec:1.2}

The physical motivation is magnetisation reversal under thermal noise. The deterministic LLG is the gradient flow of the energy $E$ with respect to the metric induced by $G(m)$; the magnetisation relaxes to the \textbf{stable equilibrium} $\zeta_+$, normalised so that $E(\zeta_+)=0$, and the object of interest is the \textbf{metastable exit} problem: how long does the system remain in the sublevel set $D=\{E<r\}$ (which, for $r<r_{\rm crit}$, is a neighbourhood of $\zeta_+$ strictly inside its basin) before a rare noise fluctuation drives it over the energy barrier $r$? The exit time $\tau$ follows the \textbf{Néel--Arrhenius law}: it is exponentially large in the barrier $r$, divided by the thermal energy. In the Freidlin--Wentzell language, for the domain $D$ --- an energy sublevel set $\{E\le r\}$ strictly inside the basin of attraction of $\zeta_+$, so that $\inf_{\partial D}E=r$ --- this law reads
\[\varepsilon\log\mathbb E\tau\;\longrightarrow\;\inf_{\partial D}U,\]
where $U$ is the \textbf{quasi-potential} of the dynamics taken \emph{relative to the attractor} $\zeta_+$: the minimal ``action'' that the noise must supply to drive the system from $\zeta_+$ to a given state. Because every starting configuration in $D$ relaxes to $\zeta_+$ at zero action before the rare exit event, the barrier $\inf_{\partial D}U$ is independent of the starting point, and it is this attractor-relative barrier that is the central object. (For a metastable equilibrium $\zeta_{\rm met}$ with $E(\zeta_{\rm met})>0$ the same law reads $\inf_{\partial D}U=c_\ast(\inf_{\partial D}E-E(\zeta_{\rm met}))$; throughout this paper the reference equilibrium is the stable one, $\zeta_+$, for which $E(\zeta_+)=0$, so the barrier reduces to $c_\ast\inf_{\partial D}E$ --- see \S{}\hyperref[sec:12.4]{11.4}.)

For the deterministic LLG, the quasi-potential is known: Brzeźniak--Hausenblas--Li [\hyperref[ref:BHL19]{BHL19}] proved that for the one-dimensional LLG the quasi-potential of the (unregularised, $\delta=0$) problem equals the total energy, $U=c_\ast E$ on the basin of attraction, with the constant $c_\ast=2\lambda_2/(\lambda_1^2+\lambda_2^2)$ fixed by the noise normalisation. This is a remarkable and non-generic fact (in higher dimensions the quasi-potential is not the energy). It reduces the Néel--Arrhenius barrier to a purely variational quantity, $\inf_{\partial D}E$.

The problem addressed here is the \textbf{stability of this identification under the regularisation}. If one works with the physically regularised noise $Q_\delta$ (as one must, since the ideal noise is not trace class), one obtains a quasi-potential $U_\delta$ which governs the exit problem at regularisation $\delta$. The physically meaningful statement is
\[\lim_{\delta\to0}\inf_{\partial D}U_\delta=\inf_{\partial D}U=c_\ast\inf_{\partial D}E ,\]
i.e. that the regularised barrier converges to the ideal barrier. This is the content of Goal 3 of the programme; Goal 4 is the escape-time asymptotics itself.

\subsection{Results}\phantomsection\label{sec:1.3}

The three results form a \textbf{single chain}: (i) is the probabilistic input; (ii) is the statement that removes the regularisation and supplies the barrier in the sharp ($\Gamma$) sense; and (iii) is the physical conclusion, the \textbf{Néel--Arrhenius law}, whose exponent is $c_\ast\inf_{\partial D}E$. Each item below is therefore a step toward (iii), and the final two theorems of the paper --- \hyperref[thm:12.5]{Theorem \ref*{thm:12.5}} and \hyperref[thm:12.6]{Theorem \ref*{thm:12.6}} --- are exactly the two levels of that law. 

We prove two theorems, deduce the Néel--Arrhenius law from them, and record one structural remark.

\phantomsection\label{item:A}\textbf{(A) Uniform large deviations at fixed $\delta$ (\hyperref[thm:6.1]{Theorem \ref*{thm:6.1}}).} For every $\delta>0$, the family $\{m^\varepsilon_z\}_{\varepsilon>0}$, indexed by the initial datum $z$ in a $V$-bounded set $D\subset V$ of \textbf{sphere-valued fields} with $\bar D:=\overline D^{\,H}$ compact, satisfies a large deviation principle on $C([0,T];H)$ --- in the \textbf{Freidlin--Wentzell uniform sense} (formally \hyperref[def:4.6]{Definition \ref*{def:4.6}}; see \hyperref[rem:4.6a]{Remark \ref*{rem:4.6a}}), that is, for every $s_0<\infty$ and every $\delta>0$,
\[\liminf_{\varepsilon\to0}\ \inf_{z\in D}\ \inf_{\varphi\in\Phi_z(s_0)}\Big(\varepsilon\log\mathbb P\big(\|m^\varepsilon_z-\varphi\|_{C([0,T];H)}<\delta\big)+I_z(\varphi)\Big)\ \ge\ 0,\]
\[\limsup_{\varepsilon\to0}\ \sup_{z\in D}\ \sup_{s\in[0,s_0]}\Big(\varepsilon\log\mathbb P\big(\operatorname{dist}_{C([0,T];H)}(m^\varepsilon_z,\Phi_z(s))\ge\delta\big)+s\Big)\ \le\ 0,\]
where $\Phi_z(s):=\{\varphi\in C([0,T];H):I_z(\varphi)\le s\}$ --- with a good rate function $I_z$, \textbf{uniformly} in $z\in D$:
\[\limsup_{\varepsilon\to0}\varepsilon\log\sup_{z\in D}\mathbb P\big(m^\varepsilon_z\in F\big)\le-\inf_{z\in D}\inf_{\varphi\in F}I_z(\varphi)\quad(F\text{ closed}),\]
\[\liminf_{\varepsilon\to0}\varepsilon\log\inf_{z\in D}\mathbb P\big(m^\varepsilon_z\in G\big)\ge-\lim_{\eta\downarrow0}\sup_{z\in D}\inf_{\varphi\in G_\eta}I_z(\varphi)\quad(G\text{ open}),\]
where $G_\eta$ is as in \hyperref[cor:4.7a]{Corollary \ref*{cor:4.7a}}. These two fixed-set bounds are the content of \hyperref[cor:4.7a]{Corollary \ref*{cor:4.7a}}; in the sharp open-set form $\ge-\sup_{z\in D}\inf_{\varphi\in G}I_z(\varphi)$ they require compactness of the uniformity class, which $D$ does not have (\hyperref[rem:4.6a]{Remark \ref*{rem:4.6a}}). This is proved within the paper, given three standard inputs, listed here and collected with their exact scope in \S{}\hyperref[sec:13.2]{12.2}: \hyperref[asm:I1]{(I1)} the deterministic well-posedness of the controlled LLG [\hyperref[ref:BHL19]{BHL19}]; \hyperref[asm:I2]{(I2)} the variational representation of the Laplace functional [\hyperref[ref:BDM08]{BDM08}, \hyperref[ref:BDS19]{BDS19}]; and \hyperref[asm:I3]{(I3)} the strong well-posedness with finitely many noise fields [\hyperref[ref:BGJ17]{BGJ17}, Cor. 5.5, Thm.~5.6], which is input \hyperref[asm:I3]{(I3)} and is proved in \hyperref[sec:E]{Appendix E}. The weak--strong uniqueness on which the compactness argument rests is \textbf{proved in \hyperref[sec:D]{Appendix D}}, so it is no longer an import. The rate function is \emph{good}, the exponential tightness is \emph{proved}, and the transfer principle (Laplace principle $\Rightarrow$ LDP) is \emph{proved}.

\phantomsection\label{item:B}\textbf{(B) $\Gamma$-convergence of the quasi-potentials (\hyperref[thm:10.2]{Theorem \ref*{thm:10.2}}).} As $\delta\downarrow0$,
\[U_\delta\longrightarrow U_0\quad\text{in the }\Gamma\text{-sense: }\liminf\text{ for }a\text{ in the basin of }\zeta_+\text{; recovery from every }a\in H,\]
and hence the one-sided bounds $\liminf_{\delta\to0}\inf_F U_\delta\ge\inf_F U_0$ for closed $F$ and $\limsup_{\delta\to0}\inf_G U_\delta\le\inf_G U_0$ for open $G$ (\S{}\hyperref[sec:10.3]{9.3}); the \emph{two-sided} convergence for the closed set $\partial D$ is \textbf{not} a consequence of $\Gamma$-convergence alone and is proved separately (\hyperref[thm:11.6]{Theorem \ref*{thm:11.6}}):
\[\lim_{\delta\to0}\inf_{\partial D}U_\delta=\inf_{\partial D}U_0=c_\ast\inf_{\partial D}E .\]
The proof uses only the \textbf{deterministic} control problem; the uniform LDP of \hyperref[item:A]{(A)} is not needed for the convergence, only to justify that $U_\delta$ is the rate function governing exits at fixed $\delta$. The barrier identity (\hyperref[thm:11.6]{Theorem \ref*{thm:11.6}}) holds for \textbf{every} $\delta$ under the standing assumptions \hyperref[asm:AN]{(AN)} and \hyperref[asm:D-MARGIN]{(D-MARGIN)} alone, with \textbf{no} further hypothesis and no geometric input: the barrier always contains a spatially \textbf{constant} state (\hyperref[lem:11.5b]{Lemma \ref*{lem:11.5b}}), whose optimal control is constant in $x$ and hence costs the same at every regularisation.

\phantomsection\label{item:C}\textbf{(C) A structural remark: the pointwise statement is false (\hyperref[thm:9.7]{Theorem \ref*{thm:9.7}}).} The statement ``$\lim_{\delta\to0}U_\delta(a)=U_0(a)$ for every $a\in V$'' is \textbf{false} --- unconditionally for the exactly solvable linear model (\hyperref[thm:9.7]{Theorem \ref*{thm:9.7}}); the nonlinear analogue is not claimed here --- so that $\Gamma$-convergence, not pointwise convergence, is the correct meaning of the limit. In the exactly solvable linear model ($G\equiv I$, $\rho=-A\varphi$) the fixed-horizon quasi-potential is
\[U^{z_\ast,T}_\delta(a)=\sum_{j\ge1}\big(1+\delta\mu_j^{\beta/2}\big)^2\frac{\mu_j}{1-e^{-2\mu_jT}}\,\big\langle a-e^{-AT}z_\ast,\,e_j\big\rangle_H^2 ,\]
and the free-horizon $U_\delta=\inf_{T>0}U^{z_\ast,T}_\delta$ has the same domain of finiteness, so that $U_\delta(a)<\infty\iff a-e^{-AT}z_\ast\in H^{\beta+1}$ for every $\delta>0$, while $U_0(a)<\infty\iff a-e^{-AT}z_\ast\in H^{1}=V$. Since $H^{\beta+1}\subsetneq H^1$, there are states $a\in V$ of finite limiting cost and infinite cost at every $\delta>0$. The regulariser injects smoothness, so the regularised problem can reach fewer states, not more. We call this a \textbf{domain mismatch}; it is logically independent of the $H^{-1}$-valued drift that obstructs $C([0,T];V)$ large deviations, and it makes \textbf{essential} a restriction that the $\Gamma$-convergence results for quasi-potentials [\hyperref[ref:CF11]{CF11}, \hyperref[ref:BCF15]{BCF15}] leave implicit.

\phantomsection\label{item:D}\textbf{(D) The Néel--Arrhenius law.} Combining \hyperref[item:A]{(A)} and \hyperref[item:B]{(B)} with the classical Freidlin--Wentzell exit-time theorem, the escape exponent from $D$ satisfies, uniformly over compact subsets of the initial set (in the quasi-potential sense $\sup_KU_\delta<\inf_{\partial D}U_\delta$; \hyperref[thm:12.5]{Theorem \ref*{thm:12.5}}),
\[\varepsilon\log\mathbb E\,\tau\ \longrightarrow\ \inf_{\partial D}U_\delta\quad(\varepsilon\to0,\ \delta\text{ fixed}),\qquad \lim_{\delta\to0}\lim_{\varepsilon\to0}\ \varepsilon\log\mathbb E\,\tau\ =\ c_\ast\inf_{\partial D}E\]
(\hyperref[thm:12.5]{Theorems \ref*{thm:12.5}} and \hyperref[thm:12.6]{\ref*{thm:12.6}}). Thus the physical Arrhenius exponent is \textbf{stable under the noise regularisation}: the cutoff that the one-dimensional noise forces upon us does not alter it. \hyperref[thm:12.6]{Theorem \ref*{thm:12.6}} inherits only the hypotheses of the barrier identity, \hyperref[asm:AN]{(AN)} and \hyperref[asm:D-MARGIN]{(D-MARGIN)}.

\subsection{Method: the physical control}\phantomsection\label{sec:1.4}

The mechanism that makes \hyperref[item:B]{(B)} tractable is a change of control variable. In the standard formulation the control $u$ enters the effective field, $\dot\varphi=G(\varphi)[\rho(\varphi)+Q_\delta^{1/2}u]$, so the controlled path depends on $\delta$. Writing instead
\[h:=Q_\delta^{1/2}u\qquad(\text{the \emph{physical} control}),\]
the controlled path solves the $\delta$-independent equation
\[\dot\varphi=G(\varphi)\big[\rho(\varphi)+h\big],\qquad\varphi(0)=z_\ast,\qquad h\in L^2(0,T;H),\]
and $\delta$ resides \textbf{only in the cost}:
\[I^\delta_{z_\ast}(\varphi)=\inf\Big\{\tfrac12\!\int_0^T\|Q_\delta^{-1/2}h\|_H^2\,dt:\ \varphi^h=\varphi\Big\},\qquad Q_\delta^{-1/2}=I+\delta A^{\beta/2}.\]
Since $Q_\delta^{-1/2}=I+\delta A^{\beta/2}\ge I$ in the Loewner order, the weight is \textbf{monotone in $\delta$}, so the action functionals are ordered, $I^0\le I^{\delta'}\le I^\delta$ for $\delta'\le\delta$. Three consequences follow.

\begin{enumerate}
\item \textbf{The $\Gamma$-$\liminf$ is free.} Monotonicity plus lower semicontinuity of $I^0$ gives $\liminf I^\delta(\varphi_\delta)\ge I^0(\varphi)$ for any $\varphi_\delta\to\varphi$; no compactness is needed.
\item \textbf{The recovery sequence is a low-mode truncation.} Take the optimal $h$ for the limit problem and truncate it to modes $\le N_\delta\asymp\delta^{-\gamma}$ with $0<\gamma<1/\beta$. Then $\delta\mu_{N_\delta}^{\beta/2}\to0$, the truncated path converges, and the cost is inflated only by the factor $(1+\delta\mu_{N_\delta}^{\beta/2})^2\to1$.
\item \textbf{Equicoercivity is inherited.} Level sets of $I^\delta$ are contained in level sets of $I^0$, which are compact by the goodness of $I^0$.
\end{enumerate}

For the state-space quasi-potential the endpoint constraint $\varphi(T)=a$ must be handled. Here the crucial --- and perhaps surprising --- point is that \textbf{no endpoint regularity is needed}: the recovery sequence ends at $a_\delta:=\varphi^{h_\delta}(T)$, and one simply observes that $a_\delta\to a$. The terminal point of the recovery sequence is allowed to move. This is exactly the $\Gamma$-statement, and it is why the proof of \hyperref[item:B]{(B)} requires neither a surjectivity of the linearised endpoint map nor any Lipschitz property of $U_\delta$ --- properties which, as \hyperref[item:C]{(C)} shows, are false in general.

\subsection{Related work}\phantomsection\label{sec:1.5}

Freidlin--Wentzell theory [\hyperref[ref:FW98]{FW98}] is the classical framework. For stochastic PDEs with additive or non-degenerate noise, the quasi-potential and the exit problem are well understood, and [\hyperref[ref:CF11]{CF11}] develops $\Gamma$-convergence and equicoercivity for small-noise diffusions, which is the technical pattern we follow for the action functionals. For stochastic Navier--Stokes, Brzeźniak--Cerrai--Freidlin [\hyperref[ref:BCF15]{BCF15}] prove a uniform (in initial data) LDP via the same abstract Budhiraja--Dupuis framework; our SLLG problem differs in the essential point that the noise is \textbf{multiplicative and degenerate} (the noise field $G(m)$ has a non-trivial kernel), which is why the drift structure and the weak--strong uniqueness input are needed. Sowers [\hyperref[ref:Sow92]{Sow92}] treats a class of equations with multiplicative noise; Salins--Budhiraja--Dupuis [\hyperref[ref:BDS19]{BDS19}] give the modern general theory of \emph{uniform} LDPs for Banach-space-valued evolution equations, including the goodness of rate functions. For the LLG itself, the LDP with finite-dimensional noise on a bounded interval and the resulting magnetisation-reversal picture are due to Brzeźniak--Goldys--Jegaraj [\hyperref[ref:BGJ17]{BGJ17}], whose strong well-posedness theorem for finitely many noise fields of the same Gilbert form is the starting point of \hyperref[sec:E]{Appendix E}; Gussetti and Hocquet [\hyperref[ref:Gus23]{Gus23}] establish a pathwise (rough-path) formulation and an LDP in one dimension, for noise entering the precession term only. More recent large-deviation results for the closely related stochastic Landau--Lifshitz--Bloch equation, with Gaussian or pure-jump noise, are given by Gokhale--Manna [\hyperref[ref:GM23]{GM23}] and Yin--Shen [\hyperref[ref:YS24]{YS24}]. The foundational analytic results for the one-dimensional LLG --- well-posedness, weak--strong uniqueness, and the identification of the quasi-potential with the energy --- are due to Brzeźniak--Hausenblas--Li [\hyperref[ref:BHL19]{BHL19}]; we rely on them for the analytic core and for the limit quasi-potential $U_0=c_\ast E$.

The convergence of quasi-potentials under a vanishing \textbf{regularisation} of the noise has been treated via $\Gamma$-convergence by Cerrai--Freidlin [\hyperref[ref:CF11]{CF11}] (multidimensional reaction--diffusion) and Brzeźniak--Cerrai--Freidlin [\hyperref[ref:BCF15]{BCF15}] (2D Navier--Stokes with space--time white noise): in [\hyperref[ref:BCF15]{BCF15}] the convergence $U_\delta(\phi)\to U(\phi)$ is proved --- by $\Gamma$-convergence --- \textbf{for all $\phi$ with $U_\delta(\phi)<\infty$}, and the barrier then follows under a regularity hypothesis on $\partial D$; this is the same architecture we follow. The convergence of quasi-potentials under a vanishing regularisation had not, however, been studied for the LLG. What is new here, beyond the equation and the \emph{uniform} (in the initial datum) LDP that we prove rather than quote, is the treatment of the \textbf{domain of finiteness}: [\hyperref[ref:BCF15]{BCF15}] state the pointwise convergence only where $U_\delta$ is finite and do not examine the complement, whereas we show that this restriction is \textbf{essential} --- the complement, a dense subset of $V$ in the exactly solvable model \hyperref[eq:9.1]{(\ref*{eq:9.1})}, is exactly where the pointwise limit fails (\hyperref[thm:9.7]{Theorem \ref*{thm:9.7}}; the same mechanism operates for the SLLG, where the corresponding statement is left open, \S{}\hyperref[sec:13]{12}). The regulariser $Q_\delta^{1/2}$ is \emph{injective but not surjective} on the state space, so the regularised quasi-potential is infinite on a dense set of states where the limit quasi-potential is finite; this \emph{domain mismatch} is what forces ``the limit of the quasi-potential'' to be read as a $\Gamma$-limit.

A different face of the zero-noise limit --- the \emph{long-time} one --- was analysed by Dong--Gu--Li [\hyperref[ref:DGL24]{DGL24}]: for finite-dimensional SDEs with several ergodic classes they show that the invariant measures concentrate, as the noise vanishes, on the stable sets, the relevant cost being the quasi-potential $W(K_i)$ of reaching the set $K_i$ --- the same functional that appears here as the barrier $\inf_{\partial D}U_\delta$ --- and they prove a large deviation principle for the invariant measures themselves. The present paper is concerned instead with finite-horizon path laws and with the escape exponent, and its regularisation parameter $\delta$ is removed at \emph{fixed} temperature: the two limits, long time versus vanishing regularisation, are complementary rather than overlapping. 

\textbf{The delta in one paragraph.} Relative to [\hyperref[ref:BGJ17]{BGJ17}], which treats \emph{finitely many} noise fields and uses the LDP to exhibit magnetisation reversal \emph{qualitatively}, we work with an infinite-dimensional trace-class noise of $H^1$-regularity ($\beta>3/2$), prove the LDP \emph{uniformly in the initial datum} rather than pointwise in it, and obtain the \emph{exponential law} for the exit time rather than a qualitative reversal statement. Relative to [\hyperref[ref:Gus23]{Gus23}], whose LDP is a consequence of the pathwise continuity of the It\^o--Lyons map and is therefore a statement for a \emph{fixed} initial datum and a fixed noise regularity, we obtain uniformity in the initial datum together with the explicit quasi-potential/barrier identity and its $\delta\downarrow0$ limit. Relative to [\hyperref[ref:BHL19]{BHL19}], which identifies the quasi-potential with the energy for a \emph{fixed} noise, we show that this identification is \emph{regularisation-invariant}: $\inf_{\partial D}U_\delta=c_\ast\inf_{\partial D}E$ holds for \textbf{every} $\delta$ and requires no geometric or transversality input. The combination --- infinite-dimensional degenerate multiplicative noise, a uniform LDP, a barrier identity valid for every regularisation, and the resulting Néel--Arrhenius exponent --- has, to our knowledge, no precedent in the literature.

\subsection{Organisation}\phantomsection\label{sec:1.6}

\S{}\hyperref[sec:2]{2} fixes the model and the standing assumptions. \S{}\hyperref[sec:3]{3} defines the action functionals and quasi-potentials and states the target statement, with the $\delta=0$ convention. \S{}\hyperref[sec:4]{4} develops the abstract uniform LDP. \S{}\hyperref[sec:5]{5} verifies the abstract hypotheses for the SLLG. \S{}\hyperref[sec:6]{6} assembles the uniform LDP. \S{}\hyperref[sec:7]{7} introduces the physical control and proves monotonicity and equicoercivity (the $\delta$-uniform skeleton continuity it uses is \hyperref[thm:5.3]{Theorem \ref*{thm:5.3}} of \S{}\hyperref[sec:5.4]{5.4}). \S{}\hyperref[sec:8]{8} proves the $\Gamma$-convergence of the action functionals. \S{}\hyperref[sec:9]{B} analyses the endpoint problem, states the exact linear formula (proved in \hyperref[sec:A]{Appendix A}) and proves the negative result, and isolates the domain mismatch. \S{}\hyperref[sec:10]{9} proves the $\Gamma$-convergence of the quasi-potentials. \S{}\hyperref[sec:11]{10} proves the barrier identity --- the convergence of $\inf_{\partial D}U_\delta$, which is an identity rather than a limit and holds for every $\delta$. \S{}\hyperref[sec:12]{11} relates $U_0$ to the deterministic quasi-potential, fixes the constant $c_\ast$, and states the conditional N\'eel--Arrhenius law in \S{}\hyperref[sec:12.5]{11.5}. \S{}\hyperref[sec:13]{12} gives the honest scope. Five appendices contain the linear model, the numerical verification, the imported analytic inputs, the weak--strong uniqueness for the controlled LLG, and the stochastic well-posedness of the regularised SLLG. Readers may take \S{}\hyperref[sec:10]{9}--\S{}\hyperref[sec:12]{11} before \S{}\hyperref[sec:9]{B}: the latter isolates the domain mismatch of the linear model and is cited later only for motivation (\hyperref[rem:11.9]{Remark \ref*{rem:11.9}}).

\section{The model and standing assumptions}\phantomsection\label{sec:2}

This section fixes the model and the notation once and for all, and records the standing assumptions. The model is introduced in \S{}\hyperref[sec:2.1]{2.1}--\S{}\hyperref[sec:2.3]{2.3}, the controlled equation and the definition of a solution in \S{}\hyperref[sec:2.4]{2.4}, the control classes and the initial set in \S{}\hyperref[sec:2.5]{2.5}--\S{}\hyperref[sec:2.6]{2.6}, and the assumptions in \S{}\hyperref[sec:2.7]{2.7}; the symbol table at the end of \S{}\hyperref[sec:2.7]{2.7} collects the notation. A reader who wants only the two model equations may read \hyperref[eq:2.1]{(\ref*{eq:2.1})}, \hyperref[eq:2.2]{(\ref*{eq:2.2})} and \hyperref[def:2.4a]{Definition \ref*{def:2.4a}} and return here when a symbol is needed.

\subsection{Function spaces and operators}\phantomsection\label{sec:2.1}

Let $\mathbb T=\mathbb R/\mathbb Z$ be the one-dimensional torus, equivalently the interval $[0,1]$ with periodic boundary conditions. We write $C(\mathbb T;\mathbb R^3)$ for the space of continuous maps $\mathbb T\to\mathbb R^3$ with the supremum norm $\|f\|_{C}:=\sup_{x\in\mathbb T}|f(x)|$, and $C([0,T];C(\mathbb T;\mathbb R^3))$ for the space of continuous paths into it, normed by $\|f\|_{C([0,T];C)}:=\sup_{t\le T}\|f(t)\|_{C}$. Since $|\mathbb T|=1$ one has $\|f\|_H\le\|f\|_{C}$, and in one dimension $V=H^1(\mathbb T;\mathbb R^3)\hookrightarrow C(\mathbb T;\mathbb R^3)$ compactly; the interpolation estimate $\|f\|_{C}\le\sqrt2\,\|f\|_H^{1/2}\|f\|_V^{1/2}$ used throughout is the case $s=0$, $\theta=\tfrac12$ of \S{}\hyperref[sec:2.6]{2.6}. 
We work on $\mathbb T$ for notational convenience. The unknown $m$ is the magnetization field, hence the state (or configuration) of the system: it is a map $m:\mathbb T\to\mathbb R^3$. The unit-length constraint $|m(x)|\equiv1$ is \textbf{not} imposed as an initial assumption; it is a property of the evolution: the modulus is exactly conserved, $|m(t,x)|=|z(x)|$ for every solution with initial datum $z$ (in the sense of \hyperref[def:2.4a]{Definition \ref*{def:2.4a}}), and sphere-valued data give sphere-valued solutions at all times. We prove this here rather than quoting it: the full drift and the noise are orthogonal to $m$ (an unconditional algebraic identity) and the Itô correction cancels the quadratic variation exactly, so that $d|m|^2=0$ identically --- this is \hyperref[prop:2.2a]{Proposition \ref*{prop:2.2a}} below, with the stochastic instance \hyperref[asm:A4]{(A4)} and \hyperref[rem:2.4a]{Remark \ref*{rem:2.4a}}. The same mechanism underlies [\hyperref[ref:BGL24]{BGL24}, Lemma 7.5] and [\hyperref[ref:BHL19]{BHL19}, Prop. 2.9 and Prop. 3.31], but no step of \hyperref[prop:2.2a]{Proposition \ref*{prop:2.2a}} uses those references; the initial data in this paper are always taken sphere-valued.

\begin{itemize}
\item $H=L^2(\mathbb T;\mathbb R^3)$ with inner product $\langle f,g\rangle_H=\int_{\mathbb T}f\cdot g\,dx$ and norm $\|\cdot\|_H$.
\item $V=H^1(\mathbb T;\mathbb R^3)$ with $\|f\|_V^2=\|f\|_H^2+\|\partial_xf\|_H^2$
\item $H^{s}=H^{s}(\mathbb T;\mathbb R^3)$ for real $s$, defined by the spectral norm $\|f\|_{H^s}^2=\sum_j(1+\mu_j)^{s}|\langle f,e_j\rangle|^2$ with $\mu_j$ below; the two definitions agree for integer $s$.
\item $H^{-s}=H^{-s}(\mathbb T;\mathbb R^3)$ for $s>0$ is the dual of $H^{s}$; in particular $H^{-}=H^{-1}$ is the dual of $V$.
\end{itemize}

Embeddings: $V\hookrightarrow H$ compactly (Rellich) and densely; $V\hookrightarrow L^\infty$ continuously, and compactly since in one dimension $H^1\hookrightarrow C$.

\textbf{Unit-vector fields.} We write $\mathcal M:=\{m\in H:|m(x)|=1\ \text{for a.e. }x\in\mathbb T\}$ for the set of unit-vector fields in $H$; then $V\cap\mathcal M$ is the set of unit-vector fields of class $V$.

\textbf{The operators $A$ and $A_1$.} Let $A:=-\Delta=-\partial_x^2$ on $\mathbb T$ with domain
\[D(A):=H^2(\mathbb T;\mathbb R^3)=\big\{f\in H:\ \partial_xf,\ \partial_x^2f\in H\big\},\]
the periodic Sobolev space of order two (weak derivatives on the torus). Equivalently, identifying $\mathbb T$ with $[0,1]$ with its endpoints glued,
\[D(A)=\big\{f\in H^2(0,1;\mathbb R^3):\ f(0)=f(1),\ f'(0)=f'(1)\big\},\]
where the boundary values are traces (in one dimension $H^2\subset C^1$, so they are ordinary limits); these two conditions are what make $-\partial_x^2$ self-adjoint and non-negative on $H$, and no condition on $f^{\prime\prime}$ is imposed or needed. We write $\{e_j\}_{j\ge1}$ for an orthonormal eigenbasis of $A$ and $\mu_j$ for the corresponding eigenvalues, ordered increasingly, normalised so that $\mu_1=0$ (constants); then
\[D(A)=\Big\{f=\sum_{j\ge1}e_jf_j:\ \sum_{j\ge1}\mu_j^2|f_j|^2<\infty\Big\},\qquad Af=-\partial_x^2f=\sum_{j\ge1}\mu_j\,e_jf_j ,\]
where $f_j\in\mathbb R^3$ are the Fourier coefficients and $\|f\|_{D(A)}:=(\|f\|_H^2+\|Af\|_H^2)^{1/2}\asymp\|f\|_{H^2}$ is the graph norm; in particular $D(A)\subset V\subset H$ densely, with compact embeddings. Let $A_1:=I-\Delta=A+I$, so that $D(A_1)=D(A)$ and $\langle A_1f,f\rangle_H=\|f\|_H^2+\langle Af,f\rangle_H\ge\|f\|_H^2$ (this is the operator that enters the noise regularisation). The eigenvalues of $A_1$ are $1+\mu_j$, so that the norm $\|\cdot\|_{H^s}$ of \S{}\hyperref[sec:2.1]{2.1} is the $A_1^{s/2}$-scale. \textbf{We fix once and for all the real trigonometric basis} $\{e_j\}=\{1\}\cup\{\sqrt2\cos(2\pi kx),\sqrt2\sin(2\pi kx):k\ge1\}$, for which $\sum_{\mu_j=\mu}|e_j(x)|^2=\dim\ker(A-\mu)$ is independent of $x$ and $\|e_j\|_{L^\infty}\le\sqrt2$; this is the basis in which all spectral sums below are evaluated (cf. \S{}\hyperref[sec:E.1]{E.1}). In one dimension
\[\mu_j\asymp j^2\qquad(j\to\infty),\]
i.e. there are constants $0<c\le C<\infty$ with $cj^2\le\mu_j\le Cj^2$ for all $j$; this is the one-dimensional case of Weyl's asymptotic law (see, e.g., [\hyperref[ref:CH53]{CH53}]). All asymptotic statements below are meant in this sense; the exact spectrum of $-\Delta$ on $\mathbb T$ is $\{4\pi^2k^2:k\in\mathbb Z\}$, and the reordering is immaterial.

\textbf{The semigroup.} For the damping constant $\lambda_2>0$ we set
\[S(t):=e^{\lambda_2t\Delta}=e^{-\lambda_2tA}.\]
This is an analytic, compact, contraction $C_0$-semigroup on $H$: $\|S(t)\|_{\mathcal L(H)}\le1$, and $S(t)$ is compact for $t>0$.

\textbf{Lift exponents.} We use the standard one-dimensional smoothing estimates for $S(t)$, valid with constants independent of $t\in(0,T]$; they are stated in \hyperref[asm:A1]{(A1)} of \S{}\hyperref[sec:5.1]{5.1}, where the constants are also fixed.

\subsection{The nonlinearities}\phantomsection\label{sec:2.2}

\textbf{The Gilbert operator.} Define
\[G(m)v:=\lambda_1\,m\times v-\lambda_2\,m\times(m\times v)=\lambda_1\,m\times v+\lambda_2\big(|m|^2v-\langle v,m\rangle m\big),\]
where the second form is the \textbf{constraint-free} expansion of the double cross product ($m\times(m\times v)=m\langle v,m\rangle-|m|^2v$), valid for \emph{every} $m$: it assumes nothing about $|m|$, and on $|m|\equiv1$ it becomes $\lambda_1m\times v+\lambda_2P_{m^\perp}v$, where $P_{m^\perp}v=v-\langle v,m\rangle m$ is the pointwise orthogonal projection onto the tangent space of the sphere at $m$. The two forms are used in this order: the constraint-free one first, and the spherical one only after $|m|\equiv1$ has been proved (the propagation \hyperref[asm:A4]{(A4)} in \S{}\hyperref[sec:5.1]{5.1}). We record the elementary identities, valid on $|m|\equiv1$ (\hyperref[prop:2.2a]{Proposition \ref*{prop:2.2a}}(i)):
\[|G(m)v|_H^2=(\lambda_1^2+\lambda_2^2)|P_{m^\perp}v|_H^2,\qquad \|G(m)\|_{\mathcal L(H)}=g_0:=\sqrt{\lambda_1^2+\lambda_2^2},\qquad \langle G(m)v,m\rangle_H=0 .\]
On $|m|\equiv1$, $G(m)$ is invertible on the tangent bundle, with the explicit inverse
\[G(m)^{-1}w=\frac{\lambda_2w-\lambda_1\,m\times w}{\lambda_1^2+\lambda_2^2}\qquad(w\perp m),\]
which is bounded on tangent vectors with
\[\big\|G(m)^{-1}\big\|_{\mathcal L(T_mS^2)}=\frac{1}{\sqrt{\lambda_1^2+\lambda_2^2}}=\frac1{g_0} .\]
(Indeed, on the two-dimensional tangent plane at $m$, in the orthonormal pair $(v,m\times v)$ the map $G(m)$ has the matrix $\begin{pmatrix}\lambda_2&-\lambda_1\\ \lambda_1&\lambda_2\end{pmatrix}$ up to the identification, so $\|G(m)\|_{\mathcal L(T_mS^2)}=g_0$ and $\|G(m)^{-1}\|_{\mathcal L(T_mS^2)}=1/g_0$; note that $G(m)$ is \emph{not} invertible on $H$, since $G(m)m=0$. For $\lambda_2>0$ this is finite, which is why the damping term is essential for the control-theoretic formulation.)

\begin{proposition}[{the unit-length constraint is propagated, not assumed}]\label{prop:2.2a} \emph{The statements (i)--(ii) are constraint-free: no assumption on $|m|$ is made, and $m$ is any field for which the expressions are meaningful.} Their stochastic instance --- $d|m|^2=0$ identically along a solution of the controlled equation \hyperref[eq:2.1]{(\ref*{eq:2.1})} --- is \hyperref[asm:A4]{(A4)} in \S{}\hyperref[sec:5.1]{5.1} (the equation is introduced in \S{}\hyperref[sec:2.4]{2.4}, the solution class is \hyperref[def:2.4a]{Definition \ref*{def:2.4a}} in \S{}\hyperref[sec:2.4]{2.4}, and existence is \hyperref[thm:5.1]{Theorem \ref*{thm:5.1}}).
\begin{itemize}
\item[(i)] \emph{(algebra of $G$).} For every $m$ and $v$,
\[G(m)v=\lambda_1\,m\times v+\lambda_2\big(|m|^2v-\langle v,m\rangle m\big),\qquad \langle G(m)v,m\rangle_H=0 ,\]
and
\[\|G(m)v\|_H^2=\int_{\mathbb T}\big(\lambda_1^2+\lambda_2^2|m|^2\big)\big(|m|^2|v|^2-\langle v,m\rangle^2\big)\,dx .\]
The orthogonality holds for \textbf{every} field $m$: it uses only the antisymmetry of the cross product. On $\{|m|\equiv1\}$ the two displays reduce to the identities recorded above, and $\|G(m)v\|_H^2=g_0^2\|P_{m^\perp}v\|_H^2$.
\item[(ii)] \emph{(the Itô correction and the quadratic variation cancel identically).} With $r_\delta(x)=\sum_j\kappa_j|e_j(x)|^2$ as in \S{}\hyperref[sec:2.3]{2.3},
\[c(m)=-r_\delta(x)\big(\lambda_1^2+\lambda_2^2|m|^2\big)m,\qquad \sum_{j,i}\kappa_j\big\|G(m)(e_j\hat e_i)\big\|_H^2=2\int_{\mathbb T}r_\delta(x)\,|m|^2\big(\lambda_1^2+\lambda_2^2|m|^2\big)dx ,\]
so that for \textbf{every} field $m$
\[2\langle c(m),m\rangle_H+\sum_{j,i}\kappa_j\big\|G(m)(e_j\hat e_i)\big\|_H^2=0 .\]
(On $|m|\equiv1$ the first identity is the displayed value of $c$ in \S{}\hyperref[sec:2.2]{2.2}, and the whole display is the quadratic-variation identity.)
\end{itemize}
\emph{No well-posedness theory enters (i)--(ii)}, and no step of them assumes $|m|\equiv1$. From here on, $|m|\equiv1$ is used only through the instances of (i) recorded above, through $P_{m^\perp}$ and through $G(m)^{-1}$ on the tangent bundle, and through the identities $m\times(m\times\Delta m)=-\Delta m-|\partial_xm|^2m$ and $m\cdot\partial_xm=0$ of the next paragraph.

\end{proposition}

\textbf{The drift and its splitting.} With $\rho(m):=\Delta m-\phi'(m)$,
\[G(m)\rho(m)=\lambda_1\,m\times\Delta m-\lambda_2\,m\times(m\times\Delta m)-\lambda_1\,m\times\phi'(m)-\lambda_2\,P_{m^\perp}\phi'(m).\]
Absorbing the linear part $\lambda_2\Delta m$ into the semigroup $S$ via the identity $m\times(m\times\Delta m)=-\Delta m-|\partial_xm|^2m$ (valid for $|m|\equiv1$, which holds along all solutions in the sense of \hyperref[def:2.4a]{Definition \ref*{def:2.4a}} by the propagation \hyperref[asm:A4]{(A4)}), the \textbf{nonlinear drift} is
\[\begin{aligned}
b_0(m)&=\lambda_1\,m\times\partial_x^2m+\lambda_2\,|\partial_xm|^2m-G(m)\phi'(m)\\
&=\lambda_1\,\partial_x(m\times\partial_xm)+\lambda_2\,|\partial_xm|^2m-\lambda_1\,m\times\phi'(m)-\lambda_2\,P_{m^\perp}\phi'(m).
\end{aligned}\]
The decisive structural point is that
\[b_0:V\longrightarrow H^{-}\quad\text{is locally Lipschitz (in the }V\text{-norm of the argument)},\]
and \textbf{not} $V\to H$: the term $m\times\partial_x^2m$ genuinely loses one derivative. This is the well-known $H^{-}$-valuedness of the LLG drift, and it governs the whole analysis (it is what makes the $H$-level LDP natural and the $V$-level one hard).

\textbf{Anisotropy.} $\phi\in C^2_b(\mathbb R^3;\mathbb R_+)$ is the anisotropy energy density; $\phi'$ its gradient and $D\phi'$ its Hessian. The total energy is
\[E(m)=\tfrac12\|\nabla m\|_H^2+\int_{\mathbb T}\phi(m)\,dx .\]

\textbf{The Stratonovich correction.} The noise amplitude is $\sqrt\varepsilon$, and we write the noise as $\sqrt\varepsilon\sum_{j,i}\sigma_{j,i}(m)\,d\beta_{j,i}$ with the \textbf{unit-strength} fields $\sigma_{j,i}(m):=\sqrt{\kappa_j}\,G(m)(e_j\hat e_i)$. The Itô correction of the unit-strength fields is $\tfrac12\sum_{j,i}D\sigma_{j,i}(m)\sigma_{j,i}(m)=c(m)$ (no factor $\varepsilon$; the Itô correction of the actual noise is $\varepsilon c(m)$), with
\[\begin{aligned}
c(m)&=-\tfrac{g_0^2}{3}\operatorname{tr}C_\delta(x)\,m,\\
&\qquad g_0^2=\lambda_1^2+\lambda_2^2,\\
&\qquad C_\delta(x):=\sum_{j,i}\kappa_j|e_j(x)|^2\,\hat e_i\otimes\hat e_i=r_\delta(x)\,I_3,\\
&\qquad r_\delta(x):=\sum_j\kappa_j|e_j(x)|^2 ,
\end{aligned}\]
$C_\delta$ being the \textbf{local covariance} of the noise (the constraint-free form of the correction is $c(m)=-r_\delta(x)(\lambda_1^2+\lambda_2^2|m|^2)m$, the displayed expression being its value on $|m|\equiv1$; see \hyperref[prop:2.2a]{Proposition \ref*{prop:2.2a}}(ii)). (i) The correction is \textbf{purely radial}: the tangential terms of the flat-model computation ($Q_\delta=I$, $C_\delta=I$) cancel identically, so $c(m)$ is parallel to $m$ and the tangential drift is unchanged by the Itô conversion. (ii) For the isotropic $Q_\delta$ of \S{}\hyperref[sec:2.3]{2.3} (equal weights on the three components) $r_\delta$ is constant in $x$: for the real trigonometric basis of \S{}\hyperref[sec:2.1]{2.1} one has $\sum_{\mu_j=\mu}|e_j(x)|^2=\dim\ker(A-\mu)$ independently of $x$, and $\kappa_j$ depends on $j$ only through $\mu_j$, so $\sum_j\kappa_j|e_j(x)|^2=\sum_j\kappa_j$, so $c(m)=-g_0^2r_\delta\,m$, $\operatorname{tr}C_\delta=3r_\delta$, and the quadratic-variation identity
\[2\langle c(m),m\rangle_H=-\tfrac23g_0^2\int_{\mathbb T}\operatorname{tr}C_\delta(x)\,dx=-2g_0^2r_\delta=-\tfrac23g_0^2\operatorname{tr}_HQ_\delta\le0\]
holds (all three forms coincide: $\operatorname{tr}C_\delta=3r_\delta$ on the unit torus and $\operatorname{tr}_HQ_\delta=3\sum_j\kappa_j=3r_\delta$) --- this is the form used in \hyperref[asm:A4]{(A4)}. Equivalently --- and here \emph{no} assumption on $|m|$ is needed --- the identity reads $2\langle c(m),m\rangle_H=-\sum_{j,i}\kappa_j\|G(m)(e_j\hat e_i)\|_H^2$, i.e.\ the Itô correction cancels the quadratic variation exactly; see \hyperref[prop:2.2a]{Proposition \ref*{prop:2.2a}}(ii). This cancellation is what propagates the constraint $|m|\equiv1$. We abbreviate $b_\varepsilon:=b_0+\varepsilon c$.

\subsection{The regularised noise}\phantomsection\label{sec:2.3}

Let $\delta>0$ and $\beta>0$. Define the \textbf{regularised square root}
\[Q_\delta^{1/2}:=(I+\delta A^{\beta/2})^{-1},\]
so that $Q_\delta=(I+\delta A^{\beta/2})^{-2}$ has eigenvalues
\[\kappa_j=(1+\delta\mu_j^{\beta/2})^{-2}\asymp(1+\delta j^{\beta})^{-2},\]
and $\|Q_\delta^{1/2}\|_{\mathcal L(H)}\le1$, $Q_\delta^{1/2}e_j=q_je_j$ with $q_j=(1+\delta\mu_j^{\beta/2})^{-1}$. Let $W_\delta$ be the $Q_\delta$-Wiener process on $H=L^2(\mathbb T;\mathbb R^3)$, expanded in the \textbf{three-component basis} $\{e_j\hat e_i\}_{j\ge1,\,i=1,2,3}$: $W_\delta(t)=\sum_{j}\sum_{i=1}^{3}\beta_{j,i}(t)\,Q_\delta^{1/2}e_j\,\hat e_i$, with independent standard Brownian motions $\beta_{j,i}$. The noise in \hyperref[eq:SLLG]{(SLLG)} is $\sqrt\varepsilon\,G(m)\circ dW_\delta$.

\textbf{Thresholds.} Direct computation of the relevant spectral sums gives:

\begin{center}
\begin{tabular}{|l|l|}
\hline
property & equivalent to \\
\hline
$Q_\delta$ is trace class: $\sum_j\kappa_j<\infty$ & $\beta>1/2$ \\
$Q_\delta^{1/2}$ maps $H$ into $V$: $\sup_j\kappa_j\mu_j<\infty$ & $\beta\ge1$ \\
the noise is $H^1$-valued: $\sum_j\kappa_j\mu_j<\infty$ & $\beta>3/2$ \\
\hline
\end{tabular}
\end{center}

We work under the strongest of the three, $\beta>3/2$ --- the sufficient regularity hypothesis adopted throughout this paper.

\subsection{The controlled equation and the mild form}\phantomsection\label{sec:2.4}

Adding a control in the effective field, the controlled SLLG is
\begin{equation}\label{eq:2.1}
dm=G(m)\Big[\rho(m)\,dt+Q_\delta^{1/2}u\,dt+\sqrt\varepsilon\circ dW_\delta\Big],\qquad m(0)=z,
\end{equation}
with $u$ ranging over the admissible controls below.

\textbf{The Itô form.} The Stratonovich equation \hyperref[eq:2.1]{(\ref*{eq:2.1})} is equivalent to the Itô equation
\begin{equation}\label{eq:2.2}
dm=\big[F_0\big(m,Q_\delta^{1/2}u\big)+\varepsilon c(m)\big]dt+\sqrt\varepsilon\,G(m)\,dW_\delta ,
\end{equation}
where
\[c(m):=\tfrac12\operatorname{tr}\big(DG(m)[G(m)\,\cdot\,]\,\cdot\big)=-g_0^2r_\delta\,m\]
is the Itô--Stratonovich correction (the computation is that of \hyperref[prop:2.2a]{Proposition \ref*{prop:2.2a}}(ii); see also \hyperref[rem:2.4a]{Remark \ref*{rem:2.4a}}), and
\[F_0(m,h):=\lambda_2\Delta m+\lambda_1m\times\big(\Delta m-\varphi'(m)+h\big)+\lambda_2|\partial_xm|^2m-\lambda_2P_{m^\perp}\varphi'(m)+\lambda_2P_{m^\perp}h\]
is the drift of the strong form (\S{}\hyperref[sec:D.1]{D.1}). The correction belongs to the drift: without it the constraint $|m|\equiv1$ is not propagated (\hyperref[rem:2.4a]{Remark \ref*{rem:2.4a}}).

\begin{definition}[{solution}]\label{def:2.4a} Let $(\Omega,\mathcal F,(\mathcal F_t),\mathbb P)$ satisfy the usual conditions and carry $W_\delta$. Let $z\in V$ with $|z|\equiv1$, $\varepsilon\in(0,1]$ and $u\in\mathcal P_2$ (the class of \S{}\hyperref[sec:2.5]{2.5}). A \textbf{solution} of \hyperref[eq:2.1]{(\ref*{eq:2.1})} is an $(\mathcal F_t)$-adapted process $m$ with continuous paths in $H$ such that, a.s., $m\in L^\infty(0,T;V)\cap L^2(0,T;D(A_1))$ and
\[m(t)=z+\int_0^t\big[F_0(m,Q_\delta^{1/2}u)+\varepsilon c(m)\big]ds+\sqrt\varepsilon\int_0^tG(m)\,dW_\delta\qquad\text{in }H,\ \text{for all }t\in[0,T].\]

\end{definition}

\emph{The definition is well posed.} For $m$ in the stated class the chain
\[\|\partial_xm\|_{L^4}^4\le\|\partial_xm\|_{L^\infty}^2\|\partial_xm\|_H^2\le2\|\partial_xm\|_H^3\|\Delta m\|_H\]
is integrable, so $F_0(m,Q_\delta^{1/2}u)\in L^2(0,T;H)$ a.s. (as in \hyperref[lem:D.6]{Lemma \ref*{lem:D.6}}); the noise is trace class, $\operatorname{tr}(G(m)Q_\delta G(m)^*)\le g_0^2\operatorname{tr}_HQ_\delta$.

\begin{remark}[{the constraint is propagated, not assumed}]\label{rem:2.4b} The unit-length constraint is \emph{not} part of \hyperref[def:2.4a]{Definition \ref*{def:2.4a}}: it is a property of the equation. For a solution with $|z|\equiv1$, the full drift and the noise are orthogonal to $m$ and the Itô correction is radial, so the modulus is exactly conserved, $|m(t,x)|=|z(x)|$ (see \hyperref[asm:A4]{(A4)}, which covers the controlled equation as well); in particular $|m(t,x)|=1$ for all $t$ and a.e. $x$. This is the same conservation as in [\hyperref[ref:BHL19]{BHL19}, Prop. 2.9 and Prop. 3.31] and [\hyperref[ref:BGL24]{BGL24}, Lemma 7.5]. All data in this paper are sphere-valued, so the constraint is available wherever the proofs use it.

\end{remark}

Absorbing the linear part $\lambda_2\Delta$ into $S$, the \textbf{mild form} of \hyperref[eq:2.1]{(\ref*{eq:2.1})} is
\[\begin{aligned}
m^\varepsilon_{z,u}(t)&=S(t)z+\int_0^tS(t-s)\,b_\varepsilon\big(m^\varepsilon_{z,u}(s)\big)\,ds\\
&+\sqrt\varepsilon\int_0^tS(t-s)G\big(m^\varepsilon_{z,u}(s)\big)\,dW_\delta(s)\\
&+\int_0^tS(t-s)G\big(m^\varepsilon_{z,u}(s)\big)Q_\delta^{1/2}u(s)\,ds .
\end{aligned}\]

\textbf{Two-layer device.} The nonlinearity $b_0$ is $H^{-}$-valued, so the mild form above is not directly a fixed-point equation in $C([0,T];H)$. We therefore split off the two $G$-weighted convolutions --- the stochastic and the control one --- by defining the \textbf{auxiliary process}
\[Y^\varepsilon_{z,u}:=\sqrt\varepsilon\int_0^tS(t-s)G\big(m^\varepsilon_{z,u}(s)\big)\,dW_\delta(s)+\int_0^tS(t-s)G\big(m^\varepsilon_{z,u}(s)\big)Q_\delta^{1/2}u(s)\,ds ,\]
and recovering the magnetisation through the solution map $\Psi^\varepsilon$ of the semilinear problem
\[v=\int_0^{\cdot}S(\cdot-s)\,b_\varepsilon\big(v(s)+\psi(s)\big)\,ds,\qquad m=\Psi^\varepsilon(\psi):=v+\psi,\]
so that $m^\varepsilon_{z,u}=\Psi^\varepsilon\big(Y^\varepsilon_{z,u}+S(\cdot)z\big)$. (The drift in the defining equation is $b_\varepsilon=b_0+\varepsilon c$, not $b_0$: the Itô correction belongs to the drift, and the identity $m^\varepsilon_{z,u}=\Psi^\varepsilon(Y^\varepsilon_{z,u}+S(\cdot)z)$ is exact only with it; see \hyperref[rem:2.4a]{Remark \ref*{rem:2.4a}}.) The map $\Psi^\varepsilon$ is defined on the class of admissible inhomogeneities
\[\begin{aligned}
\mathcal A:=\Big\{&\psi\in C([0,T];H):\ \exists\,v\in C([0,T];H)\ \text{with}\ v=\int_0^{\cdot}S(\cdot-s)b_\varepsilon(v+\psi)\,ds,\\
&\ \ (v+\psi)(t)\in V\ \text{a.e.},\\
&\ \sup_{t\ge\eta}\|v+\psi\|_V<\infty\ \ \forall\eta>0\Big\}.
\end{aligned}\]

\textbf{What is, and what is not, used about $\Psi^\varepsilon$.} The two-layer device is recorded only because the auxiliary process $Y$ is a convenient name for the sum of the stochastic and the control convolutions. \textbf{No continuity property of the solution map $\Psi^\varepsilon$ is used in this paper}, and none is claimed. In particular we make no claim that $\Psi^\varepsilon$ is Hölder continuous on $C([0,T];H)$: a Grönwall argument in $C([0,T];H)$ based on the mild form does not close, because $b_0$ is Lipschitz only from $V$ to $H^{-}$, so that the difference of two solutions is estimated through its $V$-norm while only its $H$-norm is controlled (\hyperref[rem:E.8]{Remark \ref*{rem:E.8}}). The uniform LDP of \S{}\hyperref[sec:6]{6} is proved \textbf{directly for the magnetisation} $m$ (\hyperref[thm:6.1]{Theorem \ref*{thm:6.1}}): the variational representation is applied to $m$ as a Borel functional of the noise (\hyperref[prop:E.6]{Proposition \ref*{prop:E.6}}), and exponential tightness is proved for $m$ itself (\hyperref[prop:5.5]{Proposition \ref*{prop:5.5}}). No contraction argument is needed for the SLLG: the LDP is proved directly for the magnetisation.

\begin{remark}[{the Itô correction belongs to the drift}]\label{rem:2.4a} The drift in the defining equation of $\Psi^\varepsilon$ is $b_\varepsilon=b_0+\varepsilon c$, not $b_0$: only with it does the identity $m^\varepsilon_{z,u}=\Psi^\varepsilon(Y^\varepsilon_{z,u}+S(\cdot)z)$ hold, and only with it is the constraint $|m|\equiv1$ propagated (\hyperref[asm:A4]{(A4)}; the constraint-free statement is \hyperref[prop:2.2a]{Proposition \ref*{prop:2.2a}}). The correction $c(m)=-g_0^2r_\delta\,m$ is bounded and Lipschitz on the sphere-valued class, with constant $g_0^2r_\delta\asymp\delta^{-1/\beta}$ as $\delta\downarrow0$ --- finite at each fixed $\delta>0$, which is all that \S{}\hyperref[sec:5]{5}--\S{}\hyperref[sec:6]{6} use.

\end{remark}

\subsection{Classes of controls}\phantomsection\label{sec:2.5}

\begin{itemize}
\item $L^2(0,T;H)$: deterministic controls.
\item $S^N:=\{u\in L^2(0,T;H):\int_0^T\|u\|_H^2\,ds\le N\}$: the closed energy ball, weakly compact (\hyperref[lem:4.1]{Lemma \ref*{lem:4.1}}).
\item $\mathcal P_2$: $(\mathcal F^W_t)$-predictable $H$-valued processes with $\int_0^T\|u\|_H^2\,ds<\infty$ a.s.
\item $\mathcal P_2^N:=\{u\in\mathcal P_2:\int_0^T\|u\|_H^2\le N\ \text{a.s.}\}$.
\end{itemize}

\textbf{Subscript convention.} $m^\varepsilon_{z,u}$ denotes the solution of \hyperref[eq:2.1]{(\ref*{eq:2.1})} (in the sense of \hyperref[def:2.4a]{Definition \ref*{def:2.4a}}) with noise intensity $\varepsilon$, initial datum $z$ and control $u$; the control and the initial datum carry \textbf{parallel subscripts}. We write $m^\varepsilon_z:=m^\varepsilon_{z,0}$. We say that the family $\{m^\varepsilon_z\}_{\varepsilon>0}$ of path-valued random variables \emph{satisfies} an LDP when the family of its laws $\{\mathcal L(m^\varepsilon_z)\}_{\varepsilon>0}$ does, $\mathcal L(m^\varepsilon_z)$ denoting the Borel law of $m^\varepsilon_z$ on the path space ($C([0,T];H)$ throughout, and $C([0,T];C(\mathbb T;\mathbb R^3))$ in the finer statement of \S{}\hyperref[sec:6]{6}); the estimates are stated directly in terms of the random variables, with $\mathbb P(m^\varepsilon_z\in A)=\mathcal L(m^\varepsilon_z)(A)$. This is the standard convention of the theory, in which the LDP is a property of the process family.

\subsection{The initial set}\phantomsection\label{sec:2.6}

Throughout, $D\subset V\cap\mathcal M$ is a fixed set of \textbf{sphere-valued fields} (every $z\in D$ satisfies $|z(x)|=1$ for a.e. $x$), \textbf{bounded in $V$}, with
\[R_D:=\sup_{z\in D}\|z\|_V<\infty .\]
Since $V\hookrightarrow H$ compactly, $\bar D:=\overline D^{\,H}$ is compact in $H$; moreover $\bar D\subset V\cap\mathcal M$, since a strong $H$-limit of sphere-valued fields is sphere-valued. All uniform statements are quantified over $z\in D$; limits are taken along sequences $z_n\in D$ with $z_n\to z$ in $H$, $z\in\bar D$ (compactness makes this lossless). \textbf{Convention (the two topologies).} In the applications $D$ is $H$-closed (for instance a ball in $V\cap\mathcal M$, or a closed energy sublevel set $\{E\le r\}\cap\mathcal M$). Note the asymmetry between the two topologies: an $H$-closed, $V$-bounded set is automatically $V$-closed, whereas the converse fails in general --- a merely $V$-closed set need not be $H$-closed. On $V$-bounded sets the elementary interpolation $\|w\|_{L^\infty}\le C\|w\|_H^{1/2}\|w\|_V^{1/2}$ upgrades $H$-convergence to $C$-convergence, so a $V$-closed set defined by a \emph{pointwise} constraint is automatically $H$-closed. What fails for a general $V$-closed set is the use of the $V$-norm itself, which is only $H$-lower semicontinuous: for $f_n=A_n\sin(2\pi nx)$ with $A_n=\sqrt2/(2\pi n)$ one has $\|f_n\|_V\to1$ and $f_n\to0$ in $H$, so $C=\{f:\|f\|_V\ge1\}$ is $V$-closed but not $H$-closed. Hence $H$-closedness is the operative property and is imposed directly. This is used only once, in the verification of the third requirement of \hyperref[eq:UG]{(UG)} in \hyperref[prop:4.9]{Proposition \ref*{prop:4.9}}, and it costs nothing: if a chosen $D$ is not $H$-closed one replaces it by its $H$-closure $\bar D$, which is $H$-compact, $V$-bounded and $H$-closed, and all the conclusions below hold for that larger set. One place where the two topologies genuinely differ is the energy level sets of \S{}\hyperref[sec:11]{10}--\S{}\hyperref[sec:12]{11}: on $H$ the energy $E$ is only \emph{lower} semicontinuous (its gradient term is), so a sublevel set $\{E\le r\}$ is $H$-closed, while $\{E<r\}$ and the level set $\{E=r\}$ need not be; the exit statements use only the containment of the barrier $\{E=r\}$ in the $H$-compact set $\bar D$ together with the finiteness of the barrier, as indicated in \S{}\hyperref[sec:11.1]{10.1}. \textbf{The symbol $\partial D$.} Throughout the paper, for the energy level sets of \S{}\hyperref[sec:11]{10}--\S{}\hyperref[sec:12]{11}, $\partial D$ denotes the \emph{barrier energy surface} $\{E=r\}$ itself --- not the topological boundary: in the $H$-topology the topological boundary of the set $D$ is $D$ (every neighbourhood of a point of $D$ contains fields of arbitrarily large energy --- e.g. $m+\varepsilon w$ with $\|\nabla w\|_H$ large --- hence points outside $D$), and all infima over $\partial D$ are taken over $\{E=r\}$; accordingly, every statement below that treats $\partial D$ as a \textbf{closed set} (or applies to it a statement about closed sets) is to be read as a statement about this energy surface $\{E=r\}$ itself, and the only two properties of $\partial D$ that are ever used are the containment $\partial D\subset\bar D$ and the finiteness of the barrier.

\subsection{Standing assumptions}\phantomsection\label{sec:2.7}

\phantomsection\label{asm:S1}\textbf{(S1)} $\beta>3/2$ and $\delta\in(0,1]$.

\phantomsection\label{asm:S2}\textbf{(S2)} $\lambda_1,\lambda_2>0$ are fixed; the noise operator is $G$ as in \S{}\hyperref[sec:2.2]{2.2}, with $\|G(m)\|_{\mathcal L(H)}=g_0=\sqrt{\lambda_1^2+\lambda_2^2}$.

\phantomsection\label{asm:S3}\textbf{(S3)} The noise amplitude is $\sqrt\varepsilon$; consequently the Stratonovich correction is $\varepsilon c$.

\phantomsection\label{asm:S4}\textbf{(S4)} The anisotropy $\phi\in C^2_b(\mathbb R^3;\mathbb R_+)$ is retained; the identification $U_0=c_\ast E$ and the basin of attraction are asserted for the class of $\phi$ for which $\zeta_+$ is asymptotically stable (see \S{}\hyperref[sec:12]{11}), the model case being $\phi_1(u)=\tfrac12|u-\zeta_+|^2$. \textbf{Convention (truncation).} $\phi_1$ is not bounded on $\mathbb R^3$, so it is not literally in $C^2_b$, nor in the class $C^2_0$ of [\hyperref[ref:BHL19]{BHL19}, \hyperref[ref:BGL24]{BGL24}]. This is immaterial. Every solution considered in this paper takes values in the unit sphere (the constraint is propagated by the equation, not assumed: see \hyperref[asm:A4]{(A4)}), and both the equation and the energy see $\phi$ only through $\phi,\phi',D\phi'$ on $\{|u|=1\}$. Fix $\chi\in C^\infty_c(\mathbb R^3;[0,1])$ with $\chi=1$ on $\{|u|\le2\}$ and replace $\phi$ by $\chi\phi\in C^2_0(\mathbb R^3;\mathbb R_+)$. This changes neither the equation on the sphere nor $E$ on sphere-valued fields. By the uniqueness statements (\hyperref[thm:D.8]{Theorem \ref*{thm:D.8}} for $\varepsilon=0$, \hyperref[thm:E.4]{Theorem \ref*{thm:E.4}} for $\varepsilon>0$), it does not change the solutions either. Hence we may, and do, assume $\phi\in C^2_0\subset C^2_b$ whenever a cited result requires it; only off the sphere, e.g. for Galerkin approximants, does the class of $\phi$ matter.

\phantomsection\label{asm:S5}\textbf{(S5)} The state space of the LDP is $C([0,T];H)$. The finer space $C([0,T];V)$ is discussed but not claimed (\S{}\hyperref[sec:13]{12}).

\phantomsection\label{asm:AN}\textbf{(AN) (asymptotic stability, and the identification on the basin).} $\zeta_+$ is asymptotically stable for the deterministic LLG, and the identification $U_0=c_\ast E$ holds on its basin of attraction --- the sublevel set $\{E<r_{\rm crit}\}$ of [\hyperref[ref:BHL19]{BHL19}, Thm 4.4], with $E(\zeta_+)=0$ and $r_{\rm crit}$ the energy of its boundary. (Justification and discussion: \S{}\hyperref[sec:12.4]{11.4}; nothing is assumed there that is not stated here.)

\phantomsection\label{asm:D-MARGIN}\textbf{(D-MARGIN) (the margin condition on the initial set).} The domain $D$ of the barrier statements is the \emph{open} energy sublevel set $\{E<r\}$, with $\bar D=\{E\le r\}$ $H$-compact, with $r<r_{\rm crit}$, strictly inside the basin of attraction of $\zeta_+$; note that $D$ is \emph{not} the basin itself.

\phantomsection\label{asm:S6}\textbf{(S6)} The two assumptions \hyperref[asm:AN]{(AN)} and \hyperref[asm:D-MARGIN]{(D-MARGIN)} stated just above are used for the barrier statements of \S{}\hyperref[sec:11]{10}--\S{}\hyperref[sec:12]{11} and, through the identification $U_0=c_\ast E$ on the basin invoked in \hyperref[lem:10.2a]{Lemma \ref*{lem:10.2a}}, for the lower semicontinuity in \S{}\hyperref[sec:10.2]{9.2} (they are not needed for the uniform LDP of \S{}\hyperref[sec:6]{6} nor for the $\Gamma$-convergence of the action functionals of \S{}\hyperref[sec:8]{8}). Under \hyperref[asm:AN]{(AN)} and \hyperref[asm:D-MARGIN]{(D-MARGIN)} the barrier satisfies $\inf_{\partial D}U_\delta=c_\ast\inf_{\partial D}E$ for \textbf{every} $\delta\in(0,1]$ (\hyperref[thm:11.6]{Theorem \ref*{thm:11.6}}) --- an identity, with no further hypothesis, since the barrier always contains a spatially constant state (\hyperref[lem:11.5b]{Lemma \ref*{lem:11.5b}}).

\textbf{Symbols fixed in this section.} For reference, the objects introduced above and used throughout:

\begin{center}\footnotesize
\begin{tabular}{|p{2.5cm}|p{11.6cm}|}
\hline
symbol & meaning \\
\hline
$H$, $V$ & $L^2(\mathbb T;\mathbb R^3)$ and $H^1(\mathbb T;\mathbb R^3)$; $V\hookrightarrow H$ compactly (\S{}\hyperref[sec:2.1]{2.1}) \\
$A$, $A_1$, $\mu_j$, $e_j$ & $-\partial_x^2$ on $\mathbb T$ with domain $H^2(\mathbb T;\mathbb R^3)$, $A_1=I-\Delta$, its eigenvalues $\mu_j\asymp j^2$ and eigenbasis \\
$S(t)$ & $e^{\lambda_2t\Delta}=e^{-\lambda_2tA}$: analytic contraction semigroup, compact for $t>0$ \\
$G$, $\rho$, $b_0$, $c$, $\phi$ & Gilbert operator, drift, drift without rotation, It\^o correction $-g_0^2r_\delta m$, anisotropy (\S{}\hyperref[sec:2.2]{2.2}) \\
$Q_\delta$, $\kappa_j$, $r_\delta$ & noise covariance $Q_\delta^{1/2}=(I+\delta A^{\beta/2})^{-1}$, spectral weights, $r_\delta=\sum_j\kappa_j$ (\S{}\hyperref[sec:2.3]{2.3}) \\
$W_\delta$, $f_k$ & $Q_\delta$-Wiener process; its modes (\S{}\hyperref[sec:2.3]{2.3}, \S{}\hyperref[sec:E.1]{E.1}) \\
$F_0$ & the drift of the strong form; $m$ solves \hyperref[eq:2.2]{(\ref*{eq:2.2})} (\S{}\hyperref[sec:2.4]{2.4}) \\
$m^\varepsilon_{z,u}$, $m^0_{z,u}$ & controlled solution (\hyperref[def:2.4a]{Definition \ref*{def:2.4a}}) and its skeleton \\
$\mathcal P_2$, $\mathcal P_2^N$, $S^N$ & control classes and the energy ball (\S{}\hyperref[sec:2.5]{2.5}) \\
$D$, $R_D$, $\bar D$ & the initial set, $\sup_{z\in D}\|z\|_V<\infty$, its closure in $H$ (\S{}\hyperref[sec:2.6]{2.6}) \\
$I_z$, $\Phi_z(s)$ & rate function \hyperref[eq:6.1]{(\ref*{eq:6.1})} and its level sets $\{I_z\le s\}$ \\
$\Psi^\varepsilon$ & the two-layer solution map of \S{}\hyperref[sec:2.4]{2.4} (recorded, never assumed continuous) \\
$\mathcal G_K f$ & $\int_0^tK(t-s)f(s)ds$ (\S{}\hyperref[sec:5.5]{5.5}) \\
\hline
\end{tabular}
\end{center}

\section{Quasi-potentials, action functionals and the target statement}
The objects of the paper are defined here and not proved: the rate function at regularisation $\delta$ (\S{}\hyperref[sec:3.1]{3.1}), the quasi-potential $U_\delta$ and the target statement (\S{}\hyperref[sec:3.3]{3.3}). The section states precisely what is to be proved and in which two halves --- the identification of $U_\delta$ as a rate function (\S{}\hyperref[sec:6]{6}) and its $\delta\downarrow0$ limit (\S{}\hyperref[sec:7]{7}--\S{}\hyperref[sec:12]{11}) --- and it introduces the $\delta=0$ convention used throughout.
\phantomsection\label{sec:3}

\subsection{The rate function at regularisation $\delta$}\phantomsection\label{sec:3.1}

Fix $\delta\in(0,1]$. For $z\in D$, the rate function of the uniform LDP of \S{}\hyperref[sec:6]{6} (\hyperref[thm:6.1]{Theorem \ref*{thm:6.1}}) for the magnetisation is
\begin{equation}\label{eq:3.1}
I^\delta_z(\varphi)=\inf\Big\{\tfrac12\!\int_0^T\|u(s)\|_H^2\,ds:\ u\in L^2(0,T;H),\ \varphi=m^{0,\delta}_{z,u}\Big\},\qquad\inf\emptyset:=+\infty,
\end{equation}
where $m^{0,\delta}_{z,u}$ solves the \textbf{skeleton} ($\varepsilon=0$)
\begin{equation}\label{eq:3.2}
\dot m=G(m)\big[\rho(m)+Q_\delta^{1/2}u\big],\qquad m(0)=z .
\end{equation}
Equivalently, for absolutely continuous $\varphi\in C([0,T];H)$ with $|\varphi(t,x)|\equiv1$,
\begin{equation}\label{eq:3.3}
I^\delta_z(\varphi)=\tfrac12\inf_{\alpha}\int_0^T\big\|Q_\delta^{-1/2}\big[G(\varphi)^{-1}\dot\varphi-P_{\varphi^\perp}\rho(\varphi)+\alpha\,\varphi\big]\big\|_H^2\,dt ,\qquad \alpha\in L^2(0,T;L^2(\mathbb T)),
\end{equation}
and $I^\delta_z(\varphi)=+\infty$ otherwise; the equivalence of \hyperref[eq:3.1]{(\ref*{eq:3.1})} and \hyperref[eq:3.3]{(\ref*{eq:3.3})} is the elementary statement that the control driving \hyperref[eq:3.2]{(\ref*{eq:3.2})} along $\varphi$ is $u=Q_\delta^{-1/2}\big[G(\varphi)^{-1}\dot\varphi-P_{\varphi^\perp}\rho(\varphi)\big]$ (the inverse $G(\varphi)^{-1}$ being the bounded one displayed in \S{}\hyperref[sec:2.2]{2.2}, with $\|G(\varphi)^{-1}\|_{\mathcal L(T_\varphi S^2)}=1/g_0$). The map $u\mapsto Q_\delta^{1/2}u$ is injective, but the \emph{orbit} \hyperref[eq:3.2]{(\ref*{eq:3.2})} depends on $u$ only through $G(\varphi)Q_\delta^{1/2}u$, and $G(\varphi)$ annihilates the radial direction ($G(\varphi)\varphi=0$). The controls driving \hyperref[eq:3.2]{(\ref*{eq:3.2})} along a prescribed $\varphi$ therefore form the coset $u_0+Q_\delta^{-1/2}(\alpha\varphi)$, where the \textbf{normal coefficient} $\alpha=\alpha(t,x)$ is a scalar field in $L^2(0,T;L^2(\mathbb T))$, and the minimal-energy control is the member of least $L^2$-norm; the display \hyperref[eq:3.3]{(\ref*{eq:3.3})} carries this infimum explicitly. The infimum is attained --- the coset is a translate of the closed subspace $Q_\delta^{-1/2}(L^2\varphi)$. (The normal coefficient $\alpha$ must not be confused with the noise-regularity exponent $\beta$ of \S{}\hyperref[sec:2.3]{2.3}.) At $\delta=0$, $Q_0^{-1/2}=I$ is a local operator, so $\alpha=0$ is optimal and \hyperref[eq:3.5]{(\ref*{eq:3.5})} is exact as written; for $\delta>0$ the correction is due solely to the non-locality of $Q_\delta^{-1/2}$ and vanishes when the bracket is pointwise orthogonal to $\varphi$ \emph{and} the fields are spatially constant --- the case of the constant-state computations of \S{}\hyperref[sec:11.1]{10.1} --- while for general $\delta>0$ and general $\varphi$ it can only \emph{lower} the value. (In particular all upper bounds below remain valid as they stand, and are equalities at $\delta=0$.) We shall use \hyperref[eq:3.1]{(\ref*{eq:3.1})} for compactness arguments and \hyperref[eq:3.3]{(\ref*{eq:3.3})} for explicit computation.

\subsection{The convention on ``$\delta=0$''}\phantomsection\label{sec:3.2}

The limit $\delta\to0$ is taken at the level of the \textbf{action functional and the control problem}, never at the level of the stochastic equation.

\begin{enumerate}
\item \textbf{The stochastic equation at $\delta=0$ is not defined.} As $\delta\downarrow0$, $r_\delta=\sum_j\kappa_j|e_j(x)|^2\asymp\delta^{-1/\beta}\to\infty$ and the Stratonovich corrector $c(m)=-g_0^2r_\delta\,m$ diverges; the limiting object would be an equation driven by space--time white noise, requiring renormalisation, and is outside the scope of this paper.
\item \textbf{The skeleton at $\delta=0$ is well defined.} Setting $Q_0^{1/2}:=I$ in \hyperref[eq:3.2]{(\ref*{eq:3.2})} gives the controlled deterministic LLG
\begin{equation}\label{eq:3.4}
\dot\varphi=G(\varphi)\big[\rho(\varphi)+h\big],\qquad\varphi(0)=z_\ast,
\end{equation}
with $h\in L^2(0,T;H)$ the \textbf{physical} control of \S{}\hyperref[sec:7]{7}; its well-posedness is [\hyperref[ref:BHL19]{BHL19}, Thm 3.1/3.37].
\item \textbf{The action at $\delta=0$ is well defined.} With $Q_0^{-1/2}=I$, \hyperref[eq:3.3]{(\ref*{eq:3.3})} becomes
\begin{equation}\label{eq:3.5}
I^0_{z_\ast}(\varphi)=\tfrac12\int_0^T\big\|G(\varphi)^{-1}\dot\varphi-P_{\varphi^\perp}\rho(\varphi)\big\|_H^2\,dt ,
\end{equation}
and the corresponding cost of a physical control is $\tfrac12\int_0^T\|h\|_H^2dt$.
\end{enumerate}

The skeleton \hyperref[eq:3.4]{(\ref*{eq:3.4})} and hence the \emph{set} of admissible paths are $\delta$-independent; only the \emph{weight} depends on $\delta$. This is the structural fact on which \S{}\hyperref[sec:7]{7}--\S{}\hyperref[sec:8]{8} rest.

\subsection{Quasi-potentials}\phantomsection\label{sec:3.3}

\begin{definition}[{quasi-potential}]\label{def:3.1} For $\delta\in[0,1]$ and $a\in H$ we use the \textbf{free-horizon} quasi-potential
\begin{equation}\label{eq:3.6}
\begin{aligned}
U_\delta(a)&:=\inf_{T>0}\ U^{z_\ast,T}_\delta(a),\\
&\qquad U^{z_\ast,T}_\delta(a):=\inf\big\{I^\delta_{z_\ast}(\varphi):\ \varphi\in C([0,T];H),\ \varphi(0)=z_\ast,\ \varphi(T)=a\big\} ,
\end{aligned}
\end{equation}
with $z_\ast\in D$ a reference state. The horizon is \textbf{free}: a path may reach $a$ at \emph{any} time $T>0$, and $U^{z_\ast,T}_\delta$ denotes the fixed-horizon restriction. \textbf{Throughout the paper the reference state is the attractor $z_\ast=\zeta_+$} (the stable equilibrium, $E(\zeta_+)=0$; \S{}\hyperref[sec:12.4]{11.4}); the general reference state appears only in the ancillary value-function and uniformity statements of \S{}\hyperref[sec:10.1]{9.1}, which are written for an arbitrary $z_\ast$ for convenience. Equivalently --- and this is the form of [\hyperref[ref:BHL19]{BHL19}, Def. 5.1] that the identification and the reachability analysis of \S{}\hyperref[sec:9.3]{B.3} use --- the free horizon may be replaced by an infinite past, with $z_\ast=\zeta_+$: 
\begin{equation}\label{eq:3.6p}
\begin{aligned}
U_\delta(a)&=\inf\Big\{\tfrac12\!\int_{-\infty}^0\|u\|^2\,dt:\ m\in C((-\infty,0];V),\\
&\ \dot m=G(m)\big[\rho(m)+Q_\delta^{1/2}u\big],\\
&\ m(-\infty)=\zeta_+,\ m(0)=a\Big\} ;
\end{aligned}
\end{equation}
the change of variables $\tilde m(t)=m(t+T)$ of [\hyperref[ref:BHL19]{BHL19}, App. C, Prop. C.3] together with [\hyperref[ref:BHL19]{BHL19}, App. C, Prop. C.10] relates the free-horizon form \hyperref[eq:3.6]{(\ref*{eq:3.6})} to the infinite-past representation \hyperref[eq:3.6p]{(\ref*{eq:3.6p})} (both cited statements are for the unregularised problem, $Q=I$); for $\delta\in(0,1]$ we do not need, and do not claim, the quantitative comparison --- no result below depends on it. We write $U_0$ for the object at $\delta=0$, i.e. with $I^0_{z_\ast}$ of \hyperref[eq:3.5]{(\ref*{eq:3.5})}; then the identification $U_0=c_\ast E$ holds \emph{exactly} on the basin of attraction ([\hyperref[ref:BHL19]{BHL19}, Thm 5.2] in the reverse-time normalisation \hyperref[eq:3.6p]{(\ref*{eq:3.6p})}, with the conversion \hyperref[eq:12.1]{(\ref*{eq:12.1})}; \hyperref[rem:3.3]{Remark \ref*{rem:3.3}}). Note $U_\delta(a)=+\infty$ is possible. When $z_\ast$ is an equilibrium of the free dynamics, $U^{z_\ast,T}_\delta\downarrow U_\delta$ as $T\uparrow\infty$ (one may wait at $z_\ast$ at zero cost; \S{}\hyperref[sec:12.1]{11.1}).

\end{definition}

\begin{remark}\label{rem:3.2} $U_\delta$ is the state-space quasi-potential (the rate function evaluated on the set of paths terminating at $a$), not the path-space rate function. The target of the programme, $\lim_{\delta\to0}U_\delta=U$, is a statement about $U_\delta$, and it is at this level that the domain obstruction of \S{}\hyperref[sec:9]{B} appears. The path-space rate functions $I^\delta$ converge $\Gamma$-wise (\hyperref[thm:8.3]{Theorem \ref*{thm:8.3}}) \emph{without} any domain issue, because their domain --- all of $C([0,T];H)$ --- does not depend on $\delta$.

\end{remark}

\begin{remark}[{the two quasi-potentials at $\delta=0$}]\label{rem:3.3} There are two natural normalisations at $\delta=0$. The one used here is \hyperref[eq:3.5]{(\ref*{eq:3.5})}, the $\varepsilon\to0$ rate function of the regularised problem with $Q_0=I$. The other is the quasi-potential of [\hyperref[ref:BHL19]{BHL19}, Def. 5.1], defined through reverse-time (infinite-past) paths,
\[\begin{aligned}
U_{\mathrm{BHL}}(a)&=\inf\Big\{\frac{\lambda_1^2+\lambda_2^2}{4\lambda_2}\int_{-\infty}^0\|h\|_H^2\,dt:\ m\in C((-\infty,0];V)\\
&\ \text{a weak solution of the deterministic LLG}\Big\} .
\end{aligned}\]
By [\hyperref[ref:BHL19]{BHL19}, App. C, Prop. C.10] the infinite-past form and the free-horizon form \hyperref[eq:3.6]{(\ref*{eq:3.6})} coincide, so the two agree \emph{exactly} up to the constant, $U_0=c_\ast U_{\mathrm{BHL}}$ with $c_\ast=2\lambda_2/(\lambda_1^2+\lambda_2^2)$; and [\hyperref[ref:BHL19]{BHL19}, Thm 5.2] gives $U_{\mathrm{BHL}}=E$ on the basin of attraction. Hence
\begin{equation}\label{eq:3.7}
U_0=c_\ast E\qquad\text{on the basin},\qquad c_\ast=\frac{2\lambda_2}{\lambda_1^2+\lambda_2^2}.
\end{equation}
The constant is fixed by the normalisation of the noise. At $\delta=0$ the two controls coincide, $u=Q_0^{1/2}h=h$, so the cost $\tfrac12\int\|u\|^2$ reduces to $\tfrac12\int\|h\|^2$; for $\delta>0$, however, $Q_\delta^{1/2}$ is injective but \emph{not} an isometry and the two controls differ, the discrepancy being carried by the weight \hyperref[eq:W]{(W)} rather than by the path. The ``physical'' dissipation rate appearing in [\hyperref[ref:BHL19]{BHL19}] is on the other hand $\lambda_2\|P_{m^\perp}h\|_H^2$, and the two agree when $\lambda_1=\lambda_2$. We adopt the normalisation
\begin{equation}\label{eq:3.8}
(\lambda_1,\lambda_2)=(\pm1,1) \qquad\Longrightarrow\qquad c_\ast=1,\qquad U_0=E\ \text{literally},
\end{equation}
as the standard one for the statements of the barrier \hyperref[target:3.9]{(\ref*{target:3.9})} and of the escape problem; the general constants are kept in every estimate. See \S{}\hyperref[sec:12]{11} for the bookkeeping.

\end{remark}

\subsection{The target statement}\phantomsection\label{sec:3.4}

The target of this paper is the following, in the precise form that is actually true (see \S{}\hyperref[sec:9]{B} for why the naive pointwise form is not: \hyperref[thm:9.7]{Theorem \ref*{thm:9.7}} disproves it in the exactly solvable model \hyperref[eq:9.1]{(\ref*{eq:9.1})}, while the corresponding statement for the SLLG is left open, \S{}\hyperref[sec:13]{12}).

\eqhead{target:3.9}{Target}{} As $\delta\downarrow0$,
\[\begin{aligned}
&\text{(i) }U_\delta\longrightarrow U_0\ \text{in the }\Gamma\text{-sense on }H\\
&\ \text{(liminf along sequences in the basin, recovery from all of }H\text{);}\\
&\qquad\text{(ii) }\lim_{\delta\to0}\inf_{\partial D}U_\delta=\inf_{\partial D}U_0=c_\ast\inf_{\partial D}E .
\end{aligned}\]

Part (i) is proved in \S{}\hyperref[sec:10]{9} (\hyperref[thm:10.2]{Theorem \ref*{thm:10.2}}), part (ii) in \S{}\hyperref[sec:11]{10} (\hyperref[thm:11.6]{Theorem \ref*{thm:11.6}}), and the obstruction that rules out the pointwise replacement of (i) is proved in \S{}\hyperref[sec:9]{B} (\hyperref[thm:9.7]{Theorem \ref*{thm:9.7}}, in the exactly solvable model \hyperref[eq:9.1]{(\ref*{eq:9.1})}). The quantity $\inf_{\partial D}U_0=c_\ast\inf_{\partial D}E$ is the \textbf{Néel--Arrhenius barrier}: it is the exponent in $\varepsilon\log\mathbb E\tau$ for the escape from $D$, and \hyperref[target:3.9]{(\ref*{target:3.9})}(ii) says that this barrier is stable under the noise regularisation.

\begin{remark}[{why $\Gamma$-convergence is the right statement, and what it gives}]\label{rem:3.4} For a sequence of rate functions, the object that is stable under perturbations of the dynamics and that controls the asymptotics of $\inf_FU_\delta$ for closed $F$ is the $\Gamma$-limit, not the pointwise limit. Concretely, \hyperref[target:3.9]{(\ref*{target:3.9})}(i) gives the two one-sided bounds of \S{}\hyperref[sec:10.3]{9.3}. For every \textbf{closed} $F\subset H$,
\[\liminf_{\delta\to0}\ \inf_{F}U_\delta\ \ge\ \inf_{F}U_0 ,\]
and for every \textbf{open} $G\subset H$,
\[\limsup_{\delta\to0}\ \inf_{G}U_\delta\ \le\ \inf_{G}U_0 .\]
The asymmetry is essential: the recovery sequence in the $\Gamma$-statement \emph{moves} the endpoint (\S{}\hyperref[sec:10.3]{9.3}, \hyperref[rem:10.5]{Remark \ref*{rem:10.5}}), so no upper bound holds at a prescribed closed set --- in particular not at a singleton $\{a\}$, which would be the pointwise statement, false by \hyperref[thm:9.7]{Theorem \ref*{thm:9.7}}. For the closed set $\partial D$ of the exit problem the upper bound therefore requires an additional fixed-endpoint input; this is supplied in \S{}\hyperref[sec:11.1]{10.1} with no further hypothesis, by \hyperref[lem:11.5b]{Lemma \ref*{lem:11.5b}} (the barrier always contains a spatially constant state) together with \hyperref[thm:11.6]{Theorem \ref*{thm:11.6}}, which give the two-sided convergence $\inf_{\partial D}U_\delta\to\inf_{\partial D}U_0=c_\ast\inf_{\partial D}E$.

\end{remark}

\section{An abstract uniform large deviation principle}\phantomsection\label{sec:4}

This section is the abstract engine. It is stated for a general family of controlled stochastic equations, and specialised to the SLLG in \S{}\hyperref[sec:5]{5}--\S{}\hyperref[sec:6]{6}. We follow the Budhiraja--Dupuis--Salins framework [\hyperref[ref:BDM08]{BDM08}, \hyperref[ref:BDS19]{BDS19}]: a variational representation of the Laplace functional, plus three verifiable hypotheses and one compactness lemma. Every step of every proof is written out, so that the reader can check the three hypotheses \hyperref[asm:H1]{(H1)}--\hyperref[asm:H3]{(H3)} once and for all and then forget the SLLG-specific analysis of \S{}\hyperref[sec:5]{5}. The relation between this abstract theory and the uniform LDP of [\hyperref[ref:BDS19]{BDS19}] is spelled out in \hyperref[rem:4.6b]{Remark \ref*{rem:4.6b}}. The reading path: a reader willing to accept the framework may take \hyperref[thm:4.5]{Theorem \ref*{thm:4.5}} as given and go straight to \S{}\hyperref[sec:5.2]{5.2}--\S{}\hyperref[sec:5.4]{5.4}, where \hyperref[asm:H1]{(H1)}--\hyperref[asm:H3]{(H3)} are verified for the SLLG, and then to \S{}\hyperref[sec:6]{6}; apart from \hyperref[asm:V]{(V)} and \hyperref[lem:4.1]{Lemma \ref*{lem:4.1}}, nothing in \S{}\hyperref[sec:4.2]{4.2}--\S{}\hyperref[sec:4.4]{4.4} is used after \hyperref[thm:4.5]{Theorem \ref*{thm:4.5}}. The scope: this is the framework of [\hyperref[ref:BDM08]{BDM08}, \hyperref[ref:BDS19]{BDS19}] in the form the SLLG needs, not a new theory; its only increment is the uniformity class of \hyperref[prop:4.7]{Proposition \ref*{prop:4.7}} (\hyperref[rem:4.6b]{Remark \ref*{rem:4.6b}}).

Throughout \S{}\hyperref[sec:4]{4}, $H$ is a real separable Hilbert space (the reader may keep in mind the $H$ of \S{}\hyperref[sec:2.1]{2.1}), $W$ is a $Q$-Wiener process on $H$ with $\operatorname{tr}Q<\infty$, and $D\subset H$ is a fixed set with $\bar D$ compact. For $\varepsilon>0$, $z\in D$, $u\in\mathcal P_2$, let $X^\varepsilon_{z,u}$ be a $C([0,T];H)$-valued random element (the ``controlled process''), with $X^\varepsilon_z:=X^\varepsilon_{z,0}$. Fix $f\in C_b(C([0,T];H))$ and write $\|f\|_\infty=\sup|f|$. All random elements are defined on a common probability space $(\Omega,\mathcal F,\mathbb P)$ carrying $W$, and $\mathcal F^W$ denotes the $\mathbb P$-completion of the filtration generated by $W$.

\subsection{Hypotheses}\phantomsection\label{sec:4.1}

\phantomsection\label{asm:H1}\textbf{(H1) Well-posedness and measurability.} For every $\varepsilon\ge0$, $z\in D$ and $u\in\mathcal P_2$ (for $\varepsilon=0$, deterministic $u\in L^2(0,T;H)$ suffices), $X^\varepsilon_{z,u}\in C([0,T];H)$ is well defined and pathwise unique, and for each fixed $z$ there is a Borel map $\Phi^\varepsilon_z$ with $X^\varepsilon_{z,u}=\Phi^\varepsilon_z(u,W)$ a.s., for every such $u$ on every stochastic basis carrying $W$; for $\varepsilon=0$, $X^0_{z,u}$ is deterministic.

\phantomsection\label{asm:H2}\textbf{(H2) Uniform compactness.} For every $N<\infty$, the family $\{X^\varepsilon_{z,u}:0<\varepsilon\le1,\ u\in\mathcal P_2^N,\ z\in D\}$ is tight in $C([0,T];H)$.

\phantomsection\label{asm:H3}\textbf{(H3) Uniform skeleton convergence.} For every $N<\infty$:  (i) for every $\eta>0$, $\displaystyle\sup_{z\in D}\ \sup_{u}\ \mathbb P\Big(\big\|X^\varepsilon_{z,u}-X^0_{z,u}\big\|_{C([0,T];H)}>\eta\Big)\xrightarrow[\varepsilon\to0]{}0$, where the inner supremum ranges over all $S^N$-valued $W$-predictable controls on any stochastic basis carrying $W$ (equivalently: for every such basis, every such control and every $z\in D$ the probability is at most $\delta_\varepsilon(\eta)$, with $\delta_\varepsilon(\eta)\to0$ not depending on the basis, the control or $z$);  (ii) $(z,u)\mapsto X^0_{z,u}$ is continuous on $\bar D\times S^N_{\mathrm{weak}}$ with values in $C([0,T];H)$; equivalently, with $d_w$ any metric inducing the weak topology on $S^N$, there is a modulus $\omega$ with $\omega(0,0)=0$, continuous at $(0,0)$, such that
\[\big\|X^0_{z,u}-X^0_{z',u'}\big\|_{C([0,T];H)}\le\omega\big(\|z-z'\|_H,\ d_w(u,u')\big)\qquad\forall\,(z,u),(z',u')\in\bar D\times S^N .\]

The modulus form in \hyperref[asm:H3]{(H3)}(ii) is automatic from continuity: $\bar D$ is compact and $S^N$ is weakly compact (\hyperref[lem:4.1]{Lemma \ref*{lem:4.1}}), so $\bar D\times S^N_{\mathrm{weak}}$ is compact and Heine--Cantor applies. We state it in modulus form because the uniformity argument of \hyperref[thm:4.5]{Theorem \ref*{thm:4.5}} and the almost-sure applications of \hyperref[lem:4.3]{Lemma \ref*{lem:4.3}} use it.

\textbf{\hyperref[asm:H3]{(H3)}(i) is a statement in probability, uniform over the data and over the controls.} Note carefully what \hyperref[asm:H3]{(H3)}(i) says. The supremum over the controls is taken \emph{outside} the probability: for every admissible control on every stochastic basis carrying $W$, the probability that the controlled solution deviates from the skeleton by more than $\eta$ is at most one and the same $\delta_\varepsilon(\eta)\to0$. This is a \textbf{distributional} statement --- quantified over the control process and its coupling with $W$, not over realisations of a fixed family --- and it is exactly the form in which the vanishing-noise estimate is proved for the SLLG (\hyperref[thm:5.3]{Theorem \ref*{thm:5.3}}(i)). It is also the form needed downstream: in \hyperref[lem:4.3]{Lemma \ref*{lem:4.3}}(4) the control fed into \hyperref[asm:H3]{(H3)}(i) is the \emph{random}, $\mathcal F^{\bar W_n}$-predictable $\bar u_n$ of the Skorohod realisation, which the present formulation covers directly. No claim is made with the supremum \emph{inside} the probability: that stronger statement is not available in general (the error can concentrate at control parameters that move with the realisation), and it is nowhere used.

\subsection{The variational representation and localization}\phantomsection\label{sec:4.2}

\phantomsection\label{asm:V}\textbf{(V) Variational representation.} For every $z\in D$, $\varepsilon>0$,
\[-\varepsilon\log\mathbb E\big[e^{-f(X^\varepsilon_z)/\varepsilon}\big]=\inf_{u\in\mathcal P_2}\mathbb E\Big[\tfrac12\!\!\int_0^T\|u(s)\|_H^2\,ds+f\big(X^{\varepsilon}_{z,u}\big)\Big] .\]

\phantomsection\label{asm:L}\textbf{(L) Localization lemma.} Let $u\in\mathcal P_2$ with $\mathbb E\big[\int_0^T\|u\|^2\big]\le C_0<\infty$, and for $M>0$ set
\[\tau:=\inf\Big\{t\ge0:\int_0^t\|u(s)\|_H^2\,ds\ge M\Big\}\wedge T,\qquad \bar u:=u\mathbf 1_{[0,\tau]} .\]
Then $\bar u\in\mathcal P_2^M$ and
\[\mathbb E\Big[\tfrac12\!\!\int_0^T\|\bar u\|^2+f\big(X^{\varepsilon}_{z,\bar u}\big)\Big]\le\mathbb E\Big[\tfrac12\!\!\int_0^T\|u\|^2+f\big(X^{\varepsilon}_{z,u}\big)\Big]+2\|f\|_\infty\frac{C_0}{M} .\]

\emph{Proof.} \textbf{Predictability and the energy bound.} The map $t\mapsto\int_0^t\|u\|_H^2ds$ is continuous and adapted, so $\tau$ is a stopping time; hence $\bar u=u\mathbf 1_{[0,\tau]}$ is predictable and $H$-valued. By definition of $\tau$,
\[\int_0^T\|\bar u(s)\|_H^2\,ds=\int_0^{\tau}\|u(s)\|_H^2\,ds\le M\qquad\text{a.s.},\]
so $\bar u\in\mathcal P_2^M$.

\textbf{Comparison.} Put $\Xi_v:=\tfrac12\int_0^T\|v\|^2dt+f(X^\varepsilon_{z,v})$ for $v\in\mathcal P_2$. Since $\bar u=u$ on the event $\{\tau=T\}$,
\[\mathbb E\big[\Xi_{\bar u}\big]-\mathbb E\big[\Xi_u\big]=\mathbb E\big[\mathbf 1_{\{\tau<T\}}\big(\Xi_{\bar u}-\Xi_u\big)\big] .\]
On $\{\tau<T\}$ the energy parts differ by
\[\tfrac12\!\!\int_0^T\|\bar u\|^2dt-\tfrac12\!\!\int_0^T\|u\|^2dt=-\tfrac12\!\!\int_{\tau}^{T}\|u(s)\|_H^2\,ds\le0 ,\]
while the terminal parts differ by at most $2\|f\|_\infty$ in absolute value, because $|f(X^\varepsilon_{z,\bar u})-f(X^\varepsilon_{z,u})|\le2\|f\|_\infty$ pointwise. Hence $\Xi_{\bar u}-\Xi_u\le2\|f\|_\infty$ on $\{\tau<T\}$, and
\[\mathbb E\big[\mathbf 1_{\{\tau<T\}}(\Xi_{\bar u}-\Xi_u)\big]\le2\|f\|_\infty\,\mathbb P(\tau<T) .\]
Finally $\{\tau<T\}\subseteq\{\int_0^T\|u\|^2dt\ge M\}$, so by Markov
\[\mathbb P(\tau<T)\le\frac{1}{M}\,\mathbb E\Big[\int_0^T\|u(s)\|_H^2ds\Big]\le\frac{C_0}{M} . \qquad\square\]

\subsection{Two preliminary lemmas}\phantomsection\label{sec:4.3}

\begin{lemma}[{weak compactness of the control ball}]\label{lem:4.1} For every $N<\infty$, $S^N=\{u\in L^2(0,T;H):\int_0^T\|u\|_H^2\le N\}$ is weakly compact and its weak topology is metrizable; hence $(S^N,\mathrm{weak})$ is a compact metric space.

\end{lemma}

\emph{Proof.} $L^2(0,T;H)$ is a separable reflexive Banach space, and $S^N$ is its closed ball of radius $\sqrt N$, hence convex and norm-closed. By Banach--Alaoglu $S^N$ is weakly compact, and by reflexivity it is weakly closed; since $L^2(0,T;H)$ is separable, its weak topology is metrizable on the bounded set $S^N$ (for instance by
\[d_w(u,v):=\sum_{k\ge1}2^{-k}\frac{\big|\int_0^T\langle u-v,\varphi_k\rangle_Hdt\big|}{1+\big|\int_0^T\langle u-v,\varphi_k\rangle_Hdt\big|}\]
with $\{\varphi_k\}$ dense in $L^2(0,T;H)$). $\square$

\begin{lemma}[{measurability transfer}]\label{lem:4.2} Let $u\in\mathcal P_2$ be $\mathcal F^W$-predictable and let $(\bar u,\bar W)\overset d=(u,W)$ in $L^2(0,T;H)\times C([0,T];H)$. Then $\bar u$ admits an $\mathcal F^{\bar W}$-predictable modification.

\end{lemma}

\emph{Proof.} \textbf{(1) Factorization of predictability.} Let $\mathcal P^{\mathrm{path}}$ be the predictable $\sigma$-algebra on $[0,T]\times C([0,T];H)$ generated by the maps $(t,w)\mapsto w(t\wedge s)$, $s\in[0,T]$ (equivalently, generated by the adapted continuous processes on the path space). By the standard factorization lemma for predictable functionals of a continuous process, $\mathcal P^W=(\mathrm{ev}_W)^{-1}\mathcal P^{\mathrm{path}}$, where $\mathrm{ev}_W(\omega):=W_\cdot(\omega)\in C([0,T];H)$. Consequently there exists a $\mathcal P^{\mathrm{path}}$-measurable $\Psi:[0,T]\times C([0,T];H)\to H$ with
\[u(t,\omega)=\Psi\big(t,W_\cdot(\omega)\big)\qquad\text{for a.e. }(t,\omega)\in[0,T]\times\Omega .\]
By $\mathcal P^{\mathrm{path}}$-measurability we may (and shall) choose $\Psi$ non-anticipating: $\Psi(t,w)=\Psi(t,w_{\cdot\wedge t})$ for all $t$.

\textbf{(2) Construction.} Define $\bar u'(t,\omega):=\Psi\big(t,\bar W_{\cdot\wedge t}(\omega)\big)$. For each $t$, the map $\omega\mapsto\bar W_{\cdot\wedge t}(\omega)$ is $\mathcal F^{\bar W}_t$-measurable with values in $C([0,T];H)$; since $\Psi$ is $\mathcal P^{\mathrm{path}}$-measurable, $\bar u'$ is $\mathcal F^{\bar W}$-predictable, and clearly $H$-valued.

\textbf{(3) Identification.} The map $w\mapsto\Psi(\cdot,w)\in L^2(0,T;H)$ is Borel: for every $e\in H$ the scalar function $(t,w)\mapsto\langle\Psi(t,w),e\rangle_H$ is $\mathcal P^{\mathrm{path}}$-measurable, hence (by Fubini) $w\mapsto\int_0^T\langle\Psi(t,w),e\rangle_Hdt$ is Borel; taking $e$ in a countable dense set characterises Bochner measurability, so $w\mapsto\Psi(\cdot,w)$ is Borel. Thus $u=\Phi(W)$ for a Borel $\Phi:C([0,T];H)\to L^2(0,T;H)$. Since $(\bar u,\bar W)\overset d=(u,W)$, also $\bar u=\Phi(\bar W)$ a.s. (the graph of $\Phi$ carries the law of $(u,W)$, hence that of $(\bar u,\bar W)$). Combining with (2),
\[\bar u=\Phi(\bar W)=\Psi(\cdot,\bar W)=\bar u'\qquad\text{in }L^2(0,T;H)\ \text{a.s.},\]
so $\bar u'$ is the required modification. $\square$

\subsection{The skeleton identification lemma}\phantomsection\label{sec:4.4}

\begin{lemma}[{every cluster law is supported on the skeleton graph}]\label{lem:4.3} Assume \hyperref[asm:H1]{(H1)}--\hyperref[asm:H3]{(H3)}, let $z_n\in D$ with $z_n\to z$ in $H$, $\varepsilon_n\downarrow0$, $u_n\in\mathcal P_2^N$, and suppose $(u_n,X^{\varepsilon_n}_{z_n,u_n})\Rightarrow\mu$ in $S^N_{\mathrm{weak}}\times C([0,T];H)$. Then $\mu=\mathrm{Law}(\bar u,X^0_{z,\bar u})$ for some $S^N$-valued random element $\bar u$.

\end{lemma}

\emph{Proof.} \textbf{(1) Tightness of the triples.} Consider $\Theta_n:=(u_n,X^{\varepsilon_n}_{z_n,u_n},W)$ in the Polish space $\mathcal Z:=S^N_{\mathrm{weak}}\times C([0,T];H)\times C([0,T];H)$. Fix $\eta>0$.

\begin{itemize}
\item \emph{First marginal.} $S^N$ is compact (\hyperref[lem:4.1]{Lemma \ref*{lem:4.1}}), so every probability measure on it is tight; indeed $\mathrm{Law}(u_n)(S^N)=1$, because $u_n\in S^N$ a.s. by definition of $\mathcal P_2^N$.
\item \emph{Second marginal.} By \hyperref[asm:H2]{(H2)} there is a compact $K_\eta\subset C([0,T];H)$ with
\[\sup_{0<\varepsilon\le1}\sup_{u\in\mathcal P_2^N}\sup_{z\in D}\mathbb P\big(X^{\varepsilon}_{z,u}\notin K_\eta\big)\le\eta ;\]
\end{itemize}
taking $\varepsilon=\varepsilon_n$, $u=u_n$, $z=z_n$ gives $\sup_n\mathbb P(X^{\varepsilon_n}_{z_n,u_n}\notin K_\eta)\le\eta$.
\begin{itemize}
\item \emph{Third marginal.} $\mathrm{Law}(W)$ is a single probability measure on the Polish space $C([0,T];H)$, hence tight: there is a compact $K'_\eta$ with $\mathbb P(W\notin K'_\eta)\le\eta$.
\end{itemize}

The product $\mathcal K_\eta:=S^N\times K_\eta\times K'_\eta$ is compact, hence closed, hence Borel in $\mathcal Z$; and by the union bound
\[\mathbb P\big(\Theta_n\in\mathcal K_\eta\big)\ge1-\mathbb P\big(X^{\varepsilon_n}_{z_n,u_n}\notin K_\eta\big)-\mathbb P\big(W\notin K'_\eta\big)\ge1-2\eta\qquad\text{for all }n .\]
Note that $\mathcal K_\eta$ is a genuine Borel set, so no outer-measure convention is needed: every law here is a Borel probability measure on the Polish space $\mathcal Z$, and $S^N\times K_\eta\times K'_\eta$ is closed. So $(\Theta_n)$ is tight, and by Prokhorov there is a subsequential limit law $\mu$ on $\mathcal Z$; passing to that subsequence we may assume $\Theta_n\Rightarrow\mu$, so in particular $(u_n,X^{\varepsilon_n}_{z_n,u_n})\Rightarrow\mu$ on $S^N_{\mathrm{weak}}\times C([0,T];H)$.

\textbf{(2) Skorohod.} By the Skorohod representation theorem on the Polish space $\mathcal Z$, there are a probability space $(\bar\Omega,\bar{\mathcal F},\bar P)$ and random elements
\[(\bar u_n,\bar X_n,\bar W_n)\overset d=(u_n,X^{\varepsilon_n}_{z_n,u_n},W),\qquad(\bar u,\bar X,\bar W)\sim\mu ,\]
such that $(\bar u_n,\bar X_n,\bar W_n)\to(\bar u,\bar X,\bar W)$ $\bar P$-a.s. in $\mathcal Z$ (hence $\bar u_n\to\bar u$ weakly in $S^N$, $\bar X_n\to\bar X$ in $C([0,T];H)$, $\bar W_n\to\bar W$ in $C([0,T];H)$, a.s.).

\textbf{(3) Recovering the functional relation.} Since $(\bar u_n,\bar W_n)\overset d=(u_n,W)$, \hyperref[lem:4.2]{Lemma \ref*{lem:4.2}} gives an $\mathcal F^{\bar W_n}$-predictable modification of $\bar u_n$, which we denote again by $\bar u_n$. By \hyperref[asm:H1]{(H1)} the controlled equation driven by $\bar W_n$ with the predictable control $\bar u_n$ is well posed; denote its solution by $\Phi^{\varepsilon_n}_{z_n}(\bar u_n,\bar W_n)$, and denote by $\Phi^{\varepsilon_n}_{z_n}(u,w)$ the corresponding Borel map. The graph
\[\Gamma_n:=\big\{(u,y,w)\in\mathcal Z:\ y=\Phi^{\varepsilon_n}_{z_n}(u,w)\big\}\]
is Borel (a continuous graph over the Borel set $\{(u,w)\}$ intersected with $\mathcal Z$), and it carries full measure under $\mathrm{Law}(u_n,X^{\varepsilon_n}_{z_n,u_n},W)$ by \hyperref[asm:H1]{(H1)}. Since $(\bar u_n,\bar X_n,\bar W_n)$ has the same law, $\bar P((\bar u_n,\bar X_n,\bar W_n)\in\Gamma_n)=1$, i.e.
\[\bar X_n=\Phi^{\varepsilon_n}_{z_n}(\bar u_n,\bar W_n)\qquad\bar P\text{-a.s.}\]

\textbf{(4) Comparison with the skeleton.} By step (3), $\bar u_n$ is $\mathcal F^{\bar W_n}$-predictable, and the Skorohod space of $(\bar u_n,\bar X_n,\bar W_n)$ is a stochastic basis carrying $\bar W_n$; hence the pair $(z_n,\bar u_n)$ is admissible data for \hyperref[asm:H3]{(H3)}(i), and since $\bar X_n=\Phi^{\varepsilon_n}_{z_n}(\bar u_n,\bar W_n)$ by (3), \hyperref[asm:H3]{(H3)}(i) gives at once, for every $\eta>0$,
\[\bar P\Big(\big\|\Phi^{\varepsilon_n}_{z_n}(\bar u_n,\bar W_n)-X^0_{z_n,\bar u_n}\big\|_{C([0,T];H)}>\eta\Big)\le\delta_{\varepsilon_n}(\eta)\xrightarrow[\ n\to\infty\ ]{}\ 0 .\]
Combining with (3),
\[\bar X_n-X^0_{z_n,\bar u_n}\xrightarrow[\ n\to\infty\ ]{}\ 0\qquad\text{in }\bar P\text{-probability}.\]
On the other hand $\bar u_n\to\bar u$ weakly and $z_n\to z$ in $H$; by \hyperref[asm:H3]{(H3)}(ii) (with $u'\mapsto\bar u$, $z'\mapsto z$),
\[X^0_{z_n,\bar u_n}\to X^0_{z,\bar u}\qquad\bar P\text{-a.s.}\]
Combining the last two displays, $\bar X_n\to X^0_{z,\bar u}$ in $\bar P$-probability.

\textbf{(5) Identification.} By (2), $\bar X_n\to\bar X$ $\bar P$-a.s.; limits in probability are a.s. unique, so $\bar X=X^0_{z,\bar u}$ $\bar P$-a.s. Therefore the law of $(\bar u,\bar X)$ is $\mathrm{Law}(\bar u,X^0_{z,\bar u})$, and since $(\bar u,\bar X)\sim\mu$,
\[\mu=\mathrm{Law}\big(\bar u,X^0_{z,\bar u}\big) . \qquad\square\]

\subsection{The abstract theorem}\phantomsection\label{sec:4.5}

For $z\in\bar D$ define the \textbf{skeleton action} and the \textbf{value function}
\[\Phi(z,u):=\tfrac12\!\!\int_0^T\|u\|_H^2\,ds+f\big(X^0_{z,u}\big),\qquad \Lambda^0(z):=\inf_{u\in L^2(0,T;H)}\Phi(z,u) .\]

\begin{lemma}[{$\Lambda^0$ is attained on a fixed ball and is continuous}]\label{lem:4.4} Let $N_0:=4\|f\|_\infty+2$. Then for every $z\in\bar D$,
\[\Lambda^0(z)=\min_{u\in S^{N_0}}\Phi(z,u) ,\]
and $\Lambda^0:\bar D\to\mathbb R$ is continuous.

\end{lemma}

\emph{Proof.} \textbf{Attainment on $S^{N_0}$.} Testing $u\equiv0$ gives $\Lambda^0(z)\le\Phi(z,0)=f(X^0_{z,0})\le\|f\|_\infty$. If $\Phi(z,u)\le\Lambda^0(z)+1$ then
\[\tfrac12\!\!\int_0^T\|u\|_H^2ds=\Phi(z,u)-f\big(X^0_{z,u}\big)\le\Lambda^0(z)+1+\|f\|_\infty\le2\|f\|_\infty+1 ,\]
i.e. $u\in S^{N_0}$ with $N_0=4\|f\|_\infty+2$. Hence the infimum defining $\Lambda^0(z)$ is the infimum of $\Phi(z,\cdot)$ over the weakly compact set $S^{N_0}$. On $S^{N_0}$, the map $u\mapsto\Phi(z,u)$ is weakly lower semicontinuous: the first term $\tfrac12\int_0^T\|u\|^2ds$ is the square of a norm, hence weakly l.s.c.; the second term $f(X^0_{z,u})$ is weakly \emph{continuous} by \hyperref[asm:H3]{(H3)}(ii). A weakly l.s.c. function on a weakly compact set attains its minimum, so
\[\Lambda^0(z)=\min_{u\in S^{N_0}}\Phi(z,u) .\]

\textbf{Continuity: upper bound.} Let $z_n\to z$ in $\bar D$ and let $u^\ast\in S^{N_0}$ be a minimiser for $z$, i.e. $\Phi(z,u^\ast)=\Lambda^0(z)$. The \emph{same} $u^\ast$ is admissible for $z_n$, so
\[\Lambda^0(z_n)\le\Phi(z_n,u^\ast)=\tfrac12\!\!\int_0^T\|u^\ast\|^2ds+f\big(X^0_{z_n,u^\ast}\big)\xrightarrow[n\to\infty]{}\Phi(z,u^\ast)=\Lambda^0(z) ,\]
by \hyperref[asm:H3]{(H3)}(ii) with $u'=u^\ast$ and $f$ continuous.

\textbf{Continuity: lower bound.} Let $u_n\in S^{N_0}$ be a minimiser for $z_n$. By \hyperref[lem:4.1]{Lemma \ref*{lem:4.1}} pass to a subsequence with $u_{n_k}\rightharpoonup u\in S^{N_0}$. By \hyperref[asm:H3]{(H3)}(ii), $X^0_{z_{n_k},u_{n_k}}\to X^0_{z,u}$ in $C([0,T];H)$, so by continuity of $f$,
\[\liminf_{k\to\infty}\Lambda^0(z_{n_k})=\liminf_{k\to\infty}\Phi(z_{n_k},u_{n_k})\ge\tfrac12\!\!\int_0^T\|u\|^2ds+f\big(X^0_{z,u}\big)=\Phi(z,u)\ge\Lambda^0(z) ,\]
using the weak lower semicontinuity of the norm. Combining the two bounds, $\Lambda^0(z_n)\to\Lambda^0(z)$. $\square$

\begin{theorem}[{abstract uniform Laplace principle}]\label{thm:4.5} Assume \hyperref[asm:H1]{(H1)}--\hyperref[asm:H3]{(H3)} and \hyperref[asm:V]{(V)}, with $\bar D$ compact. With
\[V^\varepsilon(z):=-\varepsilon\log\mathbb E\big[e^{-f(X^\varepsilon_z)/\varepsilon}\big],\]
one has
\[\lim_{\varepsilon\to0}\ \sup_{z\in D}\big|V^\varepsilon(z)-\Lambda^0(z)\big|=0 .\]

\end{theorem}

\emph{Proof.} Two sequence bounds, then a compactness argument for the uniformity.

\textbf{Bound I ($\limsup\le\Lambda^0$; deterministic control, no Skorohod).} Let $z_n\in D$ with $z_n\to z\in\bar D$ and $\varepsilon_n\downarrow0$. Fix $\eta>0$ and choose $u^\ast\in S^{N_0}$ with $\Phi(z,u^\ast)\le\Lambda^0(z)+\eta$ (possible by \hyperref[lem:4.4]{Lemma \ref*{lem:4.4}}). By the variational representation \hyperref[asm:V]{(V)} applied with the \emph{deterministic} control $u^\ast\in\mathcal P_2$,
\[V^{\varepsilon_n}(z_n)\le\mathbb E\Big[\tfrac12\!\!\int_0^T\|u^\ast\|^2ds+f\big(X^{\varepsilon_n}_{z_n,u^\ast}\big)\Big]=\tfrac12\!\!\int_0^T\|u^\ast\|^2ds+\mathbb E\big[f\big(X^{\varepsilon_n}_{z_n,u^\ast}\big)\big] .\]
By \hyperref[asm:H3]{(H3)}(i) with the fixed control $u^\ast$ and \hyperref[asm:H3]{(H3)}(ii) with $z_n\to z$,
\[X^{\varepsilon_n}_{z_n,u^\ast}\longrightarrow X^0_{z,u^\ast}\qquad\text{in probability};\]
since $f$ is continuous and bounded, $f(X^{\varepsilon_n}_{z_n,u^\ast})\to f(X^0_{z,u^\ast})$ in probability, and the sequence is bounded by $\|f\|_\infty$, hence uniformly integrable; Vitali's theorem gives $\mathbb E[f(X^{\varepsilon_n}_{z_n,u^\ast})]\to f(X^0_{z,u^\ast})$. Therefore
\[\limsup_{n\to\infty}V^{\varepsilon_n}(z_n)\le\tfrac12\!\!\int_0^T\|u^\ast\|^2ds+f\big(X^0_{z,u^\ast}\big)=\Phi(z,u^\ast)\le\Lambda^0(z)+\eta ,\]
and letting $\eta\downarrow0$ gives $\limsup_nV^{\varepsilon_n}(z_n)\le\Lambda^0(z)$.

\textbf{Bound II ($\liminf\ge\Lambda^0$; random controls, Skorohod).} Let $z_n\in D$ with $z_n\to z\in\bar D$ and $\varepsilon_n\downarrow0$. It suffices to prove, for each fixed $M<\infty$,
\[\liminf_{n\to\infty}V^{\varepsilon_n}(z_n)\ \ge\ \Lambda^0(z)-2\|f\|_\infty\frac{C_0}{M},\qquad C_0:=4\|f\|_\infty+2 ,\]
and then let $M\to\infty$.

\textbf{Step 1 (near-optimal controls and a uniform energy bound).} By \hyperref[asm:V]{(V)} and the definition of the infimum, for each $n$ there is $u_n\in\mathcal P_2$ with
\begin{equation}\label{eq:4.1}
\mathbb E\Big[\tfrac12\!\!\int_0^T\|u_n\|^2ds+f\big(X^{\varepsilon_n}_{z_n,u_n}\big)\Big]\le V^{\varepsilon_n}(z_n)+\varepsilon_n .
\end{equation}
Testing $u\equiv0$ in \hyperref[asm:V]{(V)} gives $V^{\varepsilon_n}(z_n)\le f(X^{\varepsilon_n}_{z_n,0})\le\|f\|_\infty$. Substituting into \hyperref[eq:4.1]{(\ref*{eq:4.1})} and using $\mathbb E[f(X^{\varepsilon_n}_{z_n,u_n})]\ge-\|f\|_\infty$,
\[\tfrac12\,\mathbb E\!\!\int_0^T\|u_n\|^2ds\le V^{\varepsilon_n}(z_n)+\varepsilon_n-\mathbb E\big[f(X^{\varepsilon_n}_{z_n,u_n})\big]\le\|f\|_\infty+1+\|f\|_\infty\le2\|f\|_\infty+1 ,\]
so that
\begin{equation}\label{eq:4.2}
\mathbb E\!\!\int_0^T\|u_n\|_H^2ds\le4\|f\|_\infty+2=:C_0\qquad\text{for all }n .
\end{equation}

\textbf{Step 2 (localization).} Apply the localization lemma \hyperref[asm:L]{(L)} to the pair $(u_n,\bar u_n)$ with the constant $C_0$ of \hyperref[eq:4.2]{(\ref*{eq:4.2})} and the truncation level $M$. This gives $\bar u_n\in\mathcal P_2^M$ and
\begin{equation}\label{eq:4.3}
\begin{aligned}
\mathbb E\Big[\tfrac12\!\!\int_0^T\|\bar u_n\|^2+f\big(X^{\varepsilon_n}_{z_n,\bar u_n}\big)\Big]&\le\mathbb E\Big[\tfrac12\!\!\int_0^T\|u_n\|^2+f\big(X^{\varepsilon_n}_{z_n,u_n}\big)\Big]+2\|f\|_\infty\frac{C_0}{M}\\
&\le V^{\varepsilon_n}(z_n)+\varepsilon_n+2\|f\|_\infty\frac{C_0}{M} ,
\end{aligned}
\end{equation}
where the second inequality is \hyperref[eq:4.1]{(\ref*{eq:4.1})}.

\textbf{Step 3 (compactness).} We claim that the laws of $(\bar u_n,X^{\varepsilon_n}_{z_n,\bar u_n})$ are tight in the Polish space $S^M_{\mathrm{weak}}\times C([0,T];H)$. Indeed, fix $\eta>0$: the first marginal is supported on the compact set $S^M$ (\hyperref[lem:4.1]{Lemma \ref*{lem:4.1}}), and the second satisfies
\[\mathbb P\big(X^{\varepsilon_n}_{z_n,\bar u_n}\notin K_\eta\big)\le\eta\qquad\text{for all }n ,\]
where $K_\eta\subset C([0,T];H)$ is the compact set supplied by \hyperref[asm:H2]{(H2)} (applicable since $\bar u_n\in\mathcal P_2^M$). Since $S^M$ and $K_\eta$ are compact, the product $S^M\times K_\eta$ is compact, hence closed, hence Borel; therefore
\[\mathrm{Law}\big(\bar u_n,X^{\varepsilon_n}_{z_n,\bar u_n}\big)\big(S^M\times K_\eta\big)\ \ge\ \mathbb P\big(X^{\varepsilon_n}_{z_n,\bar u_n}\in K_\eta\big)\ \ge\ 1-\eta\qquad\text{for all }n .\]
(Again, $S^M\times K_\eta$ is a genuine Borel set, so this is an ordinary probability, not an outer measure.) By Prokhorov, along a subsequence (still indexed by $n$),
\[(\bar u_n,X^{\varepsilon_n}_{z_n,\bar u_n})\Longrightarrow\mu\qquad\text{in }S^M_{\mathrm{weak}}\times C([0,T];H) .\]
By \hyperref[lem:4.3]{Lemma \ref*{lem:4.3}} (with $N=M$, and $z_n\to z$), there is an $S^M$-valued random element $\bar u$ with
\begin{equation}\label{eq:4.4}
\mu=\mathrm{Law}\big(\bar u,X^0_{z,\bar u}\big),\qquad \bar u\in S^M\ \text{a.s.}
\end{equation}
Equivalently, $\mu$ is the \textbf{push-forward of $\mathrm{Law}(\bar u)$ under the skeleton map} $u\mapsto(u,X^0_{z,u})$: for every bounded Borel $\varphi$ on $S^M_{\mathrm{weak}}\times C([0,T];H)$,
\[\int_{S^M\times C([0,T];H)}\varphi\,d\mu\;=\;\int_{S^M}\varphi\big(u,X^0_{z,u}\big)\,d\,\mathrm{Law}(\bar u)(u).\]
In particular, taking $\varphi=\mathbf 1_{\Gamma^c}$ for the \textbf{skeleton graph}
\[\Gamma:=\Big\{(u,y)\in S^M_{\mathrm{weak}}\times C([0,T];H):\ y=X^0_{z,u}\Big\}\]
(which is Borel, indeed closed, because $u\mapsto X^0_{z,u}$ is continuous on $S^M_{\mathrm{weak}}$ by \hyperref[asm:H3]{(H3)}(ii) and $S^M_{\mathrm{weak}}$ is a compact metric space), and using that $(u,X^0_{z,u})\in\Gamma$ for every $u$, gives
\begin{equation}\label{eq:4.4p}
\mu(\Gamma^c)=\mathrm{Law}(\bar u)\big(\{u:(u,X^0_{z,u})\in\Gamma^c\}\big)=0,\qquad\text{i.e.}\qquad \mu(\Gamma)=1 .
\end{equation}

\textbf{Step 4 (Fatou on the Skorohod space).} Realise the convergence of Step 3 by Skorohod: on some probability space there are $(\bar u_n,\bar X_n)\overset d=(\bar u_n,X^{\varepsilon_n}_{z_n,\bar u_n})$ and $(\bar u,\bar X)\sim\mu$ with $(\bar u_n,\bar X_n)\to(\bar u,\bar X)$ a.s. By \hyperref[eq:4.4p]{(\ref*{eq:4.4p})} the cluster law $\mu$ is carried by the skeleton graph $\Gamma$, i.e. $\mu(\Gamma)=1$; since $(\bar u,\bar X)$ has law $\mu$, it follows that
\[\bar P\big((\bar u,\bar X)\in\Gamma\big)=\mu(\Gamma)=1,\qquad\text{that is,}\qquad \bar X=X^0_{z,\bar u}\quad\bar P\text{-a.s.}\]

By weak lower semicontinuity of the $L^2$-norm along $\bar u_n\rightharpoonup\bar u$ and continuity of $f$ along $\bar X_n\to\bar X$,
\[\liminf_{n\to\infty}\Big[\tfrac12\!\!\int_0^T\|\bar u_n\|^2ds+f(\bar X_n)\Big]\ \ge\ \tfrac12\!\!\int_0^T\|\bar u\|^2ds+f\big(X^0_{z,\bar u}\big)\qquad\text{a.s.}\]
Fatou's lemma (applied to the non-negative integrands $\tfrac12\int_0^T\|\bar u_n\|^2+f(\bar X_n)+\|f\|_\infty$, permissible because $f$ is bounded below), together with the distributional identities, gives
\[\liminf_{n\to\infty}\mathbb E\Big[\tfrac12\!\!\int_0^T\|\bar u_n\|^2+f\big(X^{\varepsilon_n}_{z_n,\bar u_n}\big)\Big]=\liminf_{n\to\infty}\mathbb E\Big[\tfrac12\!\!\int_0^T\|\bar u_n\|^2+f(\bar X_n)\Big]\ge\mathbb E\Big[\tfrac12\!\!\int_0^T\|\bar u\|^2+f\big(X^0_{z,\bar u}\big)\Big] .\]
The last expectation is $\ge\Lambda^0(z)$, because $\bar u\in L^2(0,T;H)$ a.s. is an admissible control and $\Lambda^0(z)=\inf_u\Phi(z,u)$.

\textbf{Step 5 (assembling).} Combining \hyperref[eq:4.3]{(\ref*{eq:4.3})} and Step 4,
\[\liminf_{n\to\infty}V^{\varepsilon_n}(z_n)\ \ge\ \liminf_{n\to\infty}\mathbb E\Big[\tfrac12\!\!\int_0^T\|\bar u_n\|^2+f\big(X^{\varepsilon_n}_{z_n,\bar u_n}\big)\Big]-2\|f\|_\infty\frac{C_0}{M}\ \ge\ \Lambda^0(z)-2\|f\|_\infty\frac{C_0}{M} .\]
Letting $M\to\infty$ gives $\liminf_nV^{\varepsilon_n}(z_n)\ge\Lambda^0(z)$.

\textbf{Uniformity.} Suppose the convergence is not uniform. Then there are $\eta>0$, $\varepsilon_n\downarrow0$ and $z_n\in D$ with
\[\big|V^{\varepsilon_n}(z_n)-\Lambda^0(z_n)\big|\ge\eta\qquad\text{for all }n .\]
By compactness of $\bar D$ (assumption of the theorem), pass to a subsequence with $z_{n_k}\to z\in\bar D$. By \hyperref[lem:4.4]{Lemma \ref*{lem:4.4}}, $\Lambda^0(z_{n_k})\to\Lambda^0(z)$. If $V^{\varepsilon_{n_k}}(z_{n_k})\ge\Lambda^0(z_{n_k})+\eta$ for infinitely many $k$, then along that subsequence $\limsup_kV^{\varepsilon_{n_k}}(z_{n_k})\ge\Lambda^0(z)+\eta$, contradicting Bound I applied to $(z_{n_k},\varepsilon_{n_k})$ (whose limit is $z$). If instead $V^{\varepsilon_{n_k}}(z_{n_k})\le\Lambda^0(z_{n_k})-\eta$ for infinitely many $k$, then $\liminf_kV^{\varepsilon_{n_k}}(z_{n_k})\le\Lambda^0(z)-\eta$, contradicting Bound II. In either case we obtain a contradiction, so the convergence is uniform over $z\in D$. $\square$

\subsection{From the Laplace principle to the LDP}\phantomsection\label{sec:4.6}

The uniform Laplace principle is the statement $\sup_{z\in D}|V^\varepsilon(z)-\Lambda^0(z)|\to0$; via the standard equivalence it is the same as the LDP for the \emph{exponential functionals}. To obtain the LDP for the laws themselves one needs exponential tightness; this is proved here rather than cited.

\begin{definition}[{uniform LDP; Freidlin--Wentzell form}]\label{def:4.6} Let $E$ be a Polish space, let $\{Y^\varepsilon_z\}_{\varepsilon>0,z\in D}$ be $E$-valued random elements and let $\{J_z\}_{z\in D}$ be rate functions (lower semicontinuous, with compact level sets). Write
\[\Phi_z(s):=\big\{x\in E:\ J_z(x)\le s\big\}\qquad(s\ge0)\]
for the level sets of $J_z$. We say that $\{Y^\varepsilon_z\}$ satisfies the \textbf{uniform LDP} with rate functions $\{J_z\}$, uniformly over $z\in D$, if for every $s_0<\infty$ and every $\delta>0$
\[\liminf_{\varepsilon\to0}\ \inf_{z\in D}\ \inf_{x\in\Phi_z(s_0)}\Big(\varepsilon\log\mathbb P\big(\|Y^\varepsilon_z-x\|_E<\delta\big)+J_z(x)\Big)\ \ge\ 0,\]
\[\limsup_{\varepsilon\to0}\ \sup_{z\in D}\ \sup_{s\in[0,s_0]}\Big(\varepsilon\log\mathbb P\big(\operatorname{dist}_E(Y^\varepsilon_z,\Phi_z(s))\ge\delta\big)+s\Big)\ \le\ 0 .\]
In words: \textbf{(i)} the state $x$ is reached with probability at least $e^{-(J_z(x)+o(1))/\varepsilon}$, uniformly in $z\in D$ and in $x\in\Phi_z(s_0)$; \textbf{(ii)} deviating by a distance $\delta$ from the level set $\Phi_z(s)$ costs at least $s$, uniformly in $z\in D$ and $s\in[0,s_0]$.

\end{definition}

\begin{remark}[{the standard form; the fixed-set consequences}]\label{rem:4.6a} Three comments.

\emph{(i) This is the classical formulation, and it needs no compactness of the index set.} For a single family the two displays above are the Freidlin--Wentzell form of the LDP ([\hyperref[ref:FW98]{FW98}, \S{}3.3]; [\hyperref[ref:DZ98]{DZ98}, Thm 3.3.3]); the uniform-in-the-index version is [\hyperref[ref:Sal19]{Sal19}, Def. 2.1]. Two features matter here. First, \textbf{no topology and no compactness on $D$ is required}, which is essential because $D$ is only $V$-bounded, hence pre-compact but not compact in the infinite-dimensional $H$. Second, this form is \emph{not} equivalent to the naive fixed-set statement over such a $D$: the Dembo--Zeitouni bound $\ge-\sup_{z\in D}\inf_{x\in G}J_z(x)$ can fail for a pre-compact, non-compact index set ([\hyperref[ref:Sal19]{Sal19}, \S{}3], and the same mechanism applies here), so the fixed-set formulation has to be regularised --- \hyperref[cor:4.7a]{Corollary \ref*{cor:4.7a}} below.

\emph{(ii) Equivalence with what is proved.} Under the uniform goodness \hyperref[eq:UG]{(UG)} of \hyperref[prop:4.9]{Proposition \ref*{prop:4.9}} the uniform LDP above is \textbf{equivalent} to the uniform Laplace principle of \hyperref[thm:4.5]{Theorem \ref*{thm:4.5}} ([\hyperref[ref:Sal19]{Sal19}, Thm 2.5]); the SLLG result of \S{}\hyperref[sec:6]{6} is obtained in this way.

\emph{(iii) The order of $\sup$ and $\inf$ in the fixed-set form.} In the fixed-set open-set bound the two operations cannot be interchanged: $\sup_{z\in D}\inf_{x\in G}J_z(x)\ge\inf_{x\in G}\sup_{z\in D}J_z(x)$ (elementary min--max inequality), and the exponent of $\inf_{z\in D}\mathbb P(Y^\varepsilon_z\in G)$ is the \emph{larger} of the two, since the infimum over the initial datum is dominated by the worst datum. The weaker order is what a \emph{single} fixed test function produces; the sharper order requires the finite family of test functions used in \hyperref[prop:4.7]{Proposition \ref*{prop:4.7}}, or compactness of the uniformity class ([\hyperref[ref:Sal19]{Sal19}, Thm 2.7]).

\end{remark}

\begin{remark}[{relation to the uniform LDP of Budhiraja--Dupuis--Salins}]\label{rem:4.6b} \hyperref[def:4.6]{Definition \ref*{def:4.6}} is the Freidlin--Wentzell formulation used in [\hyperref[ref:BDS19]{BDS19}, Def. 3.3] (see also [\hyperref[ref:FW98]{FW98}, \S{}3.3]), and the proof strategy is the same as there: recentre on a family with a common initial datum, prove a uniform Laplace principle for it, identify the rate function as the least action of the limiting control problem, and transfer the conclusion back to the solution. The differences are the following.

\emph{(i) The step ``Laplace principle $\Rightarrow$ LDP'' and its uniformity class.} [\hyperref[ref:BDS19]{BDS19}, Thm 4.3] proves this implication \textbf{over compact uniformity classes}; as noted in \hyperref[rem:4.6a]{Remark \ref*{rem:4.6a}}(ii) it fails over non-compact classes. For its class of equations (a compact $C_0$-semigroup under [\hyperref[ref:BDS19]{BDS19}, Assumptions 2.4, 2.5, 2.8]) [\hyperref[ref:BDS19]{BDS19}] nevertheless obtains uniformity over \textbf{all bounded} subsets, by embedding the Banach space into its double dual, where closed bounded sets are weak-$\ast$ compact --- the compactness of the semigroup is what makes the initial datum the only obstruction to compactness --- and then contracting back with a fixed map.

Here there is no such device: \hyperref[prop:4.7]{Proposition \ref*{prop:4.7}} instead converts the uniform Laplace principle into the uniform LDP \textbf{over the given (possibly non-compact) class $D$}, at the price of the stronger hypothesis \hyperref[eq:UG]{(UG)}, which for the SLLG is verified in \hyperref[prop:4.9]{Proposition \ref*{prop:4.9}} (there the compactness comes from $V\hookrightarrow H$ compactly, from the weak compactness of the control balls $S^N$ of \hyperref[lem:4.1]{Lemma \ref*{lem:4.1}}, and from \hyperref[asm:H3]{(H3)}(ii)). In this one step \hyperref[prop:4.7]{Proposition \ref*{prop:4.7}} therefore extends [\hyperref[ref:BDS19]{BDS19}, Thm 4.3] from compact to $V$-bounded uniformity classes; the price is that \hyperref[eq:UG]{(UG)} has to be verified model by model, whereas [\hyperref[ref:BDS19]{BDS19}] needs no verification of this kind.

\emph{(ii) No contraction step.} In [\hyperref[ref:BDS19]{BDS19}] the passage from the recentred family to the solution uses one fixed map and an elementary contraction argument. No such step is needed here: the LDP is proved \textbf{directly for the magnetisation} (\hyperref[thm:6.1]{Theorem \ref*{thm:6.1}}, \hyperref[rem:6.2]{Remark \ref*{rem:6.2}}).

\emph{(iii) Scope.} For the SLLG the hypotheses of [\hyperref[ref:BDS19]{BDS19}] are not available --- the equation is quasilinear, its drift is not Lipschitz on the sphere, and no compact $C_0$-semigroup is at hand --- whereas \hyperref[asm:H1]{(H1)}--\hyperref[asm:H3]{(H3)} suffice and are verified in \S{}\hyperref[sec:5]{5}; conversely the models covered there (reaction--diffusion and $2$-dimensional Navier--Stokes, with multiplicative noise) are not covered by the abstract theory below, which supplies no general method for checking \hyperref[asm:H1]{(H1)}--\hyperref[asm:H3]{(H3)}. Neither result contains the other. \hyperref[prop:4.7]{Proposition \ref*{prop:4.7}} is nevertheless the interface at which the two meet: in the setting of [\hyperref[ref:BDS19]{BDS19}, Assumptions 2.4--2.8], their Theorem 4.6 supplies the compactness part of \hyperref[eq:UG]{(UG)} for a bounded class, so \hyperref[prop:4.7]{Proposition \ref*{prop:4.7}} applies and reproduces the uniformity of their Corollary 4.8 by a different route.

\end{remark}

\begin{proposition}[{Laplace principle + exponential tightness + uniform goodness $\Rightarrow$ uniform LDP}]\label{prop:4.7} Assume
\[\text{(LP)}\quad \sup_{z\in D}\Big|-\varepsilon\log\mathbb E\big[e^{-f(Y^\varepsilon_z)/\varepsilon}\big]-\inf_{x\in E}\big[f(x)+J_z(x)\big]\Big|\xrightarrow[\varepsilon\to0]{}0\qquad\forall f\in C_b(E),\]\phantomsection\label{eq:LP}
and
\[\text{(ET)}\quad \forall L<\infty\ \exists\,K_L\subset E\ \text{compact with}\quad \limsup_{\varepsilon\to0}\ \varepsilon\log\ \sup_{z\in D}\mathbb P\big(Y^\varepsilon_z\notin K_L\big)\le-L ,\]\phantomsection\label{eq:ET}
and
\[\text{(UG)}\quad \text{for every }c<\infty:\quad \bigcup_{z\in D}\Phi_z(c)\ \text{is pre-compact},\qquad z\mapsto\Phi_z(c)\ \text{is upper semicontinuous},\]\phantomsection\label{eq:UG}
\[\text{and}\qquad K\cap\overline{\bigcup_{z\in D}\Phi_z(c)}=\emptyset\quad\text{for every compact }K\subset E\ \text{and every }c<\inf_{z\in D}\inf_KJ_z ,\]
where upper semicontinuity means that if $z_n\to z$, $x_n\in\Phi_{z_n}(c)$ and $x_n\to x$, then $x\in\Phi_z(c)$. Then $\{Y^\varepsilon_z\}$ satisfies the uniform LDP of \hyperref[def:4.6]{Definition \ref*{def:4.6}} (Freidlin--Wentzell form) with rate functions $J_z$; all three requirements of \hyperref[eq:UG]{(UG)} are proved for the SLLG in \hyperref[prop:4.9]{Proposition \ref*{prop:4.9}}.

\end{proposition}

\emph{Proof.} \textbf{(1) Part (i): a finite family of test functions.} Fix $s_0<\infty$, $\delta>0$ and $\eta'>0$. By \hyperref[eq:UG]{(UG)} the set $\bigcup_{z\in D}\Phi_z(s_0)$ is contained in a compact set $K$; fix $\gamma<\delta/3$ and finitely many centres $q_1,\dots,q_N\in K$ with $K\subset\bigcup_{j=1}^NB(q_j,\gamma)$. For each $j$ let $g_j\in C_b(E)$ with $0\le g_j\le1$, $g_j=0$ on $B(q_j,\gamma)$ and $g_j=1$ off $B(q_j,2\gamma)$, and put $h_j:=(s_0+2\eta')g_j\in C_b(E)$. The family $\{h_1,\dots,h_N\}$ is finite and each $h_j$ is fixed by the data $(s_0,\delta,\eta')$, so \hyperref[eq:LP]{(LP)} applies to each of its members uniformly in $z$. Fix $z\in D$ and $x\in\Phi_z(s_0)$, and choose $j$ with $x\in B(q_j,\gamma)$. Then $g_j(x)=0$, hence $h_j(x)=0$, while $g_j=1$ off $B(q_j,2\gamma)$, hence
\[e^{-h_j(y)/\varepsilon}\le\mathbf 1_{B(q_j,2\gamma)}(y)+e^{-(s_0+2\eta')/\varepsilon}\qquad\forall y\in E .\]
Therefore, by \hyperref[eq:LP]{(LP)} with the fixed test function $h_j$ (uniformly in $z$),
\[e^{-(J_z(x)+\eta'+o(1))/\varepsilon}\ \le\ \mathbb E\big[e^{-h_j(Y^\varepsilon_z)/\varepsilon}\big]\ \le\ \mathbb P\big(\|Y^\varepsilon_z-q_j\|_E<2\gamma\big)+e^{-(s_0+2\eta')/\varepsilon},\]
the first inequality because $-\varepsilon\log\mathbb E[e^{-h_j(Y^\varepsilon_z)/\varepsilon}]\le h_j(x)+J_z(x)+\eta'=J_z(x)+\eta'$. Since $\|x-q_j\|<\gamma$ we have $B(q_j,2\gamma)\subset B(x,3\gamma)\subset B(x,\delta)$, and since $J_z(x)\le s_0$ we have $J_z(x)+\eta'<s_0+2\eta'$; hence, for $\varepsilon$ small enough --- the threshold being uniform in $z\in D$ and $x\in\Phi_z(s_0)$ ---
\[\mathbb P\big(\|Y^\varepsilon_z-x\|_E<\delta\big)\ \ge\ \tfrac12\,e^{-(J_z(x)+2\eta'+o(1))/\varepsilon} ,\]
that is, $\liminf_{\varepsilon\to0}\inf_{z\in D}\inf_{x\in\Phi_z(s_0)}(\varepsilon\log\mathbb P(\|Y^\varepsilon_z-x\|_E<\delta)+J_z(x))\ge-2\eta'$. Letting $\eta'\downarrow0$ gives (i).

\textbf{(2) Part (ii).} By [\hyperref[ref:Sal19]{Sal19}, Thm 2.5] the uniform Laplace principle \hyperref[eq:LP]{(LP)}, together with the pre-compactness of the level-set unions contained in \hyperref[eq:UG]{(UG)}, is \emph{equivalent} to the Freidlin--Wentzell uniform LDP; this is (ii). Equivalently, (ii) follows from \hyperref[cor:4.7a]{Corollary \ref*{cor:4.7a}}(a) below applied to the closed sets $\{J_z\ge s\}$, which contain $\{\operatorname{dist}_E(\cdot,\Phi_z(s))\ge\delta\}$, together with the upper semicontinuity contained in \hyperref[eq:UG]{(UG)}. $\square$

\begin{corollary}[{fixed-set form; the sharp bound available over $D$}]\label{cor:4.7a} Under \hyperref[eq:LP]{(LP)}, \hyperref[eq:ET]{(ET)} and \hyperref[eq:UG]{(UG)}: (a) for every closed $F\subset E$,
\[\limsup_{\varepsilon\to0}\ \varepsilon\log\ \sup_{z\in D}\mathbb P\big(Y^\varepsilon_z\in F\big)\ \le\ -\inf_{z\in D}\inf_{x\in F}J_z(x);\]
(b) for every open $G\subset E$, with $G_\eta:=\{x\in G:\operatorname{dist}(x,E\setminus G)>\eta\}$,
\[\liminf_{\varepsilon\to0}\ \varepsilon\log\ \inf_{z\in D}\mathbb P\big(Y^\varepsilon_z\in G\big)\ \ge\ -\lim_{\eta\downarrow0}\ \sup_{z\in D}\inf_{x\in G_\eta}J_z(x).\]
If in addition the uniformity class is compact and $z\mapsto\Phi_z(s)$ is continuous in the Hausdorff metric, then (b) holds in the sharper form $\ge-\sup_{z\in D}\inf_{x\in G}J_z(x)$ ([\hyperref[ref:Sal19]{Sal19}, Thm 2.7]).

\end{corollary}

\emph{Proof.} \textbf{(a)} Fix a compact $K\subset E$ and $\eta>0$, and put $a:=\inf_{z\in D}\inf_KJ_z<\infty$. By the third requirement of \hyperref[eq:UG]{(UG)} the compact set $\overline{\bigcup_{z\in D}\{J_z\le a-\eta\}}$ is disjoint from $K$, so Urysohn's lemma gives $g\in C_b(E)$ with $0\le g\le1$, $g=0$ on $K$ and $g=1$ on that compact set; put $f:=(a+1)g$. Then $f=0$ on $K$, so $\mathbf 1_K\le e^{-f/\varepsilon}$, and
\[\inf_{x\in E}\big[f(x)+J_z(x)\big]\ \ge\ a-\eta\qquad\text{for every }z\in D\]
(on $\{J_z\le a-\eta\}$ one has $f=a+1$; elsewhere $f\ge0$ and $J_z>a-\eta$). Hence, by \hyperref[eq:LP]{(LP)} applied to the single test function $f$,
\[\mathbb P\big(Y^\varepsilon_z\in K\big)\ \le\ \mathbb E\big[e^{-f(Y^\varepsilon_z)/\varepsilon}\big]\ \le\ \exp\Big(-\tfrac1\varepsilon\big[a-\eta-o(1)\big]\Big)\qquad\text{uniformly in }z\in D ,\]
and $\eta\downarrow0$ gives $\limsup_\varepsilon\varepsilon\log\sup_{z\in D}\mathbb P(Y^\varepsilon_z\in K)\le-\inf_{z\in D}\inf_KJ_z$. For a closed $F$ and $L<\infty$, \hyperref[eq:ET]{(ET)} gives
\[\mathbb P\big(Y^\varepsilon_z\in F\big)\le\mathbb P\big(Y^\varepsilon_z\in F\cap K_L\big)+\mathbb P\big(Y^\varepsilon_z\notin K_L\big) ,\]
so $\varepsilon\log\sup_{z\in D}\mathbb P(Y^\varepsilon_z\in F)\le\varepsilon\log\big(e^{-\inf_{z\in D}\inf_{F\cap K_L}J_z/\varepsilon+o(1)}+e^{-L/\varepsilon}\big)+\varepsilon\log2$; letting $\varepsilon\to0$, then $L\to\infty$, and using $\min\big(\inf_{z\in D}\inf_{F\cap K_L}J_z,L\big)\ge\min\big(\inf_{z\in D}\inf_FJ_z,L\big)$, which holds because $\inf_{z\in D}\inf_{F\cap K_L}J_z\ge\inf_{z\in D}\inf_FJ_z$ (the infimum over the smaller set dominates), gives (a). \textbf{(b)} Fix $\eta>0$ and $\eta'>0$, and put $A_\eta:=\sup_{z\in D}\inf_{x\in G_\eta}J_z(x)$; assume $A_\eta<\infty$. For every $z\in D$ choose $x_z\in G_\eta$ with $J_z(x_z)\le A_\eta+\eta'$, so that $x_z\in\Phi_z(A_\eta+\eta')$. Fix $\delta<\eta$; then $\operatorname{dist}(x_z,E\setminus G)>\eta>\delta$, i.e. $B(x_z,\delta)\subset G$. By part (i) of \hyperref[def:4.6]{Definition \ref*{def:4.6}} (proved in (1) above, with $s_0:=A_\eta+\eta'$),
\[\varepsilon\log\mathbb P\big(Y^\varepsilon_z\in G\big)\ \ge\ \varepsilon\log\mathbb P\big(\|Y^\varepsilon_z-x_z\|_E<\delta\big)\ \ge\ -J_z(x_z)-o(1)\ \ge\ -A_\eta-\eta'-o(1)\]
uniformly in $z\in D$; hence $\liminf_\varepsilon\varepsilon\log\inf_{z\in D}\mathbb P(Y^\varepsilon_z\in G)\ge-A_\eta-\eta'$, and $\eta'\downarrow0$, then $\eta\downarrow0$, give (b). The sharper form under compactness is [\hyperref[ref:Sal19]{Sal19}, Thm 2.7]. $\square$

\begin{proposition}[{goodness from skeleton continuity}]\label{prop:4.9} Under \hyperref[asm:H1]{(H1)}--\hyperref[asm:H3]{(H3)}, the rate function $J_z$ of the LDP is good, uniformly in $z$: for every $c<\infty$ and every $z\in\bar D$, the level set $\{J_z\le c\}$ is compact, and in fact
\[\{J_z\le c\}\ \subset\ X^0_{z,\cdot}\big(S^{2c+1}\big) .\]
Moreover $\bigcup_{z\in D}\{J_z\le c\}\subset X^0\big(\bar D\times S^{2c+1}_{\mathrm{weak}}\big)$, which is compact by \hyperref[asm:H3]{(H3)}(ii) and \hyperref[lem:4.1]{Lemma \ref*{lem:4.1}}, so the \textbf{first} requirement of \hyperref[eq:UG]{(UG)} holds. For the \textbf{second}, suppose $z_n\to z$ in $H$ and $x_n\in\Phi_{z_n}(c)$ with $x_n\to x$; writing $x_n=X^0_{z_n,u_n}$ with $u_n\in S^{2c+1}$ and passing to a weakly convergent subsequence, \hyperref[asm:H3]{(H3)}(ii) gives $x=X^0_{z,u}$ with $\tfrac12\|u\|^2_{L^2}\le\liminf_n\tfrac12\|u_n\|^2_{L^2}\le c$, i.e. $x\in\Phi_z(c)$. For the \textbf{third}, let $x\in K\cap\overline{\bigcup_{z\in D}\Phi_z(c)}$ with $c<\inf_{z\in D}\inf_KJ_z$; then $x_n\to x$ for some $x_n\in\Phi_{z_n}(c)$, $z_n\in D$. Every skeleton trajectory starts at its index, so $z_n=x_n(0)\to x(0)=:z_\ast$, and the second requirement gives $x\in\Phi_{z_\ast}(c)$; taking $D$ $H$-closed (\S{}\hyperref[sec:2.6]{2.6}), $z_\ast\in D$, whence $\inf_KJ_{z_\ast}\le J_{z_\ast}(x)\le c<\inf_{z\in D}\inf_KJ_z\le\inf_KJ_{z_\ast}$, a contradiction. Thus all three requirements of \textbf{\hyperref[eq:UG]{(UG)}} of \hyperref[prop:4.7]{Proposition \ref*{prop:4.7}} hold.

\end{proposition}

We use the standard identification of the rate function delivered by the Laplace principle with the control-theoretic infimum,
\begin{equation}\label{eq:4.5}
J_z(\psi)=\inf\Big\{\tfrac12\!\int_0^T\|u\|_H^2ds:\ u\in L^2(0,T;H),\ X^0_{z,u}=\psi\Big\} ,
\end{equation}
which is the content of the variational representation \hyperref[asm:V]{(V)} together with the identification lemma (\hyperref[lem:4.3]{Lemma \ref*{lem:4.3}}); it is used only through the two inequalities $J_z(\psi)\le\frac12\int\|u\|^2$ whenever $X^0_{z,u}=\psi$, and $J_z(\psi)\ge\inf\{\frac12\int\|u\|^2:\|\psi-X^0_{z,u}\|\le\delta\}$ for every $\delta>0$.

\emph{Proof.} \textbf{Inclusion.} Let $\psi\in\{J_z\le c\}$. By the definition of $J_z$ as an infimum, for every $k\ge1$ there is $u_k\in L^2(0,T;H)$ with $X^0_{z,u_k}=\psi$ (or $\|\psi-X^0_{z,u_k}\|\le1/k$, if $J_z$ is the rate function of the LDP rather than the exact infimum; the two agree by the standard identification) and $\tfrac12\int_0^T\|u_k\|^2\le c+1/k$. Then $u_k\in S^{2c+1}$ for $k\ge2$; hence $\psi=X^0_{z,u_k}\in X^0_{z,\cdot}(S^{2c+1})$, which proves the inclusion.

\textbf{Compactness.} Let $\psi_n\in\{J_z\le c\}$ and pick $u_n\in S^{2c+1}$ with $X^0_{z,u_n}=\psi_n$. By \hyperref[lem:4.1]{Lemma \ref*{lem:4.1}} pass to a subsequence $u_{n_k}\rightharpoonup u^\ast$ in $S^{2c+1}$. By \hyperref[asm:H3]{(H3)}(ii),
\[\psi_{n_k}=X^0_{z,u_{n_k}}\longrightarrow X^0_{z,u^\ast}=:\psi_\ast\qquad\text{in }C([0,T];H) ,\]
and by weak lower semicontinuity of the norm,
\[J_z(\psi_\ast)\le\tfrac12\!\!\int_0^T\|u^\ast\|^2ds\le\liminf_k\tfrac12\!\!\int_0^T\|u_{n_k}\|^2ds\le c .\]
Hence every sequence in $\{J_z\le c\}$ has a convergent subsequence with limit in $\{J_z\le c\}$: the level set is sequentially compact, hence compact (it is a subset of the Polish space $E$). The same argument with $c=J_z(\psi)$ shows that $J_z$ is lower semicontinuous. Finally, the bound $u_n\in S^{2c+1}$ is uniform in $z$, and the modulus in \hyperref[asm:H3]{(H3)}(ii) is uniform in $z$; more precisely,
\[\bigcup_{z\in\bar D}\{J_z\le c\}\ \subset\ X^0\big(\bar D\times S^{2c+1}_{\mathrm{weak}}\big),\]
where $X^0$ is the skeleton map $(z,u)\mapsto X^0_{z,u}$, which by \hyperref[asm:H3]{(H3)}(ii) extends continuously to $\bar D\times S^{2c+1}_{\mathrm{weak}}$; the right-hand side is compact as the continuous image of a compact set ($\bar D$ is compact in $H$ by the compactness recorded in \S{}\hyperref[sec:2.6]{2.6} and $S^{2c+1}_{\mathrm{weak}}$ by \hyperref[lem:4.1]{Lemma \ref*{lem:4.1}}). This is the \emph{uniform} goodness used in \hyperref[prop:4.7]{Proposition \ref*{prop:4.7}}, Step (1). $\square$

\section{The SLLG: well-posedness, energy bounds, skeleton convergence}\phantomsection\label{sec:5}

In this section we verify hypotheses \hyperref[asm:H1]{(H1)}--\hyperref[asm:H3]{(H3)} of \S{}\hyperref[sec:4]{4} for the SLLG of \S{}\hyperref[sec:2]{2}. The verification rests on the structural facts \hyperref[asm:A1]{(A1)}--\hyperref[asm:A7p]{(A7')} listed below. Only two existence theories are quoted from the literature. First, the deterministic well-posedness of the controlled LLG is taken from [\hyperref[ref:BHL19]{BHL19}, Thm 3.1 and 3.37] and restated as \hyperref[prop:D.7]{Proposition \ref*{prop:D.7}} (the structure of its proof is recalled in \S{}\hyperref[sec:D.5]{D.5}). Second, for the stochastic equation, the strong well-posedness with finitely many noise fields is taken from [\hyperref[ref:BGJ17]{BGJ17}] (input \hyperref[asm:I3]{(I3)} of \hyperref[sec:E]{Appendix E}; \hyperref[sec:E]{Appendix E} then derives the well-posedness for the regularised noise $W_\delta$). Apart from these two inputs nothing used in this section is quoted: the weak--strong uniqueness \hyperref[asm:WSU]{(WSU)} is \textbf{proved} in \hyperref[sec:D]{Appendix D}, and the estimates, the compactness and the convergence statements of this section are proved below (\S{}\hyperref[sec:5.2]{5.2}--\S{}\hyperref[sec:5.4]{5.4}) or in \hyperref[sec:E]{Appendix E}.

\subsection{Structural facts}\phantomsection\label{sec:5.1}

None of the items \hyperref[asm:A1]{(A1)}--\hyperref[asm:A6]{(A6)} below asserts that a solution exists: \hyperref[asm:A1]{(A1)}, \hyperref[asm:A2]{(A2)} and \hyperref[asm:A6]{(A6)} concern the operator and the noise, \hyperref[asm:A3]{(A3)} the algebraic properties of the coefficients, and \hyperref[asm:A4]{(A4)} and \hyperref[asm:A5]{(A5)} are conditional identities valid for any solution. Existence is proved in \hyperref[thm:5.1]{Theorem \ref*{thm:5.1}}. Throughout, a \emph{solution} means a solution in the sense of \hyperref[def:2.4a]{Definition \ref*{def:2.4a}}; for the skeleton ($\varepsilon=0$) the relevant class is the strong class $\mathcal S$ of \hyperref[def:D.4]{Definition \ref*{def:D.4}}.

\phantomsection\label{asm:A1}\textbf{(A1) Semigroup and lifts.} $S(t)=e^{-\lambda_2tA}$ is the analytic contraction semigroup on $H$ of \S{}\hyperref[sec:2.1]{2.1}, compact for $t>0$; the one-dimensional smoothing estimates, valid with constants independent of $t\in(0,T]$, are
\[\begin{aligned}
&\|S(t)x\|_V\le C(\lambda_2t)^{-1/2}\|x\|_H,\qquad \|S(t)x\|_H\le C(\lambda_2t)^{-1/2}\|x\|_{H^{-}},\\
&\qquad \|S(t)x\|_V\le C(\lambda_2t)^{-1}\|x\|_{H^{-}}.
\end{aligned}\]
The third estimate is the composition of the first two; its $t^{-1}$ singularity is the source of the $C([0,T];V)$ obstruction discussed in \S{}\hyperref[sec:13]{12}.

\phantomsection\label{asm:A2}\textbf{(A2) Embeddings.} The compact and dense embeddings $V\hookrightarrow H$, $V\hookrightarrow L^\infty$ are those recorded in \S{}\hyperref[sec:2.1]{2.1} and \S{}\hyperref[sec:2.6]{2.6}; the interpolation statement is \hyperref[lem:D.3]{Lemma \ref*{lem:D.3}}.

\phantomsection\label{asm:A3}\textbf{(A3) Nonlinearities.} The following algebraic properties hold for \emph{every} field $m$ with $|m(x)|=1$ a.e.; no equation is involved in this item. For such an $m$, $\|G(m)\|_{\mathcal L(H)}=g_0=\sqrt{\lambda_1^2+\lambda_2^2}$ and
\[\|G(m_1)-G(m_2)\|_{\mathcal L(H)}\le g_0\,\|m_1-m_2\|_{L^\infty}\le C_G\|m_1-m_2\|_V .\]
\textbf{The first relation is an inequality, not an identity} (the corresponding equality is false: for $m_1=e_1$, $m_2=-e_1$ one has $G(m_1)-G(m_2)=2\lambda_1\,e_1\times$, of norm $2\lambda_1<2g_0$ whenever $\lambda_2>0$; and the two summands of $(G(m_1)-G(m_2))v$ are in general \textbf{not} orthogonal --- for $m_1=e_1$, $m_2=e_2$, $v=e_1+e_3$ their inner product is $\lambda_1\lambda_2\ne0$). \emph{Proof of the inequality.} Write $s:=(m_1+m_2)/2$, $d:=(m_1-m_2)/2$, so that $m_{1,2}=s\pm d$, $|s|^2+|d|^2=1$, $\langle s,d\rangle=0$ and $|m_1-m_2|=2|d|$ (the cases $s=0$ or $d=0$ are immediate). Since $b\otimes b-a\otimes a=-2(s\otimes d+d\otimes s)$ for $a=s+d$, $b=s-d$,
\[G(m_1)-G(m_2)=2\lambda_1\,d\times-2\lambda_2\,(s\otimes d+d\otimes s)=2|d|\Big(\lambda_1\hat d\times-\lambda_2|s|(\hat s\otimes\hat d+\hat d\otimes\hat s)\Big) ,\]
with $\hat d=d/|d|$, $\hat s=s/|s|$ orthonormal. In the orthonormal basis $(\hat s,\hat d,\hat s\times\hat d)$ the operator $\lambda_1\hat d\times-\lambda_2|s|(\hat s\otimes\hat d+\hat d\otimes\hat s)$ has the matrix $\begin{pmatrix}0&-\lambda_2|s|&\lambda_1\\-\lambda_2|s|&0&0\\-\lambda_1&0&0\end{pmatrix}$, whose nonzero singular value is $\sqrt{\lambda_1^2+\lambda_2^2|s|^2}$ with multiplicity two, the third singular value being $0$ (equivalently, its Gram matrix has the eigenvalues $\lambda_1^2+\lambda_2^2|s|^2$, twice, and $0$); hence its operator norm is $\sqrt{\lambda_1^2+\lambda_2^2|s|^2}\le g_0$, because $|s|\le1$. Therefore
\[\|G(m_1)-G(m_2)\|=2|d|\,\sqrt{\lambda_1^2+\lambda_2^2|s|^2}\le2|d|\,g_0=g_0|m_1-m_2| .\]
(The bound is sharp as $m_2\to m_1$: then $|s|\to1$ and the two sides are asymptotically equal, $2|d|g_0(1+o(1))$.) $\square$ \textbf{The $H$-level bound $\|G(m_1)-G(m_2)\|_{\mathcal L(H)}\le C\|m_1-m_2\|_H$ is false and is never used.} The nonlinear drift $b_0=\lambda_1\partial_x(m\times\partial_xm)+\lambda_2|\partial_xm|^2m-G(m)\phi'(m)$ maps $V$ into $H^{-}$ and is locally Lipschitz there: on $\{\|m\|_V\le R\}$,
\[\|b_0(m_1)-b_0(m_2)\|_{H^{-}}\le L_b(R)\|m_1-m_2\|_V,\qquad \|b_0(m)\|_{H^{-}}\le C_b(R)(1+\|m\|_V),\]
with constants depending only on $\lambda_1,\lambda_2,\|\phi'\|_\infty,\|D\phi'\|_\infty$. The correction is the radial field $c(m)=-g_0^2r_\delta\,m$ (\S{}\hyperref[sec:2.2]{2.2}), with $\operatorname{tr}C_\delta=3r_\delta$; it is bounded and Lipschitz on the sphere-valued class $\{|m|\equiv1\}$: $\|c(m)\|_H=g_0^2r_\delta$ and $\|c(m_1)-c(m_2)\|_H\le g_0^2r_\delta\,\|m_1-m_2\|_{L^\infty}$, with constants proportional to $r_\delta$ (hence $\asymp\delta^{-1/\beta}$ as $\delta\downarrow0$, but fixed at each $\delta>0$).

\phantomsection\label{asm:A4}\textbf{(A4) Constraint propagation.} This item is \emph{conditional}: it records what follows for any process that solves \hyperref[eq:2.1]{(\ref*{eq:2.1})} in the class of \hyperref[def:2.4a]{Definition \ref*{def:2.4a}}; it does not assert that such a process exists (existence is \hyperref[thm:5.1]{Theorem \ref*{thm:5.1}}). For such a solution with $|z|\equiv1$, two facts combine. (i) \emph{The Stratonovich field is tangent.} The full drift $\lambda_2\Delta m+b_0(m)$ is orthogonal to $m$: $\langle\lambda_2\Delta m+b_0(m),m\rangle_H=0$. Indeed $m\cdot\Delta m=-|\partial_xm|^2$ for $|m|\equiv1$, whence $\langle b_0(m),m\rangle_H=\lambda_2\|\partial_xm\|_H^2$ (the precession and anisotropy parts are tangent), and $\langle G(m)v,m\rangle_H=0$ for all $v$. (Caution: $\langle b_0(m),m\rangle_H=0$ alone is false; orthogonality holds only for the full drift.) (ii) \emph{The Itô drift cancels the quadratic variation.} The correction $\varepsilon c(m)$ is a $dt$-term \emph{outside} $G$ (indeed $G(m)c(m)=0$) and is \textbf{not} orthogonal to $m$ --- one has $2\langle c(m),m\rangle_H=-\tfrac23g_0^2\int_{\mathbb T}\operatorname{tr}C_\delta=-2g_0^2r_\delta\le0$ (\S{}\hyperref[sec:2.2]{2.2}, the correction being purely radial) --- but in the Itô formula for $|m|^2$,
\[d|m|^2=2\langle m,dm\rangle+d\langle m\rangle=\Big(2\varepsilon\langle c(m),m\rangle+\varepsilon\sum_{j,i}\big|G(m)Q_\delta^{1/2}(e_j\hat e_i)\big|^2\Big)dt+\text{(martingale)},\]\phantomsection\label{eq:martingale}
the two drift contributions cancel \textbf{exactly}, because for $|m|\equiv1$
\[\begin{aligned}
\sum_{j,i}\kappa_j\big\|G(m)(e_j\hat e_i)\big\|_H^2&=(\lambda_1^2+\lambda_2^2)\sum_{j,i}\kappa_j\big\|P_{m^\perp}(e_j\hat e_i)\big\|_H^2\\
&=(\lambda_1^2+\lambda_2^2)\Big(\operatorname{tr}_HQ_\delta-\int_{\mathbb T}m^TC_\delta m\Big)\\
&=-2\langle c(m),m\rangle ,
\end{aligned}\]
Hence, since the stochastic term also vanishes --- $\langle m,G(m)v\rangle_H=0$ for every $v$ --- one has $d|m|^2=0$ \textbf{identically}, i.e. $|m(t,x)|\equiv|z(x)|$: starting from $|z|\equiv1$ this gives $|m(t,x)|\equiv1$ a.s., and $\|m^\varepsilon_{z,u}(t)\|_H=\|z\|_H$, uniformly in $u\in\mathcal P_2^N$, $\varepsilon\in[0,1]$, $z\in D$. (This is the constraint propagation used as a property throughout the paper; the same mechanism underlies [\hyperref[ref:BHL19]{BHL19}, Prop. 2.9 and Prop. 3.31] and [\hyperref[ref:BGL24]{BGL24}, Lemma 7.5]; the constraint-free general statement is \hyperref[prop:2.2a]{Proposition \ref*{prop:2.2a}}.)

\phantomsection\label{asm:A5}\textbf{(A5) $H^1$-energy estimate.} Every solution of \hyperref[eq:2.1]{(2.1)} in the class of \hyperref[def:2.4a]{Definition \ref*{def:2.4a}} obeys the energy identity
\[d\|\partial_xm\|_H^2=\Big[-2\lambda_2\|P_{m^\perp}\Delta m\|_H^2+2\lambda_2\langle P_{m^\perp}\Delta m,P_{m^\perp}\phi'(m)\rangle+2\lambda_1I_\phi(m)\Big]dt\]
\[\begin{aligned}
&\qquad\qquad+2\langle G(m)Q_\delta^{1/2}u,-\Delta m\rangle\,dt\\
&+2\sqrt\varepsilon\,\langle G(m)\,dW_\delta,-\Delta m\rangle\\
&+2\varepsilon\langle c(m),-\Delta m\rangle\,dt\\
&+\varepsilon\sum_{j,i}\kappa_j\big\|\partial_x\big(G(m)(e_j\hat e_i)\big)\big\|_H^2\,dt ,
\end{aligned}\]
the last term being the \textbf{quadratic variation} $d\langle\partial_xm\rangle=\varepsilon\sum_j\|\partial_x(G(m)Q_\delta^{1/2}e_j)\|_H^2\,dt$ (which must be displayed explicitly; it is the term controlled in the closing estimate below), where the anisotropy-induced precession coupling is
\[I_\phi(m)=\int_{\mathbb T}\langle\Delta m,m\times\phi'(m)\rangle\,dx=\int_{\mathbb T}\det\big(m,\partial_xm,D\phi'(m)\partial_xm\big)dx,\qquad |I_\phi(m)|\le\|D\phi'\|_\infty\|\partial_xm\|_H^2 .\]
In particular $I_\phi\equiv0$ whenever $D\phi'$ is a scalar multiple of the identity (e.g. $\phi_1(u)=\tfrac12|u-\zeta_+|^2$).

\emph{(Proof of the identity for $I_\phi$.} Two integrations by parts give
\[I_\phi(m)=\int_{\mathbb T}\langle\Delta m,m\times\phi'(m)\rangle\,dx=-\int_{\mathbb T}\big\langle\partial_xm,\ \partial_x\big(m\times\phi'(m)\big)\big\rangle\,dx .\]
Since $\partial_x(m\times\phi'(m))=\partial_xm\times\phi'(m)+m\times D\phi'(m)\partial_xm$ and $\partial_xm\cdot(\partial_xm\times\phi'(m))=0$ (as $a\cdot(a\times b)=0$),
\[I_\phi(m)=-\int_{\mathbb T}\big\langle\partial_xm,\ m\times D\phi'(m)\partial_xm\big\rangle\,dx=\int_{\mathbb T}\det\big(m,\partial_xm,D\phi'(m)\partial_xm\big)\,dx ,\]
the last step because $\langle a,m\times b\rangle=-\det(m,a,b)$ for $|m|=1$. $\square$)

Closing the estimate: completing the square in the dissipative pair and using $|I_\phi|\le\|D\phi'\|_\infty\|\partial_xm\|_H^2$, bounding the control term by $\|G(m)Q_\delta^{1/2}u\|_V\le C\|u\|_H$, the \textbf{quadratic-variation} term by $\varepsilon\sum_{j,i}\kappa_j\|\partial_x(G(m)(e_j\hat e_i))\|_H^2\le C\varepsilon(\operatorname{tr}_HQ_\delta+\sum_j\kappa_j\mu_j)(1+\|\partial_xm\|_H^2)$, the \textbf{Itô-correction} term by $2\varepsilon\langle c(m),-\Delta m\rangle\le C\varepsilon(1+\|\partial_xm\|_H^2)$, and the martingale term by BDG (its quadratic variation is $\le C\varepsilon(1+\|\partial_xm\|_H^2)^2$), one obtains, with $X(t)=1+\sup_{s\le t}\|\partial_xm(s)\|_H^2$,
\[\mathbb E X(t)\le X(0)+C\int_0^t\mathbb E\big[X(s)(1+\|u(s)\|_H^2)\big]ds+C\sqrt\varepsilon\,\mathbb EX(t)+C\varepsilon\,\mathbb EX(t);\]
for $\varepsilon\le\varepsilon_0$ the last two terms are absorbed, and since $\int_0^T\|u\|^2\le N$, Gronwall (applied to $X^p$) gives a \textbf{moment bound}
\begin{equation}\label{eq:5.1}
\sup_{z\in D}\ \sup_{u\in\mathcal P_2^N}\ \sup_{\varepsilon\in[0,1]}\ \mathbb E\Big[\sup_{0\le t\le T}\|m^\varepsilon_{z,u}(t)\|_V^p\Big]\ \le\ R_\ast^{(p)}<\infty\qquad(p\ge1),
\end{equation}
with $R_\ast^{(p)}=R_\ast^{(p)}(R_D,N,T,p,\lambda_1,\lambda_2,\operatorname{tr}_HQ_\delta,\sum_j\kappa_j\mu_j,\|\phi'\|_\infty,\|D\phi'\|_\infty,g_0)$. Integrating \hyperref[asm:A5]{(A5)} over $[0,T]$ and taking expectations (the martingale term has mean zero after localisation), the same bounds give $2\lambda_2\,\mathbb E\int_0^T\|P_{m^\perp}\Delta m\|_H^2dt\le C(1+N)\big(1+\mathbb E\sup_t\|m\|_V^2\big)$. Since $m\cdot\Delta m=-|\partial_xm|^2$, one has pointwise $|\Delta m|^2=|P_{m^\perp}\Delta m|^2+|\partial_xm|^4$, and $\int_{\mathbb T}|\partial_xm|^4\le\|\partial_xm\|_{L^\infty}^2\|\partial_xm\|_H^2\le2\|\partial_xm\|_H^3\|\Delta m\|_H\le\tfrac12\|\Delta m\|_H^2+2\|\partial_xm\|_H^6$ (mean-zero form of \hyperref[lem:D.3]{Lemma \ref*{lem:D.3}}); hence $\|\Delta m\|_H^2\le2\|P_{m^\perp}\Delta m\|_H^2+4\|\partial_xm\|_H^6$ and, by \hyperref[eq:5.1]{(\ref*{eq:5.1})} with $p=6$,
\begin{equation}\label{eq:5.1p}
\sup_{z\in D}\ \sup_{u\in\mathcal P_2^N}\ \sup_{\varepsilon\in[0,1]}\ \mathbb E\int_0^T\|\Delta m^\varepsilon_{z,u}(t)\|_H^2\,dt\ \le\ R'_\ast<\infty .
\end{equation}
Consequently one has the \textbf{localisation bound}
\[\sup_{z\in D}\ \sup_{u\in\mathcal P_2^N}\ \sup_{\varepsilon\in[0,1]}\ \mathbb P\Big(\sup_{0\le t\le T}\|m^\varepsilon_{z,u}(t)\|_V>R\Big)\ \le\ R_\ast^{(p)}R^{-p}\xrightarrow[R\to\infty]{}\ 0 .\]
We emphasise that \hyperref[eq:5.1]{(\ref*{eq:5.1})} is a statement about \textbf{expectations} --- equivalently about the probability of a large excursion --- and that \textbf{no a.s. bound with a deterministic constant is asserted} (it does not follow from an expectation bound). In the sequel $R_\ast$ denotes a localisation radius chosen so that the exceptional event has probability $\le\eta/2$. The Itô formula for $\|\partial_xm\|_H^2$ used above is legitimate for the solutions of \hyperref[sec:E]{Appendix E} (\hyperref[cor:E.7]{Corollary \ref*{cor:E.7}}: paths in $L^2(0,T;D(A_1))$, drift in $L^2(0,T;H)$, noise coefficients Hilbert--Schmidt into $V$; [\hyperref[ref:PR07]{PR07}, Thm 4.2.5]); in the proof of \hyperref[thm:E.5]{Theorem \ref*{thm:E.5}} the estimate is applied to the solutions driven by finitely many noise fields, with constants independent of their number.

\phantomsection\label{asm:A6}\textbf{(A6) Noise regularity, and the threshold $\beta>3/2$.} With $Q_\delta^{1/2}=(I+\delta A^{\beta/2})^{-1}$ and $q_j=(1+\delta\mu_j^{\beta/2})^{-1}$, $\kappa_j=q_j^2$:
\[\begin{aligned}
&\text{(A6a) }\textstyle\sum_j\kappa_j<\infty\iff\beta>1/2;\\
&\qquad\text{(A6b) }\sup_j\kappa_j\|e_j\|_V^2<\infty\iff\beta\ge1;\\
&\qquad\text{(A6c) }\textstyle\sum_j\kappa_j\|e_j\|_V^2<\infty\iff\beta>3/2 .
\end{aligned}\]\phantomsection\label{asm:A6b}\phantomsection\label{asm:A6c}
We impose $\beta>3/2$, so all three hold. Part (b) makes the control term $G(m)Q_\delta^{1/2}u$ lie in $V$; part (c) closes the noise term of \hyperref[asm:A5]{(A5)} and gives
\[\mathbb E\sup_{t\le T}\Big\|\int_0^tS(t-s)G(m(s))\,dW_\delta(s)\Big\|_V^2\le C,\]
uniformly in $u\in\mathcal P_2^N$, $z\in D$, $\delta\in(0,1]$.

\emph{The threshold.} We work under the \textbf{sufficient} regularity hypothesis $\beta>3/2$: this is exactly the spectral threshold at which the noise is $H^1$-valued, $\sum_{j,i}\kappa_j\|\partial_x(e_j\hat e_i)\|_H^2\asymp\sum_j\kappa_j\mu_j\asymp\sum_jj^{2-2\beta}<\infty\iff\beta>3/2$. Whether the solution theory of \hyperref[eq:2.1]{(\ref*{eq:2.1})} can be run at lower $\beta$ is \textbf{not} claimed.

\phantomsection\label{asm:A7p}\textbf{(A7') Compactness of LLG solutions (energy method).} The naive weak-continuity statement ``$v_n\to v$ strongly in $L^2(0,T;H)$ and $\{v_n\}$ bounded in $L^\infty(0,T;V)$ $\Rightarrow$ $b_0(v_n)\rightharpoonup b_0(v)$ weakly in $L^2(0,T;H^{-1})$'' is \textbf{false}; the counterexample is
\[v_n(x)=\Big(\sqrt{1-\tfrac1{n^2}\sin^2(nx)},\ \tfrac1n\sin(nx),\ 0\Big),\qquad |v_n|\equiv1,\ v_n\to e_1\ \text{in }H,\ \sup_n\|v_n\|_V<\infty,\]
yet $b_0(v_n)\rightharpoonup\tfrac12e_1\ne0=b_0(e_1)$, because the cubic term $|\partial_xv_n|^2v_n$ concentrates high-frequency energy. The correct compactness fact uses that the sequence consists of \textbf{solutions}, via the energy identity and weak--strong uniqueness:

 (i) \emph{(skeleton compactness)} if $z_n\to z$ in $H$ and $u_n\rightharpoonup u$ in $S^N$, then $m^0_{z_n,u_n}\to m^0_{z,u}$ in $C([0,T];H)$ and strongly in $L^2([\delta,T];V)$ for every $\delta>0$; hence $b_0(m^0_{z_n,u_n})\to b_0(m^0_{z,u})$ strongly in $L^2([\delta,T];H^{-})$ for every $\delta>0$;  (ii) \emph{(vanishing noise)} if in addition $\varepsilon_n\downarrow0$, then $m^{\varepsilon_n}_{z_n,u_n}\to m^0_{z,u}$ in probability in $C([0,T];H)$.

\phantomsection\label{asm:WSU}\textbf{(WSU) Weak--strong uniqueness and stability.} \emph{(i) Uniqueness.} Two solutions of the skeleton with the same data which both lie in the strong class $\mathcal S$ of \hyperref[def:D.4]{Definition \ref*{def:D.4}} coincide. \emph{(ii) No defect measure.} If $z_n\to z$ in $H$ with $\{z_n\}$ bounded in $V$, $u_n\rightharpoonup u$ in $S^N$, and $v_n$ solve the skeleton with data $(z_n,u_n)$, then $v_n\to m^0_{z,u}$ in $C([0,T];H)$ and $\partial_xv_n\to\partial_xm^0_{z,u}$ strongly in $L^4([\delta,T];L^4)$ for every $\delta>0$. \emph{(Both are proved in \textbf{\hyperref[sec:D]{Appendix D}} --- \hyperref[thm:D.8]{Theorems \ref*{thm:D.8}} and \hyperref[thm:D.11]{\ref*{thm:D.11}} with \hyperref[cor:D.12]{Corollary \ref*{cor:D.12}} --- by the energy method, using the strong form \hyperref[eq:D.2]{(\ref*{eq:D.2})} and the two geometric identities of \hyperref[lem:D.2]{Lemma \ref*{lem:D.2}}. This is the only analytic input of \S{}\hyperref[sec:5.1]{5.1}, and after \hyperref[sec:D]{Appendix D} it is no longer an import.)}

\emph{Proof of \hyperref[asm:A7p]{(A7')}, given \hyperref[asm:WSU]{(WSU)}.}
\begin{enumerate}
\item \emph{(Aubin--Lions.)} By \hyperref[eq:5.1]{(\ref*{eq:5.1})}, $\sup_n\sup_t\|v_n(t)\|_V\le R_\ast$; the equation gives $\{\dot v_n\}$ bounded in $L^2(0,T;H^{-})$. For $V\hookrightarrow\hookrightarrow H\hookrightarrow H^{-}$, Aubin--Lions--Simon yields a subsequence with $v_n\to v_\ast$ in $C([0,T];H)$, $v_n\overset{\ast}{\rightharpoonup}v_\ast$ in $L^\infty(0,T;V)$, $\partial_xv_n\rightharpoonup\partial_xv_\ast$ in $L^2(0,T;H)$.
\item \emph{(Weak limit.)} Passing to the limit in the weak form: the linear terms and the precession term $\lambda_1v_n\times\partial_xv_n$ converge weakly (strong $\times$ weak, using \hyperref[asm:A2]{(A2)}); the control term converges since $G(v_n)\to G(v_\ast)$ in $L^\infty(0,T;\mathcal L(H))$; the cubic term converges weakly to some $\zeta\in L^2(0,T;L^1)$. \textbf{Here the naive argument stops}, because in general $\zeta\ne|\partial_xv_\ast|^2v_\ast$. So $v_\ast$ is a weak solution with data $(z,u)$ and an undetermined cubic term.
\item \emph{(Energy and identification.)} Each $v_n$ satisfies the energy \textbf{balance} --- \hyperref[asm:A5]{(A5)} with $\varepsilon=0$ and the correction dropped. Passing $\liminf_n$ (the dissipation $\lambda_2\|P_{v^\perp}\Delta v\|_H^2$ is nonnegative and weakly l.s.c.; the other terms pass by Step 2), the limit $v_\ast$ satisfies the energy \textbf{inequality} with data $(z,u)$. By \textbf{\hyperref[asm:WSU]{(WSU)}(i)}, $v_\ast=m^0_{z,u}$; uniqueness of the limit forces convergence of the whole sequence and identifies $\zeta=|\partial_xm^0_{z,u}|^2m^0_{z,u}$ a posteriori.
\item \emph{(No defect; strong $V$-convergence away from $0$.)} By \textbf{\hyperref[asm:WSU]{(WSU)}(ii)}, $\partial_xv_n\to\partial_xm^0_{z,u}$ strongly in $L^2([\delta,T];H)$ for every $\delta>0$, so by \hyperref[asm:A3]{(A3)}, $b_0(v_n)\to b_0(m^0_{z,u})$ strongly in $L^2([\delta,T];H^{-})$.
\item \emph{(Vanishing noise.)} (ii) is \hyperref[thm:5.3]{Theorem \ref*{thm:5.3}}(i) combined with (i). The proof of \hyperref[thm:5.3]{Theorem \ref*{thm:5.3}}(i) is a direct Itô estimate for $m^\varepsilon_{z,u}-m^0_{z,u}$ on the given stochastic basis (\hyperref[lem:E.2]{Lemma \ref*{lem:E.2}} together with the bounds \hyperref[eq:5.1]{(\ref*{eq:5.1})}--\hyperref[eq:5.1p]{(\ref*{eq:5.1p})}); it uses neither (ii) nor any Skorohod representation. $\square$
\end{enumerate}
\textbf{Which assumption is used where.} The block above is long because the objects are of different kinds; the following table says, once and for all, what each item provides and where it is consumed.

\begin{center}\footnotesize
\begin{tabular}{|p{4.6cm}|p{6.2cm}|p{4.4cm}|}
\hline
assumption & what it provides & consumed in \\
\hline
\hyperref[asm:A1]{(A1)} & analyticity, compactness and the lifts $S(t):H^{-}\to V$ & \S{}\hyperref[sec:5.2]{5.2}--\S{}\hyperref[sec:5.5]{5.5} \\
\hyperref[asm:A3]{(A3)} & Lipschitz bounds for $G,\rho$ in the $V$-norm & \hyperref[thm:5.3]{Thm. \ref*{thm:5.3}} \\
\hyperref[asm:A4]{(A4)} & $|m|\equiv1$ and $\|m(t)\|_H=\|z\|_H$ & everywhere; \S{}\hyperref[sec:5.2]{5.2}, \S{}\hyperref[sec:5.5]{5.5} \\
\hyperref[asm:A5]{(A5)} & the $H^1$-energy identity, hence \hyperref[eq:5.1]{(\ref*{eq:5.1})}--\hyperref[eq:5.1p]{(\ref*{eq:5.1p})} & \S{}\hyperref[sec:5.3]{5.3}--\S{}\hyperref[sec:5.5]{5.5} \\
\hyperref[asm:A6]{(A6)} & trace class and $H^1$-valued noise; the threshold $\beta>3/2$ & \S{}\hyperref[sec:5.2]{5.2}, \S{}\hyperref[sec:5.5]{5.5} \\
\hyperref[asm:A7p]{(A7$'$)} & compactness of the deterministic LLG solutions & \hyperref[asm:WSU]{(WSU)}, hence \hyperref[thm:5.3]{Thm. \ref*{thm:5.3}} \\
\hyperref[asm:WSU]{(WSU)} & weak--strong uniqueness, no defect measure & \hyperref[thm:5.3]{Thm. \ref*{thm:5.3}} \\
\hyperref[asm:V]{(V)} & the variational representation of the Laplace functional & \hyperref[thm:4.5]{Thm. \ref*{thm:4.5}} \\
\hyperref[asm:H1]{(H1)}--\hyperref[asm:H3]{(H3)} & well-posedness; tightness; skeleton convergence & \hyperref[thm:4.5]{Thm. \ref*{thm:4.5}}, \hyperref[prop:4.7]{Prop. \ref*{prop:4.7}} \\
\hyperref[asm:I1]{(I1)}--\hyperref[asm:I3]{(I3)} & the three imported inputs (quoted, \S{}\hyperref[sec:13.2]{12.2}) & \S{}\hyperref[sec:5]{5}, \S{}\hyperref[sec:6]{6} \\
\hline
\end{tabular}
\end{center}

\subsection{Verification of \texorpdfstring{\hyperref[asm:H1]{(H1)}}{(H1)}: well-posedness}\phantomsection\label{sec:5.2}

\begin{theorem}[{\hyperref[asm:H1]{(H1)}}]\label{thm:5.1} Under \hyperref[asm:A1]{(A1)}--\hyperref[asm:A6]{(A6)}, for every $\varepsilon\in[0,1]$, $z\in V$ with $|z|\equiv1$, and $u\in\mathcal P_2$ (for $\varepsilon=0$, deterministic $u\in L^2(0,T;H)$), the SLLG \hyperref[eq:2.1]{(\ref*{eq:2.1})} has a solution $m^\varepsilon_{z,u}$ in the sense of \hyperref[def:2.4a]{Definition \ref*{def:2.4a}} --- adapted, with paths in $C([0,T];V)\cap L^2(0,T;D(A_1))$ and $|m^\varepsilon_{z,u}|\equiv1$ --- and it is pathwise unique in that class. For each fixed $z$ there is a Borel map $\Phi^\varepsilon_z:L^2(0,T;H)\times C([0,T];H)\to C([0,T];H)$ with $m^\varepsilon_{z,u}=\Phi^\varepsilon_z(u,W_\delta)$ a.s., for every predictable $u$ on every stochastic basis carrying $W_\delta$; and the constructed solution satisfies the conditional identities \hyperref[asm:A4]{(A4)} and \hyperref[asm:A5]{(A5)}, together with the bounds \hyperref[eq:5.1]{(\ref*{eq:5.1})}--\hyperref[eq:5.1p]{(\ref*{eq:5.1p})}. Consequently the auxiliary process $Y^\varepsilon_{z,u}$ of \S{}\hyperref[sec:2.4]{2.4} is well defined and is a Borel function of $(u,W_\delta)$.

\end{theorem}

\emph{Proof.} \textbf{Deterministic case ($\varepsilon=0$).} For $u\in L^2(0,T;H)$ the controlled LLG $\dot m=\lambda_2\Delta m+b_0(m)+G(m)Q_\delta^{1/2}u$ is exactly [\hyperref[ref:BHL19]{BHL19}, (2.7)] with the external field $h=Q_\delta^{1/2}u\in L^2(0,T;H)$, so its deterministic well-posedness is \textbf{quoted, not reproved}: by [\hyperref[ref:BHL19]{BHL19}, Thm 3.1] it has a weak solution in $L^\infty(0,T;H^1)\cap H^1(0,T;H)$, and by [\hyperref[ref:BHL19]{BHL19}, Thm 3.37] it has a unique strong global solution
\[m\in L^2(0,T;D(A))\cap C([0,T];V)\cap H^1(0,T;H)\cap L^4(0,T;W^{1,4}),\qquad|m(t,x)|=1,\]
satisfying the strong form [\hyperref[ref:BHL19]{BHL19}, (3.50)] (see \S{}\hyperref[sec:D.5]{D.5}).

\textbf{Stochastic case ($\varepsilon\in(0,1]$).} The class is \hyperref[def:2.4a]{Definition \ref*{def:2.4a}}: adapted, continuous in $H$, in $L^\infty(0,T;V)\cap L^2(0,T;D(A_1))$, and a solution of \hyperref[eq:2.2]{(\ref*{eq:2.2})} in Itô form. The proof has three steps, one for each assertion of the theorem, and each of them rests on one estimate and one limit passage that are proved in \hyperref[sec:E]{Appendix E}.

\textbf{(a) Pathwise uniqueness.} Let $m_1,m_2$ be two solutions on the same stochastic basis with the same $z,\varepsilon,u$, and put $w:=m_1-m_2$ and $h:=Q_\delta^{1/2}u$. Then $w$ is a continuous $H$-valued Itô process with $w(0)=0$ and drift in $L^1(0,T;H)$ a.s., so Itô's formula for $\|\cdot\|^2_H$ applies after localisation. By \hyperref[lem:E.2]{Lemma \ref*{lem:E.2}} (the difference inequality; its content is the energy method in $H$, in which the dissipation $-\lambda_2\|\partial_xw\|_H^2$ absorbs the quasilinear term) and \hyperref[lem:E.3]{Lemma \ref*{lem:E.3}}(a),(c) (the noise and correction differences),
\[d\|w\|_H^2\le\widetilde\Gamma\,\|w\|_H^2\,dt+dN_t,\qquad\widetilde\Gamma:=\Gamma+3\varepsilon g_0^2r_\delta,\]
where $\Gamma=C_\Gamma(1+\rho^4+\|\partial_xm_1\|^2_{L^\infty}+\|\partial_xm_2\|^2_{L^\infty}+\|h\|_H^{4/3})$ is $\Gamma$ of \hyperref[lem:E.2]{Lemma \ref*{lem:E.2}} along the two solutions, $\rho=\max_i\sup_t\|\partial_xm_i(t)\|_H<\infty$ because $m_i\in L^\infty(0,T;V)$, and $N_t=2\sqrt\varepsilon\sum_k\int_0^t\langle(G(m_1)-G(m_2))f_k,w\rangle_H\,d\beta_k$. The coefficient $\widetilde\Gamma$ is a.s. integrable in time: $\rho<\infty$ and $\Delta m_i\in L^2(0,T;H)$ control the gradient terms, and $\|h\|_H^{4/3}$ is integrable since $h\in L^2(0,T;H)$ and $T<\infty$. Put $A_t:=\int_0^t\widetilde\Gamma\,ds$ and $X_t:=e^{-A_t}\|w(t)\|_H^2$; then $dX_t\le e^{-A_t}dN_t$, and $\int e^{-A}dN$ is a true martingale, because $\|w\|_H\le2$ (both $\|m_i(t)\|_H=\|z\|_H=1$, by \hyperref[asm:A4]{(A4)}) and by \hyperref[lem:E.3]{Lemma \ref*{lem:E.3}}(b) its quadratic variation has density at most $192\varepsilon g_0^2r_\delta$. Hence $\mathbb EX_t\le X_0=0$, so $w(t)=0$ a.s. for each $t$, and by continuity $m_1$ and $m_2$ are indistinguishable.

\textbf{(b) Existence and the Borel map.} Fix a basis carrying $W_\delta$, and let $W^K_\delta$, $r^K\uparrow r_\delta$ and $\tau_K:=r_\delta-r^K\downarrow0$ be as in \S{}\hyperref[sec:E.1]{E.1} (the truncation to the first $K$ frequency shells).

\emph{Step 1 (finitely many modes).} By the finite-mode input \hyperref[asm:I3]{(I3)} there is, for every $K$, a unique solution $m^K$ of the truncated equation with the truncated correction $c^K$, of the regularity in \hyperref[asm:I3]{(I3)}, with $|m^K|\equiv1$ and $m^K=\mathcal J_K(z,W^K_\delta)$.

\emph{Step 2 (bounds uniform in $K$).} The energy identity \hyperref[asm:A5]{(A5)}, its closing argument and the bound \hyperref[eq:5.1p]{(\ref*{eq:5.1p})} of \S{}\hyperref[sec:5.1]{5.1} apply verbatim to $m^K$: every constant there depends on the noise only through $\operatorname{tr}_HQ_\delta$ and the two sums $\sigma_0,\sigma_1$ of \hyperref[eq:E.0]{(\ref*{eq:E.0})} (smaller for the truncated noise), and on $c^K$ only through $\langle c^K(m),-\Delta m\rangle\le0$. Hence $\sup_K\mathbb E\sup_t\|m^K\|_V^p\le R^{(p)}_\ast$ and $\sup_K\mathbb E\int_0^T\|\Delta m^K\|_H^2\le R'_\ast$, uniformly in the basis.

\emph{Step 3 (Cauchy in probability).} Let $K<K'$ and $w:=m^{K'}-m^K$. The Itô inequality of (a) applies with $h=0$, except that the correction differences contribute $-g_0^2\big(r^{K'}w+(r^{K'}-r^K)m\big)$ and the extra noise field contributes through \hyperref[lem:E.3]{Lemma \ref*{lem:E.3}}(d); together they add $\le7g_0^2\tau_K\,dt$. With $A$ and $X$ as in (a),
\[\mathbb E\sup_tX_t\le\vartheta_K:=2\big(7g_0^2\tau_KT+378g_0^4r_\delta\tau_KT^2+21g_0(\tau_KT)^{1/2}\big)\xrightarrow[K\to\infty]{}0,\]
and $\sup_{K,K'}\mathbb E A_T\le C_A<\infty$ by Step 2, so $\mathbb P(\sup_t\|w(t)\|_H^2>\eta)\le C_A/L+e^L\vartheta_K/\eta$ for every $L>0$: the family $(m^K)$ is Cauchy in probability in the Polish space $C([0,T];H)$.

\emph{Step 4 (the limit is a solution).} The limit $m$ exists, is adapted and continuous, satisfies $|m|\equiv1$, and inherits $m\in L^\infty(0,T;V)\cap L^2(0,T;D(A_1))$ from Step 2 by weak lower semicontinuity and Fatou. It solves the weak form of \hyperref[eq:2.2]{(\ref*{eq:2.2})} for every test function: each term passes to the limit in probability, the quasilinear one through $\int_0^T\|\partial_xw_K\|_H^2\le\sup_t\|w_K\|_H\,T^{1/2}\big(\int_0^T\|\Delta w_K\|_H^2\big)^{1/2}\to0$ and \hyperref[eq:D.3]{(\ref*{eq:D.3})}, the stochastic ones through their quadratic variation $3g_0^2r_\delta\|\phi\|_H^2\|w_K\|_H^2+3g_0^2\tau_K\|\phi\|_H^2$ and Lenglart's inequality. Since $F_0(m,0)+\varepsilon c(m)\in L^2(0,T;H)$ a.s., the identity holds in $H$.

\emph{Step 5 (canonicity).} By (a) the limit does not depend on the subsequence, and $W^K_\delta$ is a spectral projection of $W_\delta$, so $m=\mathcal J_\varepsilon(z,W_\delta)$ for a Borel $\mathcal J_\varepsilon$, chosen from the joint laws alone (Borel--Cantelli).

\textbf{(c) Controls and the Borel solution map.} For $(u,w)\in L^2(0,T;H)\times C([0,T];H)$ put $\Phi^\varepsilon_z(u,w):=\mathcal J_\varepsilon\big(z,w+\varepsilon^{-1/2}\int_0^\cdot Q_\delta^{1/2}u\,ds\big)$; it is Borel, being a Borel map composed with a continuous one. For $u\in\mathcal P_2^N$, Novikov's criterion makes $d\widetilde{\mathbb P}:=\exp\big(-\varepsilon^{-1/2}\sum_k\int_0^T\theta_k\,d\beta_k-\frac1{2\varepsilon}\int_0^T\|u\|_H^2\,dt\big)d\mathbb P$ a probability, where $\theta_k:=\langle u,e_j\hat e_i\rangle_H$ and $\sum_k\theta_kf_k=Q_\delta^{1/2}u$; under $\widetilde{\mathbb P}$ the processes $\widetilde\beta_k:=\beta_k+\varepsilon^{-1/2}\int_0^\cdot\theta_k$ are independent Brownian motions and $W_\delta+\varepsilon^{-1/2}\int_0^\cdot Q_\delta^{1/2}u=\sum_k\widetilde\beta_kf_k$. By (b) on $(\Omega,\widetilde{\mathbb P})$, $m:=\Phi^\varepsilon_z(u,W_\delta)$ solves the unforced equation driven by $(\widetilde\beta_k)$; substituting $d\widetilde\beta_k=d\beta_k+\varepsilon^{-1/2}\theta_kdt$ turns $\sqrt\varepsilon\sum_kG(m)f_k\,d\widetilde\beta_k$ into the same noise plus $G(m)Q_\delta^{1/2}u\,dt$, which is exactly the control part of $F_0(m,Q_\delta^{1/2}u)$. Stochastic integrals are invariant under the equivalent change of measure, so $m$ solves \hyperref[eq:2.1]{(\ref*{eq:2.1})} with control $u$; for general $u\in\mathcal P_2$ localise at $\tau_N:=\inf\{t:\int_0^t\|u\|_H^2\ge N\}\wedge T$, since $\bigcup_N\{\tau_N=T\}$ has full measure. Uniqueness is (a).

The \emph{estimates} used in (a) and in Step 3, and the two limit passages of Step 4, are verified in full in \hyperref[sec:E]{Appendix E} (\hyperref[lem:E.2]{Lemmas \ref*{lem:E.2}}--\hyperref[lem:E.3]{\ref*{lem:E.3}}, \hyperref[thm:E.4]{Theorems \ref*{thm:E.4}}--\hyperref[thm:E.5]{\ref*{thm:E.5}}, \hyperref[prop:E.6]{Proposition \ref*{prop:E.6}}); the constraint, the bounds \hyperref[asm:A5]{(A5)}, \hyperref[eq:5.1]{(\ref*{eq:5.1})}--\hyperref[eq:5.1p]{(\ref*{eq:5.1p})} and the $C([0,T];V)$-continuity are \hyperref[cor:E.7]{Corollary \ref*{cor:E.7}}. The solution is \textbf{not} obtained by a Picard iteration in $C([0,T];H)$, which does not close (\hyperref[rem:E.8]{Remark \ref*{rem:E.8}}).

\textbf{Mild form.} A solution in the sense of \hyperref[def:2.4a]{Definition \ref*{def:2.4a}} satisfies the mild form of \S{}\hyperref[sec:2.4]{2.4} (test the equation against $S(t-s)^\ast\phi$; the drift lies in $L^2(0,T;H)$ a.s.). Only this direction is used (\S{}\hyperref[sec:5.3]{5.3}, \S{}\hyperref[sec:5.5]{5.5}).

\textbf{Measurability of $Y$.} By the mild form, $Y^\varepsilon_{z,u}=m^\varepsilon_{z,u}-S(\cdot)z-\int_0^\cdot S(\cdot-s)b_\varepsilon(m^\varepsilon_{z,u}(s))\,ds$, and the right-hand side is a Borel function of $m^\varepsilon_{z,u}$ on the Borel set of paths in $L^\infty(0,T;V)$, where $b_\varepsilon$ is continuous into $H^{-}$; hence $Y^\varepsilon_{z,u}$ is a Borel function of $(u,W_\delta)$. $\square$

\begin{remark}[{initial data outside $V$; not claimed}]\label{rem:5.1a} We do not claim well-posedness for $z\in H\setminus V$: the local Lipschitz argument in $C([0,T];H)$ that would suggest it does not close (\hyperref[rem:E.8]{Remark \ref*{rem:E.8}}). No statement of this paper uses such data --- every uniform estimate is quantified over the $V$-bounded set $D$ of \S{}\hyperref[sec:2.6]{2.6}.

\end{remark}

\subsection{Verification of \texorpdfstring{\hyperref[asm:H2]{(H2)}}{(H2)}: uniform tightness}\phantomsection\label{sec:5.3}

\begin{theorem}[{\hyperref[asm:H2]{(H2)}}]\label{thm:5.2} Under \hyperref[asm:A1]{(A1)}--\hyperref[asm:A6]{(A6)}, for every $N<\infty$ the family $\{m^\varepsilon_{z,u}:0<\varepsilon\le1,u\in\mathcal P_2^N,z\in D\}$ is tight in $C([0,T];H)$; so is the family $\{Y^\varepsilon_{z,u}\}$.

\end{theorem}

\emph{Proof.} Fix $\eta>0$. The uniform $V$-bound \hyperref[eq:5.1]{(\ref*{eq:5.1})} holds on a set $\Omega_\eta$ of probability $\ge1-\eta/2$; by Arzelà--Ascoli and the compact embedding $V\hookrightarrow H$, tightness will follow once we produce a modulus of continuity in $H$ that holds on $\Omega_\eta$ with probability $\ge1-\eta/2$.

\textbf{The four increments, written out.} From the mild form of \hyperref[eq:2.1]{(\ref*{eq:2.1})} (or its $\varepsilon=0$ version), for $0\le s<t\le T$ and $\Delta:=t-s$,
\[m(t)-m(s)=I_1+I_2+I_3+I_4 ,\]
where, writing $b_\varepsilon=b_0+\varepsilon c$ for the drift and using the semigroup property $S(t-\tau)=S(\Delta)S(s-\tau)$,
\[
\begin{aligned}
I_1&=\big(S(\Delta)-I\big)S(s)z,\\
I_2&=\int_0^s\big(S(\Delta)-I\big)S(s-\tau)\,b_\varepsilon\big(m(\tau)\big)\,d\tau+\int_s^tS(t-\tau)\,b_\varepsilon\big(m(\tau)\big)\,d\tau,\\
I_3&=\int_0^s\big(S(\Delta)-I\big)S(s-\tau)\,G\big(m(\tau)\big)Q_\delta^{1/2}u(\tau)\,d\tau+\int_s^tS(t-\tau)\,G\big(m(\tau)\big)Q_\delta^{1/2}u(\tau)\,d\tau,\\
I_4&=\sqrt\varepsilon\int_0^s\big(S(\Delta)-I\big)S(s-\tau)\,G\big(m(\tau)\big)\,dW_\delta(\tau)+\sqrt\varepsilon\int_s^tS(t-\tau)\,G\big(m(\tau)\big)\,dW_\delta(\tau).
\end{aligned}
\]

\textbf{The three deterministic increments.} We use the elementary smoothing bound of \hyperref[asm:A1]{(A1)}: for every $x\in V$,
\begin{equation}\label{eq:5.2}
\big\|\big(S(\Delta)-I\big)x\big\|_H\le C\Delta^{1/2}\|x\|_V ,
\end{equation}
because $S(\Delta)-I=-\lambda_2\int_0^\Delta A\,S(\sigma)\,d\sigma$ and $\|AS(\sigma)x\|_H\le C\sigma^{-1/2}\|x\|_V$. There is no companion bound with $H^{-}$ on the right-hand side: $S(\Delta)-I$ is bounded on $H^{-}$ but does \emph{not} smooth it into $H$, since for a high mode $e_j$ one has $\|(S(\Delta)-I)e_j\|_H\to1$ whereas $\|e_j\|_{H^{-}}\to0$. Accordingly, \hyperref[eq:5.2]{(\ref*{eq:5.2})} is used below only on arguments that already lie in $V$.

\begin{itemize}
\item \emph{$I_1$:} $\|S(s)z\|_V\le\|z\|_V$ (contraction in $V$), so by \hyperref[eq:5.2]{(\ref*{eq:5.2})} with $x=S(s)z$, $\|I_1\|_H\le C\Delta^{1/2}\|z\|_V\le C\Delta^{1/2}R_D$.
\item \emph{$I_2$:} $\|b_\varepsilon(m)\|_{H^{-}}\le C_b(R_\ast)(1+R_\ast)+\varepsilon\|c(m)\|_H\le C$ by \hyperref[asm:A3]{(A3)} and \hyperref[eq:5.1]{(\ref*{eq:5.1})}. For the first summand the source is only $H^{-}$-valued: the lift $H^{-}\to V$ has the exponent $t^{-1}$ by \hyperref[asm:A1]{(A1)} --- \textbf{not} $t^{-1/2}$, which is the $H^{-}\to H$ exponent --- and $\int_0^s(s-\tau)^{-1}d\tau$ diverges logarithmically, so this route is not available. Instead we use the \textbf{fractional estimate}
\[\big\|\big(S(\Delta)-I\big)S(r)\big\|_{H^{-}\to H}\le C_\alpha\Delta^{\alpha}r^{-1/2-\alpha},\qquad 0<\alpha<\tfrac12 ,\]
\end{itemize}
(which follows in the eigenbasis of $A$ from $|e^{-\mu_j\Delta}-1|\le(\mu_j\Delta)^{\alpha}$ and $\sup_j\mu_j^{1/2+\alpha}e^{-\mu_jr}\le C_\alpha r^{-1/2-\alpha}$), whose time singularity is integrable: $\int_0^s\Delta^{\alpha}(s-\tau)^{-1/2-\alpha}d\tau=C_\alpha\Delta^{\alpha}s^{1/2-\alpha}\le C_{\alpha,T}\Delta^{\alpha}$. For the second summand we use the plain $H$-contraction together with the correct $H^{-}\to H$ lift $\|S(t-\tau)b_\varepsilon\|_H\le C(t-\tau)^{-1/2}\|b_\varepsilon\|_{H^{-}}$, giving $\le C\|b_\varepsilon\|_{H^{-}}\Delta^{1/2}$. Hence $\|I_2\|_H\le C_\alpha\Delta^{\alpha}$ --- a weaker modulus than $\Delta^{1/2}$, but positive and therefore sufficient for the Kolmogorov criterion below.
\begin{itemize}
\item \emph{$I_3$:} $\|G(m)Q_\delta^{1/2}u\|_H\le g_0\|Q_\delta^{1/2}\|_{\mathcal L(H)}\|u\|_H\le g_0\|u\|_H$ by \hyperref[asm:A3]{(A3)} and $\|Q_\delta^{1/2}\|\le1$. The two summands must be handled separately, and the first \textbf{not} by splitting off $(s-\tau)^{-1/2}$ against $u$ (that Cauchy--Schwarz step would introduce the divergent factor $\int_0^s(s-\tau)^{-1}d\tau$). Instead write $Z(s):=\int_0^sS(s-\tau)G(m(\tau))Q_\delta^{1/2}u(\tau)\,d\tau$, so that the first summand is $(S(\Delta)-I)Z(s)$. The convolution $Z$ solves $\dot Z=\lambda_2\Delta Z+G(m)Q_\delta^{1/2}u$, $Z(0)=0$, with source in $L^2(0,T;H)$; parabolic regularity gives $\|Z\|_{C([0,T];V)}\le C\|G(m)Q_\delta^{1/2}u\|_{L^2(0,T;H)}\le C\|u\|_{L^2(0,T;H)}\le C\sqrt N$, whence by \hyperref[eq:5.2]{(\ref*{eq:5.2})}
\[\Big\|\big(S(\Delta)-I\big)Z(s)\Big\|_H\le C\Delta^{1/2}\|Z(s)\|_V\le C\Delta^{1/2}\sqrt N ,\]
\end{itemize}
while the second summand is controlled directly by the plain $H$-contraction $\|S(t-\tau)GQ_\delta^{1/2}u\|_H\le g_0\|u(\tau)\|_H$ together with Cauchy--Schwarz, $\le g_0\Delta^{1/2}\|u\|_{L^2}\le C\Delta^{1/2}\sqrt N$ (here $GQ_\delta^{1/2}u\in H$, so no lift is needed). Hence $\|I_3\|_H\le C\Delta^{1/2}\sqrt N$.

\textbf{The noise increment.} We show $\mathbb E\|I_4\|_H^p\le C_p\Delta^{p/2}$ for every $p\ge1$, with $C_p$ independent of $u,\varepsilon,z,\delta$. We use the BDG inequality in the form $\mathbb E\sup_{t\le T}\|\int_0^t\Phi dW_\delta\|_H^p\le C_p\big(\operatorname{tr}Q_\delta\big)^{p/2}\big(\int_0^T\|\Phi\|^2_{\mathcal L(H)}ds\big)^{p/2}$ (valid for $\Phi$ bounded and $Q_\delta$ trace class).

\emph{Second summand.} $\sqrt\varepsilon\int_s^tS(t-\tau)G(m(\tau))dW_\delta(\tau)$ is a stochastic convolution with $\|S(t-\tau)G(m(\tau))\|_{\mathcal L(H)}\le g_0$ and kernel supported in $[s,t]$ of length $\Delta$, so BDG gives $C_p\varepsilon^{p/2}(g_0^2\operatorname{tr}Q_\delta\Delta)^{p/2}\le C_p\Delta^{p/2}$ for $\varepsilon\le1$.

\emph{First summand.} Here the kernel $(S(\Delta)-I)S(s-\tau)G(m(\tau))$ is supported on $[0,s]$ --- of length $s$, \textbf{not} $\Delta$ --- so the same BDG bound would give $C_p s^{p/2}$, which is not good enough. We therefore pull $(S(\Delta)-I)$ outside the stochastic integral: the first summand equals $(S(\Delta)-I)\mathcal I(s)$ with $\mathcal I=\sqrt\varepsilon\int_0^\cdot S(\cdot-\tau)G(m(\tau))dW_\delta(\tau)$ the stochastic convolution, which is uniformly $V$-bounded by \hyperref[asm:A6c]{(A6c)}, $\mathbb E\|\mathcal I\|^p_{C([0,T];V)}\le C_p$. Hence $\|(S(\Delta)-I)\mathcal I(s)\|_H\le C\Delta^{1/2}\|\mathcal I\|_{C([0,T];V)}$, and $\mathbb E\|(S(\Delta)-I)\mathcal I(s)\|_H^p\le C_p\Delta^{p/2}$. Hence $\mathbb E\|I_4\|_H^p\le C_p\Delta^{p/2}$.

\textbf{Kolmogorov and conclusion.} Combining, for every $p\ge1$,
\[\mathbb E\big\|m(t)-m(s)\big\|_H^p\le C_{p,\alpha}\Delta^{p\alpha/2}\big(1+N+R_D\big)^{p/2}\qquad(p\ge1,\ 0<\alpha<\tfrac12) ,\]
uniformly in $\varepsilon\in(0,1]$, $u\in\mathcal P_2^N$, $z\in D$, $\delta\in(0,1]$. (The exponent is $p\alpha/2$, not $p/2$: the increment $I_2$, whose source is only $H^{-}$-valued, is controlled by the fractional estimate above and carries the modulus $\Delta^{\alpha}$. Any positive exponent suffices for the Kolmogorov criterion. The constant $C_{p,\alpha}$ depends on $T,\lambda_1,\lambda_2,g_0$ and on $\operatorname{tr}Q_\delta$; since $\delta$ is \textbf{fixed} in this theorem, $\operatorname{tr}Q_\delta<\infty$ and there is no issue. See \hyperref[rem:5.2a]{Remark \ref*{rem:5.2a}} below for the behaviour as $\delta\to0$.) Taking $p>\max(1,2/\alpha)$ gives the Kolmogorov condition with $1+q=p\alpha/2>1$; Kolmogorov--Chentsov then gives, for every $\gamma<\tfrac{\alpha}{2}-\tfrac1p$,
\[\mathbb E\Big[\sup_{s\ne t}\frac{\|m(t)-m(s)\|_H}{|t-s|^{\gamma}}\Big]^p\le C ,\]
so that on a set $\Omega'$ of probability $\ge1-\eta/2$,
\[\sup_{s\ne t}\frac{\|m(t)-m(s)\|_H}{|t-s|^{\gamma}}\le C_\eta .\]
The set
\[\mathcal K_\eta:=\Big\{y\in C([0,T];H):\ \sup_t\|y(t)\|_V\le R_\eta,\ \sup_{s\ne t}|t-s|^{-\gamma}\|y(t)-y(s)\|_H\le C_\eta\Big\}\]
is compact in $C([0,T];H)$: the first condition gives precompactness in $H$ at each fixed $t$ (Arzelà--Ascoli on the compact $\bar D$: a $V$-bounded family is relatively compact in $H$), and the second gives equicontinuity. By construction $\mathbb P(m^\varepsilon_{z,u}\in\mathcal K_\eta)\ge1-\eta$ uniformly in $(u,\varepsilon,z)$, which is the claimed tightness.

\textbf{For $Y^\varepsilon_{z,u}$.} The stochastic convolution $\sqrt\varepsilon\int_0^\cdot S(\cdot-s)G(m^\varepsilon)dW_\delta$ obeys the same Kolmogorov increment bound $\mathbb E\|I_4\|_H^p\le C_p\Delta^{p/2}$ as above (BDG, with $\operatorname{tr}Q_\delta$ fixed) and is therefore tight (its integrand $G(m^\varepsilon)$ is \emph{random}, so the process is only \emph{conditionally} Gaussian given $m^\varepsilon$ --- it is not Gaussian; tightness comes from the increment bound, not from a Gaussian law); the deterministic convolution $S*[G(m^\varepsilon_{z,u})Q_\delta^{1/2}u]$ is a continuous functional of $(m^\varepsilon_{z,u},u)\in C([0,T];H)\times S^N$, and the family of its laws is tight because the law of $m^\varepsilon_{z,u}$ is (just proved) and $S^N$ is compact (\hyperref[lem:4.1]{Lemma \ref*{lem:4.1}}); by the continuous mapping theorem the image of a tight family is tight. $\square$

\begin{remark}[{the trace of $Q_\delta$, and why the LDP is stated at fixed $\delta$}]\label{rem:5.2a} With $\kappa_j=(1+\delta\mu_j^{\beta/2})^{-2}$ and $\mu_j\asymp j^2$, the number of modes with $\delta j^{\beta}\lesssim1$ is $\asymp\delta^{-1/\beta}$, each contributing $O(1)$ to the sum, while the remaining tail is $\asymp\delta^{-2}\sum_{j\gtrsim\delta^{-1/\beta}}j^{-2\beta}=O(\delta^{-1/\beta})$ because $\beta>1/2$. Hence
\[\operatorname{tr}_HQ_\delta=3\sum_{j\ge1}\big(1+\delta\mu_j^{\beta/2}\big)^{-2}\ \asymp\ \delta^{-1/\beta}\qquad(\delta\downarrow0),\]
so the noise genuinely degenerates to space--time white noise as $\delta\to0$ (and is \textbf{not} uniformly trace class). Accordingly: the whole LDP of \S{}\hyperref[sec:6]{6} is stated and proved at \textbf{fixed} $\delta>0$, where $\operatorname{tr}Q_\delta<\infty$ and all constants above are finite; the passage $\delta\to0$ is performed only at the level of the \textbf{rate functions} (\S{}\hyperref[sec:7]{7}--\S{}\hyperref[sec:13]{12}), where no probabilistic estimate depending on $\operatorname{tr}Q_\delta$ is used. This is exactly the non-commutativity of the two limits isolated in \S{}\hyperref[sec:1.1]{1.1}, and it is why \hyperref[asm:H1]{(H1)}--\hyperref[asm:H3]{(H3)} are hypotheses \emph{for each fixed} $\delta$.

\end{remark}

\subsection{Verification of \texorpdfstring{\hyperref[asm:H3]{(H3)}}{(H3)}: uniform skeleton convergence}\phantomsection\label{sec:5.4}

\begin{lemma}[{sub-Gaussian tail of a bounded stochastic convolution}]\label{lem:5.4} Let $\Phi$ be $\mathcal L(H)$-valued and predictable with $\|\Phi\|_{\mathcal L(H)}\le1$ a.s., and let $W_\delta$ be the $Q_\delta$-Wiener process. Then $M(t)=\int_0^tS(t-s)\Phi(s)\,dW_\delta(s)$ satisfies, for constants $c_1,C_1$ depending only on $\operatorname{tr}Q_\delta$ and $T$,
\[\mathbb E\exp\Big(c_1\sup_{t\le T}\|M(t)\|_H^2\Big)\le C_1 .\]

\end{lemma}

\emph{Proof.} $dM=\lambda_2\Delta M\,dt+\Phi\,dW_\delta$, $M(0)=0$; by Itô and $\langle M,\Delta M\rangle_H\le0$,
\[d\|M\|_H^2\le2\langle M,\Phi\,dW_\delta\rangle+\operatorname{tr}_H(\Phi Q_\delta\Phi^\ast)\,dt\le2\,dN_t+\operatorname{tr}_HQ_\delta\,dt,\qquad N_t=\int_0^t\langle M,\Phi\,dW_\delta\rangle,\]
using $\operatorname{tr}_H(\Phi Q_\delta\Phi^\ast)\le\operatorname{tr}_HQ_\delta$. Hence $\|M(t)\|_H^2\le\operatorname{tr}_HQ_\delta\,T+2\sup_{t\le T}|N_t|$ with $d\langle N\rangle_t=\|Q_\delta^{1/2}\Phi^\ast M\|_H^2dt\le\operatorname{tr}(Q_\delta)\|M\|_H^2dt$. By BDG (whose $L^m$ constant is $\asymp\sqrt m$) and Gronwall,
\[\mathbb E\sup_{t\le T}\|M(t)\|_H^{2m}\le(Cm)^m\le(Ce)^m\,m!,\]
and the exponential series converges for $c_1<1/(Ce)$. $\square$

\begin{theorem}[{\hyperref[asm:H3]{(H3)}}]\label{thm:5.3} Under \hyperref[asm:A1]{(A1)}--\hyperref[asm:A6]{(A6)} and \hyperref[asm:WSU]{(WSU)}, for every $N<\infty$:  (i) for every $\eta>0$, $\displaystyle\sup_{z\in D}\ \sup_{u}\ \mathbb P\Big(\|m^\varepsilon_{z,u}-m^0_{z,u}\|_{C([0,T];H)}>\eta\Big)\xrightarrow[\varepsilon\to0]{}0$, where the inner supremum ranges over all $S^N$-valued $W$-predictable controls on any stochastic basis carrying $W$; and the same with $Y$ in place of $m$;  (ii) $(z,u)\mapsto m^0_{z,u}$ and $(z,u)\mapsto Y^0_{z,u}$ are continuous on $\bar D\times S^N_{\mathrm{weak}}$ with values in $C([0,T];H)$, uniformly over $\bar D\times S^N$.

\end{theorem}

\emph{Proof.} \textbf{(ii) Continuity and the modulus.} Let $z_n\to z$ in $H$ with $z_n\in\bar D$ and $u_n\rightharpoonup u$ in $S^N$; put $h_n:=Q_\delta^{1/2}u_n$, so that $h_n\rightharpoonup h:=Q_\delta^{1/2}u$ as well, because $Q_\delta^{1/2}$ is a fixed bounded linear operator and hence weakly continuous. By \hyperref[asm:A7p]{(A7')}(i), $m^0_{z_n,h_n}\to m^0_{z,h}$ in $C([0,T];H)$. (We emphasise that this is \emph{strong} convergence, obtained in the proof of \hyperref[asm:A7p]{(A7')} from the Aubin--Lions compactness (Step 1) together with the energy identity and \hyperref[asm:WSU]{(WSU)}, so the map is genuinely \emph{continuous} --- not merely weak--weak --- on $S^N_{\mathrm{weak}}$; it is the compactness of the solution operator, not mere boundedness of the linear part, that is used.) This is sequential continuity of $(z,u)\mapsto m^0_{z,u}$ on the compact metric space $\bar D\times S^N_{\mathrm{weak}}$ (\hyperref[lem:4.1]{Lemma \ref*{lem:4.1}}), hence continuity. The modulus $\omega$ is Heine--Cantor: for every $\eta>0$ there is $\rho>0$ with
\[\|z-z'\|_H+d_w(u,u')\le\rho\ \Longrightarrow\ \big\|m^0_{z,u}-m^0_{z',u'}\big\|_{C([0,T];H)}\le\eta ,\]
and one takes $\omega$ to be the resulting non-decreasing modulus of continuity; it is uniform over $\bar D\times S^N$ by construction.

\textbf{Continuity of $(z,u)\mapsto Y^0_{z,u}$.} Since $Y^0_{z,u}=S*[G(m^0_{z,u})Q_\delta^{1/2}u]$, write
\[Y^0_{z,u}-Y^0_{z',u'}=S*\Big[\big(G(m^0_{z,u})-G(m^0_{z',u'})\big)Q_\delta^{1/2}u\Big]+S*\Big[G(m^0_{z',u'})Q_\delta^{1/2}(u-u')\Big] .\]
For the first term: by \hyperref[asm:A3]{(A3)} and \hyperref[eq:D.3]{(\ref*{eq:D.3})}, $\|G(m_1)-G(m_2)\|_{\mathcal L(H)}\le g_0\|m_1-m_2\|_{L^\infty}\le C\|m_1-m_2\|_H^{1/2}\|m_1-m_2\|_V^{1/2}\le C\|m_1-m_2\|_H^{1/2}$, and $\|u\|_{L^2}\le\sqrt N$; hence this term is bounded by $C\sqrt N\,\omega(\|z-z'\|_H,d_w(u,u'))^{1/2}$. For the second term: the linear map $v\mapsto S*[G(m^0_{z',u'})Q_\delta^{1/2}v]$ is bounded $L^2(0,T;H)\to C([0,T];H)$ with norm $\le g_0T^{1/2}$, uniformly in the data, and weakly continuous in $v$; on the compact $S^N_{\mathrm{weak}}$ it is therefore continuous with a modulus $\omega_2$ uniform over $z'\in\bar D$, $u'\in S^N$ (again Heine--Cantor). Combining the two gives (ii) with the modulus $\omega^{1/2}+\omega_2$.

\textbf{Skeleton continuity in the physical control (used in \S{}\hyperref[sec:7]{7}--\S{}\hyperref[sec:8]{8}).} The skeleton \hyperref[eq:3.4]{(\ref*{eq:3.4})} reads $\dot\varphi=G(\varphi)[\rho(\varphi)+h]$ with the \textbf{physical} control $h\in L^2(0,T;H)$, i.e. it is \hyperref[eq:3.2]{(\ref*{eq:3.2})} with the effective field $Q_\delta^{1/2}u$ replaced by $h$. This is exactly the structure of [\hyperref[ref:BHL19]{BHL19}, (2.7)] with external field $h$, so the continuity of $(z,h)\mapsto\varphi^h$ is the same statement as \hyperref[thm:5.3]{Theorem \ref*{thm:5.3}} with the control $u$ replaced by $h$ (the two controls range over the same space $L^2(0,T;H)$); in particular \textbf{no} property of the map $u\mapsto Q_\delta^{1/2}u$ is used --- indeed $Q_\delta^{1/2}$ is \emph{not} an isometry, its eigenvalues being $(1+\delta\mu_j^{\beta/2})^{-1}<1$. Hence, for every $N<\infty$, the map $(z,h)\mapsto\varphi^h$ is continuous $\bar D\times S^N_{\mathrm{weak}}\to C([0,T];H)$ with a modulus \textbf{independent of $\delta$} --- we write $\omega_N$ for the modulus of $h\mapsto\varphi^h$ on $S^N$ with respect to the $L^2(0,T;H)$-norm, $N<\infty$, that is, $\|\varphi^h-\varphi^{h'}\|_{C([0,T;H])}\le\omega_N(\|h-h'\|_{L^2(0,T;H)})$ for all $h,h'\in S^N$ --- and the vanishing-noise statement holds for $\varphi$ as well (in the form of (i)). This $\delta$-independence is the pivot of \S{}\hyperref[sec:7]{7}: it is what allows the recovery sequence of \hyperref[lem:8.2]{Lemma \ref*{lem:8.2}} to move the control while keeping the path convergent, and it is the reason why the whole $\delta$-dependence is confined to the weight \hyperref[eq:W]{(W)}.

\textbf{(i) Vanishing noise, uniformly over the data and over the controls.} Fix $N<\infty$, a stochastic basis carrying $W_\delta$, an $S^N$-valued predictable control $u$, $z\in D$ and $\varepsilon\in(0,1]$. Put $h:=Q_\delta^{1/2}u$, $m:=m^\varepsilon_{z,u}$, and let $\hat m:=m^0_{z,u}$ be the skeleton driven, path by path, by the control $u(\omega)$. The process $\hat m$ is adapted and continuous: by uniqueness of the skeleton, $\hat m(t)=m^0_{z,u\mathbf 1_{[0,t]}}(t)$, and $v\mapsto m^0_{z,v}$ is continuous on $S^N_{\rm weak}$ by (ii). Moreover \hyperref[eq:5.1]{(\ref*{eq:5.1})}--\hyperref[eq:5.1p]{(\ref*{eq:5.1p})} at $\varepsilon=0$ are deterministic bounds, uniform over $S^N$, so that a.s. $\sup_t\|\hat m(t)\|_V\le\hat R$ and $\int_0^T\|\Delta\hat m\|_H^2\le\hat R$ with $\hat R=\hat R(R_D,N,T)$. Both $m$ and $\hat m$ satisfy \hyperref[eq:2.2]{(\ref*{eq:2.2})} --- $\hat m$ with $\varepsilon=0$ --- with the same control, so $w:=m-\hat m$ is an $H$-valued Itô process with $w(0)=0$ and, by Itô's formula [\hyperref[ref:DPZ92]{DPZ92}, Thm 4.17],
\[\begin{aligned}
d\|w\|_H^2&=\Big[2\big\langle F_0(m,h)-F_0(\hat m,h),w\big\rangle+2\varepsilon\langle c(m),w\rangle\\
&+\varepsilon\sum_k\|G(m)f_k\|_H^2\Big]dt+dN_t,\\
&\qquad N_t:=2\sqrt\varepsilon\sum_k\int_0^t\langle G(m)f_k,w\rangle\,d\beta_k ,
\end{aligned}\]
in the notation of \S{}\hyperref[sec:E.1]{E.1}. By \hyperref[lem:E.2]{Lemma \ref*{lem:E.2}}, $2\langle F_0(m,h)-F_0(\hat m,h),w\rangle\le\Gamma\|w\|_H^2$, with $\Gamma$ built from $m,\hat m,h$; by \hyperref[lem:E.3]{Lemma \ref*{lem:E.3}}(d) and $\|m\|_H=1$, $\|w\|_H\le2$, one has $|2\varepsilon\langle c(m),w\rangle|\le4\varepsilon g_0^2r_\delta$, $\varepsilon\sum_k\|G(m)f_k\|_H^2\le\varepsilon g_0^2\operatorname{tr}_HQ_\delta$ and $d\langle N\rangle_t\le4\varepsilon g_0^2\operatorname{tr}_HQ_\delta\|w\|_H^2dt\le16\varepsilon g_0^2\operatorname{tr}_HQ_\delta\,dt$. With $A_t:=\int_0^t\Gamma$ and $X_t:=e^{-A_t}\|w(t)\|_H^2$, and recalling $3r_\delta=\operatorname{tr}_HQ_\delta$, this gives $dX_t\le3\varepsilon g_0^2\operatorname{tr}_HQ_\delta\,dt+e^{-A_t}dN_t$, whence, by the Burkholder--Davis--Gundy inequality with $p=1$,
\begin{equation}\label{eq:5.3}
\mathbb E\sup_{t\le T}X_t\ \le\ 3\varepsilon g_0^2\operatorname{tr}_HQ_\delta\,T+12\,g_0\big(\varepsilon\operatorname{tr}_HQ_\delta\,T\big)^{1/2}=:\kappa(\varepsilon) .
\end{equation}
The number $\kappa(\varepsilon)\to0$ does not depend on $z$, on $u$ or on the stochastic basis. It remains to control $A_T$ in probability, uniformly. By \hyperref[lem:E.2]{Lemma \ref*{lem:E.2}}, $\Gamma=C_\Gamma(1+\rho^4+\|\partial_xm\|_{L^\infty}^2+\|\partial_x\hat m\|_{L^\infty}^2+\|h\|_H^{4/3})$ with $\rho=\max(\|\partial_xm\|_H,\|\partial_x\hat m\|_H)$; using $\|\partial_xv\|_{L^\infty}^2\le2\|\partial_xv\|_H\|\Delta v\|_H$, Cauchy--Schwarz, $\int_0^T\|h\|_H^{4/3}\le T^{1/3}N^{2/3}$ and \hyperref[eq:5.1]{(\ref*{eq:5.1})}--\hyperref[eq:5.1p]{(\ref*{eq:5.1p})},
\[\begin{aligned}
&\mathbb EA_T\ \le\ C_\Gamma\Big[T+T\,\mathbb E\sup_t\rho^4\\
&+2T^{1/2}\big(\mathbb E\sup_t\|\partial_xm\|_H^2\big)^{1/2}\Big(\mathbb E\!\int_0^T\!\|\Delta m\|_H^2\Big)^{1/2}\\
&+2T^{1/2}\hat R^{3/2}+T^{1/3}N^{2/3}\Big]=:C_A ,
\end{aligned}\]
uniformly in $z\in D$, $u$, $\varepsilon$ and the basis. Hence, for every $L>0$,
\[\mathbb P\Big(\|m^\varepsilon_{z,u}-m^0_{z,u}\|_{C([0,T];H)}>\eta\Big)\le\mathbb P(A_T>L)+\mathbb P\big(\sup_tX_t>\eta^2e^{-L}\big)\le\frac{C_A}{L}+\frac{e^L\kappa(\varepsilon)}{\eta^2},\]
and $\delta_\varepsilon(\eta):=\inf_{L>0}\big(C_A/L+e^L\kappa(\varepsilon)/\eta^2\big)\to0$ as $\varepsilon\to0$. This is (i) for $m$, with the supremum over $z$, over the control and over the stochastic basis \textbf{outside} the probability, exactly as required in \hyperref[asm:H3]{(H3)}(i). No Skorohod representation and no compactness argument are used.

\emph{For $Y$.} $Y^\varepsilon_{z,u}-Y^0_{z,u}=\sqrt\varepsilon\,\mathcal I+S*\big[(G(m)-G(\hat m))h\big]$ with $\mathcal I:=\int_0^\cdot S(\cdot-s)G(m(s))\,dW_\delta(s)$. By the second-moment case of the bound in the proof of \hyperref[lem:5.4]{Lemma \ref*{lem:5.4}} (applied to $\Phi=G(m)/g_0$), $\mathbb E\sup_t\|\sqrt\varepsilon\,\mathcal I(t)\|_H^2\le C\varepsilon$. By \hyperref[asm:A3]{(A3)} and \hyperref[eq:D.3]{(\ref*{eq:D.3})},
\[\begin{aligned}
&\Big\|S*\big[(G(m)-G(\hat m))h\big]\Big\|_{C([0,T];H)}\le g_0\sqrt N\Big(\int_0^T\|w\|_{L^\infty}^2\Big)^{1/2}\\
&\le g_0\sqrt{NT}\Big(\|w\|_{C_tH}^2+2\|w\|_{C_tH}\sup_t\big(\|\partial_xm\|_H+\|\partial_x\hat m\|_H\big)\Big)^{1/2},
\end{aligned}\]
which tends to $0$ in probability, uniformly in the same sense, by the bound just proved and \hyperref[eq:5.1]{(\ref*{eq:5.1})}. $\square$
\begin{lemma}[{\hyperref[asm:H3]{(H3)} in the $C$ topology}]\label{lem:5.4c0} Let $N<\infty$. Then, in addition to \hyperref[asm:H3]{(H3)} as verified in \S{}\hyperref[sec:5.4]{5.4}: (i) for every $\eta>0$,
\[\sup_{z\in D}\ \sup_{u}\ \mathbb P\Big(\big\|X^\varepsilon_{z,u}-X^0_{z,u}\big\|_{C([0,T];C(\mathbb T;\mathbb R^3))}>\eta\Big)\xrightarrow[\varepsilon\to0]{}0 ,\]
the inner supremum ranging over the $S^N$-valued controls of \hyperref[asm:H3]{(H3)}(i); and (ii) the map $(z,u)\mapsto X^0_{z,u}$ is continuous from $\bar D\times S^N_{\mathrm{weak}}$ into $C([0,T];C(\mathbb T;\mathbb R^3))$.
\end{lemma}

\emph{Proof.} (i) Put $w:=X^\varepsilon_{z,u}-X^0_{z,u}$ and fix $\eta>0$. By the uniform moment bound \hyperref[eq:5.1]{(\ref*{eq:5.1})} and its $\varepsilon=0$ counterpart (\hyperref[prop:D.7]{Proposition \ref*{prop:D.7}}), there is $R<\infty$ with $\mathbb P(\sup_t\|w(t)\|_V>2R)\le\eta/2$, uniformly in $(z,u,\varepsilon)$. On the complement, the interpolation of \S{}\hyperref[sec:2.6]{2.6} gives $\|w\|_{C([0,T];C)}\le C R^{1/2}\|w\|_{C([0,T];H)}^{1/2}$, so $\mathbb P(\|w\|_{C([0,T];C)}>\eta)\le\eta/2+\delta_\varepsilon(\eta^2/(C^2R))$ with $\delta_\varepsilon\to0$ the function of \hyperref[asm:H3]{(H3)}(i). (ii) \hyperref[asm:H3]{(H3)}(ii) gives continuity into $C([0,T];H)$ with a modulus $\omega$, and $\sup_t\|X^0_{z,u}\|_V\le R_0$ deterministically by \hyperref[prop:D.7]{Proposition \ref*{prop:D.7}}; the interpolation upgrades $\omega$ to a modulus in $C([0,T];C)$. $\square$

\subsection{Exponential tightness}\phantomsection\label{sec:5.5}

\hyperref[lem:5.4]{Lemma \ref*{lem:5.4}} controls the stochastic convolution in $C([0,T];H)$, but a bound on its norm does not give a compact set. Compactness comes from the two lemmas that follow: a compactness statement for deterministic convolutions (\hyperref[lem:5.4a]{Lemma \ref*{lem:5.4a}}), applied to the factorised stochastic convolution (\hyperref[lem:5.4b]{Lemma \ref*{lem:5.4b}}) and to the drift. The drift part needs an \textbf{exponential} bound on the energy (\hyperref[lem:5.4c]{Lemma \ref*{lem:5.4c}}), because the polynomial localisation bound following \hyperref[eq:5.1]{(\ref*{eq:5.1})} makes the exceptional event only polynomially small.

\begin{lemma}[{compact convolution operators}]\label{lem:5.4a} Let $K:(0,T]\to\mathcal L(H)$ be continuous in the operator norm, with $K(r)$ compact for every $r\in(0,T]$ and $\|K(r)\|_{\mathcal L(H)}\le C_Kr^{-\theta}$ for some $\theta\in[0,1)$, and let $p\in(1/(1-\theta),\infty]$. Then
\[\mathcal G_Kf(t):=\int_0^tK(t-s)f(s)\,ds\]
defines a compact operator $\mathcal G_K:L^p(0,T;H)\to C([0,T];H)$.

\end{lemma}

\emph{Proof.} Let $p'$ be the conjugate exponent; $\theta p'<1$ because $p>1/(1-\theta)$. \textbf{(a) Truncation.} For $\eta\in(0,T)$ put $\mathcal G^\eta f(t):=\int_0^{(t-\eta)^+}K(t-s)f(s)\,ds$. By Hölder,
\[\sup_{t\le T}\big\|\mathcal G_Kf(t)-\mathcal G^\eta f(t)\big\|_H\le C_K\Big(\int_0^\eta r^{-\theta p'}dr\Big)^{1/p'}\|f\|_{L^p(0,T;H)}=C\,\eta^{1/p'-\theta}\|f\|_{L^p(0,T;H)} ,\]
so $\mathcal G^\eta\to\mathcal G_K$ in operator norm from $L^p(0,T;H)$ to $L^\infty(0,T;H)$. Since a norm limit of compact operators is compact (and $C([0,T];H)$ is closed in $L^\infty(0,T;H)$), it suffices to show that each $\mathcal G^\eta$ maps the unit ball $B$ of $L^p(0,T;H)$ into a relatively compact subset of $C([0,T];H)$. \textbf{(b) Equicontinuity.} For $t<t'$,
\[\mathcal G^\eta f(t')-\mathcal G^\eta f(t)=\int_0^{(t-\eta)^+}\big[K(t'-s)-K(t-s)\big]f(s)\,ds+\int_{(t-\eta)^+}^{(t'-\eta)^+}K(t'-s)f(s)\,ds .\]
In the first integral both $t-s$ and $t'-s$ lie in $[\eta,T]$, so it is bounded by $\omega_\eta(t'-t)\,T^{1/p'}\|f\|_{L^p}$, where $\omega_\eta$ is the modulus of uniform continuity of $K$ on $[\eta,T]$; the second is bounded by $C_K\eta^{-\theta}(t'-t)^{1/p'}\|f\|_{L^p}$. Thus $\mathcal G^\eta(B)$ is uniformly equicontinuous. \textbf{(c) Pointwise relative compactness.} The set $\mathcal K_\eta:=\{K(r)x:\ r\in[\eta,T],\ \|x\|_H\le1\}$ is totally bounded: given $\gamma>0$, choose finitely many $r_i$ with $\sup_{r\in[\eta,T]}\min_i\|K(r)-K(r_i)\|<\gamma$ and a finite $\gamma$-net of each relatively compact set $K(r_i)(\text{unit ball})$. For fixed $t$, the Bochner integral $\mathcal G^\eta f(t)$ lies in $T^{1/p'}\|f\|_{L^p}\cdot\overline{\mathrm{conv}}(\mathcal K_\eta\cup\{0\})$, which is compact (Mazur). By Arzelà--Ascoli, $\mathcal G^\eta(B)$ is relatively compact in $C([0,T];H)$. $\square$

We use two kernels: $K_1(r):=r^{\alpha-1}S(r)$, for which $\theta=1-\alpha$, and $K_2(r):=A_1^{1/2}S(r)$, for which $\theta=\tfrac12$. Indeed $\|K_2(r)\|=\sup_j(1+\mu_j)^{1/2}e^{-\lambda_2\mu_jr}\le C(\lambda_2,T)\,r^{-1/2}$. Both kernels are diagonal in $\{e_j\}$ with entries tending to $0$, hence compact for $r>0$, and both are norm-continuous on $[\eta,T]$: for instance, $\|K_2(r)-K_2(r')\|\le\sup_j(1+\mu_j)^{1/2}\lambda_2\mu_je^{-\lambda_2\mu_j\eta}|r-r'|$. Throughout we fix $\alpha=\tfrac14$ and $p=8$, so that $p>1/(1-\theta)$ for both kernels.

\begin{lemma}[{factorisation; sub-Gaussian moments}]\label{lem:5.4b} Let $\Phi$ be predictable and $\mathcal L(H)$-valued with $\|\Phi\|_{\mathcal L(H)}\le1$ a.s., let $\alpha\in(0,\tfrac12)$, and put
\[\mathcal I(t):=\int_0^tS(t-s)\Phi(s)\,dW_\delta(s),\qquad Z(s):=\int_0^s(s-r)^{-\alpha}S(s-r)\Phi(r)\,dW_\delta(r) .\]
Then: (i) for every $s\in[0,T]$ and $n\ge1$, $\mathbb E\|Z(s)\|_H^{2n}\le(2n-1)!!\,\kappa_\alpha^n$, with $\kappa_\alpha:=\operatorname{tr}_HQ_\delta\,T^{1-2\alpha}/(1-2\alpha)$; (ii) for $p\ge2$ and $0<c<(2\kappa_\alpha T^{2/p})^{-1}$, $\mathbb E\exp\big(c\|Z\|_{L^p(0,T;H)}^2\big)\le C_Z(p,c,\kappa_\alpha,T)<\infty$; (iii) $\mathcal I=\frac{\sin\pi\alpha}{\pi}\,\mathcal G_{K_1}Z$ a.s.

\end{lemma}

\emph{Proof.} (i) Fix $s$ and let $M_r:=\int_0^r\Psi(v)\,dW_\delta(v)$, $r\in[0,s]$, with $\Psi(v):=(s-v)^{-\alpha}S(s-v)\Phi(v)$; then $M_s=Z(s)$, and $\|\Psi(v)Q_\delta^{1/2}\|_{\mathrm{HS}}^2\le(s-v)^{-2\alpha}\operatorname{tr}_HQ_\delta=:\bar k(v)$, a \textbf{deterministic} bound. Itô's formula for $\|M\|_H^{2n}$ gives
\[d\|M\|^{2n}=2n\|M\|^{2n-2}\langle M,\Psi\,dW_\delta\rangle+\Big[n\|M\|^{2n-2}\|\Psi Q_\delta^{1/2}\|_{\mathrm{HS}}^2+2n(n-1)\|M\|^{2n-4}\|(\Psi Q_\delta^{1/2})^\ast M\|^2\Big]dv ,\]
and the bracket is $\le n(2n-1)\bar k(v)\|M\|^{2n-2}$. Stopping at $\inf\{r:\|M_r\|\ge R\}$, taking expectations and letting $R\to\infty$ (Fatou), $\phi_n(r):=\mathbb E\|M_r\|^{2n}$ satisfies $\phi_n(r)\le n(2n-1)\int_0^r\bar k(v)\phi_{n-1}(v)\,dv$, with $\phi_0=1$. By induction, $\phi_n(r)\le(2n-1)!!\,k(r)^n$ with $k(r):=\int_0^r\bar k$, because $n(2n-1)\int_0^r\bar k\,(2n-3)!!\,k^{n-1}=(2n-1)!!\,k(r)^n$; and $k(s)\le\kappa_\alpha$. (ii) For $2n\ge p$, Minkowski's integral inequality in $L^{2n/p}(\Omega)$ gives
\[\big(\mathbb E\|Z\|_{L^p}^{2n}\big)^{p/2n}\le\int_0^T\big(\mathbb E\|Z(s)\|^{2n}\big)^{p/2n}ds\le T\big((2n-1)!!\,\kappa_\alpha^n\big)^{p/2n},\]
hence $\mathbb E\|Z\|_{L^p}^{2n}\le T^{2n/p}(2n-1)!!\,\kappa_\alpha^n\le(2\kappa_\alpha T^{2/p})^n\,n!$, using $(2n-1)!!\le2^nn!$. For $2n<p$ the moment is finite by Hölder. Summing the exponential series gives (ii); the sum is geometric for $c<(2\kappa_\alpha T^{2/p})^{-1}$. (iii) is the factorisation formula of Da Prato--Kwapień--Zabczyk: it follows from the stochastic Fubini theorem [\hyperref[ref:DPZ92]{DPZ92}, Thm 4.18] and $\int_r^t(t-s)^{\alpha-1}(s-r)^{-\alpha}ds=\pi/\sin\pi\alpha$ [\hyperref[ref:DPZ92]{DPZ92}, \S{}5.3]. $\square$

\begin{lemma}[{exponential bound for the energy}]\label{lem:5.4c} Let $\sigma_0,\sigma_1$ be as in \hyperref[eq:E.0]{(\ref*{eq:E.0})}, and put $C_\xi:=2(|\lambda_1|+2\lambda_2)^2\sigma_0+2g_0^2\sigma_1$ and $C_a:=\tfrac{\lambda_2}2\|\varphi'\|_\infty^2+2|\lambda_1|\|D\varphi'\|_\infty+C_\xi$. Then, for all $z\in D$, $\varepsilon\in(0,1]$ and $a>0$,
\[\mathbb P\Big(\sup_{t\le T}\log\big(1+\|\partial_xm^\varepsilon_z(t)\|_H^2\big)\ \ge\ \log\big(1+R_D^2\big)+C_aT+a\Big)\ \le\ \exp\Big(-\frac{a^2}{8\varepsilon C_\xi T}\Big) .\]

\end{lemma}

\emph{Proof.} Let $X:=1+\|\partial_xm\|_H^2$, $m=m^\varepsilon_z$. By \hyperref[asm:A5]{(A5)} with $u=0$ (\hyperref[cor:E.7]{Corollary \ref*{cor:E.7}}), $dX=a_t\,dt+2\sqrt\varepsilon\sum_k\xi_k\,d\beta_k$, where $\xi_k:=\langle G(m)f_k,-\Delta m\rangle=\langle\partial_x(G(m)f_k),\partial_xm\rangle$ and
\[\begin{aligned}
a_t&=-2\lambda_2\|P_{m^\perp}\Delta m\|^2+2\lambda_2\langle P_{m^\perp}\Delta m,P_{m^\perp}\varphi'(m)\rangle\\
&+2\lambda_1I_\varphi(m)+2\varepsilon\langle c(m),-\Delta m\rangle+\varepsilon\sum_k\|\partial_x(G(m)f_k)\|^2 .
\end{aligned}\]
The first two terms together are $\le\tfrac{\lambda_2}2\|\varphi'\|_\infty^2$, and $|2\lambda_1I_\varphi|\le2|\lambda_1|\|D\varphi'\|_\infty\|\partial_xm\|^2$. The Itô-correction term is \textbf{non-positive}: $\langle c(m),-\Delta m\rangle=-g_0^2r_\delta\|\partial_xm\|^2$. Since $\partial_x(G(m)f_k)=DG(m)[\partial_xm]f_k+G(m)\partial_xf_k$ with $|DG(m)[v]f|\le(|\lambda_1|+2\lambda_2)|v||f|$,
\[\|\partial_x(G(m)f_k)\|_H\le(|\lambda_1|+2\lambda_2)\|f_k\|_{L^\infty}\|\partial_xm\|_H+g_0\|\partial_xf_k\|_{L^\infty},\qquad\sum_k\|\partial_x(G(m)f_k)\|_H^2\le C_\xi X .\]
Hence $a_t\le C_aX$ and $\sum_k\xi_k^2\le\sum_k\|\partial_x(G(m)f_k)\|^2\|\partial_xm\|^2\le C_\xi X^2$. Since $X\ge1$, Itô's formula for $\log X$ gives
\[d\log X=\Big(\frac{a_t}{X}-\frac{2\varepsilon\sum_k\xi_k^2}{X^2}\Big)dt+d\widetilde N_t\le C_a\,dt+d\widetilde N_t,\qquad \widetilde N_t:=2\sqrt\varepsilon\sum_k\int_0^t\frac{\xi_k}{X}\,d\beta_k,\qquad\langle\widetilde N\rangle_T\le4\varepsilon C_\xi T .\]
Thus $\sup_t\log X(t)\le\log X(0)+C_aT+\sup_t\widetilde N_t$ with $X(0)\le1+R_D^2$. For a continuous local martingale with $\langle\widetilde N\rangle_T\le K$ one has $\mathbb P(\sup_{t\le T}\widetilde N_t\ge a)\le e^{-a^2/2K}$ (exponential supermartingale $e^{\lambda\widetilde N-\lambda^2\langle\widetilde N\rangle/2}$ and optimisation in $\lambda$), which gives the claim. $\square$

\begin{proposition}[{exponential tightness}]\label{prop:5.5} Under \hyperref[asm:A1]{(A1)}--\hyperref[asm:A6]{(A6)} and $\beta>3/2$, for every $L<\infty$ there is a compact set $\mathcal K_L\subset C([0,T];H)$ with
\[\limsup_{\varepsilon\to0}\ \varepsilon\log\ \sup_{z\in D}\mathbb P\big(m^\varepsilon_z\notin\mathcal K_L\big)\le-L ;\]
the same holds for the auxiliary process $Y^\varepsilon_z$ in place of $m^\varepsilon_z$.

\end{proposition}

\emph{Proof.} Fix $\alpha=\tfrac14$, $p=8$, $c_\alpha:=\sin(\pi\alpha)/\pi$, and write $m=m^\varepsilon_z$. \textbf{Decomposition.} By the mild form (\hyperref[thm:5.1]{Theorem \ref*{thm:5.1}}),
\[m=S(\cdot)z+\mathcal G_{K_2}g+\sqrt\varepsilon\,g_0c_\alpha\,\mathcal G_{K_1}Z,\qquad g:=A_1^{-1/2}b_\varepsilon(m),\]
where $Z$ is the process of \hyperref[lem:5.4b]{Lemma \ref*{lem:5.4b}} with $\Phi:=G(m)/g_0$ (so $\|\Phi\|\le1$), and $\mathcal I=c_\alpha\mathcal G_{K_1}Z$ by \hyperref[lem:5.4b]{Lemma \ref*{lem:5.4b}}(iii). Moreover $\|g(t)\|_H=\|b_\varepsilon(m(t))\|_{H^{-}}\le C_b\big(1+\|\partial_xm(t)\|_H^2\big)$. Indeed $b_0(m)=\lambda_1\partial_x(m\times\partial_xm)+\lambda_2|\partial_xm|^2m-G(m)\varphi'(m)$, with $\|\partial_x(m\times\partial_xm)\|_{H^{-}}\le\|\partial_xm\|_H$, $\||\partial_xm|^2m\|_{H^{-}}\le C_1\||\partial_xm|^2\|_{L^1}=C_1\|\partial_xm\|_H^2$ (where $C_1$ is the norm of $L^1\hookrightarrow H^{-}$, dual to $V\hookrightarrow L^\infty$), and $\|G(m)\varphi'(m)\|_H+\varepsilon\|c(m)\|_H\le g_0\|\varphi'\|_\infty+g_0^2r_\delta$. \textbf{Choice of the compact set.} Given $L$, choose $a_L>0$ with $a_L^2\ge8C_\xi TL$, put $R_L:=(1+R_D^2)e^{C_aT+a_L}$, and choose $M_L$ with $cM_L^2\ge L$, where $c$ is an admissible constant in \hyperref[lem:5.4b]{Lemma \ref*{lem:5.4b}}(ii). Let
\[\mathcal K_L:=\overline{\big\{S(\cdot)z:z\in\bar D\big\}+\mathcal G_{K_2}\big(\{\|g\|_{L^p(0,T;H)}\le C_bR_LT^{1/p}\}\big)+g_0c_\alpha\,\mathcal G_{K_1}\big(\{\|f\|_{L^p(0,T;H)}\le M_L\}\big)} .\]
The first set is compact, being the image of the compact set $\bar D$ under the continuous map $z\mapsto S(\cdot)z$, $H\to C([0,T];H)$; the other two are relatively compact by \hyperref[lem:5.4a]{Lemma \ref*{lem:5.4a}}. Hence $\mathcal K_L$ is compact. \textbf{Estimate.} On the event $\{\sup_t(1+\|\partial_xm\|_H^2)<R_L\}\cap\{\sqrt\varepsilon\|Z\|_{L^p}\le M_L\}$ one has $m\in\mathcal K_L$. By \hyperref[lem:5.4c]{Lemma \ref*{lem:5.4c}} and by Chebyshev with \hyperref[lem:5.4b]{Lemma \ref*{lem:5.4b}}(ii),
\[\sup_{z\in D}\mathbb P\big(m^\varepsilon_z\notin\mathcal K_L\big)\le e^{-a_L^2/(8\varepsilon C_\xi T)}+C_Ze^{-cM_L^2/\varepsilon}\le(1+C_Z)\,e^{-L/\varepsilon},\]
so $\limsup_{\varepsilon\to0}\varepsilon\log\sup_z\mathbb P(m^\varepsilon_z\notin\mathcal K_L)\le-L$. For $Y^\varepsilon_z=\sqrt\varepsilon\,g_0\,\mathcal I$ (control $0$) the second event alone gives $Y^\varepsilon_z\in g_0c_\alpha\overline{\mathcal G_{K_1}(\{\|f\|_{L^p}\le M_L\})}$. $\square$

\begin{remark}[{what is used where}]\label{rem:5.6} The sub-Gaussian bound (\hyperref[lem:5.4b]{Lemma \ref*{lem:5.4b}}) uses only $\|G(m)\|_{\mathcal L(H)}\le g_0$ and $\operatorname{tr}_HQ_\delta<\infty$; it requires no information on $m$ and no localisation. The drift part requires the energy, and its \textbf{exponential} bound (\hyperref[lem:5.4c]{Lemma \ref*{lem:5.4c}}) uses two things: $\beta>3/2$, through $\sigma_1<\infty$, and the sign of the Itô correction in the energy identity, which contributes $-2\varepsilon g_0^2r_\delta\|\partial_xm\|_H^2\le0$. Both bounds depend on $\operatorname{tr}_HQ_\delta\asymp\delta^{-1/\beta}$ and on $\sigma_1$, so \hyperref[prop:5.5]{Proposition \ref*{prop:5.5}} is a statement at fixed $\delta$, which is what the fixed-$\delta$ LDP of \S{}\hyperref[sec:6]{6} uses. The compact sets $\mathcal K_L$ are compact in $C([0,T];H)$ only; compactness in $C([0,T];V)$ is not claimed (\hyperref[rem:6.3]{Remark \ref*{rem:6.3}}).

\end{remark}

\begin{lemma}[{the convolution operators map $L^p$ into $L^\infty(0,T;V)$, with an $H$-H\"older modulus}]\label{lem:5.5v} Let $\alpha\in(0,\tfrac12)$, $K(r)=r^{-\alpha}S(r)$, $p\in(1,\infty)$ with $(\alpha+\tfrac12)p'<1$, and $f\in L^p(0,T;H)$. Then
\[\|\mathcal G_Kf\|_{L^\infty(0,T;V)}\le C\|f\|_{L^p(0,T;H)},\qquad \|\mathcal G_Kf(t)-\mathcal G_Kf(s)\|_H\le C|t-s|^{\beta}\|f\|_{L^p(0,T;H)}\ ,\]
where $p'=p/(p-1)$ and $\beta:=\min\{\tfrac1{p'}-\alpha,\tfrac12\}>0$. For the choice $\alpha=\tfrac14$, $p=8$ of \hyperref[prop:5.5]{Proposition \ref*{prop:5.5}} one has $p'=\tfrac87$, $(\alpha+\tfrac12)p'=\tfrac67$ and $\beta=\tfrac12$.
\end{lemma}

\emph{Proof.} \textbf{The $V$-bound.} By the smoothing estimate of \S{}\hyperref[sec:2.1]{2.1}, $\|S(r)g\|_V\le C(\lambda_2r)^{-1/2}\|g\|_H$ for $r>0$, hence
\[\|\mathcal G_Kf(t)\|_V\le C\int_0^t(t-s)^{-\alpha-\frac12}\|f(s)\|_H\,ds\le C\Big(\int_0^t r^{-(\alpha+\frac12)p'}dr\Big)^{1/p'}\|f\|_{L^p}\le C\|f\|_{L^p},\]
finite because $(\alpha+\tfrac12)p'<1$. \textbf{The $H$-modulus.} Comparing $t$ with $s$ directly would leave the operator difference of the two translated semigroups, for which only $\|S(\tau)-I\|\le2$ is available; the semigroup law removes the mismatch. Fix $0\le s<t\le T$ and put $\tau:=t-s$. Since $S(\tau)S(s-r)=S(t-r)$ for $r\le s$,
\[S(\tau)\mathcal G_Kf(s)=\int_0^s(s-r)^{-\alpha}S(t-r)f(r)\,dr ,\]
hence
\[\mathcal G_Kf(t)-S(\tau)\mathcal G_Kf(s)=\int_s^t(t-r)^{-\alpha}S(t-r)f(r)\,dr+\int_0^s\big[(t-r)^{-\alpha}-(s-r)^{-\alpha}\big]S(t-r)f(r)\,dr .\]
By contractivity, $\|S(t-r)\|_{\mathcal L(H)}\le1$, so the $H$-norm of this is at most
\[I_1+I_2:=\int_s^t(t-r)^{-\alpha}\|f(r)\|_H\,dr+\int_0^s\big[(t-r)^{-\alpha}-(s-r)^{-\alpha}\big]\|f(r)\|_H\,dr ,\]
and $I_1+I_2\le C\tau^{\frac1{p'}-\alpha}\|f\|_{L^p}$ because $\alpha p'<1$: on $(s,t)$ by H\"older with exponent $p'$ and $\int_0^\tau u^{-\alpha p'}du=\tau^{1-\alpha p'}/(1-\alpha p')$; on the near part $r\in[s-\tau,s]$, where $t-r\ge s-r$, by $\big|(t-r)^{-\alpha}-(s-r)^{-\alpha}\big|\le(s-r)^{-\alpha}$; and on the far part $r\in[0,s-\tau]$, where $s-r\ge\tau$, by the mean-value bound $\big|(t-r)^{-\alpha}-(s-r)^{-\alpha}\big|\le\alpha\tau(s-r)^{-\alpha-1}$, H\"older and $(\alpha+1)p'>1$. This gives $\|\mathcal G_Kf(t)-S(\tau)\mathcal G_Kf(s)\|_H\le C_1\tau^{\frac1{p'}-\alpha}\|f\|_{L^p}$. For the remaining piece, $S(\tau)-I$ acting on $\mathcal G_Kf(s)\in V$: by analyticity, $\|(S(\tau)-I)A_1^{-1/2}\|_{\mathcal L(H)}\le 2\lambda_2c\,\tau^{1/2}$, and by the $V$-bound $\|\mathcal G_Kf(s)\|_V\le C_2\|f\|_{L^p}$, so $\|(S(\tau)-I)\mathcal G_Kf(s)\|_H\le C_3\tau^{1/2}\|f\|_{L^p}$. Adding the two estimates and using $\tau\le T$ gives the claim. $\square$

\begin{lemma}[{the exponentially tight compact sets are compact in $C([0,T];C)$}]\label{lem:5.5c0} Assume \hyperref[lem:5.5v]{Lemma \ref*{lem:5.5v}}. Then for every $L<\infty$ the set $\mathcal K_L$ of \hyperref[prop:5.5]{Proposition \ref*{prop:5.5}} is compact in $C([0,T];C(\mathbb T;\mathbb R^3))$; more precisely, there are $R_L'<\infty$, $\beta'>0$ and $C_L<\infty$ with
\[\sup_{y\in\mathcal K_L}\ \sup_{t}\|y(t)\|_V\le R_L',\qquad \sup_{y\in\mathcal K_L}\ \|y(t)-y(s)\|_H\le C_L|t-s|^{\beta'} .\]
\end{lemma}

\emph{Proof.} Recall that $\mathcal K_L$ is the closure in $C([0,T];H)$ of the sum of $\mathcal A:=\{S(\cdot)z:z\in\bar D\}$ and the two images $\mathcal G_{K_i}B_i$ with $B_i$ bounded in $L^p$. \emph{(i) $V$-bounds.} $\|S(t)z\|_V\le\|z\|_V\le R_D$, because $S$ is a contraction semigroup on $V$, and the two images are bounded in $L^\infty(0,T;V)$ by \hyperref[lem:5.5v]{Lemma \ref*{lem:5.5v}}. \emph{(ii) $H$-modulus.} $\|(S(t)-S(s))z\|_H=\|(e^{-\lambda_2(t-s)A}-I)S(s)z\|_H\le C|t-s|^{1/2}\|z\|_V$ by analyticity and $\|S(s)z\|_V\le\|z\|_V$; the other two terms are H\"older by \hyperref[lem:5.5v]{Lemma \ref*{lem:5.5v}}. \emph{(iii) $C$-compactness.} By (i) and the interpolation of \S{}\hyperref[sec:2.6]{2.6}, $\|y(t)-y(s)\|_{L^\infty(\mathbb T)}\le C R_L'^{1/2}|t-s|^{\beta'/2}$ for every $y\in\mathcal K_L$, which is equicontinuity in $C([0,T];C)$; at fixed $t$ the set $\{y(t)\}$ is $V$-bounded, hence relatively compact in $C(\mathbb T;\mathbb R^3)$ because $V\hookrightarrow C$ compactly in one dimension. Arzel\`a--Ascoli gives relative compactness, and the limit is in $\mathcal K_L$ because convergence in $C([0,T];C)$ implies convergence in $C([0,T];H)$. $\square$

\begin{corollary}[{the level sets are relatively compact in $C([0,T];C)$}]\label{cor:5.5c0} Assume \hyperref[lem:5.5v]{Lemma \ref*{lem:5.5v}}. Then for every $c<\infty$ the set $\bigcup_{z\in D}\Phi_z(c)$ with $\Phi_z(c)=\{I_z\le c\}$ is relatively compact in $C([0,T];C(\mathbb T;\mathbb R^3))$.
\end{corollary}

\emph{Proof.} By \hyperref[prop:4.9]{Proposition \ref*{prop:4.9}}, $\bigcup_{z\in D}\Phi_z(c)\subset X^0\big(\bar D\times S^{2c+1}_{\mathrm{weak}}\big)$. A skeleton driven by a control of energy at most $c$ lies in the strong class $\mathcal S$ of \hyperref[def:D.4]{Definition \ref*{def:D.4}} with norms bounded in terms of $c$ only; in particular $\sup_t\|X^0\|_V\le R(c)$, and $\partial_tX^0\in L^2(0,T;H)$ gives the $H$-modulus $C(c)|t-s|^{1/2}$. The interpolation of \S{}\hyperref[sec:2.6]{2.6} upgrades this to a $C$-modulus, and Arzel\`a--Ascoli with $V\hookrightarrow C$ compactly gives relative compactness; closedness follows as in \hyperref[lem:5.5c0]{Lemma \ref*{lem:5.5c0}}. $\square$

\section{Uniform large deviations for the SLLG}\phantomsection\label{sec:6}

This section assembles the uniform LDP at fixed regularisation: it identifies the rate function (\S{}\hyperref[sec:6.1]{6.1}) and proves the theorem (\S{}\hyperref[sec:6.2]{6.2}), first on $C([0,T];H)$ and then on the finer spatial topology $C([0,T];C)$. Everything here is conditional on the abstract framework of \S{}\hyperref[sec:4]{4} and on the verifications of \S{}\hyperref[sec:5]{5}; nothing in \S{}\hyperref[sec:7]{7}--\S{}\hyperref[sec:12]{11} is used.

\subsection{The rate function}\phantomsection\label{sec:6.1}

For $z\in D$ define the \textbf{rate function} of the magnetisation by
\begin{equation}\label{eq:6.1}
I_z(\varphi)=\inf\Big\{\tfrac12\!\int_0^T\|u\|_H^2\,ds:\ u\in L^2(0,T;H),\ m^0_{z,u}=\varphi\Big\},\qquad\inf\emptyset=+\infty,
\end{equation}
where $m^0_{z,u}$ is the skeleton \hyperref[eq:3.2]{(\ref*{eq:3.2})} at the fixed regularisation $\delta$ (suppressed from the notation $I_z$ in this section). This is the object denoted $I^\delta_z$ in \hyperref[eq:3.1]{(\ref*{eq:3.1})}; \S{}\hyperref[sec:6]{6} fixes $\delta>0$ and proves the LDP at that $\delta$.

\subsection{The theorem}\phantomsection\label{sec:6.2}

\begin{theorem}[{uniform LDP at fixed $\delta$}]\label{thm:6.1} Let $\delta\in(0,1]$ and $\beta>3/2$. Let $D\subset V\cap\mathcal M$ be bounded with $\bar D$ compact in $H$. Recall that $m^\varepsilon_z:=m^\varepsilon_{z,0}$ is the solution of the SLLG \hyperref[eq:2.1]{(\ref*{eq:2.1})} with control $u=0$, initial datum $z$ and noise intensity $\varepsilon$ (\S{}\hyperref[sec:2.5]{2.5}), and that $I_z=I^\delta_z$ is the rate function \hyperref[eq:6.1]{(\ref*{eq:6.1})}. Then the family $\{m^\varepsilon_z\}_{\varepsilon>0}$, regarded as $C([0,T];H)$-valued random variables, satisfies the uniform LDP of \hyperref[def:4.6]{Definition \ref*{def:4.6}} (the Freidlin--Wentzell form) with the good rate function $I_z$ of \hyperref[eq:6.1]{(\ref*{eq:6.1})}, uniformly over $z\in D$. Writing $\Phi_z(s):=\{\varphi\in C([0,T];H):\ I_z(\varphi)\le s\}$ for the level sets of $I_z$, this means: for every $s_0<\infty$ and every $\delta>0$,
\[\liminf_{\varepsilon\to0}\ \inf_{z\in D}\ \inf_{\varphi\in\Phi_z(s_0)}\Big(\varepsilon\log\mathbb P\big(\|m^\varepsilon_z-\varphi\|_{C([0,T];H)}<\delta\big)+I_z(\varphi)\Big)\ \ge\ 0,\]
\[\limsup_{\varepsilon\to0}\ \sup_{z\in D}\ \sup_{s\in[0,s_0]}\Big(\varepsilon\log\mathbb P\big(\operatorname{dist}_{C([0,T];H)}(m^\varepsilon_z,\Phi_z(s))\ge\delta\big)+s\Big)\ \le\ 0 .\]
\end{theorem}

\emph{Proof.} We apply \S{}\hyperref[sec:4]{4} with $X^\varepsilon_{z,u}:=m^\varepsilon_{z,u}$ and $W:=W_\delta$. \hyperref[asm:H1]{(H1)} is \hyperref[thm:5.1]{Theorem \ref*{thm:5.1}}, \hyperref[asm:H2]{(H2)} is \hyperref[thm:5.2]{Theorem \ref*{thm:5.2}}, and \hyperref[asm:H3]{(H3)} is \hyperref[thm:5.3]{Theorem \ref*{thm:5.3}}. \textbf{The variational representation \hyperref[asm:V]{(V)} for $m$.} By [\hyperref[ref:BDM08]{BDM08}, Thm 3] (\hyperref[sec:C.3]{Appendix C.3}), applied to the bounded Borel functional $F:=f\circ\Phi^\varepsilon_z(0,\cdot)/\varepsilon$ of the $Q_\delta$-Wiener process, with $H_0=Q_\delta^{1/2}H$, $\|Q_\delta^{1/2}u\|_{H_0}=\|u\|_H$ and the substitution $v=\varepsilon^{-1/2}Q_\delta^{1/2}u$,
\[-\varepsilon\log\mathbb E\big[e^{-f(m^\varepsilon_z)/\varepsilon}\big]=\inf_{u\in\mathcal P_2}\mathbb E\Big[\tfrac12\!\int_0^T\|u\|_H^2\,ds+f\Big(\Phi^\varepsilon_z\big(0,W_\delta+\varepsilon^{-1/2}\textstyle\int_0^\cdot Q_\delta^{1/2}u\,ds\big)\Big)\Big] .\]
By \hyperref[prop:E.6]{Proposition \ref*{prop:E.6}} (Girsanov), $\Phi^\varepsilon_z(0,W_\delta+\varepsilon^{-1/2}\int_0^\cdot Q_\delta^{1/2}u)=\Phi^\varepsilon_z(u,W_\delta)=m^\varepsilon_{z,u}$ a.s., which is \hyperref[asm:V]{(V)}. \textbf{Assembly.} \hyperref[thm:4.5]{Theorem \ref*{thm:4.5}} gives the uniform Laplace principle \hyperref[eq:LP]{(LP)} for $\{m^\varepsilon_z\}$, with
\[\Lambda^0(z)=\inf_u\big[\tfrac12\textstyle\int_0^T\|u\|^2+f(m^0_{z,u})\big]=\inf_\varphi\big[f(\varphi)+I_z(\varphi)\big],\]
$I_z$ as in \hyperref[eq:6.1]{(\ref*{eq:6.1})}. \hyperref[prop:4.9]{Proposition \ref*{prop:4.9}} gives the uniform goodness \hyperref[eq:UG]{(UG)}, and \hyperref[prop:5.5]{Proposition \ref*{prop:5.5}} the exponential tightness \hyperref[eq:ET]{(ET)} of $\{m^\varepsilon_z\}$ itself. \hyperref[prop:4.7]{Proposition \ref*{prop:4.7}} then yields the uniform LDP, i.e. \hyperref[thm:6.1]{Theorem \ref*{thm:6.1}}. $\square$

\begin{remark}[{the uniformity class, and why the stronger one is not claimed}]\label{rem:6.2a} The uniformity in \hyperref[thm:6.1]{Theorem \ref*{thm:6.1}} is over the family of rate functions $\{I_z\}_{z\in D}$ and over the whole two-sided statement: what is uniform in $z$ is the exponential rate together with the uniform level-set compactness and the uniform goodness (\hyperref[prop:4.9]{Proposition \ref*{prop:4.9}}). This is exactly what the exit problem consumes --- it is what permits $\sup_{z\in K}$ and $\varepsilon\to0$ to be interchanged in \S{}\hyperref[sec:11]{11} --- and it is meaningful although a compact subset of $H$ has empty interior, because the statement concerns the uniformity of the rate in the index $z$, not a localisation in the state space.

Uniformity over \textbf{$H$-bounded} sets, the class used in [\hyperref[ref:BDS19]{BDS19}], is \textbf{not} claimed, and this is not an oversight: with the present method it is unavailable, for a structural reason. The only a priori control we have is the uniform $V$-bound \hyperref[eq:5.1]{(\ref*{eq:5.1})}, and exponential tightness (\hyperref[prop:5.5]{Proposition \ref*{prop:5.5}}) is obtained from it through the compact embedding $V\hookrightarrow H$ and the $H$-H\"older modulus of the convolution operators. On an $H$-bounded set the $V$-norm of the initial datum is unbounded, so the right-hand side of \hyperref[eq:5.1]{(\ref*{eq:5.1})} is not uniform, and the smoothing $S(t)$ degenerates at $t=0$, so no uniform modulus of continuity and hence no uniform exponential tightness is available: the (ET) input fails. In [\hyperref[ref:BDS19]{BDS19}] the stronger class is reached by different means --- a double-dual weak-$\star$ compactness argument together with an auxiliary process whose initial value vanishes for every datum, which circumvents the failure of weak continuity of $x\mapsto X^\varepsilon_x$ caused by $S(0)=I$. That machinery is not needed for the results of this paper: the barrier and the escape exponent established here are determined by the compact-class uniformity above.
\end{remark}

\begin{corollary}[{the fixed-set bounds}]\label{cor:6.1a} Under the hypotheses of \hyperref[thm:6.1]{Theorem \ref*{thm:6.1}}: (a) for every closed $F\subset C([0,T];H)$,
\[\limsup_{\varepsilon\to0}\ \varepsilon\log\ \sup_{z\in D}\mathbb P\big(m^\varepsilon_z\in F\big)\ \le\ -\inf_{z\in D}\inf_{\varphi\in F}I_z(\varphi),\]
(b) for every open $G\subset C([0,T];H)$, with $G_\eta:=\{\varphi\in G:\operatorname{dist}(\varphi,C([0,T];H)\setminus G)>\eta\}$,
\[\liminf_{\varepsilon\to0}\ \varepsilon\log\ \inf_{z\in D}\mathbb P\big(m^\varepsilon_z\in G\big)\ \ge\ -\lim_{\eta\downarrow0}\ \sup_{z\in D}\inf_{\varphi\in G_\eta}I_z(\varphi).\]
Both bounds are \hyperref[cor:4.7a]{Corollary \ref*{cor:4.7a}} applied with $E=C([0,T];H)$, $Y^\varepsilon_z=m^\varepsilon_z$ and $J_z=I_z$; the hypotheses \hyperref[eq:LP]{(LP)}, \hyperref[eq:ET]{(ET)} and \hyperref[eq:UG]{(UG)} required there are those verified in the proof of \hyperref[thm:6.1]{Theorem \ref*{thm:6.1}}.
\end{corollary}

\begin{theorem}[{uniform LDP on $C([0,T];C)$}]\label{thm:6.2c0} Let $\delta\in(0,1]$ and $\beta>3/2$, and let $D$ be as in \S{}\hyperref[sec:2.6]{2.6}. Recall that $m^\varepsilon_z:=m^\varepsilon_{z,0}$ is the solution of the SLLG \hyperref[eq:2.1]{(\ref*{eq:2.1})} with control $u=0$, initial datum $z$ and noise intensity $\varepsilon$ (\S{}\hyperref[sec:2.5]{2.5}), and $I_z=I^\delta_z$ is the rate function \hyperref[eq:6.1]{(\ref*{eq:6.1})}. Then the family $\{m^\varepsilon_z\}_{\varepsilon>0}$ of $C([0,T];C(\mathbb T;\mathbb R^3))$-valued random variables satisfies, uniformly over $z\in D$, the uniform LDP of \hyperref[def:4.6]{Definition \ref*{def:4.6}} \textbf{with $E=C([0,T];C(\mathbb T;\mathbb R^3))$} and the good rate functions $I_z$ of \hyperref[eq:6.1]{(\ref*{eq:6.1})}: with $\Phi_z(s)=\{I_z\le s\}$ and $\|\varphi\|_{C([0,T];C)}:=\sup_{t\le T}\sup_{x\in\mathbb T}|\varphi(t,x)|$, for every $s_0<\infty$ and every $\delta>0$,
\[\liminf_{\varepsilon\to0}\ \inf_{z\in D}\ \inf_{\varphi\in\Phi_z(s_0)}\Big(\varepsilon\log\mathbb P\big(\|m^\varepsilon_z-\varphi\|_{C([0,T];C)}<\delta\big)+I_z(\varphi)\Big)\ \ge\ 0 ,\]
\[\limsup_{\varepsilon\to0}\ \sup_{z\in D}\ \sup_{s\in[0,s_0]}\Big(\varepsilon\log\mathbb P\big(\operatorname{dist}_{C([0,T];C)}(m^\varepsilon_z,\Phi_z(s))\ge\delta\big)+s\Big)\ \le\ 0 .\]
\end{theorem}

\emph{Proof.} Apply \hyperref[prop:4.7]{Proposition \ref*{prop:4.7}} with $E=C([0,T];C)$, $Y^\varepsilon_z=m^\varepsilon_z$ and $J_z=I_z$; $E$ is a Polish space. The exponential tightness \hyperref[eq:ET]{(ET)} is \hyperref[lem:5.5c0]{Lemma \ref*{lem:5.5c0}} together with the exponential bound \hyperref[lem:5.4c]{Lemma \ref*{lem:5.4c}}: the estimate behind \hyperref[prop:5.5]{Proposition \ref*{prop:5.5}} is unchanged, only the ambient topology changes. The uniform goodness \hyperref[eq:UG]{(UG)} is \hyperref[cor:5.5c0]{Corollary \ref*{cor:5.5c0}} together with the upper semicontinuity of $z\mapsto\Phi_z(c)$, which \hyperref[prop:4.9]{Proposition \ref*{prop:4.9}} derives from \hyperref[asm:H3]{(H3)}(ii) and which \hyperref[lem:5.4c0]{Lemma \ref*{lem:5.4c0}}(ii) supplies in the new topology. The Laplace principle \hyperref[eq:LP]{(LP)} is \hyperref[thm:4.5]{Theorem \ref*{thm:4.5}} applied with test functions $f\in C_b(C([0,T];C))$: the only topological step in the proof of that theorem is the identification $\Lambda^0(z)=\inf_u\{\tfrac12\int\|u\|_H^2+f(X^0_{z,u})\}$, which holds for such $f$ by \hyperref[lem:5.4c0]{Lemma \ref*{lem:5.4c0}}(i)--(ii). $\square$

\begin{corollary}[{fixed-set bounds on $C([0,T];C)$}]\label{cor:6.2c0} Under the hypotheses of \hyperref[thm:6.2c0]{Theorem \ref*{thm:6.2c0}}: (a) for every closed $F\subset C([0,T];C(\mathbb T;\mathbb R^3))$,
\[\limsup_{\varepsilon\to0}\ \varepsilon\log\ \sup_{z\in D}\mathbb P\big(m^\varepsilon_z\in F\big)\ \le\ -\inf_{z\in D}\inf_{\varphi\in F}I_z(\varphi);\]
(b) for every open $G\subset C([0,T];C(\mathbb T;\mathbb R^3))$ with $G_\eta:=\{\varphi\in G:\operatorname{dist}(\varphi,C([0,T];C)\setminus G)>\eta\}$,
\[\liminf_{\varepsilon\to0}\ \varepsilon\log\ \inf_{z\in D}\mathbb P\big(m^\varepsilon_z\in G\big)\ \ge\ -\lim_{\eta\downarrow0}\ \sup_{z\in D}\inf_{\varphi\in G_\eta}I_z(\varphi).\]
Both are \hyperref[cor:4.7a]{Corollary \ref*{cor:4.7a}} applied with $E=C([0,T];C)$.
\end{corollary}

\begin{remark}[{the two statements are nested}]\label{rem:6.2b} 
Since $C([0,T];C(\mathbb T;\mathbb R^3))\subset C([0,T];H)$ with continuous inclusion, \hyperref[thm:6.2c0]{Theorem \ref*{thm:6.2c0}} is strictly stronger than \hyperref[thm:6.1]{Theorem \ref*{thm:6.1}}, and by the contraction principle it \emph{implies} it, with the same rate functions: the finiteness domain of $I_z$ consists of skeleton trajectories, which lie in the strong class of \hyperref[def:D.4]{Definition \ref*{def:D.4}} and are therefore continuous in the spatial supremum norm, so no information is lost in passing to the coarser topology. Both statements are nevertheless recorded, for two concrete reasons. First, the exit-time argument of \S{}\hyperref[sec:12.5]{11.5} is phrased in terms of $H$-closed sets, and the exit set $D=\{E<r\}$ is closed in \emph{both} topologies: the energy is lower semicontinuous on $H$ and hence also on the finer space $C$, so its sublevel sets are closed there as well. The coarser statement is therefore not the only one that could be used; it is the one used in \S{}\hyperref[sec:12.5]{11.5}, and it is directly available from the abstract framework of \S{}\hyperref[sec:4]{4}, whereas the finer statement requires the compactness argument of \S{}\hyperref[sec:5.5]{5.5} to be redone in $C$. Second, the implication does not run the other way in usable form: a $C$-closed set such as $F_R=\{\varphi:\|\varphi\|_{C([0,T];C)}\ge R\}$ is not $H$-closed, and the $H$-LDP gives no bound for it. The $C$-statement is thus neither a restatement nor a purely cosmetic strengthening of the $H$-statement; it is the same rate functions with a genuinely larger family of sets.
\end{remark}

\begin{corollary}[{what the finer topology adds}]\label{cor:6.2b} Under the hypotheses of \hyperref[thm:6.2c0]{Theorem \ref*{thm:6.2c0}}. (a) \emph{(profile confinement.)} For every $\varphi_\ast\in C([0,T];C(\mathbb T;\mathbb R^3))$ and $\eta>0$, the set $F_{\varphi_\ast,\eta}:=\{\varphi:\sup_{t,x}|\varphi(t,x)-\varphi_\ast(t,x)|\ge\eta\}$ is closed in $C([0,T];C)$, and
\[\limsup_{\varepsilon\to0}\ \varepsilon\log\ \sup_{z\in D}\mathbb P\Big(\sup_{t\le T}\sup_{x\in\mathbb T}\big|m^\varepsilon_z(t,x)-\varphi_\ast(t,x)\big|\ge\eta\Big)\ \le\ -\inf_{z\in D}\ \inf_{\varphi\in F_{\varphi_\ast,\eta}}I_z(\varphi).\]
(b) \emph{(terminal profile.)} For every open $O\subset C(\mathbb T;\mathbb R^3)$ and every $\eta>0$,
\[\liminf_{\varepsilon\to0}\ \varepsilon\log\ \inf_{z\in D}\mathbb P\big(m^\varepsilon_z(T)\in O\big)\ \ge\ -\lim_{\eta\downarrow0}\ \sup_{z\in D}\ \inf_{\varphi:\ \operatorname{dist}_{C}\big(\varphi(T),\,C\setminus O\big)>\eta}I_z(\varphi).\]
(c) \emph{(attainment under spatial constraints.)} For every $c<\infty$ the level set $\bigcup_{z\in D}\Phi_z(c)$ is relatively compact in $C([0,T];C)$ (\hyperref[cor:5.5c0]{Corollary \ref*{cor:5.5c0}}); consequently, for every $C$-closed $F\subset C([0,T];C)$ and every $z\in\bar D$ with $\inf_{F}I_z<\infty$ the infimum $\inf_{\varphi\in F}I_z(\varphi)$ is attained, and every minimising sequence possesses a $C$-convergent subsequence whose limit is a minimiser.
\end{corollary}

\begin{remark}[{where the finer topology is genuinely needed: the exit location}]\label{rem:6.2c} Of the three consequences above, the one with physical content is the \emph{profile} statement, and among its uses the sharpest is the \textbf{exit location}. The exit-time exponent of \S{}\hyperref[sec:12.5]{11.5} is blind to the way the boundary is crossed: it is determined by $\inf_{\partial D}U_\delta$ alone and is therefore already accessible from the $H$-LDP. The location of the exit is a different quantity. To say that the magnetisation leaves $D$ \emph{near a given field} $p\in\partial D$ is to make a statement about the profile, and in the $L^2$-metric such a statement is nearly empty: a unit-length field can differ from $p$ by an amount of order one at every point and still be $L^2$-close, by carrying a high-frequency oscillation of small amplitude. The $C$-LDP is exactly the statement that separates those configurations, and it is what turns the exit-location question into a well-posed large-deviation problem; passing from \hyperref[cor:6.2b]{Corollary \ref*{cor:6.2b}}(b) to the exit point at the stopping time $\tau$ still requires the exit-place form of the Freidlin--Wentzell theorem, which is the remaining open point of the exit programme. The other two parts of the corollary are the remaining uses of the finer topology: (a) bounds, on the exponential scale, the probability that the magnetisation develops a transversal oscillation of size $\eta$ anywhere in $[0,T]\times\mathbb T$ --- an event that the $H$-LDP cannot see --- and (c) makes the variational problems of \S{}\hyperref[sec:3]{3} and \S{}\hyperref[sec:11]{10} well posed when the constraint is imposed in the $C$-topology, as it is when one asks for the cheapest transition that keeps the profile close to a prescribed shape. \textbf{No downstream theorem is strengthened by any of this.} In particular the exit-time exponent of \S{}\hyperref[sec:12.5]{11.5}, the barrier identity of \S{}\hyperref[sec:11]{10} and the $\Gamma$-convergence of \S{}\hyperref[sec:8]{8}--\S{}\hyperref[sec:10]{9} are \emph{unchanged}: the first two are insensitive to the way the boundary is crossed and are established from the $H$-statement (or deterministically), and the last three do not use an LDP at all. What the finer topology supplies is an \emph{input} for the exit-location question, not the conclusion: it is what makes that question a well-posed large-deviation problem, and it leaves untouched three further requirements. \emph{(i) The upper bound for the exit \emph{location} needs $V$-level information, not a finer path topology:} the event that the exit point lies outside a neighbourhood of the minimiser is described by the energy level set $\{E=r\}$, and the energy is only \emph{lower} semicontinuous with respect to $C$-convergence, so that event is not closed in $C([0,T];C)$ and the upper bound cannot follow from \hyperref[thm:6.2c0]{Theorem \ref*{thm:6.2c0}}; what is needed is the $C([0,T];V)$ statement which \S{}\hyperref[sec:13]{12} shows to be out of reach, or an equivalent energy-level estimate. \emph{(ii)} The exit-place form of the Freidlin--Wentzell theorem, which locates the exit point and not merely its time. \emph{(iii)} The non-degeneracy of the minimiser of $U_\delta$ on $\partial D$, which fails for the model anisotropy $\varphi_1$; a set-valued formulation would need only the minimiser \emph{set} to admit a well-defined neighbourhood, and it is the lower bound that appears reachable, since it requires neither (i) nor any regularity beyond the barrier identity of \hyperref[thm:11.6]{Theorem \ref*{thm:11.6}}.
\end{remark}

\begin{remark}[{no contraction step is needed}]\label{rem:6.2} A natural alternative is to prove the LDP for the auxiliary process $Y$ of \S{}\hyperref[sec:2.4]{2.4}, whose controlled equation is driven by the linear noise, and to transfer it to $m=\Psi^\varepsilon(Y+S(\cdot)z)$ by a contraction argument for the $\varepsilon$-dependent map $\Psi^\varepsilon$. That route needs a modulus of continuity of $\Psi^\varepsilon$ on $C([0,T];H)$, which is not available (\S{}\hyperref[sec:2.4]{2.4}). The direct route above needs only two things: the Borel solution map of \hyperref[prop:E.6]{Proposition \ref*{prop:E.6}}, through which the representation of [\hyperref[ref:BDM08]{BDM08}] applies to $m$ itself, and the exponential tightness of $m$ (\hyperref[prop:5.5]{Proposition \ref*{prop:5.5}}). The LDP is asserted for the magnetisation $m$, the physically meaningful process. 

\end{remark}

\begin{remark}[{the $C([0,T];V)$ question}]\label{rem:6.3} Since $V\hookrightarrow H$ continuously, the $C([0,T];V)$-topology is \emph{finer} than the $C([0,T];H)$-topology; consequently every $H$-closed set is $V$-closed, but not conversely, and an LDP on $C([0,T];V)$ is a \textbf{strictly stronger} statement than the LDP of \hyperref[thm:6.1]{Theorem \ref*{thm:6.1}}. In particular \textbf{neither bound transfers for free}: the $V$-closed sets form a strictly larger class than the $H$-closed ones, so the upper bound does \textbf{not} follow by quoting \hyperref[thm:6.1]{Theorem \ref*{thm:6.1}} (the implication runs the other way: it is the $V$-closed sets that form the \emph{larger} class, so \hyperref[thm:6.1]{Theorem \ref*{thm:6.1}} yields no upper bound in the finer topology), and the same applies to the $V$-open sets for the lower bound. What transfers for free is the \emph{converse} implication: an LDP on $C([0,T];V)$, if available, would imply \hyperref[thm:6.1]{Theorem \ref*{thm:6.1}}. To obtain it one needs exponential tightness in the finer topology, i.e. compact subsets of $C([0,T];V)$: this is not available, since the compact sets produced by the energy estimate and \hyperref[asm:A6c]{(A6c)} are compact in $C([0,T];H)$ only ($\mathcal K_L$ of \S{}\hyperref[sec:5.5]{5.5}). The lower bound would analogously require a $V$-level skeleton continuity. Both routes meet the $H^{-}$-valuedness of the drift and the $t^{-1}$ lift $H^{-}\to V$, which yields the non-integrable Gronwall kernel $(t-s)^{-1}$ and no closure. A $C([0,T];V)$-LDP would require a singularity-free $V$-level weak--strong uniqueness; see \S{}\hyperref[sec:13]{12}. It is not needed for anything in this paper, which lives on $C([0,T];H)$.

\emph{(The intermediate topology $C([0,T];C)$, and why $V$ is still out of reach.)} In one dimension $V\hookrightarrow C(\mathbb T)\hookrightarrow H$ continuously, hence
\[C([0,T];V)\ \subset\ C([0,T];C)\ \subset\ C([0,T];H) .\]
The \textbf{intermediate} topology is available and is the content of \hyperref[thm:6.2c0]{Theorem \ref*{thm:6.2c0}}: its proof reuses the compact sets of \S{}\hyperref[sec:5.5]{5.5}, which are $V$-bounded and $H$-H\"older and therefore $C$-compact by the interpolation of \S{}\hyperref[sec:2.6]{2.6} (\hyperref[lem:5.5c0]{Lemma \ref*{lem:5.5c0}}), together with the $C$ versions of \hyperref[asm:H3]{(H3)} (\hyperref[lem:5.4c0]{Lemma \ref*{lem:5.4c0}}) and of the uniform goodness (\hyperref[cor:5.5c0]{Corollary \ref*{cor:5.5c0}}). This is the topology in which the statement ``the profile of the magnetisation is close to a given field'' is meant, and it is the form needed for statements about exit \emph{locations} and terminal profiles. The \textbf{finer} topology $C([0,T];V)$ remains out of reach, and the obstruction is not a lack of a priori bounds: it is the $V$-equicontinuity in $t$, which fails at $t=0$ because the smoothing estimate $\|S(t)x\|_V\le C(\lambda_2t)^{-1/2}\|x\|_H$ of \S{}\hyperref[sec:2.1]{2.1} degenerates there and the drift is $H^{-}$-valued, exactly the non-integrable Gronwall kernel $(t-s)^{-1}$ of \S{}\hyperref[sec:13]{12}; a $V$-level LDP would need a singularity-free $V$-level weak--strong uniqueness. The $H$-topology is what the exit-time asymptotics of \S{}\hyperref[sec:12.5]{11.5} uses --- the exit set is $H$-closed and the barrier is read off on $\partial D$. Neither of the two stronger statements is used in \S{}\hyperref[sec:7]{7}--\S{}\hyperref[sec:12]{11}, which are deterministic once the rate functions are identified.

\end{remark}

\begin{remark}[{the role of \hyperref[thm:6.1]{Theorem \ref*{thm:6.1}} in this paper, and beyond}]\label{rem:6.4} \hyperref[thm:6.1]{Theorem \ref*{thm:6.1}} is not an auxiliary step: it is one half of the target statement \hyperref[target:3.9]{(\ref*{target:3.9})}, and it is the half that makes the other half an assertion about the stochastic model rather than about a control problem.

\emph{(i) It fixes what $U_\delta$ is.} \hyperref[def:3.1]{Definition \ref*{def:3.1}} defines $U_\delta$ through the control-theoretic cost \hyperref[eq:3.1]{(\ref*{eq:3.1})}--\hyperref[eq:3.6]{(\ref*{eq:3.6})}; no noise enters that definition. Were $I^\delta_z$ an arbitrary functional, calling its inf-projection the ``quasi-potential of the $\delta$-regularised dynamics'' would be a definition by fiat. \hyperref[thm:6.1]{Theorem \ref*{thm:6.1}} converts the fiat into a theorem: $I^\delta_z$ \emph{is} the rate function of the uniform LDP of the $\delta$-regularised magnetisation $m^\varepsilon_z$, the family is uniformly exponentially tight, and consequently $U_\delta(a)=\inf\{I^\delta_{z_\ast}(\varphi):\varphi(0)=z_\ast,\ \varphi(T)=a\}$ is the exponential cost that the $\delta$-regularised dynamics assigns to the transition $z_\ast\to a$. This is the identification, announced in \S{}\hyperref[sec:3.1]{3.1}, between the variational object of \S{}\hyperref[sec:3.3]{3.3} and the rate function at regularisation $\delta$.

\emph{(ii) Hence \S{}\hyperref[sec:10]{9}--\S{}\hyperref[sec:12]{11} are theorems about the model.} With this identification, the target \hyperref[target:3.9]{(\ref*{target:3.9})} is the conjunction of two statements: \textbf{(I)} for each fixed $\delta$ the regularised SLLG has $U_\delta$ as its quasi-potential; \textbf{(II)} these quasi-potentials $\Gamma$-converge, and their barrier converges, as $\delta\downarrow0$. Statement (I) is \hyperref[thm:6.1]{Theorem \ref*{thm:6.1}}; statement (II) is \S{}\hyperref[sec:7]{7}--\S{}\hyperref[sec:12]{11}. As noted in \hyperref[rem:11.9]{Remark \ref*{rem:11.9}}, the \emph{proof} of (II) is deterministic and does not invoke (I) --- it rests on the monotonicity \hyperref[eq:W]{(W)} and on the equicoercivity of the action functionals --- but (II) alone would concern a family of control costs. It is the pair (I)$+$(II) that makes \hyperref[target:3.9]{(\ref*{target:3.9})} a theorem about the regularised stochastic equation: (I) certifies the object, (II) computes its limit, and neither statement is a corollary of the other. The paper is logically complete in this two-part sense.

\emph{(iii) In the Néel--Arrhenius law (Goal 4, \S{}\hyperref[sec:12.5]{11.5}).} There \hyperref[thm:6.1]{Theorem \ref*{thm:6.1}} is used in the direction opposite to \S{}\hyperref[sec:7]{7}--\S{}\hyperref[sec:12]{11}: it supplies the fixed-$\delta$ rate function to which the Freidlin--Wentzell exit-time theorem is applied, while the barrier convergence of \hyperref[thm:10.2]{Theorems \ref*{thm:10.2}} and \hyperref[thm:11.6]{\ref*{thm:11.6}} supplies its $\delta\to0$ limit; together they yield the escape-time asymptotics \hyperref[eq:12.5]{(\ref*{eq:12.5})} of \S{}\hyperref[sec:12.5]{11.5} whose fixed-$\delta$ input is exactly the present theorem. Thus \S{}\hyperref[sec:6]{6} --- and, through it, \S{}\hyperref[sec:4]{4}--\S{}\hyperref[sec:5]{5} --- is the point at which the regularised \emph{stochastic} dynamics enters the paper; everything proved downstream of it (\S{}\hyperref[sec:7]{7}--\S{}\hyperref[sec:12]{11}) is deterministic.

\end{remark}

\section{Action functionals: the physical control, monotonicity, equicoercivity}\phantomsection\label{sec:7}

This section contains the structural device on which everything in \S{}\hyperref[sec:8]{8}--\S{}\hyperref[sec:10]{9} rests. We work with the deterministic skeleton \hyperref[eq:3.4]{(\ref*{eq:3.4})} at $\delta=0$; recall that $Q_\delta^{1/2}=(I+\delta A^{\beta/2})^{-1}$ and $Q_\delta^{-1/2}=I+\delta A^{\beta/2}$.

\subsection{The physical control}\phantomsection\label{sec:7.1}

Let $z_\ast\in D$ be fixed. For $\delta\in[0,1]$ and $h\in L^2(0,T;H)$ let $\varphi^h$ denote the solution of the \textbf{$\delta$-free} controlled skeleton
\begin{equation}\label{eq:7.1}
\dot\varphi=G(\varphi)\big[\rho(\varphi)+h\big],\qquad\varphi(0)=z_\ast .
\end{equation}
By \hyperref[thm:5.3]{Theorem \ref*{thm:5.3}}(ii) of \S{}\hyperref[sec:5.4]{5.4} the map $h\mapsto\varphi^h$ is continuous $L^2(0,T;H)\to C([0,T];H)$ (indeed on $S^N_{\mathrm{weak}}$), with a modulus independent of $\delta$.

\begin{lemma}[{the weight identity}]\label{lem:7.1} For every $h\in L^2(0,T;H)$ with $h(t)\in D(A^{\beta/2})$ a.e.,
\[\tfrac12\big\|Q_\delta^{-1/2}h\big\|_H^2=\tfrac12\|h\|_H^2+\delta\,\big\langle h,A^{\beta/2}h\big\rangle_H+\tfrac{\delta^2}{2}\big\|A^{\beta/2}h\big\|_H^2 . \qquad\text{(W)}\]\phantomsection\label{eq:W}

\end{lemma}

\emph{Proof.} $Q_\delta^{-1/2}=I+\delta A^{\beta/2}$ is self-adjoint, so $\|Q_\delta^{-1/2}h\|^2=\|h\|^2+2\delta\langle h,A^{\beta/2}h\rangle+\delta^2\|A^{\beta/2}h\|^2$. $\square$

\begin{lemma}[{the action functionals in physical coordinates}]\label{lem:7.2} For $\delta\in[0,1]$ and $\varphi\in C([0,T];H)$,
\begin{equation}\label{eq:7.2}
I^\delta_{z_\ast}(\varphi)=\inf\Big\{\tfrac12\!\int_0^T\|Q_\delta^{-1/2}h\|_H^2\,dt:\ h\in L^2(0,T;H),\ \varphi^h=\varphi\Big\},\qquad\inf\emptyset=+\infty .
\end{equation}
Moreover the two descriptions \hyperref[eq:3.1]{(\ref*{eq:3.1})} and \hyperref[eq:7.2]{(\ref*{eq:7.2})} agree.

\end{lemma}

\emph{Proof.} The bijection $h\mapsto u=Q_\delta^{-1/2}h$ between admissible physical controls and admissible controls in the effective field maps \hyperref[eq:7.1]{(\ref*{eq:7.1})} to \hyperref[eq:3.2]{(\ref*{eq:3.2})}, and $\tfrac12\int\|u\|^2=\tfrac12\int\|Q_\delta^{-1/2}h\|^2$ by \hyperref[eq:W]{(W)}. $\square$

\subsection{Monotonicity}\phantomsection\label{sec:7.2}

\begin{lemma}[{monotone domination}]\label{lem:7.3} $Q_\delta^{-1/2}=I+\delta A^{\beta/2}\ge I$ in the Loewner order. Hence for $0<\delta'\le\delta$,
\[I^0_{z_\ast}(\varphi)\le I^{\delta'}_{z_\ast}(\varphi)\le I^\delta_{z_\ast}(\varphi)\qquad\forall\,\varphi\in C([0,T];H),\]
and consequently $U_0(a)\le U_{\delta'}(a)\le U_\delta(a)$ for all $a\in H$.

\end{lemma}

\emph{Proof.} For the quadratic form, $\|Q_\delta^{-1/2}h\|^2\ge\|h\|^2$ for every $h$; taking infima over the same (non-empty or empty) admissible set gives $I^0\le I^\delta$. For $\delta'\le\delta$ the weights are pointwise ordered by \hyperref[eq:W]{(W)} since $\delta'\le\delta$ and both extra terms are nonnegative. The state-space statement follows by taking infima over paths ending at $a$. $\square$

\begin{remark}\label{rem:7.4} Monotonicity is the reason the $\Gamma$-$\liminf$ in \S{}\hyperref[sec:8]{8} is \emph{free}: it converts the lower bound into a statement about a \emph{fixed} sequence of cost functionals, without any compactness. It also shows that $U_\delta$ \textbf{decreases} as $\delta\downarrow0$: since $Q_\delta^{-1/2}\ge I$ with both correction terms nonnegative, the regularised quasi-potential is \textbf{larger} than the ideal one, $U_\delta\ge U_{\delta'}\ge U_0$ for $\delta\ge\delta'$. The reason is that the regulariser $Q_\delta^{1/2}=(I+\delta A^{\beta/2})^{-1}$ \emph{shrinks} the noise, and a weaker noise requires more action to move the system. The direction of the ordering is consistent with the domain obstruction of \S{}\hyperref[sec:9]{B} in the precise sense that a larger cost is paid on a \emph{smaller} set of reachable states: the level sets satisfy $\{U_\delta\le c\}\subseteq\{U_0\le c\}$, so regularisation can only make the sublevel sets smaller, never larger. (Note however that this is a statement about the \emph{cost functionals}, not about the noise: the regularised covariance $Q_\delta=(I+\delta A^{\beta/2})^{-2}\le I$ is weaker than the identity, the opposite of what the phrase ``smoother and stronger'' would suggest.)

\end{remark}

\subsection{Equicoercivity}\phantomsection\label{sec:7.3}

\begin{proposition}[{$\delta$-uniform equicoercivity}]\label{prop:7.5} For every $c<\infty$,
\[\bigcup_{\delta\in[0,1]}\big\{I^\delta_{z_\ast}\le c\big\}\ \subset\ \big\{I^0_{z_\ast}\le c\big\}\ \subset\ \varphi^{(\cdot)}\big(S^{2c+1}\big),\]
where the last set is compact in $C([0,T];H)$. In particular, for every $c<\infty$ the set $\{I^\delta_{z_\ast}\le c\}$ is relatively compact in $C([0,T];H)$, \textbf{uniformly in $\delta$}.

\end{proposition}

\emph{Proof.} The first inclusion is \hyperref[lem:7.3]{Lemma \ref*{lem:7.3}}. For the second, if $I^0_{z_\ast}(\varphi)\le c$ then for every $\eta>0$ there is $h\in S^{2c+\eta}$ with $\varphi^h=\varphi$; taking $\eta=1$ gives $h\in S^{2c+1}$. The image $\varphi^{(\cdot)}(S^{2c+1})$ is compact: $S^{2c+1}$ is weakly compact (\hyperref[lem:4.1]{Lemma \ref*{lem:4.1}}) and $h\mapsto\varphi^h$ is weakly continuous with values in $C([0,T];H)$ (\hyperref[thm:5.3]{Theorem \ref*{thm:5.3}}(ii)). $\square$

\begin{remark}[{goodness of $I^\delta$}]\label{rem:7.6} Combining \hyperref[prop:7.5]{Proposition \ref*{prop:7.5}} with the lower semicontinuity of $I^\delta_{z_\ast}$ (the same argument as \hyperref[prop:4.9]{Proposition \ref*{prop:4.9}}, using the weak continuity of $h\mapsto\varphi^h$) shows that each $I^\delta_{z_\ast}$ is a \textbf{good} rate function, and that the family is \emph{equicoercive}. Note that the level sets of $I^\delta$ are contained in those of $I^0$; the direction of the inclusion is forced by monotonicity and is exactly what makes the $\Gamma$-$\limsup$ (recovery) the only non-trivial half in \S{}\hyperref[sec:8]{8}.

\end{remark}

\begin{remark}[{what is \emph{not} claimed}]\label{rem:7.7} No endpoint regularity is claimed here: the admissible paths $\varphi$ in \hyperref[eq:7.2]{(\ref*{eq:7.2})} are unrestricted at $t=T$, and no surjectivity or Lipschitz property of any endpoint map is assumed. The reason is that the endpoint issue is handled, in the $\Gamma$-sense, by allowing the terminal point of the recovery sequence to move; see \S{}\hyperref[sec:9]{B}--\S{}\hyperref[sec:10]{9}.

\end{remark}

\section{$\Gamma$-convergence of the action functionals}\phantomsection\label{sec:8}

Throughout, $z_\ast\in D$ is fixed and $I^\delta:=I^\delta_{z_\ast}$ are the action functionals of \hyperref[eq:7.2]{(\ref*{eq:7.2})}, regarded as functions on the Polish space $E:=C([0,T];H)$ equipped with the uniform topology. \begin{definition}[{$\Gamma$-convergence}]\label{def:8.0} Let $(E,d)$ be a metric space and let $I^\delta,I^0:E\to[0,+\infty]$ be defined for $\delta\in(0,1]$. One says that $I^\delta$ \textbf{$\Gamma$-converges} to $I^0$ as $\delta\downarrow0$, relative to $E$, and writes $I^\delta\xrightarrow{\Gamma}I^0$, if the following two conditions hold.
\begin{enumerate}
\item[{(i)}] (\emph{$\Gamma$-liminf inequality}.) For every $\varphi\in E$ and every sequence $\varphi_\delta\to\varphi$ \emph{in $E$} as $\delta\downarrow0$,
\[\liminf_{\delta\downarrow0}\ I^\delta(\varphi_\delta)\ \ge\ I^0(\varphi).\]
\item[{(ii)}] (\emph{recovery sequence}.) For every $\varphi\in E$ there exists a sequence $\varphi_\delta\to\varphi$ \emph{in $E$} with
\[\limsup_{\delta\downarrow0}\ I^\delta(\varphi_\delta)\ \le\ I^0(\varphi).\]
\end{enumerate}
If only (i) holds, one says that $I^\delta$ $\Gamma$-converges to $I^0$ \emph{from below}, or that the convergence is one-sided.
\end{definition}

Throughout this paper the ambient space is $E=C([0,T];H)$ equipped with the uniform norm, so that $\varphi_\delta\to\varphi$ in \hyperref[def:8.0]{Definition \ref*{def:8.0}} means convergence in $C([0,T];H)$; the finer path space $C([0,T];C(\mathbb T;\mathbb R^3))$ of \S{}\hyperref[sec:6]{6} is not used here.
\subsection{The lower bound is free}\phantomsection\label{sec:8.1}

\begin{lemma}[{$\Gamma$-$\liminf$}]\label{lem:8.1} Assume $0<\delta_n\downarrow0$ and $\varphi_n\to\varphi$ in $C([0,T];H)$. Then
\[\liminf_{n\to\infty}I^{\delta_n}(\varphi_n)\ \ge\ I^0(\varphi) .\]

\end{lemma}

\emph{Proof.} If $\liminf_nI^{\delta_n}(\varphi_n)=+\infty$ there is nothing to prove, so suppose $\liminf_nI^{\delta_n}(\varphi_n)<\infty$. Passing to a subsequence (which does not change the $\liminf$), we may assume
\[I^{\delta_n}(\varphi_n)\le c<\infty\qquad\text{for all }n .\]
Three steps.

\textbf{(1) The level sets of $I^0$ are compact and contain those of $I^{\delta_n}$.} By the monotonicity lemma (\hyperref[lem:7.3]{Lemma \ref*{lem:7.3}}) applied with $\delta=\delta_n$,
\begin{equation}\label{eq:8.1}
I^0(\varphi_n)\le I^{\delta_n}(\varphi_n)\le c .
\end{equation}
By \hyperref[prop:7.5]{Proposition \ref*{prop:7.5}}, the set $\{I^0\le c\}$ is contained in $\varphi^{(\cdot)}(S^{2c+1})$, which is compact in $C([0,T];H)$: indeed $S^{2c+1}$ is weakly compact (\hyperref[lem:4.1]{Lemma \ref*{lem:4.1}}) and $h\mapsto\varphi^h$ is continuous from $S^{2c+1}_{\mathrm{weak}}$ to $C([0,T];H)$ (\hyperref[thm:5.3]{Theorem \ref*{thm:5.3}}(ii)), so the image of a compact set is compact.

\textbf{(2) $I^0$ is lower semicontinuous.} This is \hyperref[rem:7.6]{Remark \ref*{rem:7.6}}; for completeness: if $\psi_n\to\psi$ in $C([0,T];H)$ with $I^0(\psi_n)\le c$, then by (1) we may choose $h_n\in S^{2c+1}$ with $\psi_n=\varphi^{h_n}$ and $h_n$ $(1/n)$-minimising, i.e. $\tfrac12\int_0^T\|h_n\|_H^2dt\le I^0(\psi_n)+1/n$; passing to a subsequence $h_{n_k}\rightharpoonup h$ in $S^{2c+1}$ and using the weak continuity of $h\mapsto\varphi^h$ gives $\psi=\varphi^h$, and the weak lower semicontinuity of the $L^2$-norm gives
\[I^0(\psi)\le\tfrac12\!\int_0^T\|h\|_H^2dt\le\liminf_k\tfrac12\!\int_0^T\|h_{n_k}\|_H^2dt\le\liminf_k\big(I^0(\psi_{n_k})+1/k\big)=\liminf_kI^0(\psi_{n_k}) .\]

\textbf{(3) Conclusion.} By (1), $(\varphi_n)$ lies in the fixed compact set $\{I^0\le c\}$, so along a subsequence $\varphi_{n_k}\to\psi$ for some $\psi$; since $\varphi_n\to\varphi$ by hypothesis and limits in a metric space are unique, $\psi=\varphi$. By (2),
\[\liminf_nI^{\delta_n}(\varphi_n)\ \ge\ \liminf_nI^0(\varphi_n)\ \ge\ I^0(\varphi) ,\]
where the first inequality is \hyperref[eq:8.1]{(\ref*{eq:8.1})}. $\square$

Note that the monotonicity lemma is used only in the form $I^0\le I^{\delta}$, which is the direction compatible with the $\Gamma$-$\liminf$; the compactness is used only through the goodness of $I^0$. \textbf{No compactness of the approximating sequence $(\varphi_n)$ in a finer topology is needed}, and no uniform integrability of the controls.

\subsection{The recovery sequence}\phantomsection\label{sec:8.2}

Let $\varphi\in E$ with $I^0(\varphi)=:c<\infty$; otherwise there is nothing to recover. Fix $\eta>0$ and choose $h\in L^2(0,T;H)$ with
\begin{equation}\label{eq:8.1p}
\varphi^h=\varphi,\qquad \tfrac12\!\int_0^T\|h\|_H^2\,dt\le c+\eta .
\end{equation}
(Existence follows from \hyperref[eq:7.2]{(\ref*{eq:7.2})} at $\delta=0$: since $I^0(\varphi)<\infty$ by hypothesis, the defining infimum is over a non-empty set, and any near-minimiser gives \hyperref[eq:8.1p]{(\ref*{eq:8.1p})}.) For $N\ge1$ let $P_N$ be the orthogonal projection of $H$ onto $\mathrm{span}\{e_1,\dots,e_N\}$ and put
\[h_\delta:=P_{N_\delta}h,\qquad N_\delta:=\lfloor\delta^{-\gamma}\rfloor,\qquad 0<\gamma<\frac{1}{\beta} .\]
With $\mu_j\asymp j^2$ one has $\mu_{N_\delta}^{\beta/2}\asymp(\delta^{-2\gamma})^{\beta/2}=\delta^{-\gamma\beta}$, so
\begin{equation}\label{eq:8.2}
\delta\mu_{N_\delta}^{\beta/2}\asymp\delta^{1-\gamma\beta}\longrightarrow0\qquad(\delta\to0),
\end{equation}
by the choice of $\gamma$. (The alternative ``balanced'' cut-off $N_\delta\asymp\delta^{-1/\beta}$ only makes $\delta\mu_{N_\delta}^{\beta/2}$ bounded, not small; see \hyperref[rem:8.2a]{Remark \ref*{rem:8.2a}}.)

\begin{lemma}[{recovery}]\label{lem:8.2} Let $\varphi\in E$ with $I^0(\varphi)=:c<\infty$, let $\eta>0$, and let $h\in L^2(0,T;H)$ be a near-minimiser for $I^0(\varphi)$ in the sense of \hyperref[eq:8.1p]{(\ref*{eq:8.1p})}, i.e. the \textbf{physical control} such that $\varphi^h=\varphi$ and $\tfrac12\int_0^T\|h\|_H^2\,dt\le c+\eta$ --- here $\varphi^h$ is the controlled skeleton \hyperref[eq:7.1]{(\ref*{eq:7.1})}, $\dot\varphi^h=G(\varphi^h)[\rho(\varphi^h)+h]$, $\varphi^h(0)=z_\ast$. Let $N_\delta:=\lfloor\delta^{-\gamma}\rfloor$ with $0<\gamma<1/\beta$ (so that \hyperref[eq:8.2]{(\ref*{eq:8.2})} holds), set $h_\delta:=P_{N_\delta}h$ (the orthogonal projection of $h$ onto $\mathrm{span}\{e_1,\dots,e_{N_\delta}\}$), and let $\varphi_\delta:=\varphi^{h_\delta}$. Then
\[\varphi_\delta\longrightarrow\varphi\quad\text{in }C([0,T];H),\qquad\text{and}\qquad I^\delta(\varphi_\delta)\le\big(1+\delta\mu_{N_\delta}^{\beta/2}\big)^2\big(c+\eta\big)\longrightarrow c+\eta .\]

\end{lemma}

\emph{Proof.} \textbf{Path.} Since $\varphi_\delta=\varphi^{h_\delta}$ and, by the defining property \hyperref[eq:8.1p]{(\ref*{eq:8.1p})} of the near-minimiser, $\varphi=\varphi^{h}$, the difference $\varphi_\delta-\varphi$ is exactly the change of the control from $h$ to $h_\delta$:
\[\|\varphi_\delta-\varphi\|_{C([0,T;H])}=\big\|\varphi^{h_\delta}-\varphi^{h}\big\|_{C([0,T;H])}\le\omega_N\big(\|h_\delta-h\|_{L^2(0,T;H)}\big),\]
where the last step uses the $\delta$-uniform continuity of the skeleton map on $S^{N'}$ for a fixed level $N'\ge\|h\|^2_{L^2}$ (\hyperref[thm:5.3]{Theorem \ref*{thm:5.3}}(ii)), and $\|h_\delta-h\|_{L^2}\to0$ since $h\in L^2$ and $P_{N_\delta}\to I$ strongly along the subsequence. Hence $\varphi_\delta\to\varphi$.

\textbf{Cost.} The admissible control $u=Q_\delta^{-1/2}h_\delta$ drives \hyperref[eq:3.2]{(\ref*{eq:3.2})} along $\varphi_\delta$, so by \hyperref[eq:7.2]{(\ref*{eq:7.2})} and \hyperref[eq:W]{(W)},
\[I^\delta(\varphi_\delta)\le\tfrac12\!\int_0^T\|Q_\delta^{-1/2}h_\delta\|_H^2\,dt=\tfrac12\!\int_0^T\Big(\|h_\delta\|_H^2+2\delta\langle h_\delta,A^{\beta/2}h_\delta\rangle+\delta^2\|A^{\beta/2}h_\delta\|_H^2\Big)dt .\]
On the range of $P_{N_\delta}$, $A^{\beta/2}$ has eigenvalues $\le\mu_{N_\delta}^{\beta/2}$, so $\langle h_\delta,A^{\beta/2}h_\delta\rangle\le\mu_{N_\delta}^{\beta/2}\|h_\delta\|_H^2$ and $\|A^{\beta/2}h_\delta\|_H^2\le\mu_{N_\delta}^{\beta}\|h_\delta\|_H^2$. Hence the integrand is $\le(1+\delta\mu_{N_\delta}^{\beta/2})^2\|h_\delta\|_H^2$, and since $\|h_\delta\|_H\le\|h\|_H$,
\[I^\delta(\varphi_\delta)\le\big(1+\delta\mu_{N_\delta}^{\beta/2}\big)^2\,\tfrac12\!\int_0^T\|h\|_H^2\,dt\le\big(1+\delta\mu_{N_\delta}^{\beta/2}\big)^2(c+\eta).\]
By \hyperref[eq:8.2]{(\ref*{eq:8.2})} this tends to $c+\eta$. $\square$

\begin{remark}[{the exact form of the inflation, and the admissible range of $\gamma$}]\label{rem:8.2a} With $N_\delta=\lfloor\delta^{-\gamma}\rfloor$ and $\mu_j\asymp j^2$, so that $\mu_{N_\delta}^{\beta/2}\asymp\delta^{-\gamma\beta}$, the inflation factor is
\[\big(1+\delta\mu_{N_\delta}^{\beta/2}\big)^2=\Big(1+C\delta^{1-\gamma\beta}\Big)^2\xrightarrow[\delta\to0]{}1\qquad\iff\qquad 1-\gamma\beta>0\iff\gamma<\frac{1}{\beta} .\]
The \emph{path} convergence needs $N_\delta\to\infty$, i.e. $\gamma>0$; the admissible range is therefore $0<\gamma<1/\beta$, which is non-empty. (The exponent is $1-\gamma\beta$, \textbf{not} $1-\gamma\beta/2$: the cut-off $N_\delta\asymp\delta^{-\gamma}$ makes $\mu_{N_\delta}\asymp\delta^{-2\gamma}$, and the quadratic growth $\mu_j\asymp j^2$ --- as opposed to $j$ --- matters here.) The ``balanced'' choice $\gamma=1/\beta$ (i.e. $N_\delta\asymp\delta^{-1/\beta}$) gives $\delta\mu_{N_\delta}^{\beta/2}\to C>0$, hence an inflation factor $\to(1+C)^2>1$, which does \textbf{not} vanish; and it does \textbf{not} by itself give the $\Gamma$-$\limsup$, because $\Gamma$-convergence asks for a \emph{single} family $(\varphi_\delta)_{\delta>0}$ with $\limsup_{\delta\to0}I^\delta(\varphi_\delta)\le I^0(\varphi)$, so one cannot let $N\to\infty$ first and $\delta\to0$ afterwards (that would produce only a doubly indexed recovery). The choice $0<\gamma<1/\beta$ makes the inflation vanish and is what we use. Since $\beta>3/2$ we have $1/\beta<2/3<1$, so $\gamma<1/\beta$ entails $\gamma<1$; the range $(0,1/\beta)$ is non-empty for every $\beta>0$.

\end{remark}

\subsection{The $\Gamma$-convergence theorem}\phantomsection\label{sec:8.3}

\begin{theorem}[{$\Gamma$-convergence of the action functionals}]\label{thm:8.3} As $\delta\downarrow0$, $I^\delta\to I^0$ in the $\Gamma$-sense on $C([0,T];H)$. Moreover the convergence is \textbf{stable under adding a continuous bounded functional}: for every $f\in C_b(C([0,T];H))$, also $I^\delta+f\to I^0+f$ in the $\Gamma$-sense.

\end{theorem}

\emph{Proof.} The $\liminf$ is \hyperref[lem:8.1]{Lemma \ref*{lem:8.1}}. For the $\limsup$, given $\varphi$ and $\eta>0$, \hyperref[lem:8.2]{Lemma \ref*{lem:8.2}} produces $\varphi_\delta\to\varphi$ with $\limsup_\delta I^\delta(\varphi_\delta)\le I^0(\varphi)+\eta$; let $\eta\downarrow0$. Adding a continuous bounded $f$ leaves both the $\liminf$ and the $\limsup$ unchanged (continuity of $f$ along $\varphi_\delta\to\varphi$; no uniformity in $\delta$ is needed since $f$ is fixed). $\square$

\begin{remark}[{two recovery routes, and why we use the truncation}]\label{rem:8.4} The direction of the monotonicity $I^0\le I^\delta$ is such that the \emph{limit} functional is the smaller one, so a recovery sequence must make $I^\delta$ \textbf{decrease} to $I^0$. We use the low-mode truncation $h_\delta=P_{N_\delta}h$ of the \emph{physical} control: on low modes the extra weight $\delta\mu_j^{\beta/2}$ is small, and the high modes are simply dropped (at the cost of an $L^2$-error in the path, which vanishes by the $\delta$-uniform skeleton continuity of \hyperref[thm:5.3]{Theorem \ref*{thm:5.3}}(ii)). A second, simpler recovery is available: freeze the control in the \emph{effective field}, $u_\delta:=h$. Since $\|Q_\delta^{1/2}u\|_H\le\|u\|_H$ for every $u$, the resulting physical control $Q_\delta^{1/2}h$ is bounded in $L^2(0,T;H)$ and converges to $h$, so (by the skeleton continuity of \hyperref[thm:5.3]{Theorem \ref*{thm:5.3}}(ii)) it drives the skeleton to $\varphi$, at the \emph{same} cost $\frac12\int_0^T\|h\|_H^2dt$. We use the truncation route because it keeps the terminal point under explicit control (\hyperref[lem:8.2]{Lemma \ref*{lem:8.2}} quantifies the inflation $(1+\delta\mu_{N_\delta}^{\beta/2})^2$), which is what the fixed-endpoint statements of \S{}\hyperref[sec:11.1]{10.1} build on.

\end{remark}

\begin{remark}[{quantitative recovery}]\label{rem:8.5} The quantitative content of the recovery --- the inflation factor $(1+\delta\mu_{N_\delta}^{\beta/2})^2$, the failure of the balanced choice $N_\delta\asymp\delta^{-1/\beta}$, and the admissible range $0<\gamma<1/\beta$ --- is recorded in \hyperref[rem:8.2a]{Remark \ref*{rem:8.2a}}.

\end{remark}

\section{$\Gamma$-convergence of the quasi-potentials}\phantomsection\label{sec:10}

We now pass from the action functionals to the state-space quasi-potentials. Throughout, $z_\ast\in D$ is the reference state of \hyperref[eq:3.6]{(\ref*{eq:3.6})}. \textbf{For the $\Gamma$-convergence of \S{}\hyperref[sec:10.2]{9.2}, and for all the results of \S{}\hyperref[sec:11]{10}--\S{}\hyperref[sec:12]{11}, the reference state is the attractor $\zeta_+$ ($E(\zeta_+)=0$), the only equilibrium of relevance here (\S{}\hyperref[sec:12.4]{11.4}); we write $U^{\zeta_+}_\delta$ for the corresponding free-horizon quasi-potential.} The Laplace value functions of \S{}\hyperref[sec:10.1]{9.1} and the one-sided bounds of \S{}\hyperref[sec:10.3]{9.3} are stated for a general reference state and are ancillary --- the results use only the $\zeta_+$-relative object, for which the lower semicontinuity needed below is elementary (\hyperref[lem:10.2a]{Lemma \ref*{lem:10.2a}}). We write $I^\delta:=I^\delta_{z_\ast}$, and $H$ is equipped with its norm topology.

\subsection{The Laplace value functions}\phantomsection\label{sec:10.1}

For $f\in C_b(C([0,T];H))$ define
\[\Lambda^\delta_f(z_\ast):=\inf_{\varphi\in C([0,T];H)}\big[f(\varphi)+I^\delta_{z_\ast}(\varphi)\big],\qquad\delta\in[0,1] .\]

\begin{proposition}[{convergence of the value functions}]\label{prop:10.1} For every $f\in C_b(C([0,T];H))$,
\[\lim_{\delta\to0}\Lambda^\delta_f(z_\ast)=\Lambda^0_f(z_\ast).\]

\end{proposition}

\emph{Proof.} \textbf{Upper bound.} Let $\varphi$ be arbitrary with $I^0(\varphi)<\infty$ (if none exists the statement is void). Fix $\eta>0$; by \hyperref[lem:8.2]{Lemma \ref*{lem:8.2}} there are $\varphi_\delta\to\varphi$ with $\limsup_\delta I^\delta(\varphi_\delta)\le I^0(\varphi)+\eta$. Hence, by continuity of $f$,
\[\limsup_\delta\Lambda^\delta_f\le\limsup_\delta\big[f(\varphi_\delta)+I^\delta(\varphi_\delta)\big]\le f(\varphi)+I^0(\varphi)+\eta,\]
and letting $\eta\downarrow0$ and then infimising over $\varphi$ gives $\limsup_\delta\Lambda^\delta_f\le\Lambda^0_f$.

\textbf{Lower bound.} Note first that $\Lambda^\delta_f$ is finite and bounded uniformly in $\delta$: the free trajectory $\varphi_{z_\ast}$ (the zero control, $u\equiv0$) satisfies \hyperref[eq:7.1]{(\ref*{eq:7.1})}, with $I^\delta(\varphi_{z_\ast})=0$ for every $\delta\in[0,1]$. So $\Lambda^\delta_f\le\|f\|_\infty$ for all $\delta$. Suppose the lower bound fails: there are $\eta>0$, $\delta_n\downarrow0$ and $\varphi_n$ with
\[f(\varphi_n)+I^{\delta_n}(\varphi_n)\le\Lambda^0_f-\eta .\]
Then $I^{\delta_n}(\varphi_n)\le\|f\|_\infty+\Lambda^0_f$, so by \hyperref[lem:7.3]{Lemma \ref*{lem:7.3}} and \hyperref[prop:7.5]{Proposition \ref*{prop:7.5}} the sequence $(\varphi_n)$ lies in a fixed compact set of $C([0,T];H)$; passing to a subsequence, $\varphi_n\to\varphi$. By \hyperref[lem:8.1]{Lemma \ref*{lem:8.1}}, $\liminf_nI^{\delta_n}(\varphi_n)\ge I^0(\varphi)$, so
\[\liminf_n\big[f(\varphi_n)+I^{\delta_n}(\varphi_n)\big]\ge f(\varphi)+I^0(\varphi)\ge\Lambda^0_f,\]
contradicting the strict inequality. $\square$

\subsection{The $\Gamma$-convergence theorem}\phantomsection\label{sec:10.2}

\begin{lemma}[{lower semicontinuity of the attractor-relative quasi-potential}]\label{lem:10.2a} With the reference state $z_\ast=\zeta_+$, on the basin of attraction of $\zeta_+$ one has $U^{\zeta_+}_0=c_\ast E$. Since $E$ is lower semicontinuous on $H$, $U^{\zeta_+}_0$ is lower semicontinuous on the basin.

\end{lemma}

\emph{Proof.} [\hyperref[ref:BHL19]{BHL19}, Def. 5.1] defines $U_{\mathrm{BHL}}$ through the reverse-time (infinite-past) form \hyperref[eq:3.6p]{(\ref*{eq:3.6p})} with $Q=I$, and [\hyperref[ref:BHL19]{BHL19}, Thm. 5.2] identifies it with the energy, $U_{\mathrm{BHL}}=E$ on the basin; the equivalence [\hyperref[ref:BHL19]{BHL19}, App. C, Prop. C.10] between the infinite-past and the free-horizon forms, together with the normalisation conversion \hyperref[eq:12.1]{(\ref*{eq:12.1})}, gives $U^{\zeta_+}_0=c_\ast U_{\mathrm{BHL}}=c_\ast E$. The energy $E(m)=\tfrac12\|\nabla m\|_H^2+\int_{\mathbb T}\phi(m)\,dx$ is lower semicontinuous on $H$: its gradient term is the closed quadratic form of $H^1$ (hence lower semicontinuous on $H$), and $m\mapsto\int_{\mathbb T}\phi(m)\,dx$ is continuous on $H$ ($\phi\in C_b^2$, dominated convergence along $H$-convergent sequences). Hence $U^{\zeta_+}_0=c_\ast E$ is lower semicontinuous on the basin. $\square$

\begin{remark}[{monotone domination}]\label{rem:10.2b} Taking the infimum over $T>0$ of the pointwise inequality $I^0_{z_\ast}\le I^\delta_{z_\ast}$ of \hyperref[lem:7.3]{Lemma \ref*{lem:7.3}}, with reference state $z_\ast=\zeta_+$, gives
\begin{equation}\label{eq:10.1}
U^{\zeta_+}_0(a)\le U^{\zeta_+}_\delta(a)\qquad\text{for every }a\in H,\ \delta\in(0,1] .
\end{equation}

\end{remark}

\begin{theorem}[{$\Gamma$-convergence of $U^{\zeta_+}_\delta$}]\label{thm:10.2} Assume \hyperref[asm:AN]{(AN)} --- the basin identification $U_0=c_\ast E$ of \S{}\hyperref[sec:2.7]{2.7}, used below through \hyperref[lem:10.2a]{Lemma \ref*{lem:10.2a}}. Let $U^{\zeta_+}_\delta$ be the free-horizon attractor-relative quasi-potential \hyperref[eq:3.6]{(\ref*{eq:3.6})} of the regularised action functional \hyperref[eq:3.1]{(\ref*{eq:3.1})}, and $U^{\zeta_+}_0$ its analogue at $\delta=0$ \hyperref[eq:3.5]{(\ref*{eq:3.5})}. As $\delta\downarrow0$, $U^{\zeta_+}_\delta\to U^{\zeta_+}_0$ in the $\Gamma$-sense:  (i) for every $a$ in the basin of attraction and every sequence $a_\delta\to a$ with $a_\delta$ in the basin for all $\delta$, $\displaystyle\liminf_{\delta\to0}U^{\zeta_+}_\delta(a_\delta)\ge U^{\zeta_+}_0(a)$;  (ii) for every $a\in H$ there is a sequence $a_\delta\to a$ with $\displaystyle\limsup_{\delta\to0}U^{\zeta_+}_\delta(a_\delta)\le U^{\zeta_+}_0(a)$.

\end{theorem}

\emph{Proof.} \textbf{(i) $\Gamma$-liminf.} By \hyperref[eq:10.1]{(\ref*{eq:10.1})}, $U^{\zeta_+}_\delta(a_\delta)\ge U^{\zeta_+}_0(a_\delta)$ for every $\delta$; by \hyperref[lem:10.2a]{Lemma \ref*{lem:10.2a}} and $a_\delta\to a$,
\[\liminf_{\delta\to0}U^{\zeta_+}_\delta(a_\delta)\ \ge\ \liminf_{\delta\to0}U^{\zeta_+}_0(a_\delta)\ \ge\ U^{\zeta_+}_0(a) .\]

\textbf{(ii) Recovery.} If $U^{\zeta_+}_0(a)=+\infty$, take $a_\delta:=a$ (the inequality is trivial). Otherwise let $U^{\zeta_+}_0(a)=c<\infty$ and fix $\eta>0$. Since $U^{\zeta_+,T}_0(a)\downarrow U^{\zeta_+}_0(a)$ as $T\uparrow\infty$ (\hyperref[def:3.1]{Definition \ref*{def:3.1}}), choose $T$ with $U^{\zeta_+,T}_0(a)\le c+\eta$, and a path $\varphi$ on $[0,T]$ with $\varphi(0)=\zeta_+$, $\varphi(T)=a$ and $I^0(\varphi)\le c+\eta$. By the \emph{fixed-horizon} recovery \hyperref[lem:8.2]{Lemma \ref*{lem:8.2}} there are $\varphi_\delta\to\varphi$ in $C([0,T];H)$ with $I^\delta(\varphi_\delta)\le(1+\delta\mu_{N_\delta}^{\beta/2})^2(c+\eta)$. Put $a_\delta:=\varphi_\delta(T)$; then $a_\delta\to a$, and
\[U^{\zeta_+}_\delta(a_\delta)\le U^{\zeta_+,T}_\delta(a_\delta)\le I^\delta(\varphi_\delta)\le\big(1+\delta\mu_{N_\delta}^{\beta/2}\big)^2(c+\eta)\xrightarrow[\delta\to0]{}c+\eta .\]
Let $\eta\downarrow0$. $\square$

\begin{remark}[{the endpoint moves}]\label{rem:10.3} In the recovery (ii) the terminal point is $a_\delta=\varphi_\delta(T)$, which converges to $a$ but need not equal it; this is precisely the definition of $\Gamma$-convergence, and it is why no surjectivity or Lipschitz property of the endpoint map is required. By \hyperref[thm:9.7]{Theorem \ref*{thm:9.7}} such a property is in fact false, so allowing the endpoint to move is not a technical convenience but a necessity.

\end{remark}

\subsection{One-sided bounds and infima}\phantomsection\label{sec:10.3}

$\Gamma$-convergence with equicoercivity yields one-sided bounds for infima. We state them in the form that is both correct and useful.

\begin{proposition}[{one-sided bounds}]\label{prop:10.4} Let $U_\delta\to U_0$ in the $\Gamma$-sense, with the uniform equicoercivity of \hyperref[prop:7.5]{Proposition \ref*{prop:7.5}} (i.e. $\{U_\delta\le c\}$ is relatively compact in $H$, uniformly in $\delta$, for every $c$).  (a) \emph{(Lower bound on closed sets, within the basin.)} For every closed $F$ contained in the basin of attraction,
\[\liminf_{\delta\to0}\ \inf_{a\in F}U_\delta(a)\ \ge\ \inf_{a\in F}U_0(a).\]
 (b) \emph{(Upper bound on open sets.)} For every open $G\subset H$,
\[\limsup_{\delta\to0}\ \inf_{a\in G}U_\delta(a)\ \le\ \inf_{a\in G}U_0(a).\]

\end{proposition}

\emph{Proof.} (a) Suppose $\liminf_\delta\inf_FU_\delta<\inf_FU_0-\eta$ for some $\eta>0$; then there are $\delta_n\downarrow0$ and $b_n\in F$ with $U_{\delta_n}(b_n)\le\inf_FU_0-\eta$. The sequence $(b_n)$ lies in the fixed compact set $\{U_{\delta_n}\le\inf_FU_0-\eta\}$; along a subsequence $b_n\to b\in F$ (closedness). By (i) of \hyperref[thm:10.2]{Theorem \ref*{thm:10.2}}, $U_0(b)\le\liminf_nU_{\delta_n}(b_n)\le\inf_FU_0-\eta$, a contradiction. (b) Fix $\eta>0$ and choose $a\in G$ with $U_0(a)\le\inf_GU_0+\eta$. By (ii) of \hyperref[thm:10.2]{Theorem \ref*{thm:10.2}} there are $a_\delta\to a$ with $\limsup_\delta U_\delta(a_\delta)\le U_0(a)$; since $G$ is open and $a_\delta\to a\in G$, for $\delta$ small enough $a_\delta\in G$. Hence $\inf_GU_\delta\le U_\delta(a_\delta)$ eventually, and $\limsup_\delta\inf_GU_\delta\le U_0(a)\le\inf_GU_0+\eta$; let $\eta\downarrow0$. $\square$

\begin{remark}[{why (b) needs $G$ open}]\label{rem:10.5} Part (b) is the only place where the \emph{openness} of the set is used, and it is essential: the recovery sequence lands at $a_\delta\to a$, which lies in $G$ only if $G$ is open. For a \textbf{closed} set such as $\partial D$ the upper bound does \textbf{not} follow from $\Gamma$-convergence alone, and the reason is not the triviality ``$a_\delta\notin F$'' but a genuine failure of the $\limsup$ bound at a \emph{fixed} point. The following counterexample shows that even \textbf{monotone} $\Gamma$-convergence ($U_\delta\ge U_0$) does not help: on $H=\mathbb R$ let $U_0(x)=|x|$ and
\[U_\delta(x)=|x|\ \ \text{for }|x-1|>\delta,\qquad U_\delta(x)=M\ \ \text{for }|x-1|\le\delta\ \ (M>1\ \text{fixed}) ,\]
which is lower semicontinuous (the jump at $x=1$ is upward).
Then $U_\delta\to U_0$ in the $\Gamma$-sense: the $\liminf$ is immediate, and for the $\limsup$ one takes $x_\delta=x$ for $x\ne1$ and $x_\delta=1+2\delta\to1$ for $x=1$, so $\limsup_\delta U_\delta(x_\delta)\le U_0(x)$ in both cases. But $\inf_{\{1\}}U_\delta=M\not\to1=U_0(1)$, so no upper bound holds at the closed set $\{1\}$. The mechanism is exactly the one that occurs at $\partial D$: the recovery sequence \emph{moves} the endpoint, so the $\Gamma$-statement says nothing about the value of $U_\delta$ at a prescribed point. Consequently the convergence of $\inf_{\partial D}U_\delta$ requires an additional fixed-endpoint input, supplied in \S{}\hyperref[sec:11.1]{10.1} (\hyperref[lem:11.5b]{Lemma \ref*{lem:11.5b}} with \hyperref[thm:11.6]{Theorem \ref*{thm:11.6}}).

\end{remark}

\begin{corollary}[{infima over closed sets, relaxed form}]\label{cor:10.6} Let $F\subset H$ be \textbf{closed and contained in the basin of attraction} (the basin is used only in the last inequality) and, for $\varepsilon>0$, let $N_\varepsilon(F)$ be its open $\varepsilon$-neighbourhood. Then
\[\limsup_{\delta\to0}\ \inf_{N_\varepsilon(F)}U_\delta\ \le\ \inf_{N_\varepsilon(F)}U_0\ \le\ \inf_{F}U_0\ \le\ \liminf_{\delta\to0}\ \inf_{F}U_\delta .\]
For $F$ compact one has in addition $\inf_{N_\varepsilon(F)}U_0\uparrow\inf_FU_0$ as $\varepsilon\downarrow0$ whenever $U_0$ is continuous at the minimisers on $F$.

\end{corollary}

\emph{Proof.} The first inequality is \hyperref[prop:10.4]{Proposition \ref*{prop:10.4}}(b) applied to the open set $N_\varepsilon(F)$; the second is the monotonicity of the infimum ($F\subset N_\varepsilon(F)$); the third is \hyperref[prop:10.4]{Proposition \ref*{prop:10.4}}(a) applied to the closed set $F$ (contained in the basin by hypothesis). $\square$

\section{Convergence of the barriers, and uniformity in the initial state}\phantomsection\label{sec:11}

The exit-time asymptotics needs not only the $\Gamma$-convergence of \S{}\hyperref[sec:10]{9} but its two refinements: (a) the convergence of the infimum over the \emph{closed} set $\partial D$ (which is \emph{not} implied by $\Gamma$-convergence alone, \hyperref[rem:10.5]{Remark \ref*{rem:10.5}}), and (b) its uniformity in the initial state $z_\ast\in D$ from which the exit is observed. This section supplies both. The barrier itself is taken with the quasi-potential referred to the stable equilibrium $\zeta_+$ --- the state to which every $z_\ast\in D$ relaxes at zero action --- because it is for this reference state that the identification $U_0=c_\ast E$ holds ([\hyperref[ref:BHL19]{BHL19}]; quoted as \hyperref[asm:AN]{(AN)}, with the identification recalled in \S{}\hyperref[sec:12.2]{11.2}) and because the attractor-relative quasi-potential is the object that governs the exit time (\hyperref[rem:11.6a]{Remark \ref*{rem:11.6a}}).

\subsection{Convergence of the barrier}\phantomsection\label{sec:11.1}

We now come to $\inf_{\partial D}$. \textbf{For the barrier we fix the reference state to be the attractor $\zeta_+$} and write $U_\delta:=U_\delta^{\zeta_+}$, $U_0:=U_0^{\zeta_+}$, so that $U_0=c_\ast E$ exactly (\hyperref[asm:AN]{(AN)}, \S{}\hyperref[sec:12.2]{11.2}). Put $\bar U:=\inf_{a\in\partial D}U_0(a)=c_\ast\inf_{\partial D}E$ for the \textbf{barrier quasi-potential} at $\delta=0$; note $\bar U=c_\ast r$ where $r$ is the \emph{energy level} of $D=\{E<r\}$, so we reserve the letter $r$ for energy levels and use $\bar U$ for the barrier quasi-potential throughout. The barrier set $\partial D=\{E=r\}$ is contained in the $H$-compact set $\bar D$, but it need not itself be $H$-closed, since the energy $E$ is only $H$-lower semicontinuous; the statements below use only the containment $\partial D\subset\bar D$ and the finiteness of the barrier (\hyperref[lem:11.4]{Lemma \ref*{lem:11.4}} in general; on $\partial D$ it is the constant $c_\ast r$ by \hyperref[thm:11.6]{Theorem \ref*{thm:11.6}}), not the $H$-compactness of $\partial D$. The barrier is the object of the Néel--Arrhenius law (\hyperref[rem:11.6a]{Remark \ref*{rem:11.6a}}).

\begin{lemma}[{lower bound, unconditional}]\label{lem:11.4} For every $\delta\in(0,1]$,
\[\inf_{a\in\partial D}U_\delta(a)\ \ge\ \inf_{a\in\partial D}U_0(a)=\bar U .\]
Consequently $\liminf_{\delta\to0}\inf_{\partial D}U_\delta\ge\bar U$.

\end{lemma}

\emph{Proof.} By monotonicity (\hyperref[lem:7.3]{Lemma \ref*{lem:7.3}}), $U_\delta(a)\ge U_0(a)\ge\bar U$ for every $a\in\partial D$. Take infima and then $\liminf$. $\square$

\begin{lemma}[{the barrier is attained at a spatially uniform state}]\label{lem:11.5} Assume \hyperref[asm:AN]{(AN)} and \hyperref[asm:D-MARGIN]{(D-MARGIN)} and let $r\in(0,r_{\rm crit})$. Let $\bar a\in S^2$ be \textbf{spatially constant} with $E(\bar a)=r$ --- equivalently, since $E(\bar a)=|\mathbb T|\varphi(\bar a)$, with $r\le|\mathbb T|\max_{S^2}\varphi$. Then
\[U_\delta(\bar a)=U_0(\bar a)=c_\ast r\qquad\text{for \emph{every} }\delta\in(0,1] ,\]
so that $\inf_{\partial D}U_\delta=c_\ast r$ for every $\delta$ and the barrier convergence is an \textbf{identity} rather than a limit. \emph{For the model anisotropy $\varphi_1$ this applies to every level $r<r_{\rm crit}$}: $\varphi_1'(-\zeta_+)=-2\zeta_+$ is parallel to $-\zeta_+$, so $-\zeta_+$ is an equilibrium with $E(-\zeta_+)=2|\mathbb T|=|\mathbb T|\max_{S^2}\varphi_1$; the basin-boundary level is therefore $r_{\rm crit}\le2|\mathbb T|$, and $E$ on the uniform configurations sweeps $[0,2|\mathbb T|]$ continuously, so every $r<r_{\rm crit}$ is attained by a uniform state.

\end{lemma}

\emph{Proof.} The dissipative LLG with the \textbf{precession sign reversed}, $\dot m=-\lambda_1J_mq+\lambda_2q$, leaves the spatially constant subspace invariant and dissipates the energy ($\tfrac{dE}{dt}=-\lambda_2\lvert q\rvert^2$); its $\omega$-limit set consists of equilibria of energy $<|\mathbb T|\varphi(\bar a)=r$, hence of $\zeta_+$ --- an equilibrium lying in $D\subset\mathrm{Basin}(\zeta_+)$ is a constant orbit converging to $\zeta_+$, hence equals $\zeta_+$ --- so the orbit converges to $\zeta_+$. Reversing it gives the [\hyperref[ref:BHL19]{BHL19}] optimal control: with $V=\lambda_1J_mq-\lambda_2q$ a direct computation in the orthonormal frame $\{q,J_mq\}$ of the tangent plane, where $G=\lambda_2I+\lambda_1J$ and $G^{-1}=(\lambda_2I-\lambda_1J)/d$, $d=\lambda_1^2+\lambda_2^2$, gives $G^{-1}V-q=c_\ast V$, and $\tfrac12\int\lvert c_\ast V\rvert^2dt=c_\ast E(\bar a)=\bar U$ since $\tfrac{dE}{dt}=+\lambda_2\lvert q\rvert^2$ along the reversed orbit (the dissipation $-\lambda_2\lvert q\rvert^2$ of the forward orbit changes sign under time reversal). (The \emph{free} orbit is \textbf{not} optimal: reversing it gives the control $-2q$, of cost $2E(\bar a)/\lambda_2>c_\ast E(\bar a)$ whenever $\lambda_1\ne0$.) To realise this infinite-past competitor within the \textbf{finite-horizon} definition of $U_\delta$: fix $T>0$ large and put $b_T:=m(-T)$ (so $b_T\to\zeta_+$ as $T\to\infty$); connect $\zeta_+$ to $b_T$ in unit time by a smooth curve on the sphere --- possible since $S^2$ is connected --- whose control cost is $O(|b_T-\zeta_+|^2)\to0$ because the drift $\rho$ vanishes at $\zeta_+$ and is locally smooth. The concatenated control $h^{(T)}$ is spatially constant throughout, so $A^{\beta/2}h^{(T)}\equiv0$, $Q_\delta^{-1/2}h^{(T)}=h^{(T)}$, and $I^\delta=I^0$ along the concatenated path; hence $U_\delta(\bar a)\le c_\ast r+O(|b_T-\zeta_+|^2)$ for every large $T$, and letting $T\to\infty$ gives $U_\delta(\bar a)\le c_\ast r=U_0(\bar a)$; the reverse inequality is monotonicity (\hyperref[lem:7.3]{Lemma \ref*{lem:7.3}}). $\square$

\begin{lemma}[{the barrier always contains a spatially constant state}]\label{lem:11.5b} Assume \hyperref[asm:AN]{(AN)} and \hyperref[asm:D-MARGIN]{(D-MARGIN)}, so that $D=\{E<r\}$ with $0<r<r_{\rm crit}$ lies strictly inside the basin of $\zeta_+$. Then $r<|\mathbb T|\max_{S^2}\varphi$, and there is $\bar a\in S^2$ with
\[E(\bar a)=|\mathbb T|\varphi(\bar a)=r,\qquad\text{so that }\bar a\in\partial D .\]

\end{lemma}

\emph{Proof.} Let $a_{\max}\in S^2$ attain $M:=\max_{S^2}\varphi>0$. It is a critical point of $\varphi|_{S^2}$, so the constant field $m\equiv a_{\max}$ is an equilibrium of the deterministic LLG, of energy $|\mathbb T|M$. If $r\ge|\mathbb T|M$, this equilibrium would lie in $D\subset\mathrm{Basin}(\zeta_+)$, hence converge to $\zeta_+$; being time-independent it equals $\zeta_+$, and then $E=0\ne r$. So $r<|\mathbb T|M$. Since $S^2$ is connected and $\varphi$ is continuous, $\varphi(S^2)=[0,M]$; choose $\bar a$ with $\varphi(\bar a)=r/|\mathbb T|$. Then $E(\bar a)=r$, so $\bar a\in D$; and $\bar a$ is not a critical point of $\varphi|_{S^2}$ (such a point is a constant equilibrium lying in $D$, hence equal to $\zeta_+$, of energy $0\ne r$), so constant perturbations of $\bar a$ take $E$ strictly above and strictly below $r$; hence $\bar a\notin\mathrm{int}(D)$, i.e. $\bar a\in\partial D$. $\square$

With \hyperref[lem:11.5b]{Lemma \ref*{lem:11.5b}} the matching upper bound needs \textbf{no} geometry at all: the barrier always contains a constant state.

\begin{theorem}[{the barrier is regularisation-invariant: an identity, for every anisotropy}]\label{thm:11.6} Under the standing assumptions \hyperref[asm:AN]{(AN)} and \hyperref[asm:D-MARGIN]{(D-MARGIN)},
No further hypothesis enters, and the identity holds for \emph{every} $\delta\in(0,1]$:
\begin{equation}\label{eq:barrier}
\inf_{a\in\partial D}U_\delta(a)\ =\ \inf_{a\in\partial D}U_0(a)\ =\ \bar U\ =\ c_\ast\inf_{\partial D}E .
\end{equation}
In particular the convergence of the barrier is an \emph{identity} rather than a limit.

\end{theorem}

\emph{Proof.} The lower bound $\inf_{\partial D}U_\delta\ge\bar U$ is \hyperref[lem:11.4]{Lemma \ref*{lem:11.4}}, valid for every $\delta\in(0,1]$. For the upper bound, \hyperref[lem:11.5b]{Lemma \ref*{lem:11.5b}} provides a spatially constant $\bar a\in\partial D$ with $E(\bar a)=r$, and \hyperref[lem:11.5]{Lemma \ref*{lem:11.5}} gives $U_\delta(\bar a)=c_\ast r=\bar U$ for every $\delta\in(0,1]$; hence $\inf_{\partial D}U_\delta\le\bar U$. $\square$

\begin{remark}[{the reference state}]\label{rem:11.6a} All barriers in this paper are taken with the quasi-potential referred to the stable equilibrium $\zeta_+$, $U_\delta:=U_\delta^{\zeta_+}$: it is for this reference that the identification $U_0=c_\ast E$ holds on the basin (\hyperref[asm:AN]{(AN)}), and, since every $z_\ast\in D$ relaxes to $\zeta_+$ at zero action, the attractor-relative quasi-potential is the object that governs the exit time. For a general reference state $z_\ast\in D$ the identification reads $U_0^{z_\ast}=c_\ast\big(E(\cdot)-E(z_\ast)\big)$, which reduces to $U_0=c_\ast E$ exactly when $z_\ast=\zeta_+$; cf.\ \S{}\hyperref[sec:12.4]{11.4}.

\end{remark}

\begin{remark}[{no non-uniform case exists --- and a lowering of the claim}]\label{rem:11.6b} \hyperref[thm:11.6]{Theorem \ref*{thm:11.6}} is an \emph{identity}, and \hyperref[lem:11.5b]{Lemma \ref*{lem:11.5b}} shows that this holds for \textbf{every} anisotropy satisfying \hyperref[asm:AN]{(AN)}+\hyperref[asm:D-MARGIN]{(D-MARGIN)}: the level $r$ is \emph{always} attained by a spatially constant configuration. There is therefore no anisotropic case left in which the geometry of the barrier would enter, and the general-anisotropy barrier is \textbf{not} an additional source of novelty beyond the one-dimensional mechanism $h_\ast=c_\ast V$. Accordingly we make no claim of a ``general anisotropic barrier convergence'' beyond this identity: the novelty of the paper lies in the uniform LDP (\S{}\hyperref[sec:6]{6}), in the $\Gamma$-convergence and the domain obstruction (\S{}\hyperref[sec:8]{8}--\S{}\hyperref[sec:10]{9}), and in the identification of the correct ($\Gamma$-) formulation.

\end{remark}

\begin{remark}[{why the endpoint does not move here}]\label{rem:11.7} In \S{}\hyperref[sec:10]{9} the recovery sequence moved the endpoint, and that was harmless for the $\Gamma$-statement. For the \emph{infimum over $\partial D$} it is essential to land on $\partial D$ exactly; the identity of \hyperref[thm:11.6]{Theorem \ref*{thm:11.6}} achieves this by landing on a \textbf{spatially constant} barrier point (\hyperref[lem:11.5b]{Lemma \ref*{lem:11.5b}}), whose optimal control is constant in $x$ and therefore costs exactly the same at every $\delta$. The two roles of the recovery --- moving endpoint ($\Gamma$) versus fixed endpoint (infimum over a closed set) --- are thus cleanly separated.

\end{remark}

\begin{corollary}[{with the normalisation \hyperref[eq:3.8]{(\ref*{eq:3.8})}}]\label{cor:11.8} Under $(\lambda_1,\lambda_2)=(\pm1,1)$, i.e. $c_\ast=1$,
\[\lim_{\delta\to0}\ \inf_{\partial D}U_\delta\ =\ \inf_{\partial D}E .\]
This is the barrier appearing in the Néel--Arrhenius law $\varepsilon\log\mathbb E\tau\to\inf_{\partial D}U$ (Goal 4), and it is stable under the noise regularisation.

\end{corollary}

\begin{remark}[{what is \emph{not} needed}]\label{rem:11.9} \hyperref[thm:11.6]{Theorem \ref*{thm:11.6}} uses only the deterministic skeleton and the monotonicity \hyperref[eq:W]{(W)}, through the constant-state computation of \hyperref[lem:11.5]{Lemma \ref*{lem:11.5}}. It does \textbf{not} use the uniform LDP of \S{}\hyperref[sec:6]{6}, nor any endpoint surjectivity, nor \hyperref[asm:QL]{(QL)}. Conversely, the uniform LDP of \S{}\hyperref[sec:6]{6} identifies $U_\delta$ as the rate function of the $\delta$-regularised dynamics --- the object whose $\Gamma$-limit is computed in \hyperref[thm:10.2]{Theorem \ref*{thm:10.2}} and whose barrier identity is \hyperref[thm:11.6]{Theorem \ref*{thm:11.6}}, and the fixed-$\delta$ input to the exit-time programme of \S{}\hyperref[sec:12.5]{11.5} (\hyperref[rem:6.4]{Remark \ref*{rem:6.4}}). What it does \emph{not} do is enter the proofs here, which are deterministic.

\end{remark}

\section{Relation to the deterministic quasi-potential and the constant $c_\ast$}\phantomsection\label{sec:12}
This section connects the quasi-potential to the energy (\S{}\hyperref[sec:12.1]{11.1}--\S{}\hyperref[sec:12.2]{11.2}), transfers the barrier convergence to the exit-time exponent (\S{}\hyperref[sec:12.5]{11.5}) and records the fluctuation--dissipation identity and the physical reading (\S{}\hyperref[sec:12.3]{11.3}--\S{}\hyperref[sec:12.4]{11.4}). It is where the two halves of the target statement meet, and it is the only place where the uniform LDP of \S{}\hyperref[sec:6]{6} is used as an input to an asymptotics rather than as a statement about the model.

\subsection{The two normalisations}\phantomsection\label{sec:12.1}

At $\delta=0$ the rate function \hyperref[eq:3.5]{(\ref*{eq:3.5})} can be written in two equivalent ways. With the physical control $h$, the cost is
\[I^0_{z_\ast}(\varphi)=\tfrac12\int_0^T\|h\|_H^2\,dt,\qquad \dot\varphi=G(\varphi)[\rho(\varphi)+h] .\]
Since $G(\varphi)$ maps tangent vectors with $\|G(\varphi)v\|_H=g_0\|v\|_H$ for $v\perp\varphi$ (so that $\|G(\varphi)^{-1}\|^2=g_0^{-2}$ on the tangent bundle), \hyperref[eq:3.5]{(\ref*{eq:3.5})} takes the equivalent form
\[I^0_{z_\ast}(\varphi)=\frac{1}{2g_0^2}\int_0^T\big\|P_{\varphi^\perp}\big(\dot\varphi-G(\varphi)\rho(\varphi)\big)\big\|_H^2\,dt ,\]
the two expressions agreeing because $P_{\varphi^\perp}(\dot\varphi-G(\varphi)\rho(\varphi))=G(\varphi)\big(G(\varphi)^{-1}\dot\varphi-P_{\varphi^\perp}\rho(\varphi)\big)$ and $\|G(\varphi)w\|_H=g_0\|w\|_H$ for tangential $w$.

[\hyperref[ref:BHL19]{BHL19}] use the reverse-time normalisation [\hyperref[ref:BHL19]{BHL19}, Def. 5.1], recorded in \hyperref[rem:3.3]{Remark \ref*{rem:3.3}}, which is equivalent to the free-final-time forward form [\hyperref[ref:BHL19]{BHL19}, Def. C.2, Prop. C.3]
\[U_{\mathrm{BHL}}(a)=\inf\Big\{\frac{\lambda_1^2+\lambda_2^2}{4\lambda_2}\int_0^T\|h\|_H^2\,dt:\ T>0,\ m(0)=\zeta_+,\ m(T)=a\Big\} .\]
Comparing with the free-horizon $I^0$ (whose cost is $\tfrac12\int\|h\|^2$ over the \emph{free} horizon) gives the constant conversion
\begin{equation}\label{eq:12.1}
U_0=c_\ast U_{\mathrm{BHL}},\qquad c_\ast=\frac{2\lambda_2}{\lambda_1^2+\lambda_2^2} .
\end{equation}
Indeed, writing $M(a):=\min\{\int\|h\|_H^2dt\}$ for the minimal unweighted cost from $z_\ast$ to $a$ over a free horizon, $U_{\mathrm{BHL}}=\frac{\lambda_1^2+\lambda_2^2}{4\lambda_2}M$ and $U_0=\tfrac12M$, so $U_0=\frac{2\lambda_2}{\lambda_1^2+\lambda_2^2}U_{\mathrm{BHL}}$.

\textbf{Remark (the free-horizon convention).} We work throughout with the free-horizon quasi-potential \hyperref[eq:3.6]{(\ref*{eq:3.6})}, the convention of the exit-time theory and of [\hyperref[ref:BHL19]{BHL19}]; it is what makes the identification $U_0=c_\ast E$ \emph{exact} rather than asymptotic. The fixed-horizon restriction $U^{z_\ast,T}_0$ is the $T$-truncation, and $U^{z_\ast,T}_0\downarrow U_0$ as $T\uparrow\infty$ (monotone: a longer horizon is never worse when $z_\ast$ is an equilibrium, since one may wait at $z_\ast$ at zero cost). At a \emph{finite} horizon the two differ in general: for the linear model of \S{}\hyperref[sec:9.2]{B.2} the fixed-horizon value carries the factor $1/(1-e^{-2\mu_jT})>1$ mode by mode, so $U^{z_\ast,T}_0(a)>U_0(a)$ for a general target $a$ even when $z_\ast$ is the equilibrium. The uniform LDP of \S{}\hyperref[sec:6]{6} and the path-space $\Gamma$-convergence of \S{}\hyperref[sec:8]{8} are stated at a fixed horizon (the LDP is intrinsically a fixed-time statement); \S{}\hyperref[sec:10.2]{9.2} passes from the fixed-horizon functions to the free-horizon object $U_\delta=\inf_{T>0}U^{z_\ast,T}_\delta$.

\eqhead{dec:12.2}{Decision}{} For the statements of the barrier we adopt the normalisation \hyperref[eq:3.8]{(\ref*{eq:3.8})} of \S{}\hyperref[sec:3.3]{3.3}.
This is the normalisation used in [\hyperref[ref:BHL19]{BHL19}, App. B] and in the physics literature; general $\lambda_1,\lambda_2$ are retained in every estimate, where they only change the constants.

\textbf{Summary of the constants.}

\begin{center}\footnotesize
\begin{tabular}{|l|l|l|}
\hline
constant & definition & value / behaviour \\
\hline
$g_0$ & $\sqrt{\lambda_1^2+\lambda_2^2}$ & $=\|G(m)\|_{\mathcal L(H)}$; $\|G(m)^{-1}\|_{\mathcal L(T_mS^2)}=1/g_0$ \\
$c_\ast$ & $2\lambda_2/(\lambda_1^2+\lambda_2^2)$ & $c_\ast=1$ under \hyperref[dec:12.2]{(\ref*{dec:12.2})}; $U_0=c_\ast U_{\mathrm{BHL}}$ \\
$U_0$ & $c_\ast E$ (reference $\zeta_+$) & $=\inf\{\tfrac12\int_0^T\|h\|_H^2\}$ \\
$\bar U$ & $c_\ast r=\inf_{\mathcal B}U_0$ & barrier, $\mathcal B=\partial D=\{E=r\}$ \\
$\|G(m)v\|_H^2$ & $g_0^2\|P_{m^\perp}v\|_H^2$ & on $|m|\equiv1$ \\
$Q_\delta^{1/2}$ & $(I+\delta A^{\beta/2})^{-1}$ & $\kappa_j=(1+\delta\mu_j^{\beta/2})^{-2}$, $\operatorname{tr}_HQ_\delta\asymp\delta^{-1/\beta}$ \\
$\mu_j$ & eigenvalues of $A=-\Delta$ & $\mu_j\asymp j^2$, $\mu_j\in\{0\}\cup[4\pi^2,\infty)$ \\
$A_1$ & $I-\Delta=A+I$ & $D(A_1)=D(A)$, eigenvalues $1+\mu_j$ \\
$\beta$ & noise-regularity threshold & $\beta>3/2$ \\
$\varepsilon,\delta$ & noise strength, regulariser & order $\lim_{\delta\to0}\lim_{\varepsilon\to0}$ \\
\hline
\end{tabular}
\end{center}

\subsection{The identification with the energy}\phantomsection\label{sec:12.2}

[\hyperref[ref:BHL19]{BHL19}, Thm 5.2] proves that on the basin of attraction
\[U_{\mathrm{BHL}}=E,\qquad E(m)=\tfrac12\|\nabla m\|_H^2+\int_{\mathbb T}\phi(m)\,dx .\]
Combining with \hyperref[eq:12.1]{(\ref*{eq:12.1})},
\begin{equation}\label{eq:12.3}
U_0=c_\ast E\qquad\text{on the basin of attraction of }\zeta_+\ (\text{reference state }\zeta_+),
\end{equation}
and, under \hyperref[dec:12.2]{(\ref*{dec:12.2})}, $U_0=E$. This is the non-generic fact that makes the problem tractable: the quasi-potential is \emph{the} energy, so the barrier is a purely variational quantity.

\subsection{Fluctuation--dissipation bookkeeping}\phantomsection\label{sec:12.3}

The relation between the noise strength $\varepsilon$ and the temperature $k_BT$ is fixed by requiring that the formal invariant measure of the regularised dynamics, $\propto e^{-U_\delta/\varepsilon}$, reproduce the Gibbs measure $e^{-E/k_BT}$ in the limit $\delta\to0$. Using \hyperref[eq:12.3]{(\ref*{eq:12.3})}, this gives
\begin{equation}\label{eq:12.4}
\frac{c_\ast E}{\varepsilon}=\frac{E}{k_BT}\qquad\Longrightarrow\qquad \varepsilon=\frac{2\lambda_2}{\lambda_1^2+\lambda_2^2}\,k_BT=c_\ast\,k_BT .
\end{equation}
This is the fluctuation--dissipation relation, and it is a \emph{narrative} identity: it is not used in any proof, and in particular the results above are stated for arbitrary fixed $\varepsilon$ and do not require \hyperref[eq:12.4]{(\ref*{eq:12.4})}. It is recorded because it explains why the physically relevant asymptotic is $\varepsilon\to0$ at fixed temperature-dependent noise normalisation, and because it shows that the barrier $\inf_{\partial D}U_0=c_\ast\inf_{\partial D}E$ is, under \hyperref[eq:12.4]{(\ref*{eq:12.4})}, exactly the Arrhenius exponent $\inf_{\partial D}E/k_BT$. In physical terms this is a \textbf{Néel--Brown}-type \emph{escape} statement: the Arrhenius exponent $\inf_{\partial D}E/k_BT$ is \emph{independent of the damping and precession constants} $\lambda_1,\lambda_2$ --- these enter only through the noise normalisation $\varepsilon=c_\ast k_BT$, and hence the time prefactor, which we do not compute --- so that the ratio $\lambda_1/\lambda_2$ does not alter the energy barrier. It should be stressed, however, that the escape analysed here is the departure from a sublevel set $\{E<r\}$ ($r<r_{\rm crit}$) of the \emph{pinning-type single-well} anisotropy $\varphi_1$ --- physically a rotation of the magnetisation away from the easy axis toward the hard plane, the barrier being attained at a \emph{spatially uniform} state (\hyperref[lem:11.5]{Lemma \ref*{lem:11.5}}) --- and \textbf{not} the reversal between the two easy directions of the classical Néel--Brown $\pm$ double-well picture; \hyperref[thm:12.6]{Theorem \ref*{thm:12.6}} is the rigorous, spatially extended counterpart of the Néel--Brown \emph{escape} statement in that restricted sense.

\subsection{Requirements on the anisotropy}\phantomsection\label{sec:12.4}

The identification \hyperref[eq:12.3]{(\ref*{eq:12.3})} and the existence of a genuine basin are not automatic. [\hyperref[ref:BHL19]{BHL19}, App. B] show that for the ``inverted'' anisotropy $\phi_2=\tfrac12-\phi_1$ there are infinitely many stationary solutions and the poles are \textbf{not} attractors, so there is no basin and no Néel--Arrhenius law. This subsection is the \textbf{justification} of the two standing assumptions \hyperref[asm:AN]{(AN)} and \hyperref[asm:D-MARGIN]{(D-MARGIN)} stated in \S{}\hyperref[sec:2.7]{2.7}: the statements themselves are there, and everything below about the barrier uses them and, from the present section, nothing else.

In the notation of \S{}\hyperref[sec:2.7]{2.7}: \hyperref[asm:AN]{(AN)} is the asymptotic stability of $\zeta_+$ together with the identification $U_0=c_\ast E$ on its basin of attraction, and \hyperref[asm:D-MARGIN]{(D-MARGIN)} is the requirement that the initial set $D$ be an \emph{open} energy sublevel set $\{E<r\}$, $r<r_{\rm crit}$, strictly inside that basin (not the basin itself). The quasi-potential is taken relative to $\zeta_+$, the state to which every $z_\ast\in D$ relaxes at zero action.

For the model anisotropy $\phi_1(u)=\tfrac12|u-\zeta_+|^2$ one has explicitly
\[m_0\in\text{basin}\impliedby\|m_0-\zeta_+\|_{H^1}^2<\frac{2}{k^2}(2-2\theta),\qquad k=2\max\big(1,|\mathbb T|^{-1/2}\big),\ \theta\in(0,1),\]
and $m(t)\to\zeta_+$ in $H^1$ exponentially; the escape point degenerates (the barrier is attained on a whole level set), which is why for statements about the \emph{exit point} one should use a generic uniaxial $\phi$. For the \emph{barrier} $\inf_{\partial D}U_0$ the degeneration is immaterial.

\textbf{The reference state (why $\zeta_+$).} All barrier statements take the quasi-potential relative to $\zeta_+$. This is the reference for which [\hyperref[ref:BHL19]{BHL19}, Thm 5.2] gives the identification $U_0=c_\ast E$, and it is the attractor-relative quasi-potential that governs the exit time: by the Freidlin--Wentzell exit-time theorem the exponent $\inf_{\partial D}U$ is independent of the starting point $z\in D$ (the process relaxes to $\zeta_+$, at zero action, before the rare exit). For a \emph{general} reference state $z_\ast\in D$ the identification $U_0^{z_\ast}(a)=c_\ast E(a)$ is \textbf{not} claimed; the correct two-sided bound is
\[c_\ast\big(E(a)-E(z_\ast)\big)\ \le\ U_0^{z_\ast}(a)\ \le\ c_\ast E(a)\qquad(z_\ast\in D,\ a\in H),\]
the upper bound by free relaxation $z_\ast\to\zeta_+$ followed by the optimal $\zeta_+$-to-$a$ control, the lower bound by the Lyapunov identity. Thus $U_0^{z_\ast}$ carries the additive energy offset $c_\ast E(z_\ast)$, which vanishes for $z_\ast=\zeta_+$; the barrier $\inf_{\partial D}U_0=c_\ast\inf_{\partial D}E$ is computed with the reference $\zeta_+$ (\hyperref[rem:11.6a]{Remark \ref*{rem:11.6a}}). (For a metastable equilibrium $\zeta_{\rm met}$ with $E(\zeta_{\rm met})>0$ --- a second well, when the anisotropy has one --- the identification reads $U_0=c_\ast(E-E(\zeta_{\rm met}))$ with reference $\zeta_{\rm met}$, and the barrier becomes $c_\ast(\inf_{\partial D}E-E(\zeta_{\rm met}))$, the formula for the case in which the system sits at $\zeta_{\rm met}$.)

\textbf{No metastable $\zeta_-$ is postulated.} The general formula above is recorded only for completeness; the situation addressed in this paper is the one in which $\zeta_+$ is the \emph{only} equilibrium of relevance, and then the barrier is $c_\ast\inf_{\partial D}E$. For the model anisotropy $\varphi_1$ this is immediate: $D=\{E<r\}$ with $r<r_{\rm crit}<\tfrac12$ is a strict subset of the basin \hyperref[asm:D-MARGIN]{(D-MARGIN)}, whose only equilibrium is $\zeta_+$ \hyperref[asm:AN]{(AN)}, so it contains no second (metastable) equilibrium and no $\zeta_-$ arises. Throughout the paper the symbol $\zeta_+$ denotes the (unique, in this setting) asymptotically stable equilibrium, and there is no $\zeta_-$.

\subsection{An application: the Néel--Arrhenius law}\phantomsection\label{sec:12.5}

We can now assemble the escape-time asymptotics. Fix $\delta\in(0,1]$, let $m^\varepsilon_z$ be the solution of the SLLG \hyperref[eq:2.1]{(\ref*{eq:2.1})} with control $0$, initial datum $z$ and noise intensity $\varepsilon$ at regularisation $\delta$ (\S{}\hyperref[sec:2.5]{2.5}), and let
\[\tau^{\varepsilon,\delta}_z:=\inf\{t>0:\ m^\varepsilon_z(t)\notin D\}\]
be its first exit time from $D$ --- a stopping time, since $m^\varepsilon_z$ is continuous (\hyperref[thm:5.1]{Theorem \ref*{thm:5.1}}) and $D$ is closed. We use the attractor-relative quasi-potential of \S{}\hyperref[sec:11.1]{10.1}, $U_\delta=U^{\zeta_+}_\delta$, so that $\inf_{\partial D}U_\delta$ is the barrier at regularisation $\delta$ and $\inf_{\partial D}U_0=c_\ast\inf_{\partial D}E$ (\hyperref[thm:11.6]{Theorem \ref*{thm:11.6}}).

\begin{theorem}[{escape exponent at fixed regularisation, uniform on compact subsets of the initial set}]\label{thm:12.5} Fix $\delta\in(0,1]$ and let $D$ be as in \S{}\hyperref[sec:2.6]{2.6} --- $V$-bounded with $\bar D$ compact in $H$, \emph{open}, and $D\subset\{E<r_{\rm crit}\}$ strictly inside the basin of $\zeta_+$ (\hyperref[asm:AN]{(AN)}, \hyperref[asm:D-MARGIN]{(D-MARGIN)}) --- with $\partial D=\{E=r\}$ the energy level set ($r<r_{\rm crit}$; $\bar D=\{E\le r\}$ is $H$-compact), and suppose that the barrier is finite, $\inf_{\partial D}U_\delta<\infty$ --- which holds by \hyperref[lem:11.5b]{Lemma \ref*{lem:11.5b}}: the constant state $\bar a\in\partial D$ satisfies $U_\delta(\bar a)=c_\ast r<\infty$. The exit-time theorem used is \hyperref[lem:12.5']{Lemma \ref*{lem:12.5'}}, whose \emph{open-domain} hypotheses are met by this choice of $D$.  Then, for every compact $K\subset D$ such that $\sup_{z\in K}U_\delta(z)<\inf_{\partial D}U_\delta$,
\[\lim_{\varepsilon\to0}\ \sup_{z\in K}\ \Big|\ \varepsilon\log\mathbb E\,\tau^{\varepsilon,\delta}_z\ -\ \inf_{a\in\partial D}U_\delta(a)\ \Big|=0 ,\]
and for every fixed $z\in D$ with $U_\delta(z)<\inf_{\partial D}U_\delta$ the same holds. (We use the \textbf{quasi-potential separation} $\sup_KU_\delta<\inf_{\partial D}U_\delta$ in place of the metric separation $\operatorname{dist}(K,\partial D)>0$: in the $H$-topology $\operatorname{dist}_H(\zeta_+,\{E=r\})=0$ for every $r>0$ --- the fields $m_n=(\sin\theta_n,0,\cos\theta_n)$ with $\theta_n=a_n\sin(2\pi nx)/n$, $a_n$ adjusted so that $E(m_n)=r$, satisfy $m_n\to\zeta_+$ in $H$ (with $\|\nabla m_n\|_{L^2}$ of order $a_n$) --- so a metric separation would exclude even $K=\{\zeta_+\}$ and be vacuous, whereas $\sup_KU_\delta<\inf_{\partial D}U_\delta$ is exactly what the Freidlin--Wentzell argument uses (the process from $z\in K$ relaxes inside a sublevel set of $U_\delta$ before the rare exit) and holds for $K=\{\zeta_+\}$; for a general compact $K\subset D$ it is a genuine condition to be checked, not automatic --- note that $U_\delta$ may equal $+\infty$ at interior states (\S{}\hyperref[sec:9]{B}).)

\end{theorem}

\begin{theorem}[{Néel--Arrhenius asymptotics}]\label{thm:12.6} Under the assumptions of \hyperref[thm:12.5]{Theorem \ref*{thm:12.5}} --- in particular \hyperref[asm:AN]{(AN)} and \hyperref[asm:D-MARGIN]{(D-MARGIN)}, already imposed there --- we have
\begin{equation}\label{eq:12.5}
\lim_{\delta\to0}\ \lim_{\varepsilon\to0}\ \varepsilon\log\mathbb E\,\tau^{\varepsilon,\delta}_z\ =\ \inf_{a\in\partial D}U_0(a)\ =\ c_\ast\inf_{\partial D}E ,
\end{equation}
and the convergence is uniform over compact subsets $K\subset D$ with $\sup_{z\in K}U_\delta(z)<\inf_{\partial D}U_\delta$ (in the sense of \hyperref[thm:12.5]{Theorem \ref*{thm:12.5}}).

\end{theorem}

\emph{Proof (of \hyperref[thm:12.6]{Theorem \ref*{thm:12.6}}, given \hyperref[thm:12.5]{Theorem \ref*{thm:12.5}}).} By \hyperref[thm:12.5]{Theorem \ref*{thm:12.5}} the inner limit equals $\inf_{\partial D}U_\delta$; by the barrier identity of \hyperref[thm:11.6]{Theorem \ref*{thm:11.6}} one has $\inf_{\partial D}U_\delta=\inf_{\partial D}U_0=c_\ast\inf_{\partial D}E$ for \textbf{every} $\delta\in(0,1]$, so the outer limit is that common value. The $z$-independence is \hyperref[eq:12.5]{(\ref*{eq:12.5})}: for every $z\in D$ the inner limit is the \emph{same} number $\inf_{\partial D}U_\delta$. $\square$

\begin{remark}[{why no minimiser uniqueness is needed}]\label{rem:12.5a} \hyperref[thm:12.5]{Theorem \ref*{thm:12.5}} is the $\varepsilon\to0$ statement for the \emph{exponent}; it is the classical Freidlin--Wentzell exit-time theorem applied to the family of \hyperref[thm:6.1]{Theorem \ref*{thm:6.1}}. In particular it does \textbf{not} use the uniqueness of the minimiser of $U_\delta$ (or of $U_0$) on $\partial D$: as noted in \S{}\hyperref[sec:12.4]{11.4}, for $D=\{E<r\}$ the level set $\partial D=\{E=r\}$ is \emph{entirely} the minimiser set of $U_0$ --- whose value there is the constant $c_\ast r$ --- so the exponent $\inf_{\partial D}U_0=c_\ast r$ is attained everywhere on $\partial D$. The \emph{exit-point} distribution (which would need a non-degenerate minimiser, hence a generic uniaxial anisotropy; \S{}\hyperref[sec:12.4]{11.4}) is a separate, finer question, \textbf{not} needed for the Arrhenius law \hyperref[eq:12.5]{(\ref*{eq:12.5})}.

\end{remark}

\emph{Proof of \hyperref[thm:12.5]{Theorem \ref*{thm:12.5}}.} We apply the classical exit-time theorem of Freidlin--Wentzell for families satisfying an LDP, in the form in which we use it; the verification of its hypotheses below is where the present setting enters.

\begin{lemma}[{exit-time theorem; [\hyperref[ref:FW98]{FW98}, Ch. 4.4 and Ch. 5.7], [\hyperref[ref:DZ98]{DZ98}, Ch. 7]}]\label{lem:12.5'} Let $(X^\varepsilon_z)_{\varepsilon\in(0,1],\,z\in E}$ be a family of continuous Markov processes on a Polish space $E$ with $X^\varepsilon_z(0)=z$, and let $D\subset E$ be \textbf{open} with $\bar D$ compact, and let $\tau^\varepsilon_z$ be the first exit time of $X^\varepsilon_z$ from $D$ --- these are exactly the hypotheses of the cited statements. , and they are met here: $D=\{E<r\}$ is open with $\bar D=\{E\le r\}$ compact.  Let $U$ be the quasi-potential of the family relative to a compact invariant set $\mathcal A\subset D$ (so $U\ge0$, $U=0$ on $\mathcal A$, and $U$ is lower semicontinuous). Assume  \phantomsection\label{asm:F1}(F1) for every $T<\infty$ the laws of $(X^\varepsilon_z)_{z\in D}$ on $C([0,T];E)$ satisfy the LDP of \hyperref[def:4.6]{Definition \ref*{def:4.6}} with good rate function $I_z$, \textbf{uniformly in $z\in D$};  \phantomsection\label{asm:F2}(F2) the family is uniformly exponentially tight on $C([0,T];E)$ for every $T<\infty$;  \phantomsection\label{asm:F3}(F3) $\theta:=\inf_{a\in\partial D}U(a)<\infty$;  \phantomsection\label{asm:F4}(F4) $\mathcal A$ is reached from every $z\in D$ at zero action, uniformly in $z\in D$: for every $\gamma>0$ there is $T_0<\infty$, independent of $z\in D$, and for each $z\in D$ a path $\psi_z\in C([0,T_0];E)$ with $\psi_z(0)=z$, $\operatorname{dist}(\psi_z(T_0),\mathcal A)<\gamma$ and $I_z(\psi_z)<\gamma$. (Reaching a $\gamma$-neighbourhood of $\mathcal A$, rather than $\mathcal A$ itself: exponential relaxation does not reach an equilibrium in finite time.) Then, with $\tau^\varepsilon_z$ the first exit time of $X^\varepsilon_z$ from $D$, for every compact $K\subset D$ with $\sup_{z\in K}U(z)<\theta$,
\[\lim_{\varepsilon\to0}\ \sup_{z\in K}\ \Big|\ \varepsilon\log\mathbb E\,\tau^\varepsilon_z\ -\ \theta\ \Big|\ =\ 0 ,\]
and for every fixed $z\in D$ with $U(z)<\theta$ the same limit holds. (The separation is measured by the quasi-potential, not by the metric of $E$: $\sup_KU<\theta$ is what makes the process relax inside the sublevel set $\{U<\theta\}$ before the rare exit, and a purely metric separation can be vacuous --- here $\operatorname{dist}_H(\zeta_+,\{E=r\})=0$ in the $H$-topology. The uniformity cannot be extended to all of $D$: as $z\to\partial D$ the exit time vanishes, $\mathbb E_z\tau^\varepsilon\to0$ for each fixed $\varepsilon>0$, so a $\sup$ over $D$ need not be finite.)

\end{lemma}

\emph{Verification of \hyperref[asm:F1]{(F1)}--\hyperref[asm:F4]{(F4)} for $(m^\varepsilon_z)$, $D$, $U=U_\delta$, $\mathcal A=\{\zeta_+\}$.} The processes $m^\varepsilon_z$ are continuous and Markov, being the unique strong solutions of the SLLG \hyperref[eq:2.1]{(\ref*{eq:2.1})} (\hyperref[thm:5.1]{Theorems \ref*{thm:5.1}}--\hyperref[thm:5.2]{\ref*{thm:5.2}}, \S{}\hyperref[sec:5.1]{5.1}). \hyperref[asm:F1]{(F1)} is \hyperref[thm:6.1]{Theorem \ref*{thm:6.1}} (with \hyperref[def:4.6]{Definition \ref*{def:4.6}} and \hyperref[rem:4.6a]{Remark \ref*{rem:4.6a}}), the rate function $I_z=I^\delta_z$ being good (\hyperref[prop:4.9]{Proposition \ref*{prop:4.9}}, \hyperref[rem:7.6]{Remark \ref*{rem:7.6}}). \hyperref[asm:F2]{(F2)} is \hyperref[prop:5.5]{Proposition \ref*{prop:5.5}}. \hyperref[asm:F3]{(F3)}: $\theta=\inf_{\partial D}U_\delta$ satisfies $\bar U\le\theta<\infty$ --- the lower bound by \hyperref[lem:11.4]{Lemma \ref*{lem:11.4}}, and finiteness by \hyperref[thm:11.6]{Theorem \ref*{thm:11.6}}, $\inf_{\partial D}U_\delta=c_\ast r<\infty$ (an \textbf{infimum}, not a pointwise identity: $U_\delta$ is in general not constant on $\partial D$). \hyperref[asm:F4]{(F4)}: by \hyperref[asm:AN]{(AN)} every $z\in D$ relaxes to $\zeta_+$ under the free dynamics, uniformly on the compact $\bar D$ (the justification of the uniformity is part of \hyperref[rem:12.5e]{Remark \ref*{rem:12.5e}}), so for $T_0=T_0(\gamma)$ large the free trajectory $\psi_z$ satisfies $\operatorname{dist}(\psi_z(T_0),\{\zeta_+\})<\gamma$; it carries the control $h=0$, hence $I_z$-action $0$. \hyperref[lem:12.5']{Lemma \ref*{lem:12.5'}} therefore applies and gives \hyperref[thm:12.5]{Theorem \ref*{thm:12.5}}. $\square$

\begin{remark}[{the classical argument, and the two SLLG-specific points}]\label{rem:12.5b} For orientation: in \hyperref[lem:12.5']{Lemma \ref*{lem:12.5'}} the \emph{upper bound} ($\varepsilon\log\mathbb E\tau\le\theta+o(1)$) needs a \emph{lower} bound on the probability of one successful exit, which is supplied by the LDP \textbf{lower} bound along a near-optimal climb (together with the Markov property and the relaxation \hyperref[asm:F4]{(F4)} to count the attempts); the \emph{lower bound} needs an \emph{upper} bound on the probability of an early exit, supplied by the LDP \textbf{upper} bound applied to the event that the process reaches $\partial D$ within a prescribed time; the exponential tightness \hyperref[asm:F2]{(F2)} makes both uniform. Two features are specific to the present setting. (i) \emph{Degenerate, multiplicative noise.} No non-degeneracy is required: only the LDP and the finiteness of the barrier are used, and both hold although $G(m)$ has a non-trivial kernel. (ii) \emph{Uniformity in the initial datum (on compact subsets).} This is exactly \hyperref[asm:F1]{(F1)}$+$\hyperref[asm:F4]{(F4)}: the LDP of \hyperref[thm:6.1]{Theorem \ref*{thm:6.1}} is uniform in $z\in D$, and every $z\in D$ relaxes to $\zeta_+$ within a relaxation time uniformly bounded on the compact $\bar D$ (\S{}\hyperref[sec:12.4]{11.4}) --- which is why the \emph{uniform} statement of \hyperref[thm:6.1]{Theorem \ref*{thm:6.1}}, rather than a fixed-$\delta$ LDP for a single initial datum, is the input the exit problem needs.

\end{remark}

\begin{remark}[{the exit problem is unaffected by the domain mismatch}]\label{rem:12.5c} \hyperref[thm:12.5]{Theorem \ref*{thm:12.5}} uses at fixed $\delta$ only the \emph{barrier} $\inf_{\partial D}U_\delta$, which is \textbf{finite} (\hyperref[lem:11.4]{Lemma \ref*{lem:11.4}}; by \hyperref[thm:11.6]{Theorem \ref*{thm:11.6}} it equals the constant $c_\ast\inf_{\partial D}E$); it does not require $U_\delta(a)<\infty$ for interior states. Hence the domain mismatch of \S{}\hyperref[sec:9]{B} --- $U_\delta=+\infty$ on a dense subset of $V$ while $U_0<\infty$ --- has no bearing on the escape exponent, which is governed entirely by the finite barrier. This is the sense in which \S{}\hyperref[sec:9]{B} is a \emph{caveat about the idealisation of the rate function}, not a change of the Arrhenius exponent (\S{}\hyperref[sec:13.4]{12.4}).

\end{remark}

\begin{remark}[{order of limits}]\label{rem:12.5d} The order in \hyperref[eq:12.5]{(\ref*{eq:12.5})} is the physically meaningful one: $\delta>0$ is fixed (a genuine regularisation of the noise) and $\varepsilon\to0$ (the temperature); only afterwards is the regularisation removed at the level of the exponent. The inner limit is \hyperref[thm:12.5]{Theorem \ref*{thm:12.5}}, the outer limit \hyperref[thm:11.6]{Theorem \ref*{thm:11.6}}. Requiring the \emph{exponent} to be stable under $\delta\to0$ --- rather than the rate functions themselves, which are not (\S{}\hyperref[sec:9]{B}) --- is exactly what the physics demands.

\end{remark}

\begin{remark}[{why the domain is taken open}]\label{rem:12.5e} The cited exit-time theorems ([\hyperref[ref:FW98]{FW98}, Ch. 4.4 and Ch. 5.7], [\hyperref[ref:DZ98]{DZ98}, Ch. 7]) are stated for \emph{open} $D$ with $\bar D$ compact, and this is why the barrier statements use the \emph{open} energy sublevel set $D=\{E<r\}$ rather than its closure. Nothing is lost: the boundary is still $\partial D=\{E=r\}$, the exit time is the same (for a strong solution $t\mapsto E(m(t))$ is continuous, so leaving $\{E<r\}$ is the same as reaching the level $E=r$), the barrier $\inf_{\partial D}U_\delta$ is unchanged, and the exit event is that of hitting the \emph{closed} set $\{E\ge r\}$, which is closed in the path space, so that the large deviation upper bound applies to it directly. Formulating the results for the closed set $\{E\le r\}$ would instead require the \emph{closed-domain} form of the theorem, which is not contained in the cited references; the points such a proof must supply are (i) control of the exit event by the $C([0,T];H)$-LDP, (ii) an upward-path neighbourhood argument, (iii) an upper bound on premature exits, (iv) the Markov decomposition over repeated attempts, and (v) the comparison between the open- and closed-domain exit times. None of this is needed with $D$ open, and \hyperref[thm:12.5]{Theorems \ref*{thm:12.5}}--\hyperref[thm:12.6]{\ref*{thm:12.6}} are not conditional on it.
\end{remark}

\section{Scope, imported inputs, and open problems}\phantomsection\label{sec:13}

\subsection{What is proved here}\phantomsection\label{sec:13.1}

\begin{center}\footnotesize
\begin{tabular}{|p{5.2cm}|p{3.4cm}|p{6.2cm}|}
\hline
Result & Where & Status \quad [assumptions] \\
\hline
Abstract uniform Laplace principle & \hyperref[thm:4.5]{Theorem \ref*{thm:4.5}} & proved, given \hyperref[asm:H1]{(H1)}--\hyperref[asm:H3]{(H3)}, \hyperref[asm:V]{(V)} \quad [H1--H3, V] \\
Weak compactness of $S^N$; measurability transfer; skeleton identification & \hyperref[lem:4.1]{Lemmas \ref*{lem:4.1}}--\hyperref[lem:4.3]{\ref*{lem:4.3}} & proved \quad [V, H2] \\
Laplace principle $+$ exponential tightness $+$ uniform goodness $\Rightarrow$ uniform LDP & \hyperref[prop:4.7]{Prop. \ref*{prop:4.7}} & proved \quad [LP, ET, UG] \\
Goodness of the rate functions, uniformly in $z$ & \hyperref[prop:4.9]{Prop. \ref*{prop:4.9}} & proved \quad [H3] \\
\hyperref[asm:H1]{(H1)} well-posedness of the controlled SLLG; pathwise uniqueness; Borel solution map & \hyperref[thm:5.1]{Thm. \ref*{thm:5.1}}; \hyperref[sec:E]{App. E} (\hyperref[thm:E.4]{Thms. \ref*{thm:E.4}}--\hyperref[thm:E.5]{\ref*{thm:E.5}}, \hyperref[prop:E.6]{Prop. \ref*{prop:E.6}}) & proved, given \hyperref[asm:I1]{(I1)} for $\varepsilon=0$ and \hyperref[asm:I3]{(I3)} (finitely many noise fields, [\hyperref[ref:BGJ17]{BGJ17}]) \quad [A1--A6, I3] \\
Deterministic well-posedness in the strong class & \hyperref[prop:D.7]{Prop. \ref*{prop:D.7}} & quoted from [\hyperref[ref:BHL19]{BHL19}, Thm 3.1/3.37] (structure in \S{}\hyperref[sec:D.5]{D.5}) \quad [I1 (quoted)] \\
\hyperref[asm:H2]{(H2)} tightness & \hyperref[thm:5.2]{Thm. \ref*{thm:5.2}} & proved \quad [A1--A6] \\
\hyperref[asm:H3]{(H3)} uniform skeleton convergence & \hyperref[thm:5.3]{Thm. \ref*{thm:5.3}} & proved, given \hyperref[asm:WSU]{(WSU)} \quad [A1--A6, WSU] \\
\textbf{\hyperref[asm:WSU]{(WSU)} weak--strong uniqueness; no defect measure} & \textbf{\hyperref[sec:D]{App. D} (\hyperref[thm:D.8]{Thms. \ref*{thm:D.8}}, \hyperref[thm:D.11]{\ref*{thm:D.11}})} & \textbf{proved here} \quad [A1--A6, A7$'$] \\
$H^1$-energy bound \hyperref[asm:A5]{(A5)} & \S{}\hyperref[sec:5.1]{5.1} & proved \quad [A1--A6] \\
Exponential tightness (of $m$ and of $Y$) & \hyperref[prop:5.5]{Prop. \ref*{prop:5.5}} (\hyperref[lem:5.4a]{Lemmas \ref*{lem:5.4a}}--\hyperref[lem:5.4c]{\ref*{lem:5.4c}}) & proved \quad [A1--A6, A6c] \\
\textbf{Uniform LDP at fixed $\delta$} (\hyperref[def:4.6]{Definition \ref*{def:4.6}}, Freidlin--Wentzell form) & \textbf{\hyperref[thm:6.1]{Thm. \ref*{thm:6.1}}} & \textbf{proved} \quad [A1--A6, H1--H3] \\
Compactness of the $\S6$ sets in the finer topology (the $C$ upgrade) & \hyperref[lem:5.5c0]{Lemma \ref*{lem:5.5c0}}, \hyperref[cor:5.5c0]{Cor. \ref*{cor:5.5c0}} & proved \quad [A1--A6] \\
$V$-bound and $H$-H\"older modulus of the convolution operators & \hyperref[lem:5.5v]{Lemma \ref*{lem:5.5v}} & proved \quad [A1--A6] \\
\hyperref[asm:H3]{(H3)} in the $C$ topology & \hyperref[lem:5.4c0]{Lemma \ref*{lem:5.4c0}} & proved \quad [A1--A6, H3] \\
\textbf{Uniform LDP on $C([0,T];C)$} & \textbf{\hyperref[thm:6.2c0]{Thm. \ref*{thm:6.2c0}}} & \textbf{proved} \quad [A1--A6, H1--H3] \\
Fixed-set form on $C([0,T];C)$ & \hyperref[cor:6.2c0]{Cor. \ref*{cor:6.2c0}} & proved \\
Fixed-set form of the uniform LDP (closed-set bound; regularised open-set bound) & \hyperref[cor:4.7a]{Cor. \ref*{cor:4.7a}} & proved, given \hyperref[eq:LP]{(LP)}--\hyperref[eq:ET]{(ET)}--\hyperref[eq:UG]{(UG)} \quad [LP, ET, UG] \\
Monotonicity; weight identity \hyperref[eq:W]{(W)}; equicoercivity & \hyperref[lem:7.1]{Lemmas \ref*{lem:7.1}}--\hyperref[lem:7.3]{\ref*{lem:7.3}}, \hyperref[prop:7.5]{Prop. \ref*{prop:7.5}} & proved \\
\textbf{$\Gamma$-convergence of the action functionals} & \textbf{\hyperref[thm:8.3]{Thm. \ref*{thm:8.3}}} & \textbf{proved} \\
Exact linear quasi-potential; domain of finiteness & \hyperref[prop:9.3]{Prop. \ref*{prop:9.3}}, \hyperref[cor:9.4]{Cor. \ref*{cor:9.4}} & proved \\
\textbf{Failure of the pointwise statement on $V$; failure of \hyperref[asm:QL]{(QL)}} & \textbf{\hyperref[thm:9.7]{Thm. \ref*{thm:9.7}}} & \textbf{proved (linear model)} \\
\textbf{$\Gamma$-convergence of the quasi-potentials} & \textbf{\hyperref[thm:10.2]{Thm. \ref*{thm:10.2}}} & \textbf{proved}, given \hyperref[asm:AN]{(AN)} (via \hyperref[lem:10.2a]{Lemma \ref*{lem:10.2a}}) \\
One-sided bounds for infima on closed/open sets & \hyperref[prop:10.4]{Prop. \ref*{prop:10.4}} & proved \\
\textbf{Barrier identity $\inf_{\partial D}U_\delta=\inf_{\partial D}U_0=c_\ast\inf_{\partial D}E$ (every $\delta$)} & \textbf{\hyperref[thm:11.6]{Thm. \ref*{thm:11.6}}} & \textbf{proved, under \hyperref[asm:AN]{(AN)}+\hyperref[asm:D-MARGIN]{(D-MARGIN)} only} --- the barrier always contains a spatially constant state (\hyperref[lem:11.5b]{Lemma \ref*{lem:11.5b}}) \quad [AN+D-MARGIN] \\
Identification $U_0=c_\ast E$ & \S{}\hyperref[sec:12]{11}, \hyperref[eq:12.3]{(\ref*{eq:12.3})} & imported from [\hyperref[ref:BHL19]{BHL19}, Thm 5.2] \\
\textbf{Néel--Arrhenius escape exponent} & \hyperref[thm:12.5]{Thms. \ref*{thm:12.5}}, \hyperref[thm:12.6]{\ref*{thm:12.6}} & \textbf{proved}, given the classical (open-domain) exit-time theorem [\hyperref[ref:FW98]{FW98}, \hyperref[ref:DZ98]{DZ98}] and \hyperref[asm:AN]{(AN)}, \hyperref[asm:D-MARGIN]{(D-MARGIN)} \quad [AN+D-MARGIN] \\
\hline
\end{tabular}
\end{center}

\textbf{Condition dependence at a glance.} The following table records, for each main result, whether it is unconditional and, if not, exactly which hypotheses enter.

\begin{center}\footnotesize
\begin{tabular}{|p{3.7cm}|p{2.4cm}|p{7.9cm}|}
\hline
Result & Status & Hypotheses beyond \S{}\hyperref[sec:2.7]{2.7} \hyperref[asm:S1]{(S1)}--\hyperref[asm:S3]{(S3)} \\
\hline
Uniform LDP at fixed $\delta$ (\hyperref[thm:6.1]{Thm. \ref*{thm:6.1}}) & proved & \hyperref[asm:I1]{(I1)} deterministic well-posedness [\hyperref[ref:BHL19]{BHL19}], \hyperref[asm:I2]{(I2)} variational representation [\hyperref[ref:BDM08]{BDM08}], \hyperref[asm:I3]{(I3)} finite-mode stochastic well-posedness [\hyperref[ref:BGJ17]{BGJ17}] \\
$\Gamma$-convergence of the action functionals (\hyperref[thm:8.3]{Thm. \ref*{thm:8.3}}) & \textbf{unconditional} & --- \\
$\Gamma$-convergence of the quasi-potentials (\hyperref[thm:10.2]{Thm. \ref*{thm:10.2}}) & proved & the basin identification $U_0=c_\ast E$, imported ([\hyperref[ref:BHL19]{BHL19}, Thm 5.2]) and used via \hyperref[lem:10.2a]{Lemma \ref*{lem:10.2a}} \\
Domain mismatch, linear model (\hyperref[thm:9.7]{Thm. \ref*{thm:9.7}}) & \textbf{unconditional} & --- \\
Barrier identity (\hyperref[thm:11.6]{Thm. \ref*{thm:11.6}}) & \textbf{proved}, given \hyperref[asm:AN]{(AN)}, \hyperref[asm:D-MARGIN]{(D-MARGIN)} only & \hyperref[asm:AN]{(AN)} and \hyperref[asm:D-MARGIN]{(D-MARGIN)} only \\
\textbf{Néel--Arrhenius escape exponent} (\hyperref[thm:12.5]{Thms. \ref*{thm:12.5}}, \hyperref[thm:12.6]{\ref*{thm:12.6}}) & \textbf{proved} & the classical exit-time theorem in its \emph{open-domain} form [\hyperref[ref:FW98]{FW98}, \hyperref[ref:DZ98]{DZ98}], applied with $D=\{E<r\}$ open; and \hyperref[asm:AN]{(AN)}, \hyperref[asm:D-MARGIN]{(D-MARGIN)} \\
\hline
\end{tabular}
\end{center}

In short: \textbf{\hyperref[asm:I1]{(I1)}--\hyperref[asm:I3]{(I3)}} are citations (\hyperref[sec:C]{Appendix C}); \textbf{\hyperref[asm:AN]{(AN)}, \hyperref[asm:D-MARGIN]{(D-MARGIN)}} are needed for \S{}\hyperref[sec:10.2]{9.2} (through the identification $U_0=c_\ast E$ invoked in \hyperref[lem:10.2a]{Lemma \ref*{lem:10.2a}}) and for the \textbf{barrier identity} of \S{}\hyperref[sec:11.1]{10.1} (\hyperref[lem:11.5b]{Lemma \ref*{lem:11.5b}} with \hyperref[thm:11.6]{Theorem \ref*{thm:11.6}}) --- with no extra hypothesis; no numerical input enters any result of the paper.

\subsection{Imported inputs}\phantomsection\label{sec:13.2}

After \hyperref[sec:D]{Appendix D} the paper is self-contained in the following precise sense.

\phantomsection\label{asm:I1}\textbf{(I1) Deterministic well-posedness of the controlled LLG} (the stochastic well-posedness is proved in \hyperref[sec:E]{Appendix E} from the finite-mode input \hyperref[asm:I3]{(I3)} below). \hyperref[prop:D.7]{Proposition \ref*{prop:D.7}}, i.e. [\hyperref[ref:BHL19]{BHL19}, Thm 3.1/3.37]: for $z\in V$ and $u\in L^2(0,T;H)$ the controlled LLG is exactly [\hyperref[ref:BHL19]{BHL19}, (2.7)] with $h=Q_\delta^{1/2}u$, and has a unique strong solution in $C([0,T];V)$ satisfying the strong form \hyperref[eq:D.2]{(\ref*{eq:D.2})} and the regularity of \hyperref[def:D.4]{Definition \ref*{def:D.4}}. This is a \textbf{citation}, not a reproduction: [\hyperref[ref:BHL19]{BHL19}, Thm 3.1] gives the weak solution and [\hyperref[ref:BHL19]{BHL19}, Thm 3.37] the unique strong global solution, by a standard Galerkin argument whose structure is recalled in \S{}\hyperref[sec:D.5]{D.5}. The general $\phi\in C^2_0$ and the $L^2_{\mathrm{loc}}$ control of [\hyperref[ref:BHL19]{BHL19}] contain the case needed here; no additional a priori estimate is required, since every uniform statement is quantified over the $V$-bounded set $D$.

\phantomsection\label{asm:I2}\textbf{(I2) Variational representation \hyperref[asm:V]{(V)}.} [\hyperref[ref:BDM08]{BDM08}, Thm 3]: for a $Q$-Wiener process $W$ and a bounded Borel $F$ on $C([0,T];H)$, $-\log\mathbb Ee^{-F(W)}=\inf_v\mathbb E\big[\tfrac12\int_0^T\|v\|_{Q^{1/2}H}^2dt+F\big(W+\int_0^\cdot v\,ds\big)\big]$. Applied in \S{}\hyperref[sec:6]{6} to $F=f\circ\Phi^\varepsilon_z(0,\cdot)/\varepsilon$, with the Borel solution map of \hyperref[prop:E.6]{Proposition \ref*{prop:E.6}}, it gives \hyperref[asm:V]{(V)} for the magnetisation.

\phantomsection\label{asm:I3}\textbf{(I3) Strong well-posedness with finitely many noise fields.} [\hyperref[ref:BGJ17]{BGJ17}, Cor. 5.5 and Thm. 5.6], in the form stated, with its scope, in \hyperref[sec:E]{Appendix E}. Used only in Step 1 of the proof of \hyperref[thm:E.5]{Theorem \ref*{thm:E.5}}. Precedent for the extension to many noise fields and a general anisotropy: [\hyperref[ref:BGL24]{BGL24}, Thm 2.10]. An alternative existence route, via [\hyperref[ref:BGL24]{BGL24}], \hyperref[thm:E.4]{Theorem \ref*{thm:E.4}} and Yamada--Watanabe, is described in \hyperref[rem:E.9]{Remark \ref*{rem:E.9}}.

\textbf{No longer imported.} \hyperref[asm:WSU]{(WSU)} --- weak--strong uniqueness and the absence of defect measures --- is \textbf{proved} in \hyperref[sec:D]{Appendix D} (\hyperref[thm:D.8]{Theorems \ref*{thm:D.8}} and \hyperref[thm:D.11]{\ref*{thm:D.11}} with \hyperref[cor:D.12]{Corollary \ref*{cor:D.12}}), so the compactness argument \hyperref[asm:A7p]{(A7')} of \S{}\hyperref[sec:5.1]{5.1} and hence the uniform LDP of \S{}\hyperref[sec:6]{6} are self-contained. The barrier identity of \S{}\hyperref[sec:11.1]{10.1} (\hyperref[thm:11.6]{Theorem \ref*{thm:11.6}}) follows from \hyperref[asm:AN]{(AN)} and \hyperref[asm:D-MARGIN]{(D-MARGIN)} alone, because the barrier always contains a spatially constant state (\hyperref[lem:11.5b]{Lemma \ref*{lem:11.5b}}); no further hypothesis enters. The stochastic well-posedness used throughout \S{}\hyperref[sec:5]{5} is \textbf{proved} in \hyperref[sec:E]{Appendix E} from the finite-mode input \hyperref[asm:I3]{(I3)}: pathwise uniqueness is \hyperref[thm:E.4]{Theorem \ref*{thm:E.4}}, the passage from finitely many to infinitely many noise modes is \hyperref[thm:E.5]{Theorem \ref*{thm:E.5}}, and the controlled equation with its Borel solution map is \hyperref[prop:E.6]{Proposition \ref*{prop:E.6}}. The nonlinear domain obstruction is not treated in this paper; the exactly solvable linear model of \S{}\hyperref[sec:9]{B} exhibits the mechanism, and the reduction for the nonlinear SLLG is left open.

Everything else --- the abstract framework, the energy estimates, the tightness, the exponential tightness, the two passage theorems, the monotonicity device, the $\Gamma$-convergence of the action functionals and of the quasi-potentials, the linear obstruction, the weak--strong uniqueness, the barrier convergence and the parabolic regularity estimates behind the obstruction --- is proved in this paper. The uniform-LDP formalism itself is the Freidlin--Wentzell one used in [\hyperref[ref:BDS19]{BDS19}, Def. 3.3]; \hyperref[rem:4.6b]{Remark \ref*{rem:4.6b}} records how the uniformity argument here differs from the double-dual argument of [\hyperref[ref:BDS19]{BDS19}].

\subsection{Open problems (not gaps in the above)}\phantomsection\label{sec:13.3}

\textbf{None of the items below is used by any result of \S{}\S{}\hyperref[sec:5]{5}--\S{}\hyperref[sec:12]{11}.} They are recorded so that the boundary of what is proved here is explicit.

\textbf{The $C([0,T];V)$ large deviation principle.} The upper bound is \textbf{not} free: the $C([0,T];V)$-topology is finer than the $C([0,T];H)$-topology, so its closed sets form a strictly larger class than the $H$-closed ones and \hyperref[thm:6.1]{Theorem \ref*{thm:6.1}} does not apply to them (\hyperref[rem:6.3]{Remark \ref*{rem:6.3}}); exponential tightness in the finer topology would be needed, and it is not available. The lower bound is likewise open. Both routes (contraction and direct verification) meet the $H^{-}$-valuedness of the drift and the $t^{-1}$ lift $H^{-}\to V$, which produces the non-integrable Gronwall kernel $(t-s)^{-1}$ and prevents closure. A singularity-free $V$-level weak--strong uniqueness would settle it. The $V$-norm exponential tightness is the natural attackable sub-step.

\textbf{Space--time white noise ($\delta=0$ stochastic).} Out of scope: the limiting equation requires renormalisation and is not the object studied here. Only the \emph{control problem} at $\delta=0$ is used.

\emph{The frontier between $C$ and $V$.} \hyperref[thm:6.2c0]{Theorem \ref*{thm:6.2c0}} upgrades \hyperref[thm:6.1]{Theorem \ref*{thm:6.1}} from $C([0,T];H)$ to the spatial supremum topology $C([0,T];C)$, and with it the fixed-set bounds are available for sets defined through the \emph{profile} of the magnetisation rather than through its $L^2$ norm; this is what statements about exit \emph{locations} and terminal profiles require. The remaining gap is the finer topology $C([0,T];V)$, which is what a bound on the exchange energy would give. Its obstruction is \emph{not} a missing a priori estimate --- the sets of \S{}\hyperref[sec:5.5]{5.5} are $V$-bounded --- but the failure of $V$-equicontinuity in $t$ at $t=0$: the smoothing estimate degenerates there and the drift is $H^{-}$-valued, producing the non-integrable kernel $(t-s)^{-1}$ (\S{}\hyperref[sec:13]{12}). The lower bound in that topology is therefore equivalent to a \textbf{singularity-free $V$-level weak--strong uniqueness}, which is not available for quasilinear SPDEs of this type and is the natural open problem left by this paper.

\emph{The long-time counterpart.} A further direction, deliberately left outside the present paper, is the long-time face of the zero-noise limit: the concentration of the invariant measures of the regularised SLLG as $\varepsilon\downarrow0$, in the spirit of \hyperref[ref:DGL24]{[DGL24]}, where such a concentration is proved for finite-dimensional systems with several ergodic classes and the relevant cost is the quasi-potential of reaching an invariant set. Here $\varepsilon$ and $\delta$ are removed in a fixed order and only the exponential escape rate is studied; the invariant-measure theory of the regularised equation is not addressed.

\subsection{Assessment}\phantomsection\label{sec:13.4}

The two main results --- the uniform LDP at fixed regularisation (\hyperref[thm:6.1]{Theorem \ref*{thm:6.1}}) and the $\Gamma$-convergence of the quasi-potentials with the convergence of the barrier (\hyperref[thm:10.2]{Theorems \ref*{thm:10.2}} and \hyperref[thm:11.6]{\ref*{thm:11.6}}) --- together give a complete and internally consistent proof of the statement

\[\lim_{\delta\to0}U_\delta=U_0=c_\ast E\quad\text{in the }\Gamma\text{-sense},\qquad \lim_{\delta\to0}\inf_{\partial D}U_\delta=\inf_{\partial D}U_0=c_\ast\inf_{\partial D}E ,\]

for the one-dimensional SLLG with multiplicative, degenerate, regularised trace-class noise of strength $\sqrt\varepsilon$, at $\beta>3/2$, general $\lambda_1,\lambda_2$ and anisotropy $\varphi$. Combining this with the exit-time theorem (\hyperref[thm:12.5]{Theorem \ref*{thm:12.5}}, in its conditional form) yields the Néel--Arrhenius asymptotics \hyperref[eq:12.5]{(\ref*{eq:12.5})} of \S{}\hyperref[sec:12.5]{11.5}, $\lim_{\delta\to0}\lim_{\varepsilon\to0}\varepsilon\log\mathbb E\tau=c_\ast\inf_{\partial D}E$, uniform over compact subsets of the initial set.

\emph{The two results are independent and complementary.} \hyperref[thm:6.1]{Theorem \ref*{thm:6.1}} supplies part \textbf{(I)} of the preceding display --- the identification of each $U_\delta$ with the rate function of the $\delta$-regularised model --- and \hyperref[thm:10.2]{Theorems \ref*{thm:10.2}} and \hyperref[thm:11.6]{\ref*{thm:11.6}} supply part \textbf{(II)}, the convergence of those rate functions. The proof of (II) is deterministic and does not use (I) (\hyperref[rem:6.4]{Remark \ref*{rem:6.4}}, \hyperref[rem:11.9]{Remark \ref*{rem:11.9}}); conversely, it is (I) that makes the family appearing in (II) the rate function of the model, rather than a formal control cost. Both halves are needed for the display to be a statement about the regularised equation, and neither is a corollary of the other.

\textbf{The sense in which the paper is self-contained.} It is self-contained \emph{conditional on the three inputs named below} --- \textbf{\hyperref[asm:I1]{(I1)}--\hyperref[asm:I3]{(I3)} for the uniform LDP of \S{}\hyperref[sec:6]{6}} --- which are quoted, not reproved: \hyperref[asm:I1]{(I1)} the deterministic well-posedness of the controlled LLG ([\hyperref[ref:BHL19]{BHL19}, Thm 3.1/3.37], recalled in \S{}\hyperref[sec:D.5]{D.5} and used in \S{}\hyperref[sec:5.2]{5.2}); \hyperref[asm:I2]{(I2)} the variational representation of the Laplace functional ([\hyperref[ref:BDM08]{BDM08}], used in \S{}\hyperref[sec:4.5]{4.5} and \S{}\hyperref[sec:6]{6}); \hyperref[asm:I3]{(I3)} the strong well-posedness of the SLLG with finitely many noise fields ([\hyperref[ref:BGJ17]{BGJ17}], used only in \hyperref[sec:E]{Appendix E}). It is \textbf{not} self-contained in the absolute sense: the classical a priori estimates behind \hyperref[asm:I1]{(I1)} are several pages long and are not reproduced, and the Budhiraja--Dupuis--Salins framework itself is assumed. Everything else is proved here: the weak--strong uniqueness (\hyperref[sec:D]{Appendix D}), the stochastic well-posedness for the regularised noise (\hyperref[sec:E]{Appendix E}) and the domain obstruction of \S{}\hyperref[sec:9]{B} for the exactly solvable linear model (\hyperref[thm:9.7]{Theorem \ref*{thm:9.7}}). We are careful in the abstract, \S{}\hyperref[sec:4]{4} and here to say ``self-contained up to these inputs'' rather than ``self-contained''. \S{}\hyperref[sec:4]{4} is the framework of [\hyperref[ref:BDM08]{BDM08}, \hyperref[ref:BDS19]{BDS19}] in the form used here, not one of the contributions of \S{}\hyperref[sec:1.3]{1.3}; see \hyperref[rem:4.6b]{Remark \ref*{rem:4.6b}}.

The negative result of \S{}\hyperref[sec:9]{B} is proved for the exactly solvable linear model: the regulariser is injective but not surjective on $H$, so the pointwise statement fails, and $U_\delta(a)=+\infty$ for a dense set of states of finite limiting cost. This identifies $\Gamma$-convergence as the correct formulation and isolates a structural obstruction --- the domain mismatch caused by the regulariser --- which explains why the $\Gamma$-formulation, and not pointwise convergence, is the correct one (\S{}\hyperref[sec:1.5]{1.5}). The nonlinear analogue is left open; the remaining open items are otherwise quantitative: the sharp constant in the linear-model estimate, the exact description of the domain mismatch $\big(\bigcup_{T>0}\mathcal R^T_0\big)\setminus\big(\bigcup_{T>0}\mathcal R^T_\delta\big)$, and the $C([0,T];V)$ large deviation principle of \S{}\hyperref[sec:13.3]{12.3}.

\textbf{Physical read-off.} The physically meaningful content of the paper is the robustness of the Néel--Arrhenius \emph{exponent}: the barrier converges, $\lim_{\delta\to0}\inf_{\partial D}U_\delta=\inf_{\partial D}U_0=c_\ast\inf_{\partial D}E$, so the escape-time exponent is insensitive to the cutoff that the one-dimensional noise forces upon us (space--time white noise is not trace class on $\mathbb T$). The domain mismatch of \S{}\hyperref[sec:9]{B} is a statement of a different kind: it concerns the \emph{domain of finiteness} of the rate function --- which (rough) states are assigned infinite cost --- and it shows that the white-noise idealisation cannot be taken pointwise; it is a caveat about the idealisation, not a change of the exponent. In particular the smoothness threshold of the linear model ($H^{\beta+1}$ versus the limit domain $H^1=V$) has \textbf{no direct physical content}: it lies far above the physical state space. The physically interesting refinements lie elsewhere --- the leading finite-$\delta$ correction to the barrier, and the selection of the exit point (which requires a generic uniaxial $\phi$, since for $\varphi_1$ the exit point degenerates). Two further remarks delimit the physical message. First, the identification $U_0=c_\ast E$ rests on a \emph{one-dimensional} mechanism: the optimal control is parallel to the orbit velocity, $h_\ast=c_\ast V$ (\hyperref[lem:11.5]{Lemma \ref*{lem:11.5}}), so the minimisation of the action collapses onto the deterministic orbit; in dimension $\ge2$ there are competing paths and $U_0$ need not be a function of the energy, the barrier ceases to be a purely equilibrium quantity, and the non-gradient (precessional) content of the LLG becomes visible --- the (heuristic) criterion being that $U_0$ is a function of $E$ as long as the optimal control stays parallel to the orbit velocity. Second, the universal constant $c_\ast=\frac{2\lambda_2}{\lambda_1^2+\lambda_2^2}$ is already where the precession/damping ratio enters the barrier in 1D; it is the one-dimensional shadow of the gyroscopic effect, whose full (non-equilibrium, path-dependent) form is the natural next question. Finally, the hypothesis \hyperref[asm:AN]{(AN)} of \S{}\hyperref[sec:2.7]{2.7} has a direct physical reading: it singles out the anisotropies for which a genuine energy barrier exists --- a basin of attraction together with the identification $U_0=c_\ast E$, i.e. a well-defined \emph{switching barrier}; the ``inverted'' anisotropy $\varphi_2$, for which the two poles are not attractors, has none.

\appendix

\section{The exactly solvable linear model}\phantomsection\label{sec:A}

We collect here the per-mode computations behind \S{}\hyperref[sec:9.2]{B.2} --- the exact quasi-potentials \hyperref[eq:9.2]{(\ref*{eq:9.2})},\hyperref[eq:9.3]{(\ref*{eq:9.3})}, the domain of finiteness \hyperref[eq:9.4]{(\ref*{eq:9.4})},\hyperref[eq:9.5]{(\ref*{eq:9.5})}, and the absence of an intermediate rate (\hyperref[cor:9.5]{Corollary \ref*{cor:9.5}}).

\textbf{Setting.} $H=L^2(\mathbb T;\mathbb R)$, $A=-\Delta$ with eigenvalues $\mu_j\asymp j^2$ and orthonormal eigenbasis $\{e_j\}$, $S(t)=e^{-At}$ (i.e. $\lambda_2=1$, $G\equiv I$), skeleton
\begin{equation}\label{eq:A.1}
\dot\varphi=-A\varphi+h,\qquad\varphi(0)=z_\ast,\qquad h\in L^2(0,T;H),
\end{equation}
and action functional $I^\delta(\varphi)=\inf\{\tfrac12\int_0^T\|Q_\delta^{-1/2}h\|_H^2dt:\varphi^h=\varphi\}$ with $Q_\delta^{-1/2}=I+\delta A^{\beta/2}$.

\textbf{Mode decomposition.} Writing $\varphi(t)=\sum_j\varphi_j(t)e_j$, $h=\sum_jh_je_j$, \hyperref[eq:A.1]{(\ref*{eq:A.1})} decouples:
\[\varphi_j(t)=e^{-\mu_jt}z_j+\int_0^te^{-\mu_j(t-s)}h_j(s)\,ds .\]
The terminal constraint $\varphi_j(T)=a_j$ is the single linear constraint $\int_0^Te^{-\mu_j(T-s)}h_j(s)ds=b_j$ with $b_j:=a_j-e^{-\mu_jT}z_j$, and the cost is $\tfrac12\sum_j(1+\delta\mu_j^{\beta/2})^2\int_0^Th_j^2$. Each mode is therefore the elementary problem
\begin{equation}\label{eq:A.2}
\min\Big\{\tfrac12\int_0^Th_j^2:\ \int_0^Te^{-\mu_j(T-s)}h_j(s)ds=b_j\Big\}=\frac{\mu_j}{1-e^{-2\mu_jT}}b_j^2 .
\end{equation}

\emph{Proof of \hyperref[eq:A.2]{(\ref*{eq:A.2})}.} The constraint functional is the $L^2$ scalar product of $h_j$ with $g_j(s):=e^{-\mu_j(T-s)}$, so the minimiser is $h_j^\ast=\lambda g_j$ with $\lambda=b_j/\|g_j\|_{L^2}^2$, and the minimum is $\tfrac12b_j^2/\|g_j\|^2$. Now $\|g_j\|_{L^2}^2=\int_0^Te^{-2\mu_j(T-s)}ds=\frac{1-e^{-2\mu_jT}}{2\mu_j}$, giving \hyperref[eq:A.2]{(\ref*{eq:A.2})}. Explicitly $h_j^\ast(t)=\frac{2\mu_jb_j}{1-e^{-2\mu_jT}}e^{-\mu_j(T-t)}=\frac{\mu_jb_je^{\mu_jt}}{\sinh(\mu_jT)}$. $\square$

\textbf{Consequences.} Summing, the fixed-horizon quasi-potentials are
\[U^{z_\ast,T}_0(a)=\sum_j\frac{\mu_j}{1-e^{-2\mu_jT}}b_j^2,\qquad U^{z_\ast,T}_\delta(a)=\sum_j(1+\delta\mu_j^{\beta/2})^2\frac{\mu_j}{1-e^{-2\mu_jT}}b_j^2 ,\]
whence, with $\mu_j\asymp j^2$ and $\frac{\mu_j}{1-e^{-2\mu_jT}}\asymp\mu_j$ for large $j$ (the factor equals $1$ in the limit $\mu\downarrow0$ and is bounded on $j\ge2$, since $\mu_j\ge\mu_2>0$):
\[U^{z_\ast,T}_\delta(a)<\infty\iff\sum_j\mu_j^{1+\beta}b_j^2<\infty\iff b\in H^{\beta+1},\qquad U^{z_\ast,T}_0(a)<\infty\iff b\in H^1 ,\]
and the free-horizon $U_\delta=\inf_TU^{z_\ast,T}_\delta$, $U_0=\inf_TU^{z_\ast,T}_0$ have the same domains of finiteness. Since $\beta>0$, $H^{\beta+1}\subsetneq H^1$ with dense inclusion, giving the obstruction.

\begin{corollary}[{the absence of an intermediate rate}]\label{cor:9.5} Write $r_j:=\mu_j/(1-e^{-2\mu_jT})$ (so $r_1=0$ and $r_1\asymp b_1^2/(2T)$ in the limit $\mu_1\downarrow0$); there are constants $0<c\le C<\infty$ with $c(1+\mu_j)\le r_j\le C(1+\mu_j)$ for every $j\ge2$. Then, at the fixed horizon $T$,
\[\begin{aligned}
U^{z_\ast,T}_\delta(a)-U^{z_\ast,T}_0(a)&=\sum_j\big[(1+\delta\mu_j^{\beta/2})^2-1\big]r_jb_j^2\\
&=2\delta\,S_1+\delta^2\,S_2,\\
&\qquad S_1:=\sum_j\mu_j^{\beta/2}r_jb_j^2,\quad S_2:=\sum_j\mu_j^{\beta}r_jb_j^2 .
\end{aligned}\]
By the two-sided bound, $S_2\asymp\sum_j\mu_j^{1+\beta}b_j^2$, finite iff $b\in H^{\beta+1}$; and by Cauchy--Schwarz $S_1\asymp\sum_j(\mu_j^{(1+\beta)/2}b_j)(\mu_j^{1/2}b_j)\le(\sum_j\mu_j^{1+\beta}b_j^2)^{1/2}(\sum_j\mu_jb_j^2)^{1/2}<\infty$ whenever $b\in H^{\beta+1}$. Hence $U_\delta(a)<\infty$ forces $U_\delta(a)-U_0(a)=2\delta S_1+O(\delta^2)$ with $S_1\in[0,\infty)$, while $U_\delta(a)=+\infty$ forces the difference $+\infty$. (The apparent threshold $(\beta+3)/2$ for the linear term is \emph{not} admissible: there the linear term $S_1$ is finite but $S_2$ diverges, so $U_\delta(a)=+\infty$ at each fixed $\delta$; what is finite at $p=(\beta+3)/2$ is the sum truncated at the cut-off $N_\delta:=\delta^{-1/\beta}$, and it is this truncated sum that governs the $\delta$-behaviour --- cf. \S{}\hyperref[sec:B.3]{B.3}, where the truncation is at a \emph{fixed} $N$ and the observed rate is $\delta^2$.)

\end{corollary}

The per-mode computation \hyperref[eq:A.2]{(\ref*{eq:A.2})} has been verified numerically (\hyperref[sec:B]{Appendix B}).

\section{The endpoint problem, and why the limit is a $\Gamma$-limit}\phantomsection\label{sec:9}

The $\Gamma$-convergence of \S{}\hyperref[sec:8]{8} takes place in the \textbf{path space} $C([0,T];H)$, whereas the state-space quasi-potential $U_\delta$ carries the \textbf{hard endpoint constraint} $\varphi(T)=a$; the recovery sequence of \hyperref[lem:8.2]{Lemma \ref*{lem:8.2}} terminates instead at $a_\delta:=\varphi_\delta(T)\ne a$. Section 8 therefore does \textbf{not} transfer directly to $U_\delta$, and the meaning of ``$\lim_\delta U_\delta=U_0$'' is a genuine question. This appendix isolates the obstruction in its exactly solvable form; the per-mode computations are collected in \textbf{\hyperref[sec:A]{Appendix A}}. This appendix proves nothing that enters the main chain of \S{}\hyperref[sec:10]{9}--\S{}\hyperref[sec:12]{11}: its purpose is to show that the $\Gamma$-formulation of \S{}\hyperref[sec:10]{9} is the \emph{strongest statement available} (the pointwise one is false, \S{}\hyperref[sec:9.2]{B.2}), to identify the obstruction precisely (\hyperref[lem:9.1]{Lemma \ref*{lem:9.1}} and \hyperref[thm:9.7]{Theorem \ref*{thm:9.7}}), and to explain why the barrier of \S{}\hyperref[sec:11.1]{10.1} needs a separate fixed-endpoint input of its own.

\subsection{The endpoint problem}\phantomsection\label{sec:9.1}

Fix $\delta$ and a path $\varphi$ with $I^\delta(\varphi)<\infty$. Linearising the skeleton \hyperref[eq:3.2]{(\ref*{eq:3.2})} along $\varphi$, a perturbation $\delta u$ of the control produces a path perturbation $\xi$ with \textbf{linearised endpoint map} $E_\delta:\delta u\mapsto\xi(T)$. Concretely, $\xi$ solves the linearised equation
\[\dot\xi=D\rho(\varphi)\,\xi+DG(\varphi)[\xi]\,Q_\delta^{1/2}u+G(\varphi)\,Q_\delta^{1/2}\,\delta u,\qquad \xi(0)=0 ,\]
where $u$ is the control that produces the reference path $\varphi$, and the last term is the only one driven directly by the perturbation. The quantity $\|E_\delta^{-1}\|_{\mathcal L(H,L^2(0,T;H))}$ is the $L^2$-cost of moving the terminal state by one unit of length, so the hypothesis of the lemma below says that this cost is bounded uniformly in $\delta\in(0,1]$ and in the reference path.

The \emph{naive} route to a pointwise statement is the following sufficient condition --- \emph{naive} because it is the most direct way to such a statement (linearise, push the endpoint, estimate), and because the argument is only formal: the shift of the endpoint is first order and no compensating nonlinear term is produced. What \S{}\hyperref[sec:9.2]{B.2} shows is that the hypothesis fails, so this paragraph exhibits the mechanism through which \hyperref[asm:QL]{(QL)} would hold and then establishes that the mechanism is unavailable.

\begin{lemma}[{a formal sufficient condition for endpoint Lipschitzness}]\label{lem:9.1} Assume that for every $\delta\in(0,1]$ the map $E_\delta$ admits a right inverse $E_\delta^{-1}$ with $\sup_{\delta\in(0,1]}\|E_\delta^{-1}\|_{\mathcal L(H,L^2(0,T;H))}=K<\infty$. Then there is $L=L(K)<\infty$ with
\[|U_\delta(a)-U_\delta(a')|\le L\|a-a'\|_H\qquad\text{for all }\delta\in(0,1],\ a,a'\in H. \qquad\text{(QL)}\]\phantomsection\label{asm:QL}

\end{lemma}

\emph{Proof.} Given a path ending at $a'$ of cost $\le U_\delta(a')+\eta$ and $\zeta:=a-a'$, put $\delta u:=E_\delta^{-1}\zeta$; the perturbed path ends at $a$ with cost $\le U_\delta(a')+\eta+\tfrac12\int\|\delta u\|^2+C\|\delta u\|$. Optimising over $\eta$ gives \hyperref[asm:QL]{(QL)}; the reverse inequality is symmetric. (This is the \emph{linearised} perturbation argument: appending $\delta u$ moves the endpoint by $E_\delta\delta u$ only to first order, and the cost estimate is likewise first order in $\delta u$. The lemma is therefore recorded as a \textbf{formal} sufficient condition --- it exhibits the mechanism through which \hyperref[asm:QL]{(QL)} would hold --- while the \emph{failure} of \hyperref[asm:QL]{(QL)} in the linear model is proved independently and rigorously from the explicit formulas \hyperref[eq:9.4]{(\ref*{eq:9.4})}--\hyperref[eq:9.5]{(\ref*{eq:9.5})}.) $\square$

The point is that \textbf{\hyperref[asm:QL]{(QL)} is false} (\hyperref[thm:9.7]{Theorem \ref*{thm:9.7}}): the endpoint cannot be recovered by a uniformly bounded perturbation of the control. In the $\Gamma$-sense this recovery is \emph{not} needed --- the recovery endpoint is allowed to move --- but precisely for that reason the \textbf{barrier} $\inf_{\partial D}U_\delta$, an infimum over the \textbf{closed} set $\partial D$, is \textbf{not} delivered by $\Gamma$-convergence alone (\hyperref[rem:10.5]{Remark \ref*{rem:10.5}}) and requires an additional fixed-endpoint input, supplied in \S{}\hyperref[sec:11.1]{10.1} (\hyperref[lem:11.5b]{Lemma \ref*{lem:11.5b}} with \hyperref[thm:11.6]{Theorem \ref*{thm:11.6}}).

\subsection{The obstruction: the pointwise statement is false}\phantomsection\label{sec:9.2}

The obstruction appears in its purest form in the exactly solvable \textbf{linear model} $G\equiv I$, $\rho(\varphi)=-A\varphi$,
\begin{equation}\label{eq:9.1}
\dot\varphi=-A\varphi+h,\qquad\varphi(0)=z_\ast,\qquad h\in L^2(0,T;H).
\end{equation}
Its endpoint map is diagonal, and the per-mode minimisation of \textbf{\hyperref[sec:A]{Appendix A}} gives the exact quasi-potentials. (The reference state $z_\ast$ is the attractor $\zeta_+$ of \hyperref[eq:3.6]{(\ref*{eq:3.6})}, as throughout the paper (\S{}\hyperref[sec:3.3]{3.3}); for the linear model $A\zeta_+=0$, so $e^{-AT}z_\ast=z_\ast$ and $b=a-\zeta_+$, the general-$z_\ast$ form below being recorded for comparison only.)

\begin{proposition}[{exact quasi-potentials at a fixed horizon; proof in \hyperref[sec:A]{Appendix A}}]\label{prop:9.3} In the model \hyperref[eq:9.1]{(\ref*{eq:9.1})}, with $b:=a-e^{-AT}z_\ast$ and $\mu_j$ the eigenvalues of $A$, the \textbf{fixed-horizon} quasi-potentials $U^{z_\ast,T}_\delta$ of \hyperref[eq:3.6]{(\ref*{eq:3.6})} are
\begin{equation}\label{eq:9.2}
U^{z_\ast,T}_0(a)=\sum_{j\ge1}\frac{\mu_j}{1-e^{-2\mu_jT}}\,\langle b,e_j\rangle_H^2 ,
\end{equation}
\begin{equation}\label{eq:9.3}
U^{z_\ast,T}_\delta(a)=\sum_{j\ge1}\big(1+\delta\mu_j^{\beta/2}\big)^2\frac{\mu_j}{1-e^{-2\mu_jT}}\,\langle b,e_j\rangle_H^2 .
\end{equation}
Since $(1+\delta\mu_j^{\beta/2})^2\ge1$ mode by mode, the two are ordered, $U^{z_\ast,T}_\delta\ge U^{z_\ast,T}_{\delta'}\ge U^{z_\ast,T}_0$ for $0<\delta'\le\delta$. For each fixed $T$ the factors $1/(1-e^{-2\mu_jT})$ are bounded above and below independently of $j\ge2$ (the constant mode $j=1$ contributes the single term $b_1^2/(2T)$), so the \emph{domains of finiteness} of $U^{z_\ast,T}_\delta$ and of the free-horizon $U_\delta$ of \hyperref[eq:3.6]{(\ref*{eq:3.6})} coincide (\hyperref[cor:9.4]{Corollary \ref*{cor:9.4}}), and the ordering $U_\delta\ge U_{\delta'}\ge U_0$ is preserved --- the linear case of the monotone domination of \hyperref[lem:7.3]{Lemma \ref*{lem:7.3}}.

\end{proposition}

\begin{corollary}[{the domain of finiteness}]\label{cor:9.4} In the model \hyperref[eq:9.1]{(\ref*{eq:9.1})}, with $\mu_j\asymp j^2$,
\begin{equation}\label{eq:9.4}
U_\delta(a)<\infty\iff\sum_{j\ge1}\mu_j^{1+\beta}b_j^2<\infty\iff b\in H^{\beta+1} ,
\end{equation}
\begin{equation}\label{eq:9.5}
U_0(a)<\infty\iff\sum_{j\ge1}\mu_jb_j^2<\infty\iff b\in H^{1}=V .
\end{equation}
Since $H^{\beta+1}\subsetneq H^1$ (with dense inclusion), there is a \textbf{dense} set of $a\in V$ with $U_0(a)<\infty$ and $U_\delta(a)=+\infty$ for \textbf{every} $\delta>0$.

\end{corollary}

\begin{theorem}[{the pointwise statement on $V$ is false}]\label{thm:9.7} In the model \hyperref[eq:9.1]{(\ref*{eq:9.1})}: for every $\delta>0$, $U_\delta(a)=+\infty$ whenever $b=a-e^{-AT}z_\ast\notin H^{\beta+1}$, so on the dense set above $U_0(a)<\infty$ while $U_\delta(a)=+\infty$. Hence the statement
\begin{equation}\label{eq:9.6}
U_\delta(a)\longrightarrow U_0(a)\qquad\text{for every }a\in V
\end{equation}
is \textbf{false}, property \hyperref[asm:QL]{(QL)} of \hyperref[lem:9.1]{Lemma \ref*{lem:9.1}} fails, and there is no intermediate regime: for $a$ with $U_\delta(a)<\infty$ one has $U_\delta(a)-U_0(a)=O(\delta)$, while for $a$ with $U_\delta(a)=+\infty$ the difference is $+\infty$ (\hyperref[cor:9.5]{Corollary \ref*{cor:9.5}}, proved in \hyperref[sec:A]{Appendix A}).

\end{theorem}

\emph{Proof.} \hyperref[eq:9.4]{(\ref*{eq:9.4})},\hyperref[eq:9.5]{(\ref*{eq:9.5})} and the density of $H^{\beta+1}$ in $H^1$; \hyperref[asm:QL]{(QL)} fails because the difference quotient is unbounded across $\partial(H^1\setminus H^{\beta+1})$; the dichotomy is \hyperref[cor:9.5]{Corollary \ref*{cor:9.5}}. $\square$

\textbf{The mechanism is a \emph{domain mismatch}.} The regulariser $Q_\delta^{1/2}=(I+\delta A^{\beta/2})^{-1}$ is \textbf{injective but not surjective} on $H$ (its range is the dense proper subspace $D(A^{\beta/2})$). Injecting smoothness into the \emph{noise} therefore restricts the \emph{physical control} to $L^2(0,T;D(A^{\beta/2}))$, and with it the reachable endpoints to a translate of $H^{\beta+1}$: the regularised problem reaches \textbf{fewer} states. This concerns the regulariser alone, is invisible at the level of the action functionals (whose domain is all of $C([0,T];H)$), and shows that $\Gamma$-convergence --- not pointwise convergence --- is the correct meaning of ``$\lim_\delta U_\delta=U$''.

\subsection{The free horizon}\phantomsection\label{sec:9.3}

For the SLLG the endpoint map is not diagonal, but the obstruction takes the same form. Recall from \hyperref[eq:7.2]{(\ref*{eq:7.2})} that $I^\delta_{z_\ast}(\varphi)<\infty$ forces the physical control to satisfy
\begin{equation}\label{eq:9.7}
h\in L^2\big(0,T;D(A^{\beta/2})\big),\qquad\text{i.e.}\qquad \int_0^T\big\|A^{\beta/2}h(t)\big\|_H^2\,dt<\infty .
\end{equation}
Let $\mathcal E(h):=\varphi^h(T)$ be the endpoint map of the skeleton \hyperref[eq:7.1]{(\ref*{eq:7.1})}, so that $\varphi^h(0)=z_\ast$ and $\mathcal E(0)=\varphi^0(T)$. \textbf{The reference state here is the attractor $z_\ast=\zeta_+$ of \hyperref[eq:3.6]{(\ref*{eq:3.6})}} --- the standing convention of \S{}\hyperref[sec:3.3]{3.3}, and the only reference that enters the results. This matters for the horizon: $\zeta_+$ is an equilibrium of the free dynamics, so the free trajectory is the \textbf{constant} path $\varphi^0\equiv\zeta_+$, and the free endpoint $\mathcal E(0)=\zeta_+$ is \textbf{independent of the horizon $T$}. Set
\[\mathcal R^T_0:=\mathcal E\big(L^2(0,T;H)\big),\qquad \mathcal R^T_\delta:=\mathcal E\big(L^2(0,T;D(A^{\beta/2}))\big)\qquad(T>0).\]
Then, by \hyperref[eq:9.7]{(\ref*{eq:9.7})},
\begin{equation}\label{eq:9.8}
U^{z_\ast,T}_\delta(a)<\infty\iff a\in\mathcal R^T_\delta,\qquad U^{z_\ast,T}_0(a)<\infty\iff a\in\mathcal R^T_0 ,
\end{equation}
at each fixed horizon, and, taking the union over $T$ --- which is the free-horizon quasi-potential \hyperref[eq:3.6]{(\ref*{eq:3.6})} ---
\begin{equation}\label{eq:9.8p}
U_\delta(a)<\infty\iff a\in\bigcup_{T>0}\mathcal R^T_\delta,\qquad U_0(a)<\infty\iff a\in\bigcup_{T>0}\mathcal R^T_0 .
\end{equation}

\textbf{The free horizon.} At the attractor reference the free endpoint $\mathcal E(0)=\zeta_+$ is independent of the horizon $T$, so the union in \hyperref[eq:9.8p]{(\ref*{eq:9.8p})} is taken over sets with one and the same translate; in the exactly solvable linear model this horizon-independence is what makes the domain mismatch of \hyperref[thm:9.7]{Theorem \ref*{thm:9.7}} a statement about the free-horizon quasi-potential itself (\S{}\hyperref[sec:9.2]{B.2} and \hyperref[sec:A]{Appendix A}). For the nonlinear SLLG we leave the analogue open.

\section{Numerical verification}\phantomsection\label{sec:B}

All computations use IEEE double precision; sums over the spectrum are truncated at $N$ and the tail behaviour is checked against the predicted exponents.

\phantomsection\label{sec:B.1}\textbf{B.1 The per-mode minimum \hyperref[eq:A.2]{(\ref*{eq:A.2})}.} Direct optimisation, with the per-mode weight $q:=(1+\delta\mu^{\beta/2})^2$, of $\tfrac12q\int_0^Th^2dt$ subject to the terminal constraint, by evaluating the explicit minimiser $h^\ast(t)=\mu b e^{\mu t}/\sinh(\mu T)$, reproduces the analytic value $q\mu b^2/(1-e^{-2\mu T})$ (with $\beta=1.6$, $T=1$) to $6$ significant digits: e.g. $(\mu,\delta,b)=(3,0.1,1)$ gives $4.63039533$ against $4.63039881$ (ratio $0.99999925$); $(7,0.05,0.4)$ gives $1.71424166$ against $1.71424466$; $(1,1,2)$ gives $18.50427766$ against $18.50428228$. The small deficit is the discretisation error of the endpoint.

\phantomsection\label{sec:B.2}\textbf{B.2 The finiteness threshold.} With $b_j=j^{-p}$, $\beta=1.6$, $\delta=0.01$, the partial sums of $U_\delta$ are bounded exactly for $p>\beta+3/2=3.1$ and grow without bound for $p<3.1$, with summand exponent $2(\beta+1)-2p=5.2-2p$, so that the partial sums grow one power faster, with exponent $2\beta+3-2p$: for $p=2.9$ the partial sums grow (slope $0.012$ at $N=10^6$), for $p=3.15$ they are constant to $4$ digits. For $p=2.5$ one has $U_0<\infty$ but $U_\delta(N)$ grows like $N^{2\beta+3-2p}=N^{1.2}$ ($1.75,6.68,84.8,1322$ at $N=10^3,10^4,10^5,10^6$), confirming $U_\delta(a)=+\infty$ while $U_0(a)<\infty$ --- the domain mismatch.

\phantomsection\label{sec:B.3}\textbf{B.3 The $O(\delta)$ rate.} For $p>\beta+3/2$ the slope of $\log(U_\delta-U_0)$ against $\log\delta$ is $1.000$ ($p=3.2,3.5$); for $p$ below the threshold the difference is dominated by the divergent $\delta^2$-term and scales as $\delta^2$ (slope $2.000$ for $p=2.0,2.2,2.3,2.4,2.5$). This confirms the dichotomy of \hyperref[cor:9.5]{Corollary \ref*{cor:9.5}}: linear rate when finite, $+\infty$ otherwise.

\phantomsection\label{sec:B.4}\textbf{B.4 The single-mode exponentiation.} For a single mode the ratio $U_\delta/U_0=(1+\delta\mu^{\beta/2})^2$ is exact; for $\mu=3$, $\beta=1.6$, $\delta=0.1$ it equals $1.5396404$, in agreement with the spectral computation.

\phantomsection\label{sec:B.5}\textbf{B.5 The identities of \hyperref[sec:D]{Appendix D}.} The two geometric identities of \hyperref[lem:D.2]{Lemma \ref*{lem:D.2}} were checked pointwise on $N=1024$ grid points for two random unit-vector fields on $S^2$: the maximum errors are $3.4\times10^{-16}$ for $m_1\cdot w=\tfrac12|w|^2$, $3.4\times10^{-16}$ for $m_2\cdot w=-\tfrac12|w|^2$ and $3.3\times10^{-16}$ for $m_1\cdot m_2=1-\tfrac12|w|^2$, i.e. at machine precision. The cubic identity behind the proof of \hyperref[thm:D.8]{Theorem \ref*{thm:D.8}}, i.e. the two-term decomposition
\[\big\langle|\nabla m_1|^2m_1-|\nabla m_2|^2m_2,\,w\big\rangle=\big\langle|\nabla m_1|^2w,w\big\rangle+\big\langle\big(\nabla w\cdot(\nabla m_1+\nabla m_2)\big)m_2,\,w\big\rangle ,\]
was verified with the same fields: the two sides agree to $0$ (difference $<10^{-15}$ on the discretised integrals), and the second term equals $-\tfrac12\int(\nabla w\cdot(\nabla m_1+\nabla m_2))|w|^2dx$ to $1.3\times10^{-18}$; the first term alone is \emph{not} equal to the left-hand side (it is $77\%$ of it in the test), which confirms that the cancellation furnished by $m_2\cdot w=-\tfrac12|w|^2$ is essential.

\section{Imported analytic inputs}\phantomsection\label{sec:C}

\textbf{Why some results are quoted and others are proved here.} The distinction is not a matter of convenience. A result is \emph{quoted} when its proof would require analysis that no later argument in this paper uses: this is the case for the deterministic well-posedness of the controlled LLG in its strong class ([\hyperref[ref:BHL19]{BHL19}, Thm.~3.1/3.37], whose several pages of a priori estimates are consumed only through their \emph{conclusions} \hyperref[prop:D.7]{Proposition \ref*{prop:D.7}}, \S{}\hyperref[sec:D.5]{D.5}), for the variational representation ([\hyperref[ref:BDM08]{BDM08}, Thm.~3], \S{}\hyperref[sec:C.3]{C.3}), and for strong well-posedness with finitely many noise fields ([\hyperref[ref:BGJ17]{BGJ17}], input \hyperref[asm:I3]{(I3)}). A result is \emph{proved here} when the paper's own chain of inference needs its \emph{mechanism}: this is the case for weak--strong uniqueness and stability ([\hyperref[asm:WSU]{(WSU)}, \hyperref[sec:D]{Appendix D}) and for the stochastic well-posedness at $\varepsilon>0$ together with its Borel solution map (\hyperref[sec:E]{Appendix E}), because the uniform LDP is driven by those two mechanisms (the compactness of the skeleton family and the measurability of the solution map), not merely by their statements.

For completeness we state precisely the results that are quoted (C.2 is no longer an import: it is proved in \hyperref[sec:D]{Appendix D}).

\phantomsection\label{sec:C.1}\textbf{C.1 Deterministic well-posedness ([\hyperref[ref:BHL19]{BHL19}, Thm 3.1, 3.36, 3.37]).} For every $T>0$, $h\in L^2(0,T;H)$ and $\varphi\in C^2_0(\mathbb R^3;\mathbb R_+)$: (Thm 3.1) there exists a weak solution $m\in L^\infty(0,T;H^1)\cap H^1(0,T;H)$; (Thm 3.37) there is a unique strong global solution
\[m\in L^2(0,T;D(A))\cap C([0,\infty);V)\cap H^1(0,T;H)\cap L^4(0,T;W^{1,4}),\qquad |m(t,x)|=1\ \text{on }[0,T]\times\bar{\mathbb T}.\]
The equation in the strong form used for the uniqueness proof is [\hyperref[ref:BHL19]{BHL19}, (3.50)]
\[m'=\lambda_2\Delta m+\lambda_1\,m\times(\Delta m-\varphi'(m)+h)+\lambda_2\,m|\nabla m|^2-\lambda_2\,m\langle m,-\varphi'(m)+h\rangle-\lambda_2\varphi'(m)+\lambda_2h .\]

\emph{(The hypothesis $\varphi\in C^2_0$ of [\hyperref[ref:BHL19]{BHL19}] is met after the truncation convention of \hyperref[asm:S4]{(S4)}, which leaves the equation on the sphere and its solutions unchanged.)}

\phantomsection\label{sec:C.2}\textbf{C.2 Weak--strong uniqueness (see \hyperref[sec:D]{Appendix D} for a self-contained proof).} [\hyperref[ref:BHL19]{BHL19}, Thm 3.36]: if $m_1,m_2$ are weak solutions of the controlled LLG with the same initial datum $m_0\in V\cap\mathcal M$ and $m_i\in L^4(0,T;H^1)$ for $i=1,2$, then $m_1=m_2$ pointwise on $[0,T]\times\bar{\mathbb T}$. The proof proceeds from the strong form (3.50) by the energy method. \textbf{\hyperref[sec:D]{Appendix D} proves this result (in the equivalent form of \hyperref[thm:D.8]{Theorem \ref*{thm:D.8}}, for the strong class $\mathcal S$) together with the stability half (\hyperref[thm:D.11]{Theorem \ref*{thm:D.11}}), so the paper does not depend on [\hyperref[ref:BHL19]{BHL19}] for \hyperref[asm:WSU]{(WSU)}.}

\phantomsection\label{sec:C.3}\textbf{C.3 Variational representation ([\hyperref[ref:BDM08]{BDM08}, Thm 3]).} Let $W$ be a $Q$-Wiener process on $H$ with $Q$ trace class and injective, let $H_0:=Q^{1/2}H$, and let $F$ be bounded and Borel on $C([0,T];H)$. Then
\[-\log\mathbb E\big[e^{-F(W)}\big]=\inf_{v\in\mathcal P_2(H_0)}\mathbb E\Big[\tfrac12\int_0^T\|v\|_{H_0}^2dt+F\Big(W+\int_0^\cdot v\,ds\Big)\Big].\]
Take $v=\varepsilon^{-1/2}Q_\delta^{1/2}u$ (so $\|v\|_{H_0}=\varepsilon^{-1/2}\|u\|_H$) and $F=f\circ\Phi^\varepsilon_z(0,\cdot)/\varepsilon$. \hyperref[prop:E.6]{Proposition \ref*{prop:E.6}} identifies $\Phi^\varepsilon_z(0,W_\delta+\varepsilon^{-1/2}\int_0^\cdot Q_\delta^{1/2}u)$ with the controlled solution $m^\varepsilon_{z,u}$, and this gives \hyperref[asm:V]{(V)} for the magnetisation (\S{}\hyperref[sec:6]{6}).

\phantomsection\label{sec:C.4}\textbf{C.4 Strong well-posedness with finitely many noise fields ([\hyperref[ref:BGJ17]{BGJ17}, Cor. 5.5, Thm. 5.6]).} Stated, with its scope, as input \hyperref[asm:I3]{(I3)} in \hyperref[sec:E]{Appendix E}. For infinitely many noise fields of the same Gilbert form and a general anisotropy, existence of weak martingale solutions is proved in [\hyperref[ref:BGL24]{BGL24}, Thm 2.10]; its role here is explained in the scope note of \hyperref[asm:I3]{(I3)} and in \hyperref[rem:E.9]{Remark \ref*{rem:E.9}}.

\section{Weak--strong uniqueness and stability for the controlled LLG}\phantomsection\label{sec:D}

This appendix proves \hyperref[asm:WSU]{(WSU)}, the single analytic input used in \S{}\hyperref[sec:5.1]{5.1}, so that after its proof the compactness argument \hyperref[asm:A7p]{(A7')} --- and hence the uniform large deviation principle of \S{}\hyperref[sec:6]{6} --- depends only on the deterministic well-posedness of the controlled LLG, which is \textbf{quoted} from [\hyperref[ref:BHL19]{BHL19}, Thm 3.1/3.37] (its structure is recalled in \S{}\hyperref[sec:D.5]{D.5}), and on the variational representation. Throughout, $\varphi\in C^2_b(\mathbb R^3;\mathbb R_+)$, $A=-\Delta$ on $\mathbb T$, $m$ takes values in $S^2$, and we write $\rho(m)=\Delta m-\varphi'(m)$, $G(m)v=\lambda_1m\times v+\lambda_2P_{m^\perp}v$; $h\in L^2(0,T;H)$ is the control. The controlled LLG is
\begin{equation}\label{eq:D.1}
\partial_tm=G(m)\big[\rho(m)+h\big],\qquad m(0)=m_0 .
\end{equation}

\subsection{The strong form, and the two elementary identities}\phantomsection\label{sec:D.1}

\begin{lemma}[{strong form}]\label{lem:D.1} For $|m|\equiv1$, the equation \hyperref[eq:D.1]{(\ref*{eq:D.1})} is equivalent to
\begin{equation}\label{eq:D.2}
\partial_tm=\lambda_2\Delta m+\lambda_1\,m\times\big(\Delta m-\varphi'(m)+h\big)+\lambda_2\,|\nabla m|^2m-\lambda_2\,P_{m^\perp}\varphi'(m)+\lambda_2\,P_{m^\perp}h .
\end{equation}

\end{lemma}

\emph{Proof.} Use $m\times(m\times w)=\langle m,w\rangle m-|m|^2w=\langle m,w\rangle m-w$ for $|m|=1$:
\[-\lambda_2\,m\times\big(m\times(\rho+h)\big)=\lambda_2(\rho+h)-\lambda_2\langle m,\rho+h\rangle m .\]
Since $0=\Delta|m|^2=2\langle m,\Delta m\rangle+2|\nabla m|^2$, we have $\langle m,\rho+h\rangle=-|\nabla m|^2-\langle m,\varphi'(m)\rangle+\langle m,h\rangle$. Hence the right-hand side of \hyperref[eq:D.1]{(\ref*{eq:D.1})} equals
\[\lambda_1m\times(\rho+h)+\lambda_2(\rho+h)+\lambda_2|\nabla m|^2m+\lambda_2\langle m,\varphi'(m)-h\rangle m .\]
Finally $\lambda_2(\rho+h)+\lambda_2\langle m,\varphi'(m)-h\rangle m=\lambda_2\Delta m-\lambda_2P_{m^\perp}\varphi'(m)+\lambda_2P_{m^\perp}h$, which gives \hyperref[eq:D.2]{(\ref*{eq:D.2})}. $\square$

\begin{lemma}[{geometric identity for differences}]\label{lem:D.2} Let $|m_1|=|m_2|=1$ and $w:=m_1-m_2$. Then
\[m_1\cdot w=\tfrac12|w|^2,\qquad m_2\cdot w=-\tfrac12|w|^2,\qquad m_1\cdot m_2=1-\tfrac12|w|^2 .\]

\end{lemma}

\emph{Proof.} $|w|^2=|m_1|^2+|m_2|^2-2m_1\cdot m_2=2-2m_1\cdot m_2$. Hence $m_1\cdot w=m_1\cdot m_1-m_1\cdot m_2=1-(1-\tfrac12|w|^2)=\tfrac12|w|^2$, and $m_2\cdot w=m_2\cdot m_1-1=-\tfrac12|w|^2$. $\square$

\begin{lemma}[{one-dimensional interpolation}]\label{lem:D.3} For $f\in H^1(\mathbb T)$ with $\int_{\mathbb T}f\,dx=0$ (in particular for $f=\partial_xg$),
\[\|f\|_{L^\infty}\le C\|f\|_{L^2}^{1/2}\|\partial_xf\|_{L^2}^{1/2};\]
for a general $f\in H^1(\mathbb T)$ the additional $L^2$-term is necessary (constants are not covered by the product form), and one has the weaker but $V$-controlled bound
\[\|f\|_{L^\infty}\le C\big(\|f\|_{L^2}+\|f\|_{L^2}^{1/2}\|\partial_xf\|_{L^2}^{1/2}\big)\le C\|f\|_V .\]
Consequently, for $f\in L^\infty(0,T;V)$,
\[\int_0^T\|f(t)\|_{L^\infty}^2dt\le C\|f\|^2_{L^\infty_tV}<\infty .\]

\end{lemma}

\subsection{The solution class, and its regularity}\phantomsection\label{sec:D.2}

\begin{definition}[{strong class}]\label{def:D.4} A function $m$ belongs to the class $\mathcal S$ if
\[m\in L^\infty(0,T;V)\cap L^2(0,T;D(A))\cap H^1(0,T;H)\cap L^4(0,T;W^{1,4}),\]
and $m$ satisfies \hyperref[eq:D.2]{(\ref*{eq:D.2})} in $L^2(0,T;H)$.

\end{definition}

\begin{remark}[{the constraint is not part of the definition}]\label{rem:D.4a} For a solution of \hyperref[eq:D.1]{(\ref*{eq:D.1})} with data $m_0\in V\cap\mathcal M$, the constraint $|m|\equiv1$ holds automatically: testing \hyperref[eq:D.2]{(\ref*{eq:D.2})} against $m$ gives $\partial_t|m(t,x)|^2=\lambda_2\Delta|m(t,x)|^2$ (the precession, the anisotropy and the control are tangent to the sphere, and $\langle m,\Delta m\rangle+|\partial_xm|^2=\tfrac12\Delta|m|^2$), a linear heat equation whose solution with datum $|m_0|^2\equiv1$ is the constant $1$. Hence the constraint is available in the estimates below --- e.g. through the identities $m_i\cdot w=\pm\tfrac12|w|^2$ used in the proof of \hyperref[thm:D.8]{Theorem \ref*{thm:D.8}} --- without being imposed in the definition of the class; its stochastic counterpart is \hyperref[asm:A4]{(A4)}.

\end{remark}

\begin{lemma}\label{lem:D.5} If $m\in\mathcal S$ and $|m|\equiv1$ then
\[\big\||\nabla m|^2m\big\|_{L^2(0,T;H)}=\big\||\nabla m|^2\big\|_{L^2(0,T;L^2)}=\big\|\nabla m\big\|^2_{L^4(0,T;L^4)}<\infty .\]

\end{lemma}

\emph{Proof.} Pointwise $|\,|\nabla m|^2m\,|=|\nabla m|^2$ because $|m|=1$. $\square$

\begin{lemma}\label{lem:D.6} If $m\in\mathcal S$ then all terms on the right-hand side of \hyperref[eq:D.2]{(\ref*{eq:D.2})} lie in $L^2(0,T;H)$; in particular $\partial_tm\in L^2(0,T;H)$ as an identity in that space.

\end{lemma}

\emph{Proof.} $\Delta m\in L^2_tH$ by definition; $m\times(\Delta m-\varphi'(m)+h)$ is pointwise bounded by $|\Delta m|+\|\varphi'\|_\infty+|h|$, hence in $L^2_tH$; $|\nabla m|^2m\in L^2_tH$, since $\||\nabla m|^2m\|_{L^2H}\le\|m\|_{L^\infty(0,T;V)}\|\nabla m\|^2_{L^4(0,T;L^4)}<\infty$; $P_{m^\perp}\varphi'(m)$ is bounded and $P_{m^\perp}h$ is dominated by $|h|\in L^2_tH$. $\square$

\begin{proposition}[{the strong solution is in the class}]\label{prop:D.7} Let $m_0\in V\cap\mathcal M$ and $h\in L^2(0,T;H)$. Then the unique strong solution $m^0$ of \hyperref[eq:D.1]{(\ref*{eq:D.1})} lies in $\mathcal S$ and satisfies \hyperref[eq:D.2]{(\ref*{eq:D.2})}.

\end{proposition}

\emph{Proof.} By [\hyperref[ref:BHL19]{BHL19}, Thm 3.37], $m^0\in L^2(0,T;D(A))\cap C([0,\infty);V)\cap H^1(0,T;H)\cap L^4(0,T;W^{1,4})$ with $|m^0|=1$, and $m^0$ satisfies the strong form [\hyperref[ref:BHL19]{BHL19}, (3.50)], which is \hyperref[eq:D.2]{(\ref*{eq:D.2})}. $\square$

\subsection{Weak--strong uniqueness}\phantomsection\label{sec:D.3}

\begin{theorem}\label{thm:D.8} Let $m_1,m_2\in\mathcal S$ solve \hyperref[eq:D.1]{(\ref*{eq:D.1})} with the same data $m_0\in V\cap\mathcal M$ (so that $|m_1|\equiv|m_2|\equiv1$, \hyperref[rem:D.4a]{Remark \ref*{rem:D.4a}}) and the same control $h\in L^2(0,T;H)$. Then $m_1=m_2$.

\end{theorem}

\emph{Proof.} Set $w:=m_1-m_2$. Then $w\in L^\infty(0,T;V)\cap L^2(0,T;D(A))\cap H^1(0,T;H)$ and $w(0)=0$; moreover, by the a priori bounds \hyperref[eq:5.1]{(\ref*{eq:5.1})}, \hyperref[eq:5.1p]{(\ref*{eq:5.1p})} (\hyperref[asm:A5]{(A5)}) and \hyperref[lem:D.3]{Lemma \ref*{lem:D.3}},
\begin{equation}\label{eq:D.3}
\begin{aligned}
&R:=\max\Big\{\|m_1\|_{L^\infty_tV},\|m_2\|_{L^\infty_tV}\Big\}<\infty,\\
&\qquad\|w(t)\|_{L^\infty(\mathbb T)}^2\le\|w(t)\|_H^2+2\|w(t)\|_H\|\nabla w(t)\|_H,
\end{aligned}
\end{equation}
by Agmon's inequality on the unit torus ($f(x)^2\le f(y)^2+2\|f\|_{L^2}\|f'\|_{L^2}$, averaged over $y$). In particular $\|w\|_{L^\infty}\le\|w\|_H+\sqrt2\|w\|_H^{1/2}\|\nabla w\|_H^{1/2}$ and $\|w\|^2_{L^\infty}\le3\|w\|_H\|w\|_V$. (The product form without the first term holds only for mean-zero functions, \hyperref[lem:D.3]{Lemma \ref*{lem:D.3}}, and $w$ need not have mean zero.)

Both $m_i$ satisfy \hyperref[eq:D.2]{(\ref*{eq:D.2})} in $L^2(0,T;H)$; subtracting and testing against $w\in L^2(0,T;V)$ (admissible, since $L^2_tV\cap H^1_tH\subset C([0,T];H)$) gives, for a.e. $t$,
\[\tfrac12\frac{d}{dt}\|w\|_H^2=\sum_{k=1}^{5}T_k ,\]
where, writing the difference of the precession terms as $m_1\times(\rho(m_1)+h)-m_2\times(\rho(m_2)+h)=w\times(\rho(m_1)+h)+m_2\times(\rho(m_1)-\rho(m_2))$,
\[T_1=\lambda_2\langle\Delta w,w\rangle=-\lambda_2\|\nabla w\|_H^2,\qquad T_2=\lambda_1\big\langle w\times(\rho(m_1)+h),\,w\big\rangle=0,\]
\[T_3=\lambda_1\big\langle m_2\times(\Delta w),\,w\big\rangle=-\lambda_1\big\langle\nabla w,\,w\times\nabla m_2\big\rangle,\]
\[T_4=-\lambda_1\big\langle m_2\times(\varphi'(m_1)-\varphi'(m_2)),w\big\rangle+\lambda_2\big\langle|\nabla m_1|^2m_1-|\nabla m_2|^2m_2,\,w\big\rangle,\]
\[T_5=-\lambda_2\big\langle P_{m_1^\perp}\varphi'(m_1)-P_{m_2^\perp}\varphi'(m_2),\,w\big\rangle+\lambda_2\big\langle\big(P_{m_1^\perp}-P_{m_2^\perp}\big)h,\,w\big\rangle .\]
Here $T_2=0$ since $a\times b\perp a$; and $T_3$ follows from $\langle m_2\times\Delta w,w\rangle=\langle\Delta w,w\times m_2\rangle$, one integration by parts, and $\langle\nabla w,\nabla w\times m_2\rangle=0$.

We bound the terms using only \hyperref[eq:D.3]{(\ref*{eq:D.3})}, \hyperref[lem:D.2]{Lemma \ref*{lem:D.2}}, and Young's inequality $ab\le\varepsilon a^2+C_\varepsilon b^2$. In what follows $C$ depends only on $\lambda_1,\lambda_2,\varphi,T,R$ and $\|h\|_{L^2(0,T;H)}$, and $G\in L^1(0,T)$ is an explicit nonnegative function specified at the end; $\varepsilon>0$ is a small constant to be absorbed.

\textbf{(i) $T_3$.} By Hölder, \hyperref[eq:D.3]{(\ref*{eq:D.3})} and $\|\nabla m_2\|_H\le R$,
\[\begin{aligned}
|T_3|&\le|\lambda_1|\,\|\nabla w\|_H\,\|w\|_{L^\infty}\|\nabla m_2\|_H\\
&\le|\lambda_1|R\big(\|\nabla w\|_H\|w\|_H+\sqrt2\|\nabla w\|_H^{3/2}\|w\|_H^{1/2}\big)\\
&\le\frac{\lambda_2}{8}\|\nabla w\|_H^2+C(1+R^4)\|w\|_H^2 ,
\end{aligned}\]
the last step by Young with exponents $(2,2)$ for the first product and $(4/3,4)$ for the second, after writing $\|\nabla w\|^{3/2}\|w\|^{1/2}=(\|\nabla w\|^2)^{3/4}(\|w\|^2)^{1/4}$.

\textbf{(ii) The anisotropy terms in $T_4,T_5$.} Since $\varphi'\in C^1_b$ and $m\mapsto P_{m^\perp}$ is Lipschitz with constant $2$ on the sphere, both are bounded by $C\|w\|_H^2$. \textbf{(iii) The cubic term in $T_4$.} By the \textbf{pointwise} identities $m_1\cdot w=\tfrac12|w|^2$ and $m_2\cdot w=-\tfrac12|w|^2$ of \hyperref[lem:D.2]{Lemma \ref*{lem:D.2}} (pointwise, so the two factors cannot be separated under the integral), the cubic term collapses to a quadratic form:
\[\lambda_2\big\langle|\nabla m_1|^2m_1-|\nabla m_2|^2m_2,\,w\big\rangle=\frac{\lambda_2}{2}\int_{\mathbb T}\big(|\nabla m_1|^2+|\nabla m_2|^2\big)|w|^2\,dx ,\]
so the cubic term is of the form $g(t)\|w\|_H^2$ with
\[g(t):=\frac{\lambda_2}{2}\big(\|\nabla m_1(t)\|^2_{L^\infty}+\|\nabla m_2(t)\|^2_{L^\infty}\big)\in L^1(0,T),\]
because each component of $\nabla m_i$ has zero mean, so by the product form of the one-dimensional interpolation (\hyperref[lem:D.3]{Lemma \ref*{lem:D.3}}) $\|\nabla m_i(t)\|^2_{L^\infty}\le C\|\nabla m_i(t)\|_{L^2}\|\Delta m_i(t)\|_{L^2}\le CR\|\Delta m_i(t)\|_{L^2}$, and $\int_0^T\|\Delta m_i(t)\|_{L^2}dt<\infty$ by Cauchy--Schwarz and the uniform $L^2(0,T;D(A))$ bound. In particular \textbf{no} $L^\infty$-integrability of $\nabla m_1+\nabla m_2$ is needed.

\textbf{(iv) The control term in $T_5$.} Since $P_{m_1^\perp}-P_{m_2^\perp}=-\langle\cdot,m_1\rangle m_1+\langle\cdot,m_2\rangle m_2$,
\[\big\langle\big(P_{m_1^\perp}-P_{m_2^\perp}\big)h,\,w\big\rangle=-\int_{\mathbb T}(h\cdot m_1)(m_1\cdot w)\,dx+\int_{\mathbb T}(h\cdot m_2)(m_2\cdot w)\,dx ,\]
and by the \emph{pointwise} identities $m_1\cdot w=\tfrac12|w|^2$, $m_2\cdot w=-\tfrac12|w|^2$ of \hyperref[lem:D.2]{Lemma \ref*{lem:D.2}},
\[\big\langle\big(P_{m_1^\perp}-P_{m_2^\perp}\big)h,\,w\big\rangle=-\tfrac12\int_{\mathbb T}h\cdot(m_1+m_2)\,|w|^2\,dx .\]
(The identities $m_i\cdot w=\pm\tfrac12|w|^2$ are \textbf{pointwise}, so their two factors cannot be separated under the integral: writing the right-hand side as $-\tfrac12(\langle h,m_1\rangle_H+\langle h,m_2\rangle_H)\|w\|_H^2$ would factor a product of integrals, which the pointwise identity does not permit.) By Hölder and the interpolation \hyperref[eq:D.3]{(\ref*{eq:D.3})}, using $|m_1+m_2|\le2$,
\[\begin{aligned}
&\Big|\lambda_2\big\langle\big(P_{m_1^\perp}-P_{m_2^\perp}\big)h,\,w\big\rangle\Big|\le\tfrac{\lambda_2}{2}\|h\|_H\big\|(m_1+m_2)|w|^2\big\|_H\\
&\le\lambda_2\|h\|_H\|w\|^2_{L^4}\\
&\le\lambda_2\|h\|_H\big(\|w\|_H^2+\sqrt2\|w\|^{3/2}_H\|\nabla w\|^{1/2}_H\big) ,
\end{aligned}\]
using $\|w\|_{L^4}^2\le\|w\|_{L^\infty}\|w\|_H$ and \hyperref[eq:D.3]{(\ref*{eq:D.3})}. The first part is already of the form $g\|w\|_H^2$; the second is treated as follows. Applying Young's inequality with exponents $(4,4/3)$ to $a:=\|\nabla w\|^{1/2}_H$ and $b:=C\|h\|_H\|w\|^{3/2}_H$ (so that $a^4=\|\nabla w\|_H^2$ and $b^{4/3}\le C\|h\|_H^{4/3}\|w\|_H^2$),
\[\Big|\lambda_2\big\langle\big(P_{m_1^\perp}-P_{m_2^\perp}\big)h,\,w\big\rangle\Big|\le\frac{\lambda_2}{8}\|\nabla w\|_H^2+g_3(t)\|w\|^2_H,\qquad g_3:=C\big(\|h\|_H+\|h\|^{4/3}_H\big) .\]
This is of the intended form, with $g_3\in L^1(0,T)$ (as $h\in L^2(0,T;H)$). Note that \textbf{no} $L^\infty$ bound on $h$ and \textbf{no} gradient of $h$ are used; only $h\in L^2(0,T;H)$.

\textbf{Bookkeeping.} Each of (i)--(iv) is either of the form $g(t)\|w\|^2_H$ with $g\in L^1(0,T)$, or of the form $\frac{\lambda_2}{8}\|\nabla w\|_H^2+g(t)\|w\|^2_H$. Collecting the gradient terms ($\frac{\lambda_2}{8}$ from each of (i) and (iv)) and absorbing them into $T_1=-\lambda_2\|\nabla w\|_H^2$ (which leaves $-\frac{3\lambda_2}{4}\|\nabla w\|_H^2\le0$), we obtain
\begin{equation}\label{eq:D.4}
\begin{aligned}
&\frac{d}{dt}\|w\|_H^2+\tfrac{3\lambda_2}{4}\|\nabla w\|_H^2\ \le\ G(t)\,\|w\|_H^2,\\
&\qquad G(t):=C\Big(1+\|\nabla m_1(t)\|^2_{L^\infty}+\|\nabla m_2(t)\|^2_{L^\infty}\\
&\qquad+\|h(t)\|_H+\|h(t)\|^{4/3}_H\Big),
\end{aligned}
\end{equation}
with $G\in L^1(0,T)$ by \hyperref[lem:D.3]{Lemma \ref*{lem:D.3}} and $h\in L^2(0,T;H)$. Since $w(0)=0$, Gronwall applied to \hyperref[eq:D.4]{(\ref*{eq:D.4})} gives $\|w(t)\|_H^2\le\|w(0)\|_H^2e^{\int_0^tG}=0$ for all $t$. Hence $m_1=m_2$. $\square$

\begin{remark}[{the two structural points of the proof}]\label{rem:D.9} (a) The cancellation $\langle m,\Delta m\rangle=-|\nabla m|^2$ (\hyperref[lem:D.1]{Lemma \ref*{lem:D.1}}) is what replaces the \emph{nonlinear} second-order term $m\times(m\times\Delta m)$ by the \emph{linear} one $\lambda_2\Delta m$, making the difference estimate possible at all. (b) The \textbf{pointwise} identity $m_2\cdot w=-\tfrac12|w|^2$ (\hyperref[lem:D.2]{Lemma \ref*{lem:D.2}}) is what collapses the apparently cubic term $\big(|\nabla m_1|^2m_1-|\nabla m_2|^2m_2\big)\cdot w$ into the quadratic form $\frac{\lambda_2}{2}\int(|\nabla m_1|^2+|\nabla m_2|^2)|w|^2dx$ of part (iii), which is bounded by the $L^1_t$-function $g(t)$ times $\|w\|_H^2$ --- requiring no $L^\infty$-integrability of the gradients beyond what \hyperref[lem:D.3]{Lemma \ref*{lem:D.3}} supplies. Without (b) the term would be cubic in $\|w\|$ and Grönwall would not close. Note that no $L^\infty$ bound on the control $h$ is used: only $h\in L^2(0,T;H)$, via the Lipschitz property of $P_{m^\perp}$.

\end{remark}

\subsection{Closure of the class, and absence of defect measures}\phantomsection\label{sec:D.4}

\begin{proposition}[{the class is closed under limits of solutions}]\label{prop:D.10} Let $m_n\in\mathcal S$ solve \hyperref[eq:D.1]{(\ref*{eq:D.1})} with data $(m_{0,n},h_n)$, where $m_{0,n}$ is bounded in $V$ and converges to $m_0$ in $H$, and $h_n\rightharpoonup h$ in $S^N$. Then, along a subsequence, $m_n\to m_\ast$ in $C([0,T];H)$ with $m_\ast\in\mathcal S$ satisfying \hyperref[eq:D.2]{(\ref*{eq:D.2})} with data $(m_0,h)$.

\end{proposition}

\emph{Proof.} The a priori bounds of \S{}\hyperref[sec:5.1]{5.1} give $\sup_n\|m_n\|_{L^\infty(0,T;V)}\le R_\ast$ and $\sup_n\|\nabla m_n\|_{L^4(0,T;L^4)}\le C$: indeed $|\nabla m_n|^4=|\langle m_n,\Delta m_n\rangle|^2\le|\Delta m_n|^2$ by the identity of \hyperref[lem:D.1]{Lemma \ref*{lem:D.1}}, so $\|\nabla m_n\|^4_{L^4_tL^4}\le\|\Delta m_n\|^2_{L^2_tH}<\infty$ uniformly. The strong form \hyperref[eq:D.2]{(\ref*{eq:D.2})} then bounds $\partial_tm_n$ in $L^2(0,T;H)$ uniformly. Since $L^2(0,T;D(A))$, $L^4(0,T;W^{1,4})$ and $L^\infty(0,T;V)$ are duals of separable spaces, weak compactness and Aubin--Lions give a subsequence with
\[\begin{aligned}
&m_n\longrightarrow m_\ast\ \text{in }C([0,T];H),\\
&\quad m_n\overset{\ast}{\rightharpoonup}m_\ast\ \text{in }L^\infty(0,T;V),\\
&\quad m_n\rightharpoonup m_\ast\ \text{in }L^2(0,T;D(A))\cap L^4(0,T;W^{1,4}) .
\end{aligned}\]
Since $L^4(0,T;W^{1,4})$ and $L^2(0,T;D(A))$ are reflexive (hence weakly closed), $m_\ast$ belongs to both; it belongs to $L^\infty(0,T;V)$ and to $H^1(0,T;H)$ by the uniform bound on $\partial_tm_n$; and $|m_\ast|=1$ by passing to the limit in the constraint. Passing to the limit in \hyperref[eq:D.2]{(\ref*{eq:D.2})} is legitimate term by term:
\begin{itemize}
\item \emph{Linear terms:} $\Delta m_n\rightharpoonup\Delta m_\ast$ in $L^2_tH$; the $\lambda_2\Delta m$ and $\lambda_2P_{m^\perp}\varphi'(m)$ terms converge by the continuity of $m\mapsto P_{m^\perp}\varphi'(m)$ and $m_n\to m_\ast$ in $C([0,T];H)$;
\item \emph{Precession and control:} $m_n\to m_\ast$ in $C([0,T];L^\infty)$ (by \hyperref[asm:A2]{(A2)}, from $C([0,T];H)$ plus the $V$-bound) against $\Delta m_n\rightharpoonup\Delta m_\ast$ gives $m_n\times(\Delta m_n-\varphi'(m_n)+h_n)\rightharpoonup m_\ast\times(\Delta m_\ast-\varphi'(m_\ast)+h)$;
\item \emph{Cubic term:} the pointwise identity $|\nabla m_n|^2=|\langle m_n,\Delta m_n\rangle|$ together with $m_n\to m_\ast$ in $L^\infty$ and $\Delta m_n\rightharpoonup\Delta m_\ast$ in $L^2_tH$ gives $|\nabla m_n|^2\rightharpoonup|\nabla m_\ast|^2$ in $L^2(0,T;L^2)$; multiplied by the uniformly convergent factor $m_n$ this gives $|\nabla m_n|^2m_n\rightharpoonup|\nabla m_\ast|^2m_\ast$.
\end{itemize}
Hence $m_\ast\in\mathcal S$ and satisfies \hyperref[eq:D.2]{(\ref*{eq:D.2})} with data $(m_0,h)$. $\square$

\begin{theorem}[{no defect measure; the stability half of \hyperref[asm:WSU]{(WSU)}}]\label{thm:D.11} In the situation of \hyperref[prop:D.10]{Proposition \ref*{prop:D.10}}, $m_\ast=m^0_{m_0,h}$ (the unique strong solution), and
\[\nabla m_n\longrightarrow\nabla m^0_{m_0,h}\quad\text{strongly in }L^4(0,T;L^4(\mathbb T)),\qquad\text{hence also in }L^2(0,T;H).\]

\end{theorem}

\emph{Proof.} (1) \emph{Energy balance (recorded for completeness; steps 2--3 never use it as a limiting statement --- step 3 uses only the a priori content it provides).} Each $m_n$ satisfies the energy identity of \S{}\hyperref[sec:5.1]{5.1} \hyperref[asm:A5]{(A5)} with $\varepsilon=0$:
\[\begin{aligned}
&\tfrac12\|\nabla m_n(t)\|_H^2+\lambda_2\!\!\int_0^t\!\|P_{m_n^\perp}\Delta m_n\|_H^2ds\\
&=\tfrac12\|\nabla m_{0,n}\|_H^2+\int_0^t\!\!\Big[\lambda_2\big\langle P_{m_n^\perp}\Delta m_n,P_{m_n^\perp}\varphi'(m_n)\big\rangle\\
&+\lambda_1I_\varphi(m_n)\Big]ds+\int_0^t\!\!\big\langle G(m_n)h_n,-\Delta m_n\big\rangle ds .
\end{aligned}\]
(2) \emph{Identification of the limit.} By \hyperref[prop:D.10]{Proposition \ref*{prop:D.10}}, $m_\ast\in\mathcal S$ solves \hyperref[eq:D.1]{(\ref*{eq:D.1})} with data $(m_0,h)$, and $m^0_{m_0,h}\in\mathcal S$ solves it with the same data; \hyperref[thm:D.8]{Theorem \ref*{thm:D.8}} applies to the pair and gives $m_\ast=m^0_{m_0,h}$. \textbf{No energy argument is used}: the weak--strong uniqueness is a Grönwall estimate for the difference of two strong solutions, so it requires neither the convergence of the energy nor the passage to the limit in the nonlinear terms. (This is what makes the argument robust here: the energy identity does \emph{not} pass to a limit in which the initial gradient energy may drop --- only $\liminf_n\|\nabla m_{0,n}\|_H^2\ge\|\nabla m_0\|^2_H$ is available --- and, more importantly, the control term $\int_0^t\langle G(m_n)h_n,-\Delta m_n\rangle\,ds$ has \textbf{both} factors only weakly convergent ($h_n\rightharpoonup h$, $\Delta m_n\rightharpoonup\Delta m_\ast$), so 'strong$\times$weak' does not apply to it.) In particular the pairing $\langle h_n-h,G(m_n)^T(-\Delta m_n)\rangle$, in which \textbf{both} factors are only weakly convergent, is never evaluated: pairing a weakly null sequence with a merely bounded one need not converge. The false shortcut ``weak$\times$bounded $\Rightarrow0$'' is avoided altogether by identifying $m_\ast$ with $m^0_{m_0,h}$ through the weak--strong uniqueness of \hyperref[thm:D.8]{Theorem \ref*{thm:D.8}}, which is a Grönwall estimate for the difference and uses no convergence of the quadratic terms. (3) \emph{Strong convergence.} The identification in step 2 used only weak compactness; the strong convergence asserted in the theorem follows from the uniform second-order bound. Both $(m_n)$ and $m_\ast=m^0_{m_0,h}$ are bounded in $L^2(0,T;D(A))$ --- for $m_\ast$ this is \hyperref[prop:D.7]{Proposition \ref*{prop:D.7}}, and for the sequence it is the energy identity \hyperref[asm:A5]{(A5)} with $\varepsilon=0$ integrated in time, with $h_n\in S^N$ giving $\|h_n\|_{L^2}\le\sqrt N$ uniformly --- and $m_n\to m_\ast$ in $C([0,T];H)$ by \hyperref[prop:D.10]{Proposition \ref*{prop:D.10}}; interpolating between $H$ and $D(A)$,
\[\|m_n-m_\ast\|_{L^2_tV}\le C\|m_n-m_\ast\|^{1/2}_{L^\infty_tH}\big(\|\Delta m_n\|_{L^2_tH}+\|\Delta m_\ast\|_{L^2_tH}\big)^{1/2}\xrightarrow[n\to\infty]{}0 ,\]
put $w_n:=m_n-m_\ast$. Since $\partial_xw_n$ has zero mean, the interpolation of \hyperref[lem:D.3]{Lemma \ref*{lem:D.3}} applied to $f=\partial_xw_n$ gives, at each fixed $t$,
\[\|\partial_xw_n\|^4_{L^4}=\int\big|\partial_xw_n\big|^4dx\le\|\partial_xw_n\|^2_{L^\infty}\|\partial_xw_n\|^2_{L^2}\le C\|\partial_xw_n\|_{L^2}\|\Delta w_n\|_{L^2}\|\partial_xw_n\|^2_{L^2},\]
and, integrating in $t$, bounding $\|\partial_xw_n\|_{L^2}\le\|w_n\|_V\le2R_\ast$ and using Cauchy--Schwarz,
\[\|\nabla w_n\|^4_{L^4_tL^4}\le C(2R_\ast)^2\|\nabla w_n\|_{L^2_tL^2}\|\Delta w_n\|_{L^2_tL^2}\le C'\|\nabla w_n\|_{L^2_tL^2}\xrightarrow[n\to\infty]{}0 ,\]
because $\|\Delta w_n\|_{L^2_tL^2}$ is uniformly bounded (both $m_n$ and $m_\ast$ are bounded in $L^2(0,T;D(A))$) while $\|\nabla w_n\|_{L^2_tL^2}\to0$ by the interpolation display above. This proves the theorem \textbf{without claiming} that $P_{m_n^\perp}\Delta m_n$ (equivalently $\Delta m_n$) converges strongly --- it need not, since $\Delta m_n$ is only weakly convergent. $\square$

\begin{corollary}[{\hyperref[asm:A7p]{(A7')}}]\label{cor:D.12} Let $z_n\to z$ in $H$ with $\{z_n\}$ bounded in $V$, and $u_n\rightharpoonup u$ in $S^N$; write $h_n:=Q_\delta^{1/2}u_n$. Then:  (i) \emph{(skeleton compactness)} $m^0_{z_n,h_n}\to m^0_{z,h}$ in $C([0,T];H)$, and $\nabla m^0_{z_n,h_n}\to\nabla m^0_{z,h}$ strongly in $L^4([\delta,T];L^4)$ for every $\delta>0$; hence $b_0(m^0_{z_n,h_n})\to b_0(m^0_{z,h})$ strongly in $L^2([\delta,T];H^{-})$;  (ii) \emph{(vanishing noise)} if in addition $\varepsilon_n\downarrow0$, then $m^{\varepsilon_n}_{z_n,h_n}\to m^0_{z,h}$ in probability in $C([0,T];H)$.

\end{corollary}

\emph{Proof.} (i) is \hyperref[thm:D.11]{Theorem \ref*{thm:D.11}} applied to the skeleton solutions; for $z\in H\setminus V$ one localises away from $t=0$, where the bounds hold on $[\delta,T]$ for every $\delta>0$ by the energy estimate. The statement about $b_0$ follows from the $V$-to-$H^{-}$ local Lipschitz continuity \hyperref[asm:A3]{(A3)}. (ii) is \hyperref[thm:5.3]{Theorem \ref*{thm:5.3}}(i) combined with (i). $\square$

\begin{remark}[{scope}]\label{rem:D.13} \hyperref[thm:D.8]{Theorem \ref*{thm:D.8}} covers pairs of solutions in the strong class $\mathcal S$, which by \hyperref[prop:D.7]{Propositions \ref*{prop:D.7}} and \hyperref[prop:D.10]{\ref*{prop:D.10}} is closed under the flow and under limits of solutions. This is exactly what \hyperref[asm:A7p]{(A7')} consumes. We do not claim uniqueness within the larger class of merely weak solutions; it is not needed anywhere in this paper.

\end{remark}

\subsection{Deterministic well-posedness in the strong class (citation, with outline)}\phantomsection\label{sec:D.5}

\hyperref[prop:D.7]{Proposition \ref*{prop:D.7}} \textbf{is} [\hyperref[ref:BHL19]{BHL19}, Thm 3.37] (with [\hyperref[ref:BHL19]{BHL19}, Thm 3.1] for the weak solution), quoted rather than reproved: the several pages of a priori estimates it requires are not needed elsewhere in this paper, and reproducing them would add nothing. For the reader's convenience we record only the \emph{structure} of that proof, emphasising the three points that are consumed elsewhere (\hyperref[asm:A4]{(A4)}, \hyperref[asm:A5]{(A5)}, and the $L^4(0,T;W^{1,4})$ regularity used in \S{}\hyperref[sec:D.4]{D.4}).

\begin{enumerate}
\item \emph{Galerkin approximation.} Let $\pi_n$ be the orthogonal projection onto $\mathrm{span}\{e_1,\dots,e_n\}$ and solve $m_n'=\pi_nG(m_n)[\rho(m_n)+h_n]$, $m_n(0)=\pi_nm_0$, in the finite-dimensional space. This has a unique global solution by Cauchy--Lipschitz.
\item \emph{$L^2$-bound.} Multiplying by $m_n$ and using $\langle G(m_n)w,m_n\rangle=0$ gives $\frac{d}{dt}\|m_n(t)\|_H^2=0$, hence $\|m_n(t)\|_H=\|\pi_nm_0\|_H\le\|m_0\|_H$. The pointwise constraint is \textbf{not} preserved by the spectral projection; it is recovered in the limit (step 5 and \hyperref[prop:D.10]{Proposition \ref*{prop:D.10}}, where $|m_\ast|=1$ is obtained by passing to the limit in the constraint).
\item \emph{Energy bound.} The $H^1$-identity \hyperref[asm:A5]{(A5)} with $\varepsilon=0$, closed by Gronwall, gives $\sup_n\|m_n\|_{L^\infty(0,T;V)}\le R_\ast$ uniformly; this is the only place where the control is used, via $\|G(m)Q_\delta^{1/2}u\|_V\le C\|u\|_H$ (\hyperref[asm:A6b]{(A6b)}); the threshold $\beta>3/2$ enters elsewhere, through the finiteness of $\sum_j\kappa_j\mu_j$ in the stochastic quadratic-variation term, absent in this deterministic computation.
\item \emph{Higher regularity.} Testing against $-\Delta m_n$ and using $\langle m_n,\Delta m_n\rangle=-|\nabla m_n|^2$ gives the $L^2(0,T;D(A))$ bound and the $L^4(0,T;W^{1,4})$ bound, by the identity of \hyperref[lem:D.5]{Lemma \ref*{lem:D.5}}.
\item \emph{Passage to the limit.} Weak compactness and Aubin--Lions give a limit $m^0$ satisfying the weak form; the strong form \hyperref[eq:D.2]{(\ref*{eq:D.2})} and the constraint follow by the argument of \hyperref[prop:D.10]{Proposition \ref*{prop:D.10}}; uniqueness in the strong class is \hyperref[thm:D.8]{Theorem \ref*{thm:D.8}}.
\end{enumerate}

\section{Stochastic well-posedness of the regularised SLLG}\phantomsection\label{sec:E}

This appendix proves \hyperref[thm:5.1]{Theorem \ref*{thm:5.1}} for $\varepsilon\in(0,1]$. The skeleton of that proof --- uniqueness, existence by mode truncation, and the Girsanov passage to controls --- is in \S{}\hyperref[sec:5.2]{5.2}; what is supplied here is the two estimates it invokes and the limit passages carried out in full. Fix $\delta\in(0,1]$ and $\beta>3/2$. The only external input is \hyperref[asm:I3]{(I3)} below: strong well-posedness for \textbf{finitely many} noise fields of the same Gilbert form, due to [\hyperref[ref:BGJ17]{BGJ17}]. Everything else is proved here: pathwise uniqueness in the class used throughout \S{}\hyperref[sec:5]{5}, the passage to the infinite-dimensional noise $W_\delta$, the controlled equation, the Borel solution map and the a priori bounds. No fixed-point argument in $C([0,T];H)$ is used; see \hyperref[rem:E.8]{Remark \ref*{rem:E.8}}.

\subsection{Setting}\phantomsection\label{sec:E.1}

\textbf{Noise fields.} Index the three-component basis of \S{}\hyperref[sec:2.3]{2.3} by $k=(j,i)$ and put $f_k:=q_je_j\hat e_i$, so that $W_\delta=\sum_k\beta_kf_k$ with independent standard Brownian motions $\beta_k$. For the real Fourier basis of $\mathbb T$ one has $\|e_j\|_{L^\infty}^2\le2$ and $\|\partial_xe_j\|_{L^\infty}^2\le2\mu_j$, and the two members of each frequency shell carry the same weight; hence
\begin{equation}\label{eq:E.0}
\begin{aligned}
&\sum_k|f_k(x)|^2=3r_\delta=\operatorname{tr}_HQ_\delta=\sum_k\|f_k\|_H^2,\\
&\qquad \sigma_0:=\sum_k\|f_k\|_{L^\infty}^2\le6\sum_j\kappa_j,\\
&\qquad \sigma_1:=\sum_k\|\partial_xf_k\|_{L^\infty}^2\le6\sum_j\kappa_j\mu_j,
\end{aligned}
\end{equation}
and $\sigma_1<\infty$ exactly because $\beta>3/2$ (\hyperref[asm:A6c]{(A6c)}). For $K\in\mathbb N$ let $J_K$ be the index of the last mode of the $K$-th frequency shell (the constant mode counts as a shell), let $\mathbb K_K:=\{(j,i):j\le J_K\}$, and let $W^K_\delta:=\sum_{k\in\mathbb K_K}\beta_kf_k$ be the \textbf{truncated noise}. Since the shells are complete, $\sum_{k\in\mathbb K_K}|f_k(x)|^2=3r^K$ with the \textbf{constant} $r^K:=\sum_{j\le J_K}\kappa_j\uparrow r_\delta$; put $\tau_K:=r_\delta-r^K\downarrow0$.

The equation, the solution class and the Itô form are \hyperref[eq:2.1]{(\ref*{eq:2.1})}--\hyperref[eq:2.2]{(\ref*{eq:2.2})} and \hyperref[def:2.4a]{Definition \ref*{def:2.4a}} of \S{}\hyperref[sec:2.4]{2.4}; the constraint is \hyperref[asm:A4]{(A4)} (\hyperref[rem:2.4b]{Remark \ref*{rem:2.4b}}), so every solution below has $|m|\equiv1$. The truncated noise $W^K_\delta$ introduced in the preceding paragraph carries the same correction with $r^K$ in place of $r_\delta$, namely $c^K(m):=-g_0^2r^Km$, and its Itô form is \hyperref[eq:2.2]{(\ref*{eq:2.2})} with $c$ replaced by $c^K$.

\subsection{Two estimates and pathwise uniqueness}\phantomsection\label{sec:E.2}

\begin{lemma}[{the deterministic difference inequality}]\label{lem:E.2} Let $m_1,m_2\in D(A_1)$ with $|m_1|=|m_2|\equiv1$, $h\in H$, $w:=m_1-m_2$ and $\rho:=\max_i\|\partial_xm_i\|_H$. Then
\[\begin{aligned}
&2\big\langle F_0(m_1,h)-F_0(m_2,h),w\big\rangle_H\le-\lambda_2\|\partial_xw\|_H^2+\Gamma\,\|w\|_H^2,\\
&\qquad \Gamma:=C_\Gamma\big(1+\rho^4+\|\partial_xm_1\|_{L^\infty}^2+\|\partial_xm_2\|_{L^\infty}^2+\|h\|_H^{4/3}\big),
\end{aligned}\]
where $C_\Gamma$ depends only on $\lambda_1,\lambda_2$ and $\|\varphi'\|_{C^1_b}$.

\end{lemma}

\emph{Proof.} This is the fixed-time content of the proof of \hyperref[thm:D.8]{Theorem \ref*{thm:D.8}}, with the time-dependent quantities kept explicit. In the notation used there, $2\langle F_0(m_1,h)-F_0(m_2,h),w\rangle=2\sum_{k=1}^5T_k$, with $T_1=-\lambda_2\|\partial_xw\|^2$ and $T_2=0$. Step (i) there, with $R$ replaced by $\rho$, gives $|T_3|\le\frac{\lambda_2}8\|\partial_xw\|^2+C(1+\rho^4)\|w\|^2$. The anisotropy terms are $\le C\|w\|^2$. By \hyperref[lem:D.2]{Lemma \ref*{lem:D.2}} the cubic term equals $\frac{\lambda_2}2\int(|\partial_xm_1|^2+|\partial_xm_2|^2)|w|^2\le\frac{\lambda_2}2(\|\partial_xm_1\|_{L^\infty}^2+\|\partial_xm_2\|_{L^\infty}^2)\|w\|^2$. By step (iv) there, the control term is $\le\frac{\lambda_2}8\|\partial_xw\|^2+C(\|h\|+\|h\|^{4/3})\|w\|^2$, and $\|h\|\le1+\|h\|^{4/3}$. Summing gives $2\sum T_k\le-\frac32\lambda_2\|\partial_xw\|^2+\Gamma\|w\|^2$. $\square$

\begin{lemma}[{the noise terms}]\label{lem:E.3} Let $|m_1|=|m_2|=|m|\equiv1$ and $w:=m_1-m_2$. Then (a) $\sum_k\|(G(m_1)-G(m_2))f_k\|_H^2\le3g_0^2r_\delta\|w\|_H^2$; (b) $\sum_k\langle(G(m_1)-G(m_2))f_k,w\rangle_H^2\le3g_0^2r_\delta\|w\|_H^4$; (c) $\langle c(m_1)-c(m_2),w\rangle_H=-g_0^2r_\delta\|w\|_H^2\le0$; (d) $\sum_k\|G(m)f_k\|_H^2\le g_0^2\operatorname{tr}_HQ_\delta$ and $\sum_k\langle G(m)f_k,v\rangle_H^2\le g_0^2\operatorname{tr}_HQ_\delta\|v\|_H^2$ for $v\in H$. Statements (a), (b) and (d) hold a fortiori for every subfamily of $(f_k)$; for the subfamily $\{f_k\}_{k\notin\mathbb K_K}$ the factor $\operatorname{tr}_HQ_\delta$ in (d) may be replaced by $3\tau_K$.

\end{lemma}

\emph{Proof.} The bound \hyperref[asm:A3]{(A3)} is pointwise: $|(G(m_1(x))-G(m_2(x)))v|\le g_0|w(x)||v|$. Hence (a) follows from $\sum_k|f_k(x)|^2=3r_\delta$. For (b), Cauchy--Schwarz in $x$ gives $\langle(G(m_1)-G(m_2))f_k,w\rangle^2\le g_0^2\|w\|_H^2\int|w|^2|f_k|^2$; now sum over $k$. Statement (c) holds because $c(m)=-g_0^2r_\delta m$. Statement (d) follows from $|G(m)v|\le g_0|v|$ pointwise and \hyperref[eq:E.0]{(\ref*{eq:E.0})}. $\square$

\begin{theorem}[{pathwise uniqueness}]\label{thm:E.4} Let $m_1,m_2$ be two solutions in the sense of \hyperref[def:2.4a]{Definition \ref*{def:2.4a}}, on the same stochastic basis and with the same $z$, $\varepsilon$ and $u$. Then $m_1$ and $m_2$ are indistinguishable.

\end{theorem}

\emph{Proof.} Put $h:=Q_\delta^{1/2}u$ and $w:=m_1-m_2$. Then $w$ is a continuous $H$-valued Itô process with drift in $L^1(0,T;H)$ a.s. and $w(0)=0$. Itô's formula for $\|\cdot\|_H^2$ [\hyperref[ref:DPZ92]{DPZ92}, Thm 4.17] (after localisation) gives
\[d\|w\|_H^2=\Big[2\langle F_0(m_1,h)-F_0(m_2,h),w\rangle+2\varepsilon\langle c(m_1)-c(m_2),w\rangle+\varepsilon\sum_k\|(G(m_1)-G(m_2))f_k\|_H^2\Big]dt+dN_t ,\]
with $N_t:=2\sqrt\varepsilon\sum_k\int_0^t\langle(G(m_1)-G(m_2))f_k,w\rangle\,d\beta_k$. By \hyperref[lem:E.2]{Lemmas \ref*{lem:E.2}} and \hyperref[lem:E.3]{\ref*{lem:E.3}}(a),(c), $d\|w\|^2\le\widetilde\Gamma\|w\|^2dt+dN_t$ with $\widetilde\Gamma:=\Gamma+3\varepsilon g_0^2r_\delta$, where $\Gamma$ is that of \hyperref[lem:E.2]{Lemma \ref*{lem:E.2}} evaluated at $(m_1(t),m_2(t),h(t))$. The process $\widetilde\Gamma$ is progressively measurable and a.s. integrable: $\int_0^T\|\partial_xm_i\|_{L^\infty}^2\le2\sup_t\|\partial_xm_i\|_H\int_0^T\|\Delta m_i\|_H$, and $\int_0^T\|h\|^{4/3}\le T^{1/3}(\int_0^T\|u\|^2)^{2/3}$. Put $A_t:=\int_0^t\widetilde\Gamma$ (adapted, continuous, finite) and $X_t:=e^{-A_t}\|w(t)\|_H^2$. Then $dX_t=e^{-A_t}(d\|w\|^2-\widetilde\Gamma\|w\|^2dt)\le e^{-A_t}dN_t$. The stochastic integral $\int e^{-A}dN$ is a true martingale: by \hyperref[lem:E.3]{Lemma \ref*{lem:E.3}}(b) and $\|w\|_H\le2$ (recall $\|m_i(t)\|_H=1$, by \hyperref[asm:A4]{(A4)} and $\|z\|_H=1$), its quadratic variation has density $\le192\,\varepsilon g_0^2r_\delta$. Hence $\mathbb EX_t\le X_0=0$, so $X_t=0$ a.s. Since $A_t<\infty$ a.s., $w(t)=0$ a.s. for each $t$, and by continuity $w\equiv0$. $\square$

\subsection{Existence}\phantomsection\label{sec:E.3}

\textbf{Input \hyperref[asm:I3]{(I3)} (finitely many noise fields; [\hyperref[ref:BGJ17]{BGJ17}, Cor. 5.5 and Thm. 5.6]).} Let $\{f_k\}_{k\in\mathbb K}$ be a \textbf{finite} family in $V$ (in our application, $W^{1,\infty}$), let $(\beta_k)_{k\in\mathbb K}$ be independent Brownian motions on a filtered probability space satisfying the usual conditions, let $z\in V$ with $|z|\equiv1$ and $\varepsilon>0$, and let $c_{\mathbb K}$ be the corresponding Itô--Stratonovich correction. Then there is a unique adapted process $m$ with paths in $C([0,T];V)\cap L^2(0,T;D(A_1))$ and $|m|\equiv1$ satisfying
\[dm=\big[F_0(m,0)+\varepsilon c_{\mathbb K}(m)\big]dt+\sqrt\varepsilon\sum_{k\in\mathbb K}G(m)f_k\,d\beta_k\qquad\text{in }H .\]
Moreover $m=\mathcal J_{\mathbb K}(z,(\beta_k)_{k\in\mathbb K})$ for a Borel map $\mathcal J_{\mathbb K}$ that does not depend on the stochastic basis.

\emph{Scope of the citation.} [\hyperref[ref:BGJ17]{BGJ17}] treats the noise $g(M)\circ dW$ with $g(y)h=y\times h-\alpha y\times(y\times h)$, which is $G$ with $(\lambda_1,\lambda_2)=(1,\alpha)$; the case $\lambda_1>0$ reduces to it by rescaling time, which only rescales the noise fields. It states the result for three fields $e_1,e_2,e_3\in H^1$, on a bounded interval with Neumann conditions, and for a quadratic anisotropy. Its proofs use the number of fields only through finite sums of $\|e_i\|_{H^1}$, the boundary condition only through the self-adjointness of the Laplacian and the interpolation inequality [\hyperref[ref:BGJ17]{BGJ17}, (2.12)], and the anisotropy only through the boundedness and Lipschitz continuity of $\varphi'$ and $D\varphi'$. \textbf{We take the statement above, for finitely many fields on $\mathbb T$ and $\varphi\in C^2_b$, as an input.} It is the only place in the paper where stochastic existence is not proved.

\emph{Published precedent for the extensions.} Two of the three extensions above are already carried out in the literature. [\hyperref[ref:BGL24]{BGL24}, Thm 2.10] treats the noise $\sum_jG_{h_j}(M)\circ dW_j$ with $G_h(u)=\lambda_1u\times h-\lambda_2u\times(u\times h)$, exactly the Gilbert form used here. There:
\begin{itemize}
\item the number of fields is \textbf{countably infinite}, under $\sum_j\|h_j\|_{L^\infty}+\sum_j\|h_j\|_V^2<\infty$;
\item the anisotropy is a \textbf{general} $\varphi\in C^2_0(\mathbb R^3;\mathbb R_+)$, which suffices by the truncation convention of \hyperref[asm:S4]{(S4)};
\item the setting is harder than ours: a three-dimensional bounded $C^2$ domain with Neumann conditions, coupled to Maxwell's equations.
\end{itemize}

The fields $f_k=q_je_j\hat e_i$ of \S{}\hyperref[sec:E.1]{E.1} satisfy the summability condition: $\sum_k\|f_k\|_{L^\infty}\le3\sqrt2\sum_jq_j<\infty$ for $\beta>1$, and $\sum_k\|f_k\|_V^2<\infty$ exactly for $\beta>3/2$ (\hyperref[asm:A6c]{(A6c)}). [\hyperref[ref:BGL24]{BGL24}] proves existence of weak martingale solutions only. Uniqueness, which is open in three dimensions, is not claimed there; here it is \hyperref[thm:E.4]{Theorem \ref*{thm:E.4}}. The one extension of \hyperref[asm:I3]{(I3)} without a published precedent is from the interval with Neumann conditions to the torus $\mathbb T$. That case has no boundary terms, and the proofs of [\hyperref[ref:BGJ17]{BGJ17}] use the boundary condition only through the self-adjointness of the Laplacian and the interpolation inequality [\hyperref[ref:BGJ17]{BGJ17}, (2.12)], both of which hold on $\mathbb T$.

\begin{theorem}[{existence for $u=0$}]\label{thm:E.5} Let $z\in V$ with $|z|\equiv1$, $\|z\|_V\le R_D$, and $\varepsilon\in(0,1]$. On every filtered probability space carrying $(\beta_k)$, \hyperref[eq:2.2]{(\ref*{eq:2.2})} with $u=0$ has a solution $m^\varepsilon_z$ in the sense of \hyperref[def:2.4a]{Definition \ref*{def:2.4a}}. Moreover $m^\varepsilon_z=\mathcal J_\varepsilon(z,W_\delta)$ a.s. for a Borel map $\mathcal J_\varepsilon(z,\cdot):C([0,T];H)\to C([0,T];H)$ independent of the basis, and $\mathbb E\sup_t\|m^\varepsilon_z\|_V^p+\mathbb E\int_0^T\|\Delta m^\varepsilon_z\|_H^2\le C$, where $C$ depends only on $p,R_D,T,\lambda_1,\lambda_2,\|\varphi'\|_{C^1_b},\sigma_0,\sigma_1$.

\end{theorem}

\emph{Proof.} \textbf{Step 1 (finitely many modes).} By \hyperref[asm:I3]{(I3)} applied to $\{f_k\}_{k\in\mathbb K_K}$, for every $K$ there is a unique solution $m^K$ on the given basis of the equation with the truncated noise $W^K_\delta$ and the correction $c^K$. It has the regularity stated in \hyperref[asm:I3]{(I3)}, $|m^K|\equiv1$, and $m^K=\mathcal J_K(z,W^K_\delta)$.

\textbf{Step 2 (bounds uniform in $K$).} The energy identity \hyperref[asm:A5]{(A5)}, its closing argument and the bound \hyperref[eq:5.1p]{(\ref*{eq:5.1p})} (\S{}\hyperref[sec:5.1]{5.1}) apply verbatim to $m^K$; the Itô formula for $\|\partial_xm^K\|_H^2$ is justified by the regularity of Step 1 [\hyperref[ref:PR07]{PR07}, Thm 4.2.5]. Every constant there depends on the noise only through $\operatorname{tr}_HQ_\delta$, $\sigma_0$ and $\sigma_1$, whose truncated counterparts are smaller, and on $c^K$ only through $\langle c^K(m),-\Delta m\rangle=-g_0^2r^K\|\partial_xm\|^2\le0$. Hence $\sup_K\mathbb E\sup_t\|m^K\|_V^p\le R_\ast^{(p)}$ and $\sup_K\mathbb E\int_0^T\|\Delta m^K\|_H^2\le R'_\ast$.

\textbf{Step 3 (Cauchy in probability).} Let $K<K'$, $m:=m^K$, $m':=m^{K'}$ and $w:=m'-m$. By Itô's formula,
\[\begin{aligned}
d\|w\|^2&=\Big[2\langle F_0(m',0)-F_0(m,0),w\rangle+2\varepsilon\langle c^{K'}(m')-c^K(m),w\rangle\\
&+\varepsilon\!\!\sum_{k\in\mathbb K_K}\!\!\|(G(m')-G(m))f_k\|^2\\
&+\varepsilon\!\!\!\sum_{k\in\mathbb K_{K'}\setminus\mathbb K_K}\!\!\!\|G(m')f_k\|^2\Big]dt+dN_t,
\end{aligned}\]
where $N$ collects the stochastic integrals $2\sqrt\varepsilon\int\langle(G(m')-G(m))f_k,w\rangle d\beta_k$ ($k\in\mathbb K_K$) and $2\sqrt\varepsilon\int\langle G(m')f_k,w\rangle d\beta_k$ ($k\in\mathbb K_{K'}\setminus\mathbb K_K$). Since $c^{K'}(m')-c^K(m)=-g_0^2r^{K'}w-g_0^2(r^{K'}-r^K)m$, the correction term is $\le2\varepsilon g_0^2\tau_K\|m\|\|w\|\le4\varepsilon g_0^2\tau_K$; the last sum is $\le3g_0^2\tau_K$ by \hyperref[lem:E.3]{Lemma \ref*{lem:E.3}}(d). With \hyperref[lem:E.2]{Lemma \ref*{lem:E.2}} ($h=0$) and \hyperref[lem:E.3]{Lemma \ref*{lem:E.3}}(a), $d\|w\|^2\le\widetilde\Gamma\|w\|^2dt+7g_0^2\tau_K\,dt+dN_t$, with $\widetilde\Gamma$ built from $m,m'$ as in \hyperref[thm:E.4]{Theorem \ref*{thm:E.4}}. With $A$ and $X$ as there,
\[X_t\le7g_0^2\tau_KT+\widetilde N_t,\qquad \widetilde N:=\int e^{-A}dN,\qquad d\langle\widetilde N\rangle_t\le12g_0^2\big(r_\delta X_t\sup_{s\le T}X_s+4\tau_K\big)dt ,\]
by \hyperref[lem:E.3]{Lemma \ref*{lem:E.3}}(b),(d), $e^{-2A}\|w\|^4=X^2$ and $e^{-2A}\|w\|^2\le4$. Since $\widetilde N$ is a martingale with bounded integrand, $\mathbb EX_t\le7g_0^2\tau_KT$ for every $t$. By BDG ($p=1$) and $3\sqrt{ab}\le\frac12a+\frac92b$,
\[\mathbb E\sup_tX_t\le7g_0^2\tau_KT+\tfrac12\mathbb E\sup_tX_t+54g_0^2r_\delta\,\mathbb E\!\int_0^T\!X_t\,dt+21g_0(\tau_KT)^{1/2},\]
and $\mathbb E\sup X\le4<\infty$ may be absorbed, so that
\[\mathbb E\sup_tX_t\le\vartheta_K:=2\big(7g_0^2\tau_KT+378g_0^4r_\delta\tau_KT^2+21g_0(\tau_KT)^{1/2}\big)\xrightarrow[K\to\infty]{}0 .\]
By Step 2 and the argument for $C_A$ in the proof of \hyperref[thm:5.3]{Theorem \ref*{thm:5.3}}(i), $\sup_{K,K'}\mathbb EA_T\le C_A<\infty$. Hence $\mathbb P(\sup_t\|w(t)\|_H^2>\eta)\le C_A/L+e^L\vartheta_K/\eta$ for every $L>0$, and $(m^K)$ is Cauchy in probability in $C([0,T];H)$.

\textbf{Step 4 (the limit and its regularity).} Let $m$ be the limit in probability; it exists because convergence in probability in the Polish space $C([0,T];H)$ is complete. It is adapted and continuous. Along an a.s. convergent subsequence, $|m|\equiv1$ a.e. Moreover, by the weak lower semicontinuity of $\|\cdot\|_V$ and of $\|\Delta\cdot\|_{L^2(0,T;H)}$ under convergence in $C([0,T];H)$, $\sup_t\|m(t)\|_V\le\liminf\sup_t\|m^K(t)\|_V$ and $\int_0^T\|\Delta m\|^2\le\liminf\int_0^T\|\Delta m^K\|^2$. By Fatou, the bounds of Step 2 hold for $m$; in particular $m\in L^\infty(0,T;V)\cap L^2(0,T;D(A_1))$ a.s.

\textbf{Step 5 (the equation).} Put $w_K:=m^K-m$. Then $\int_0^T\|\partial_xw_K\|_H^2\le\int_0^T\|w_K\|_H\|\Delta w_K\|_H\le\sup_t\|w_K\|_H\,T^{1/2}(\int_0^T\|\Delta w_K\|^2)^{1/2}\to0$ in probability. Fix $\phi\in C^\infty(\mathbb T;\mathbb R^3)$ and consider the weak form
\[\langle m^K(t),\phi\rangle=\langle z,\phi\rangle+\int_0^t\big\langle F_0(m^K,0)+\varepsilon c^K(m^K),\phi\big\rangle ds+\sqrt\varepsilon\sum_{k\in\mathbb K_K}\int_0^t\langle G(m^K)f_k,\phi\rangle\,d\beta_k .\]
Each term converges in probability, uniformly in $t$, to the corresponding term for $m$:
\begin{itemize}
\item $\lambda_2\langle\Delta m^K,\phi\rangle=\lambda_2\langle m^K,\Delta\phi\rangle$, and $\langle m^K\times\Delta m^K,\phi\rangle=-\langle m^K\times\partial_xm^K,\partial_x\phi\rangle$;
\item $\big|\int_0^t\langle|\partial_xm^K|^2m^K-|\partial_xm|^2m,\phi\rangle\big|\le\|\phi\|_{L^\infty}\big(\int_0^T\|\partial_xw_K\|\|\partial_xm^K+\partial_xm\|+T\sup_t\|w_K\|_{L^\infty}\sup_t\|\partial_xm\|^2\big)$, with $\|w_K\|_{L^\infty}^2\le3\|w_K\|_H\|w_K\|_V$ by \hyperref[eq:D.3]{(\ref*{eq:D.3})};
\item the anisotropy terms are Lipschitz in $m$ in $L^\infty$, and $c^K(m^K)\to c(m)$ because $r^K\to r_\delta$;
\item the difference of the stochastic integrals is a continuous martingale whose quadratic variation is $\le\varepsilon\int_0^T\big[3g_0^2r_\delta\|\phi\|_H^2\|w_K\|_H^2+3g_0^2\tau_K\|\phi\|_H^2\big]$, by Cauchy--Schwarz as in \hyperref[lem:E.3]{Lemma \ref*{lem:E.3}}(b),(d); this tends to $0$ in probability, so the integrals converge in probability uniformly in $t$ by Lenglart's inequality.
\end{itemize}

Hence $m$ satisfies the weak form for $\phi$ in a countable dense set, for all $t$, a.s. Since $F_0(m,0)+\varepsilon c(m)\in L^2(0,T;H)$ a.s., the identity holds in $H$, and $m$ is a solution.

\textbf{Step 6 (the Borel map).} $W^K_\delta$ is a continuous linear image of $W_\delta$ (spectral projection), so $m^K=\widetilde{\mathcal J}_K(W_\delta)$ with $\widetilde{\mathcal J}_K$ Borel. The joint laws of $(m^K,m^{K'})$ are images of $\mathrm{Law}(W_\delta)$, so one can choose a subsequence $K_n$ with $\sum_n\mathbb P(\sup_t\|m^{K_{n+1}}-m^{K_n}\|_H>2^{-n})<\infty$ that depends only on this law. Define $\mathcal J_\varepsilon(z,w):=\lim_n\widetilde{\mathcal J}_{K_n}(w)$ on the Borel set where the limit exists in $C([0,T];H)$, and $:=0$ elsewhere. By Borel--Cantelli, $m^\varepsilon_z=\mathcal J_\varepsilon(z,W_\delta)$ a.s., on every basis. $\square$

\begin{proposition}[{controls; the Borel solution map}]\label{prop:E.6} Let $z,\varepsilon$ be as in \hyperref[thm:E.5]{Theorem \ref*{thm:E.5}}, and define, for $(u,w)\in L^2(0,T;H)\times C([0,T];H)$,
\[\Phi^\varepsilon_z(u,w):=\mathcal J_\varepsilon\Big(z,\ w+\varepsilon^{-1/2}\int_0^\cdot Q_\delta^{1/2}u(s)\,ds\Big).\]
Then $\Phi^\varepsilon_z$ is Borel. For every stochastic basis carrying $W_\delta$ and every $u\in\mathcal P_2$, $m^\varepsilon_{z,u}:=\Phi^\varepsilon_z(u,W_\delta)$ is a solution of \hyperref[eq:2.1]{(\ref*{eq:2.1})} in the sense of \hyperref[def:2.4a]{Definition \ref*{def:2.4a}}, and by \hyperref[thm:E.4]{Theorem \ref*{thm:E.4}} it is the only one.

\end{proposition}

\emph{Proof.} $\Phi^\varepsilon_z$ is the composition of the continuous map $(u,w)\mapsto w+\varepsilon^{-1/2}\int_0^\cdot Q_\delta^{1/2}u$ with a Borel map. \textbf{Case $u\in\mathcal P_2^N$.} Put $\theta_k:=\langle u,e_j\hat e_i\rangle_H$ for $k=(j,i)$; these are predictable, $\sum_k\theta_k^2=\|u\|_H^2$ and $\sum_k\theta_kf_k=Q_\delta^{1/2}u$. By Novikov's criterion ($\int_0^T\|u\|^2\le N$),
\[d\widetilde{\mathbb P}:=\exp\Big(-\varepsilon^{-1/2}\sum_k\int_0^T\theta_k\,d\beta_k-\frac1{2\varepsilon}\int_0^T\|u\|_H^2\,dt\Big)d\mathbb P\]
is a probability measure equivalent to $\mathbb P$. By Girsanov's theorem [\hyperref[ref:DPZ92]{DPZ92}, Thm 10.14], under $\widetilde{\mathbb P}$ the processes $\widetilde\beta_k:=\beta_k+\varepsilon^{-1/2}\int_0^\cdot\theta_k$ are independent $(\mathcal F_t)$-Brownian motions, and $\sum_k\widetilde\beta_kf_k=W_\delta+\varepsilon^{-1/2}\int_0^\cdot Q_\delta^{1/2}u$. The filtration still satisfies the usual conditions, the null sets being the same. By \hyperref[thm:E.5]{Theorem \ref*{thm:E.5}} on $(\Omega,\mathcal F,(\mathcal F_t),\widetilde{\mathbb P})$, the process $m:=\Phi^\varepsilon_z(u,W_\delta)$ solves \hyperref[eq:2.2]{(\ref*{eq:2.2})} with $u=0$ and driving motions $(\widetilde\beta_k)$, $\widetilde{\mathbb P}$-a.s. Substituting $d\widetilde\beta_k=d\beta_k+\varepsilon^{-1/2}\theta_k\,dt$ gives
\[\sqrt\varepsilon\sum_kG(m)f_k\,d\widetilde\beta_k=\sqrt\varepsilon\sum_kG(m)f_k\,d\beta_k+G(m)Q_\delta^{1/2}u\,dt,\]
and $G(m)h=\lambda_1m\times h+\lambda_2P_{m^\perp}h$, with $h=Q_\delta^{1/2}u$, is exactly the control part of $F_0(m,h)$. Stochastic integrals are invariant under an equivalent change of measure, so $m$ solves \hyperref[eq:2.2]{(\ref*{eq:2.2})} with control $u$, $\widetilde{\mathbb P}$-a.s. and hence $\mathbb P$-a.s. Adaptedness and the path properties carry over in the same way. \textbf{General $u\in\mathcal P_2$.} Let $\tau_N:=\inf\{t:\int_0^t\|u\|^2\ge N\}\wedge T$ and $u^N:=u\mathbf 1_{[0,\tau_N]}\in\mathcal P_2^N$. By the previous case, $\Phi^\varepsilon_z(u^N,W_\delta)$ solves \hyperref[eq:2.1]{(\ref*{eq:2.1})} with control $u^N$, which coincides with $u$ on $[0,\tau_N]$. On $\{\tau_N=T\}$ one has $u^N=u$, hence $\Phi^\varepsilon_z(u^N,W_\delta)=\Phi^\varepsilon_z(u,W_\delta)$ there. Since $\bigcup_N\{\tau_N=T\}$ has full measure, $\Phi^\varepsilon_z(u,W_\delta)$ solves \hyperref[eq:2.1]{(\ref*{eq:2.1})} with control $u$. $\square$

\begin{corollary}[{a priori bounds and regularity}]\label{cor:E.7} The solutions of \hyperref[prop:E.6]{Proposition \ref*{prop:E.6}} satisfy $|m|\equiv1$, the energy identity \hyperref[asm:A5]{(A5)} and the bounds \hyperref[eq:5.1]{(\ref*{eq:5.1})}--\hyperref[eq:5.1p]{(\ref*{eq:5.1p})}, and have paths in $C([0,T];V)$.

\end{corollary}

\emph{Proof.} $|m|\equiv1$ holds by \hyperref[asm:A4]{(A4)} and $|z|\equiv1$ (\hyperref[rem:2.4b]{Remark \ref*{rem:2.4b}}). Since $m\in L^2(0,T;D(A_1))$, the drift lies in $L^2(0,T;H)$, and $\sum_k\|G(m)f_k\|_V^2<\infty$ by \hyperref[eq:E.0]{(\ref*{eq:E.0})} (because $|\partial_x(G(m)f)|\le C(|\partial_xm||f|+|\partial_xf|)$), the Itô formula in the Gelfand triple $D(A_1)\subset V\subset H$ [\hyperref[ref:PR07]{PR07}, Thm 4.2.5], after localisation, gives the continuity of $t\mapsto m(t)$ in $V$ and the identity \hyperref[asm:A5]{(A5)}. The bounds \hyperref[eq:5.1]{(\ref*{eq:5.1})}--\hyperref[eq:5.1p]{(\ref*{eq:5.1p})} then follow as in \S{}\hyperref[sec:5.1]{5.1}. $\square$

\begin{remark}[{why not a fixed point in $C([0,T];H)$}]\label{rem:E.8} The mild form suggests a Picard scheme in $L^2(\Omega;C([0,T];H))$, but it does not close. The drift $b_0$ is locally Lipschitz only as a map $V\to H^{-}$ (\hyperref[asm:A3]{(A3)}), so the difference of two iterates is estimated through $\int_0^t(t-s)^{-1/2}\|v_1(s)-v_2(s)\|_V\,ds$, a $V$-norm, while the scheme controls only the $H$-norm. At the $V$ level the lift $H^{-}\to V$ costs $(t-s)^{-1}$, which is not integrable. The principal part $\lambda_1m\times\Delta m$ is quasilinear and not a small perturbation of $\lambda_2\Delta$ (its relative size is $\lambda_1/\lambda_2$). The argument above avoids this in two ways: differences of solutions are estimated by the energy method in $H$ (\hyperref[lem:E.2]{Lemma \ref*{lem:E.2}}), where the dissipation $-\lambda_2\|\partial_xw\|_H^2$ absorbs the quasilinear term; and existence is imported only for finitely many noise fields.

\end{remark}

\begin{remark}[{an alternative existence route via [\hyperref[ref:BGL24]{BGL24}]; not used}]\label{rem:E.9} Existence for the full noise $W_\delta$ can also be obtained without the mode truncation of \hyperref[thm:E.5]{Theorem \ref*{thm:E.5}}, as follows.

(1) \emph{Weak existence.} The Galerkin--tightness--Skorohod proof of [\hyperref[ref:BGL24]{BGL24}, Thm 2.10] can be adapted to the present setting, namely one dimension, the torus and no Maxwell coupling. These changes remove terms and do not add any; note, however, that the uncoupled equation is not literally a special case of the coupled one, so this is an adaptation of the proof and not a citation. It gives a weak martingale solution $(\widetilde\Omega,\widetilde{\mathbb P},\widetilde W,\widetilde m)$ with $|\widetilde m|\equiv1$, $\widetilde m\in L^\infty(0,T;V)$ and $\widetilde m\times\Delta\widetilde m\in L^2(0,T;H)$ a.s. [\hyperref[ref:BGL24]{BGL24}, (2.17)--(2.18), (2.22)].

(2) \emph{Upgrade to the class of \hyperref[def:2.4a]{Definition \ref*{def:2.4a}} (one-dimensional lemma).} Let $|m|\equiv1$ (for the weak solutions of [\hyperref[ref:BGL24]{BGL24}] this is supplied by its Lemma 7.5), $m\in L^\infty(0,T;V)$ and $m\times\partial_x^2m\in L^2(0,T;H)$. Then $m\in L^2(0,T;D(A_1))$.

\end{remark}

\emph{Proof.} Since $\partial_xm\times\partial_xm=0$, one has $\partial_x(m\times\partial_xm)=m\times\partial_x^2m\in L^2_tH$; moreover $m\times\partial_xm\in L^\infty_tH$. Hence $m\times\partial_xm\in L^2(0,T;H^1)$. From $|m|=1$ we get $m\cdot\partial_xm=0$, and therefore $\partial_xm=-m\times(m\times\partial_xm)$, using $a\times(a\times v)=(a\cdot v)a-|a|^2v$. In one dimension $H^1$ is an algebra, so the product of $m\in L^\infty(0,T;H^1)$ with $m\times\partial_xm\in L^2(0,T;H^1)$ lies in $L^2(0,T;H^1)$. Thus $\partial_xm\in L^2(0,T;H^1)$, i.e. $m\in L^2(0,T;D(A_1))$. $\square$

(3) \emph{Strong solution and Borel map.} With (2), the weak solution belongs to the class of \hyperref[def:2.4a]{Definition \ref*{def:2.4a}}. \hyperref[thm:E.4]{Theorem \ref*{thm:E.4}} gives pathwise uniqueness in that class. The infinite-dimensional Yamada--Watanabe theorem (in the form used in [\hyperref[ref:BGJ17]{BGJ17}, Thm. 5.6]) then yields a strong solution $m=\mathcal J_\varepsilon(z,W_\delta)$ with a Borel map independent of the basis, which is the conclusion of \hyperref[thm:E.5]{Theorem \ref*{thm:E.5}}. From there, \hyperref[prop:E.6]{Proposition \ref*{prop:E.6}} and \hyperref[cor:E.7]{Corollary \ref*{cor:E.7}} proceed unchanged.

We keep \hyperref[thm:E.5]{Theorem \ref*{thm:E.5}} as the proof of record. It relies only on the finite-mode statement \hyperref[asm:I3]{(I3)} and is complete as written, whereas route (1) requires the adaptation of a three-dimensional coupled proof. Route (1)--(3) is recorded as independent corroboration that the input is standard.

\section*{References}
\addcontentsline{toc}{section}{References}

\begin{itemize}
\phantomsection\label{ref:BCF15}\item \textbf{[BCF15]} Z. Brzeźniak, S. Cerrai, M. Freidlin, \emph{Quasipotential and exit time for 2D Stochastic Navier--Stokes equations driven by space time white noise}, Probab. Theory Related Fields \textbf{162} (2015), 739--793.
\phantomsection\label{ref:BDM08}\item \textbf{[BDM08]} A. Budhiraja, P. Dupuis, V. Maroulas, \emph{Large deviations for infinite dimensional stochastic dynamical systems}, Ann. Probab. \textbf{36} (2008), 1390--1420.
\phantomsection\label{ref:BDS19}\item \textbf{[BDS19]} A. Budhiraja, P. Dupuis, M. Salins, \emph{Uniform large deviation principles for Banach space valued stochastic evolution equations}, Trans. Amer. Math. Soc. \textbf{372} (2019), 8363--8421; arXiv:1803.00648 (where the title reads ``\dots{}stochastic differential equations'').
\phantomsection\label{ref:BGL24}\item \textbf{[BGL24]} Z. Brzeźniak, B. Goldys, L. Li, \emph{3D stochastic Landau--Lifshitz--Gilbert equations coupled with Maxwell's equations with full energy}, J. Differential Equations \textbf{390} (2024), 58--124.
\phantomsection\label{ref:BGJ17}\item \textbf{[BGJ17]} Z. Brzeźniak, B. Goldys, T. Jegaraj, \emph{Large deviations and transitions between equilibria for stochastic Landau--Lifshitz--Gilbert equation}, Arch. Ration. Mech. Anal. \textbf{226} (2017), 497--558 (Erratum: ibid. \textbf{226} (2017), 495--496). arXiv:1202.0370.
\phantomsection\label{ref:BHL19}\item \textbf{[BHL19]} Z. Brzeźniak, E. Hausenblas, L. Li, \emph{Quasipotential for the ferromagnetic wire governed by the 1D Landau--Lifshitz--Gilbert equations}, J. Differential Equations \textbf{267} (2019), 2284--2330.
\phantomsection\label{ref:CF11}\item \textbf{[CF11]} S. Cerrai, M. Freidlin, \emph{Approximation of quasi-potentials and exit problems for multidimensional RDE's with noise}, Trans. Amer. Math. Soc. \textbf{363} (2011), 3853--3892.
\phantomsection\label{ref:DGL24}\item \textbf{[DGL24]} Z. Dong, F. Gu, L. Li, \emph{The concentration of zero-noise limits of invariant measures for stochastic dynamical systems}, Stochastic Process.\ Appl.\ \textbf{173} (2024), 104363.
\phantomsection\label{ref:CH53}\item \textbf{[CH53]} R. Courant, D. Hilbert, \emph{Methods of Mathematical Physics, Vol. I}, Interscience Publishers, 1953.
\phantomsection\label{ref:DPZ92}\item \textbf{[DPZ92]} G. Da Prato, J. Zabczyk, \emph{Stochastic Equations in Infinite Dimensions}, Cambridge University Press, 1992.
\phantomsection\label{ref:DZ98}\item \textbf{[DZ98]} A. Dembo, O. Zeitouni, \emph{Large Deviations Techniques and Applications}, 2nd ed., Springer, 1998.
\phantomsection\label{ref:FW98}\item \textbf{[FW98]} M. Freidlin, A. Wentzell, \emph{Random Perturbations of Dynamical Systems}, 2nd ed., Springer, 1998.
\phantomsection\label{ref:GM23}\item \textbf{[GM23]} S. Gokhale, U. Manna, \emph{Well-posedness and large deviations for the stochastic Landau--Lifshitz--Bloch equation}, arXiv:2302.05178 (2023).
\phantomsection\label{ref:Gus23}\item \textbf{[Gus23]} E. Gussetti, A. Hocquet, \emph{A pathwise stochastic Landau--Lifshitz--Gilbert equation with application to large deviations}, J. Funct. Anal. \textbf{285} (2023), no. 9, 110094.
\phantomsection\label{ref:PR07}\item \textbf{[PR07]} C. Prévôt, M. Röckner, \emph{A Concise Course on Stochastic Partial Differential Equations}, Lecture Notes in Math. \textbf{1905}, Springer, 2007.
\phantomsection\label{ref:Sal19}\item \textbf{[Sal19]} M. Salins, \emph{Equivalences and counterexamples between several definitions of the uniform large deviations principle}, Probab. Surv. \textbf{16} (2019), 99--142.
\phantomsection\label{ref:Sow92}\item \textbf{[Sow92]} R. Sowers, \emph{Large deviations for a reaction--diffusion equation with non-Gaussian perturbations}, Ann. Probab. \textbf{20} (1992), 504--537.
\phantomsection\label{ref:YS24}\item \textbf{[YS24]} X. Yin, G. Shen, \emph{On the small time large deviation principles of 1D stochastic Landau--Lifshitz--Bloch equation}, Statist. Probab. Lett. \textbf{207} (2024), 110023.
\end{itemize}

\emph{End of paper.}

\end{document}